\documentclass[a4paper,11pt]{article}
\usepackage[utf8]{inputenc}

\usepackage[proportional,scaled=1.064]{erewhon}
\usepackage[T1]{fontenc}
\usepackage{geometry}
\selectfont

\usepackage[dvipsinames]{xcolor}
\usepackage[absolute,overlay]{textpos}
\usepackage{graphicx} 

\usepackage{lipsum}
\usepackage{array}
\usepackage{caption,subcaption}
\usepackage{multicol}
\usepackage{afterpage}
\usepackage{setspace}
\usepackage{diagbox}
\usepackage{pgffor}
\usepackage[utf8]{inputenc} % allow utf-8 input
\usepackage[T1]{fontenc}    % use 8-bit T1 fonts
\usepackage{hyperref}       % hyperlinks
\usepackage{url}            % simple URL typesetting
\usepackage{booktabs}       % professional-quality tables
\usepackage{amsfonts}       % blackboard math symbols
\usepackage{nicefrac}       % compact symbols for 1/2, etc.
\usepackage{microtype}      % microtypography
\usepackage{dsfont}            % for indicator function
\usepackage{import}
\usepackage[toc,page]{appendix}
\usepackage{minitoc,etoc}
\usepackage{tcolorbox}

\usepackage{amsmath}
\usepackage{amssymb,amsthm,amsfonts}
\usepackage{graphicx}

\usepackage{mathrsfs}
\usepackage{nicefrac}

\usepackage{tikz}
\allowdisplaybreaks

\definecolor{darkmidnightblue}{HTML}{003366}    
\definecolor{midnightblue}{HTML}{0059b3}
\definecolor{chromered}{HTML}{f14233}

\definecolor{darkpowderblue}{rgb}{0.0, 0.2, 0.6}
\definecolor{dukeblue}{rgb}{0.0, 0.0, 0.61}

\hypersetup{
    colorlinks = true,
    citecolor= darkmidnightblue,
    urlcolor=darkmidnightblue,
    breaklinks=true,
    linkcolor = darkmidnightblue,
    linkbordercolor = {white},
}

\usepackage{cleveref}
\usepackage{autonum}
\usepackage{algorithm}
\usepackage{algpseudocode}

\usepackage[absolute,overlay]{textpos}

\newtheorem{theorem}{Theorem}
\newtheorem{lemma}[theorem]{Lemma}

\newtheorem{proposition}{Proposition}
\theoremstyle{remark}
\newtheorem{assumption}{Assumption}
\newtheorem*{notation}{Notation}
\newtheorem{remark}{Remark}
\newtheorem{definition}[theorem]{Definition}
\newtheorem{example}{Example}

\makeatletter
\DeclareRobustCommand{\varamalg}{%
  \mathbin{\mathpalette\var@malg\perp}%
}

\newcommand\var@malg[2]{%
  \rlap{$\m@th#1#2$}\mkern6mu{#1#2}%
}
\makeatother

\makeatletter
\def\@tvsp{\mathchoice{{}\mkern-4.5mu}{{}\mkern-4.5mu}{{}\mkern-2.5mu}{}}
\def\ltrivert{\left|\@tvsp\left|\@tvsp\left|}
\def\rtrivert{\right|\@tvsp\right|\@tvsp\right|}
\makeatother

\makeatletter
\def\@tvsp{\mathchoice{{}\mkern-4.5mu}{{}\mkern-4.5mu}{{}\mkern-2.5mu}{}}
\def\llangle{\langle\@tvsp\langle}
\def\rrangle{\rangle\@tvsp\rangle}
\makeatother

\newcommand{\esp}{\mathbb{E}}
\newcommand{\prob}{\mathbb{P}}
\newcommand{\probn}{\mathbf{P}}

\newcommand{\re}{\mathbb{R}}

\newcommand{\tr}{\mathsf{tr}}

\newcommand{\KL}{\mathsf{KL}}
\DeclareMathOperator*{\op}{op}

\DeclareMathOperator*{\rank}{rank}

\DeclareMathOperator*{\TP}{RSC}
\DeclareMathOperator*{\IP}{IP}
\DeclareMathOperator*{\PP}{PP}
\DeclareMathOperator*{\ATP}{ARSC}
\DeclareMathOperator*{\MP}{MP}

\def\bfDelta{\boldsymbol{\Delta}}

\def\bfGamma{\boldsymbol{\Gamma}}
\def\bfSigma{\boldsymbol{\Sigma}}
\def\bfTheta{\boldsymbol{\Theta}}
\def\bfXi{\boldsymbol{\Xi}}

\def\bfB{\mathbf{B}}

\def\bfI{\mathbf{I}}

\def\bfP{\mathbf{P}}
\def\bfQ{\mathbf{Q}}

\def\bfV{\mathbf{V}}
\def\bfW{\mathbf{W}}
\def\bfX{\mathbf{X}}
\def\bfY{\mathbf{Y}}

\def\bb{\boldsymbol b}

\def\boldf{\boldsymbol f}

\def\bx{\boldsymbol x}
\def\by{\boldsymbol y}
\def\bu{\boldsymbol u}
\def\bv{\boldsymbol v}

\def\bz{\boldsymbol z}
\def\bf0{\mathbf{0}}

\def\bomega{\boldsymbol\omega}

\def\btheta{\boldsymbol\theta}
\def\bxi{\boldsymbol\xi}

\def\calA{\mathcal A}
\def\calB{\mathcal B}
\def\calC{\mathcal C}

\def\calE{\mathcal E}
\def\calF{\mathcal F}
\def\calG{\mathcal G}
\def\calI{\mathcal I}

\def\calJ{\mathcal J}
\def\calL{\mathcal L}

\def\calN{\mathcal N}
\def\calO{\mathcal O}
\def\calP{\mathcal P}
\def\calR{\mathcal R}
\def\calS{\mathcal S}
\def\calQ{\mathcal Q}
\def\calU{\mathcal U}
\def\calV{\mathcal V}

\def\sfC{\mathsf{C}}

\def\sa{\mathsf{a}}
\def\sb{\mathsf{b}}
\def\sc{\mathsf{c}}
\def\sd{\mathsf{d}}
\def\sf{\mathsf{f}}

\def\frM{\mathfrak{M}}
\def\frS{\mathfrak{S}}

\def\frX{\mathfrak{X}}

\def\mbB{\mathbb{B}}
\def\mbC{\mathbb{C}}
\def\mbS{\mathbb{S}}
\def\mbN{\mathbb{N}}
\def\mbX{\mathbb{X}}

\def\mcP{\mathscr{P}}

\def\mdR{\mathds{R}}

\newcommand{\vertiii}[1]{{\left\vert\kern-0.25ex\left\vert\kern-0.25ex\left\vert #1 
    \right\vert\kern-0.25ex\right\vert\kern-0.25ex\right\vert}}
\newcommand{\verti}[1]{{\left\vert\kern-0.4ex #1 
\kern-0.4ex\right\vert}}
\newcommand{\Lpnorm}[1]{\verti{\,#1\,}_{p}}

\newcommand{\dist}{\mathsf{d}}

\DeclareMathOperator*{\argmin}{argmin}

\DeclareMathOperator*{\eff}{\tiny{eff}}

\title{Outlier-robust additive matrix decomposition}

\author{Philip Thompson}

\begin{document}

\maketitle

\begin{abstract}
We study least-squares trace regression when the parameter is the sum of a $r$-low-rank matrix and a $s$-sparse matrix and a fraction $\epsilon$ of the labels is corrupted. For subgaussian distributions and feature-dependent noise, we highlight three needed design properties, each one derived from a different process inequality: a ``product process inequality'', ``Chevet's inequality'' and a ``multiplier process inequality''. These properties handle, simultaneously, additive decomposition, label contamination and design-noise interaction. They imply the near-optimality of a tractable estimator with respect to the effective dimensions  $d_{\eff,r}$ and $d_{\eff,s}$ of the low-rank and sparse components, 
$\epsilon$ and the failure probability $\delta$. The near-optimal rate is
$
\mathsf{r}(n,d_{\eff,r}) + \mathsf{r}(n,d_{\eff,s}) + \sqrt{(1+\log(1/\delta))/n}
+ \epsilon\log(1/\epsilon), 
$
where $\mathsf{r}(n,d_{\eff,r})+\mathsf{r}(n,d_{\eff,s})$ is the optimal rate in average with no contamination. Our estimator is adaptive to $(s,r,\epsilon,\delta)$ and, for fixed absolute constant $c>0$, it attains the mentioned rate with probability $1-\delta$ uniformly over all 
$\delta\ge\exp(-cn)$. Without matrix decomposition, our analysis also entails optimal bounds for a robust estimator adapted to the noise variance. Our estimators are based on ``sorted'' versions of Huber's loss. We present simulations matching the theory. In particular, it reveals the superiority of ``sorted'' Huber's losses over the classical Huber's loss.
\end{abstract}

\section{Introduction}\label{s:intro}

Outlier-robust estimation has been a topic studied for many decades since the seminal work by Huber \cite{1964huber}. 
%\cite{2004huber}, 
%Tukey 
%\cite{1960tukey}
%\cite{1962tukey} and Hampel 
%\cite{1968hampel}
%\cite{1971hampel}. 
One of the objectives of the field is to device estimators which are less sensitive to outlier contamination. The formalization of outlyingness and the construction of robust estimators matured in several directions. One common assumption is that the adversary can only change a fraction $\epsilon$ of the original sample. %\cite{1985yatracos, 1992donoho:gasko}
For an extensive overview we refer, e.g., to 
\cite{2005hampel:ronchetti:rousseeuw:stahel, 2006maronna:martin:yohai, 2011huber:ronchetti} and references therein. %\cite{2017yu:yao}.

Within a very general framework, the minimax optimality of several robust estimation problems has been recently obtained in a series of elegant works by \cite{2016chen:gao:ren,2018chen:gao:ren,2020gao}. The construction, however, is based on Tukey's depth, a hard computational problem in higher dimensions. A recent trend of research, initiated by   \cite{2016diakonikolas:kane:karmalkar:price:stewart,2016lai:rao:vempala}, has focused on the optimality of robust estimators within computationally tractable algorithms. The \emph{oblivious} model assumes the contamination is independent of the original  sample. In the \emph{adversarial} model, the outliers may depend arbitrarily on the sample. For instance, optimal mean estimators for the adversarial model can now be computed in nearly-linear time \cite{2019cheng:diakonikolas:ge, 2019dong:hopkins:li,  2019depersin:lecue, 2020dalalyan:minasyan}.  We refer to  \cite{2019diakonikolas:kane} for an extensive survey.  

In the realm of robust linear regression, two broad lines of investigations exist: (1) one in which only the response (label) is contaminated and (2) the more general setting in which the covariates (features) are also corrupted  \cite{2019diakonikolas:kamath:kane:li:steinhardt:stewart,   2019diakonikolas:kong:stewart}. Model (1), albeit less general, has been considered in many applications and studied in numerous past and recent works \cite{2008candes:randall, 2013li, 2019suggala:bhatia:ravikumar:jain, 2020gao:lafferty, 2020pesmse:flammarion,2023diakonikolas:karmalkar:park:tzamos}. It has also some connection with the problems of robust matrix completion \cite{2011candes:li:ma:wright, 2011chen:xu:caramanis:sanghavi, 2013chen:jalali:sanghavi:caramanis, 2013li,  2017klopp:lounici:tsybakov} and matrix decomposition \cite{2011chandrasekaran:sanghavi:parrilo:willsky, 2011candes:li:ma:wright, 2011xu:caramanis:sanghavi, 2011hsu:kakade:zhang,  2012agarwal:negahban:wainwright}.          
Both models (1)-(2) have been considered assuming adversarial or oblivious contamination. For instance, an interesting property of model (1) with oblivious contamination is the existence of consistent estimators, a property not shared by the adversary model. See for instance \cite{2014tsakonas:jalden:sidiropoulos:ottersten, 2017consistent:robust:regression, 2019suggala:bhatia:ravikumar:jain, 2020gao:lafferty, 2020pesmse:flammarion}.

In this work, we consider a new model: robust trace regression with additive matrix decomposition (RTRMD). It corresponds to least-squares trace regression when the parameter is the sum of a low-rank matrix and a sparse matrix and, simultaneously, the labels are adversarially contaminated. We assume there are at most $o$ outliers and the sample size $n$ is much smaller than the extrinsic dimension $p$. We focus on subgaussian distributions and pay attention to the following points: 
\begin{itemize}
\item[\rm{(a)}] \emph{Adversarial label corruption in high dimensions.} The parameter is a $d_1\times d_2$ matrix and $n\ll p:=d_1d_2$. RTRMD includes, as particular cases, $s$-sparse linear regression \cite{1996tibshirani} and noisy $r$-low-rank matrix sensing \cite{2011negahban:wainwright}. One practical appeal of the established theory of high-dimensional estimation is the existence of efficient estimators adapted to 
$(s,r)$ --- without resorting to Lepski's method. Likewise, we assume no knowledge of $(s,r,o)$.
\item[\rm{(b)}] \emph{Noise heterogeneity.} The majority of the literature on (robust) least-squares regression, within the framework of $M$-estimation with decomposable regularizers \cite{2012negahban:ravikumar:wainwright:yu},  assumes feature-independent noise. We avoid this assumption and identify design properties and concentration inequalities needed for this case. See Section \ref{s:MP:in:M-estimation} for a discussion. 
\item[\rm{(c)}] \emph{Subgaussian rates and uniform confidence level.}  Minimax rates are defined on average or as a function of the failure probability $\delta$. For instance, the first seminal bounds in sparse linear regression were of the form $\sqrt{s\log(p/s\delta)/n}$ --- optimal in average but suboptimal in $\delta$. The optimal rate is the ``subgaussian'' rate $\sqrt{s\log (p/s)/n}+\sqrt{\log(1/\delta)/n}$ --- for which $\log(1/\delta)$ does not multiply the ``effective dimension''.\footnote{The decoupling of $\log(1/\delta)$ with the dimension is of major concern in the recent literature of heavy-tailed estimation \cite{2019lugosi:mendelson-survey}. There are significant additional challenges. For instance, the optimal estimator depends on $\delta$ \cite{2016devroye:lerasle:lugosi:oliveira}.} A second point is to what extent the estimator \emph{depends on $\delta$}. Let $c>0$ be an absolute constant.  
Without knowing $\delta$, is there  an estimator that ``automatically'' attains  the optimal rate \emph{across all} $\delta\ge\exp(-cn)$ with probability at least $1-\delta$? For sparse linear regression,  \cite{2018bellec:lecue:tsybakov} was the first  to answer these points affirmatively when the \emph{noise is independent of the features}. We ask the same questions for the general model RTRMD in case the noise is feature-dependent. See Section \ref{s:MP:in:M-estimation} for a discussion. 
\item[\rm{(d)}] \emph{Matrix decomposition.} It is considered in \cite{2009wright:ganesh:rao:peng:ma, 2011chandrasekaran:sanghavi:parrilo:willsky, 2011candes:li:ma:wright, 2011xu:caramanis:sanghavi, 2011hsu:kakade:zhang, 2011mccoy:tropp} assuming the ``incoherence'' condition and in \cite{2012agarwal:negahban:wainwright} assuming the milder ``low-spikeness'' condition. See also Chapter 7 of the book \cite{2015hastie:tibshirani:wainwright}. These works do not consider label contamination. A general framework is proposed in \cite{2012agarwal:negahban:wainwright} assuming a specific design property (see their Definition 2). In the applications considered in \cite{2012agarwal:negahban:wainwright}, this property is straightforwardly satisfied: either the design is the identity or the fixed design is invertible. For instance,  multi-task learning has an invertible design in high dimensions (one has $n\ge d_1$ albeit $n\ll d_1d_2$). In this regard, trace regression with additive matrix decomposition is fundamentally a different model: with high probability, the random design is singular in the regime $n\ll d_1d_2$. To our knowledge, there is currently no optimal statistical theory for RTRMD --- even without label contamination. Under assumptions (a)-(c), this work identifies three design properties to prove optimality for this problem: 
$\PP$, $\IP$ and $\MP$. Respectively, each one is derived from a different process inequality: a ``product process inequality'', ``Chevet's inequality'' and a ``multiplier process inequality''. See Sections  \ref{ss:related:work:trace:regression:decomposition} and \ref{s:subgaussian:properties}. 
\end{itemize}

The rest of the paper is organized as follows. We start presenting some useful notation. Section \ref{s:framework} states our general framework, Section \ref{s:estimators} presents our estimators and Section \ref{s:motivating:examples} exemplifies with concrete models. Section \ref{s:main:result:1} state new results for these models. It also presents reference to later sections concerning technical contributions. In Section \ref{s:related:work}, we review related literature and compare with our results. In Section \ref{s:multiplier:process:main}, we state two concentration inequalities for the multiplier and product processes. In Section \ref{s:subgaussian:properties}, we state our needed design properties and apply these inequalities to prove them. Section \ref{s:MP:in:M-estimation} presents a preliminary discussion. Sections \ref{s:proof:main:paper}-\ref{s:proof:main:paper:q=1} presents the proofs of our main results, namely, Theorems \ref{thm:improved:rate} and \ref{thm:improved:rate:q=1'}. Finally, Sections \ref{s:simulation} and \ref{s:discussion} finish with simulations and a final discussion. Additional proofs are presented in the Supplemental Material.

\begin{notation}
We set $\mdR^p:=\re^{d_1\times d_2}$. For a vector $\bv$, $\Vert\bv\Vert_k$ denotes its $\ell_k$-norm ($1\le k\le\infty$) and 
$\Vert\bv\Vert_0$ is the number of its nonzero entries. We use similar notation for matrices (considered as vectors). We denote the  Frobenius norm by $\Vert\cdot\Vert_F$, the nuclear norm by $\Vert\cdot\Vert_N$ and the operator norm by $\Vert\cdot\Vert_{\op}$. Given norm $\calR$ and
$\bfV\in\mdR^{p}\setminus\{0\}$,
$
\Psi_{\calR}(\bfV):=\nicefrac{\calR(\bfV)}{\Vert\bfV\Vert_F}.
$  
The inner product in $\mdR^p$ will be denoted by $\llangle\bfV,\bfW\rrangle=\tr(\bfV^\top\bfW)$. For vectors, we use the notation 
$\langle\cdot,\cdot\rangle$. If $a\le Cb$ for some absolute constant $C>0$, we write $a\lesssim b$ or $a\le\calO(b)$. We write $a\asymp b$ if $a\lesssim b$ and $b\lesssim a$. Given $\ell\in\mathbb{N}$, 
$[\ell]:=\{1,\ldots, \ell\}$. The
$\psi_2$-Orlicz norm will be denoted by $|\cdot|_{\psi_2}$. Throughout the paper, given $\ell\in\mathbb{N}$, $A^{(\ell)}:=A/\sqrt{\ell}$ whenever $A$ is a number, vector or function. 

We now recall the definition of the Slope norm \cite{2015bogdan:berg:sabatti:su:candes}. Given nonincreasing positive sequence $\bomega:=\{\omega_i\}_{i\in[n]}$, the Slope norm at a point 
$\bu\in\re^n$ is defined by
$
\Vert\bu\Vert_\sharp:=\sum_{i\in[n]}\omega_i\bu_i^\sharp,
$
where $\bu_1^\sharp\ge\ldots\ge\bu_n^\sharp$ denotes the nonincreasing rearrangement of the absolute coordinates of $\bu$. Throughout this paper, unless otherwise stated, $\bomega\in\re^n$ will be the sequence with coordinates $\omega_i=\sqrt{\log(An/i)}$ for some $A\ge2$. Recall that $\Omega:=\{\sum_{i=1}^o\omega_i^2\}^{1/2}\asymp o\log(n/o)$  \cite{2018bellec:lecue:tsybakov}. With some abuse of notation, 
$\Vert\cdot\Vert_\sharp$ will also denote the Slope norm in $\re^p$ with sequence $\bar w_j=\sqrt{\log(\bar A p/j)}$ for some $\bar A\ge2$.

$\bfX\in\mdR^p$ will denote a random matrix with distribution $\Pi$ and covariance operator $\frS$ --- seen as a vector, $\bfSigma$ is its covariance matrix. Given $[\bfV,\bfW,\bu]\in(\mdR^p)^2\times\re^n$, we define the bilinear form
$
\llangle\bfV,\bfW\rrangle_\Pi:=\esp[\llangle\bfX,\bfV\rrangle\llangle\bfX,\bfW\rrangle]=\llangle\frS(\bfV),\bfW\rrangle
$
and the pseudo-norms
$
\Vert\bfV\Vert_\Pi:=\llangle\bfV,\bfV\rrangle_\Pi^{1/2}
$, 
$\Vert[\bfV,\bfW,\bu]\Vert_\Pi:=\{\Vert\bfV\Vert_\Pi^2+\Vert\bfW\Vert_\Pi^2+\Vert\bu\Vert_2^2\}^{1/2}$, 
$\Vert[\bfV,\bfW]\Vert_\Pi:=\Vert[\bfV,\bfW,\bf0]\Vert_\Pi$
and 
$\Vert[\bfV,\bu]\Vert_\Pi:=\Vert[\bfV,\bf0,\bu]\Vert_\Pi$.
Given $\mbC\subset\mdR^p$, let
$
\mu(\mbC):=\sup_{\bfV\in\mbC}\nicefrac{\Vert\bfV\Vert_F}{\Vert\bfV\Vert_\Pi}.
$
Next, we define the unit balls
$
\mbB_\Pi:=\{\bfV\in\mdR^p:\Vert\bfV\Vert_\Pi\le1\},
$
$
\mbB_F:=\{\bfV\in\mdR^p:\Vert\bfV\Vert_F\le1\},
$
$
\mbB_\ell^k:=\{\bv\in\re^k:\Vert\bv\Vert_\ell\le1\}
$ 
and, for given norm $\calR$ in $\re^k$,
$
\mbB_\calR:=\{\bv:\calR(\bv)\le1\}
$. 
All the corresponding unit spheres will take the symbol $\mbS$. 
Finally, the \emph{Gaussian width} of a compact set 
$\calB\subset\re^{k\times \ell}$ is the quantity
$
\mathscr G(\calB):=\esp[\sup_{\bfV\in\calB}\llangle\bfV,\bfXi\rrangle],
$
where $\bfXi\in\re^{k\times \ell}$ is random matrix with iid $\calN(0,1)$ entries. 
\end{notation}

\section{General framework}\label{s:framework}
\begin{assumption}[Adversarial label contamination]\label{assump:label:contamination}
Let $\{(y_i^\circ,\bfX_i^\circ)\}_{i\in[n]}$ be an iid copy of a feature-label pair $(\bfX,y)\in\mdR^p\times\re$. We assume available a sample
$\{(y_i,\bfX_i)\}_{i\in[n]}$ such that 
$\bfX_i=\bfX_i^\circ$ for all $i\in[n]$ and at most $o$ arbitrary outliers modify the label sample $\{y_i^\circ\}_{i\in[n]}$.
\end{assumption}

Letting $\by=(y_i)_{i\in[n]}$, our underlying model is
\begin{align}
\by=\boldf+\sqrt{n}\btheta^*+\bxi, \label{equation:structural:equation}
\end{align} 
where $\boldf:=(f_i)_{i\in[n]}$ and 
$\bxi:=(\xi_i)_{i\in[n]}$ are, respectively, iid copies of unknown random variables $f,\xi\in\re$ and $\btheta^*\in\re^{n}$ is an arbitrary unknown vector with at most $o$ nonzero coordinates.\footnote{Nothing more is assumed on $\btheta^*$. For instance, it can depend arbitrarily on the data set. We are not concerned with $\btheta^*$ itself and rather see it as a nuisance parameter.}
We assume $\xi$ and $\bfX$ are centered and 
$\esp[\xi\bfX]=0$. The number 
$\epsilon:=o/n$ is referred as the ``contamination fraction''. 

Define the \emph{design operator} 
$\frX:\mdR^p\rightarrow\re^n$ with components 
$\frX_i(\bfV):=\llangle\bfX_i,\bfV\rrangle$. In its general form, our goal is to estimate 
$\boldf$ assuming that, for some $[\bfB,\bfGamma]\in\mdR^p$, the average approximation error 
$
(\nicefrac{1}{n})\Vert\frX(\bfB+\bfGamma)-\boldf\Vert_2^2
$
is ``small''. To be precise, assuming $\epsilon\le c$ for some constant $c\in(0,1/2)$, we would like to design an estimator $[\hat\bfB,\hat\bfGamma]$ satisfying, with probability at least $1-\delta$, ``oracle inequalities'' of the form 
\begin{align}
\Vert\frX^{(n)}(\hat\bfB+\hat\bfGamma)-\boldf^{(n)}\Vert_2^2
\le \inf_{[\bfB,\bfGamma]\in\calF}\left\{
C\Vert\frX^{(n)}(\bfB+\bfGamma)-\boldf^{(n)}\Vert_2^2
+ r_{\bfB,\bfGamma}(n,d_{\eff},\epsilon,\delta)
\right\}.\label{oracle:inequality:eq}
\end{align}
In above, $C\approx1$ is an universal constant, $\calF$ is a class of parameters associated to well known parsimonious properties --- e.g, sparsity or low-rankness --- and $r_{\bfB,\bfGamma}(n,d_{\eff},\epsilon,\delta)$ is an appropriate rate that depends on the ``effective dimension'' $d_{\eff}$ of $\calF$. In Section \ref{s:main:result:1}, we specify different classes of interest, each associated to the examples of Section \ref{s:motivating:examples}.

When the model is ``well-specified'', there is $[\bfB^*,\bfGamma^*]\in\calF$ such that 
\begin{align}
[\bfB^*,\bfGamma^*]\in\argmin_{[\bfB,\bfGamma]\in(\mdR^p)^2}\esp\left[y-\llangle\bfX,\bfB+\bfGamma\rrangle\right]^2. \label{equation:least-squares:regression}
\end{align}
This corresponds to \eqref{equation:structural:equation} with $f:=\llangle\bfX,\bfB^*+\bfGamma^*\rrangle$. In this case, 
$\Vert\frX^{(n)}(\bfB^*+\bfGamma^*)-\boldf^{(n)}\Vert_2=0$, Assumption \ref{assump:label:contamination} and \eqref{equation:structural:equation} are equivalent and the rate $r_{\bfB^*,\bfGamma^*}(n,d_{\eff},\epsilon,\delta)$ is the minimax optimal rate for the well-specified model. More broadly, our goal is to obtain, under the assumptions discussed in the introduction, inequalities of the form \eqref{oracle:inequality:eq}, ensuring optimal statistical guarantees for least-squares trace regression in presence of either additive matrix decomposition, label contamination or inexact parsimony. 

 \section{Our estimators}\label{s:estimators}

Our estimators are based on the following class of losses.
\begin{tcolorbox}
\begin{definition}[Sorted Huber-type losses]\label{def:sorted:Huber:loss}
Define the functions 
$
\rho_1(\bu):=\Vert\bu\Vert_2
$
and
$
\rho_2(\bu):=\frac{1}{2}\Vert\bu\Vert_2^2
$
over $\re^n$. For $q\in\{1,2\}$ and $\tau>0$, let
$\rho_{\tau\bomega,q}:\re^n\rightarrow\re_+$ be the infimal convolution of $\rho_q$ and $\tau\Vert\cdot\Vert_\sharp$, i.e., 
$$
\rho_{\tau\bomega,q}(\bu):=\min_{\bz\in\re^n}\rho_q(\bu-\bz)+\tau\Vert\bz\Vert_{\sharp}.
$$
Finally, define the loss
$
\calL_{\tau\bomega,q}(\bfB):=\rho_{\tau\bomega,q}\left(\nicefrac{\by-\frX(\bfB)}{\sqrt{n}}\right).
$
\end{definition}
\end{tcolorbox}
These losses are convex. For instance, when $q=2$, $\rho_{\tau\bomega,2}$ is the optimal value of the problem defining the proximal mapping of $\tau\Vert\cdot\Vert_\sharp$. When 
$\omega_1=\ldots=\omega_n=1$, separability implies the explicit expression:
\begin{align}
\calL_{\tau\bomega,2}(\bfB)=\tau^2\sum_{i=1}^n\Phi\left(\frac{y_i-\frX_i(\bfB)}{\tau\sqrt{n}}\right),
\end{align}
where $\Phi:\re\rightarrow\re$ is the Huber's function $\Phi(t)=\min\{(\nicefrac{1}{2})t^2,|t|-\nicefrac{1}{2}\}$. Thus, Huber regression corresponds to $M$-estimation with the loss $\calL_{\tau\bomega,2}$ with \emph{constant} weighting sequence $\bomega$. In this work we advocate the use of the loss  
$\calL_{\tau\bomega,2}$ with \emph{varying weights}. Throughout this paper, we fix the sequence 
$\bomega:=(\omega_i)_{i\in[n]}$ to be 
$
\omega_i:=\sqrt{\log(A n/i)}
$
for some $A\ge2$. It corresponds to a ``Sorted'' generalization of Huber's loss.

In high-dimensions, we additionally use regularization norms $(\calR,\calS)$. We consider the estimator
\begin{align}\label{equation:sorted:Huber:general}
\begin{array}{ccl}
[\hat\bfB,\hat\bfGamma]&\in
&\argmin_{[\bfB,\bfGamma]\in(\mdR^p)^2}\calL_{\tau\bomega,2}(\bfB+\bfGamma)+\lambda\calR(\bfB)+\chi\calS(\bfGamma)\\
&&\mbox{s.t.}\quad	 \Vert\bfB\Vert_\infty\le\sa,
\end{array}
\end{align}
where $\lambda,\chi,\tau>0$ and $\sa\in(0,\infty]$ are tuning parameters. If we are not concerned with matrix decomposition, we instead consider, for $q\in\{1,2\}$, estimators of the form 
\begin{align}\label{equation:sorted:Huber}
\begin{array}{ccl}
\hat\bfB&\in
&\argmin_{\bfB\in\mdR^p}\calL_{\tau\bomega,q}(\bfB)+\lambda\calR(\bfB).
\end{array}
\end{align} 

By adding an extra variable $\hat\btheta$ --- aiming in estimating the nuisance parameter $\btheta^*$ --- the solution  
of \eqref{equation:sorted:Huber:general} can be equivalently found solving
\begin{align}
\begin{array}{ccl}
&&\min_{[\bfB,\bfGamma,\btheta]\in(\mdR^p)^2\times\re^n}\frac{1}{2n}\sum_{i=1}^n\left(y_i-\llangle\bfX_i,\bfB+\bfGamma\rrangle+\sqrt{n}\btheta_i\right)^2+\lambda\calR(\bfB)+\chi\calS(\bfGamma)
+\tau\Vert\btheta\Vert_\sharp\\
&&\mbox{s.t.}\quad	 \Vert\bfB\Vert_\infty\le \sa.
\label{equation:aug:slope:rob:estimator:general}
\end{array}
\end{align}
Estimator \eqref{equation:aug:slope:rob:estimator:general} is a concrete example of  regularization with three norms.

Similarly, the solution  of \eqref{equation:sorted:Huber} with $q=2$ can be found solving 
\begin{align}
\begin{array}{ccl}
&&\min_{[\bfB,\btheta]\in\mdR^p\times\re^n}\frac{1}{2n}\sum_{i=1}^n\left(y_i-\llangle\bfX_i,\bfB\rrangle+\sqrt{n}\btheta_i\right)^2
+\lambda\calR(\bfB)+\tau\Vert\btheta\Vert_\sharp.
\label{equation:aug:slope:rob:estimator:q=2}
\end{array}
\end{align}
The practical appeal of these problems is that they can be computed by alternated convex optimization using standard Lasso, Slope or nuclear norm solvers  \cite{2015bogdan:berg:sabatti:su:candes}.

Finally, the solution of \eqref{equation:sorted:Huber} with $q=1$ can be found solving
\begin{align}
\begin{array}{ccl}
&&\min_{[\bfB,\btheta]\in\mdR^p\times\re^n}\left\{\frac{1}{n}\sum_{i=1}^n\left(y_i-\llangle\bfX_i,\bfB\rrangle+\sqrt{n}\btheta_i\right)^2\right\}^{\frac{1}{2}}+\lambda\calR(\bfB)+\tau\Vert\btheta\Vert_\sharp.
\label{equation:aug:slope:rob:estimator:q=1}
\end{array}
\end{align}
When $\btheta\equiv\bf0$ and $\calR=\Vert\cdot\Vert_1$, the above estimator corresponds to the square-root lasso estimator \cite{2011belloni:chernozhukov:wang}. When not concerned with matrix decomposition, we will show that the robust estimator  \eqref{equation:aug:slope:rob:estimator:q=1} achieves the same guarantees as \eqref{equation:aug:slope:rob:estimator:q=2} --- under the same set of assumptions (a)-(c) in Section \ref{s:intro} --- while being adaptive to the variance of the noise. Of course, the computational cost of \eqref{equation:aug:slope:rob:estimator:q=1} is higher. 

\begin{remark}[Huber loss versus ``Sorted'' Huber loss]
Let us illustrate considering the well-specified model with 
$\bfGamma^*\equiv\bf0$ and $q=2$. It is well known that, in linear regression, Huber's estimator can be cast as a least-squares estimator in the augmented variable 
$[\bb,\btheta]$ with the penalization $\Vert\btheta\Vert_1$ \cite{2011she:owen, 2016donoho:montanari}.
In this work, we penalize $\Vert\btheta\Vert_\sharp$ --- i.e., we fit with the loss in Definition \ref{def:sorted:Huber:loss}. Figure \ref{fig.robust.linear.reg:HS} plots the estimation error as a function of $\epsilon$ in sparse linear regression using synthetic contaminated data. We refer to Section \ref{s:simulation} for details. We can see that the ``Sorted'' Huber loss significantly outperforms Huber regression. In this work, we present near-optimal bounds trying to explain this significant empirical observation. Roughly, the intuition is that the former loss assigns more weight to outliers with larger magnitude. Huber regression processes all label data points indistinguishably.
\begin{figure}
\hspace*{\fill}%
\includegraphics[scale=0.27]{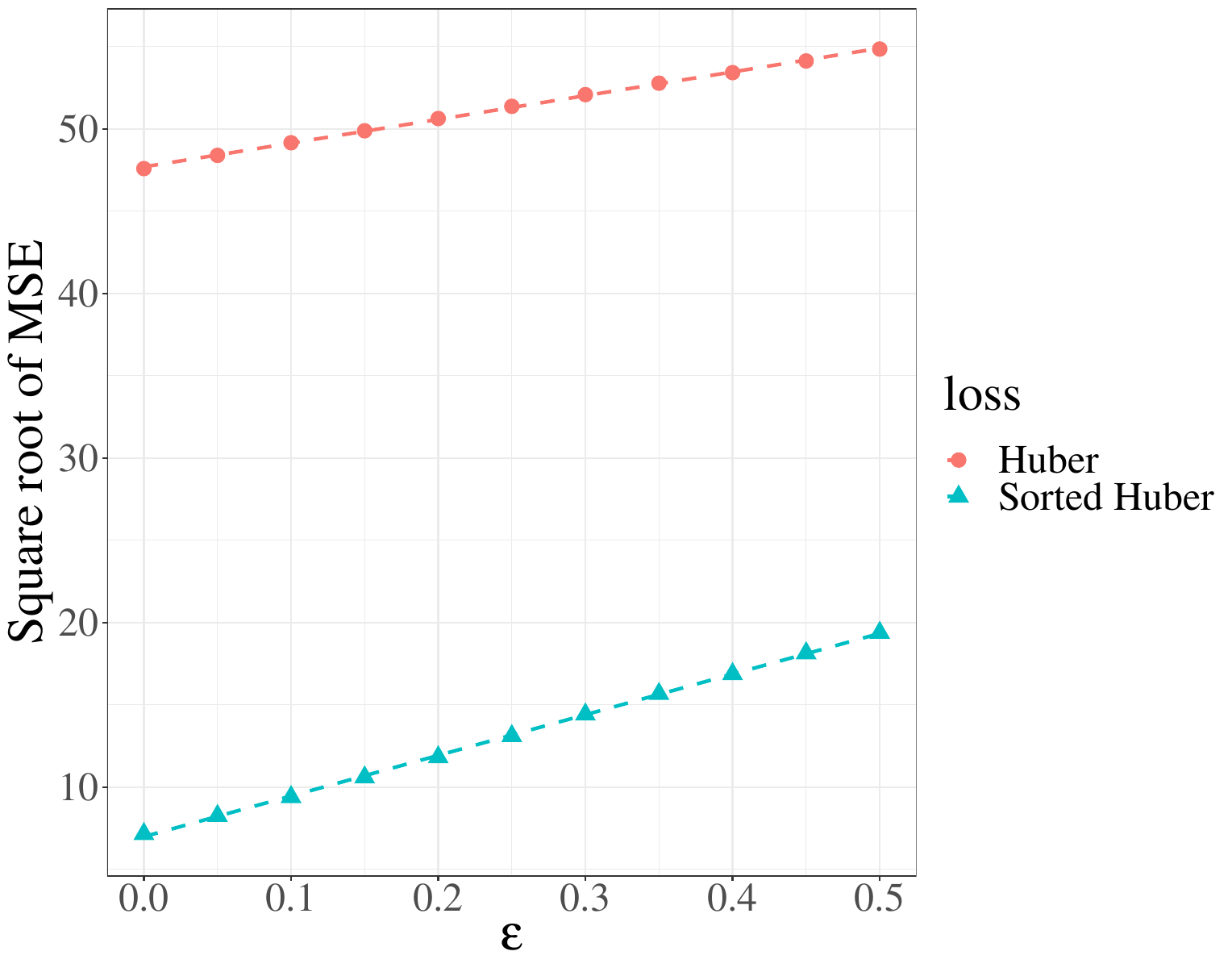}
\hspace*{\fill}%
\caption{Huber vs ``Sorted'' Huber losses in sparse regression: 
$\sqrt{\texttt{MSE}}$ versus $\epsilon$.}\label{fig.robust.linear.reg:HS}
\end{figure}
\end{remark}

\section{Motivating examples}\label{s:motivating:examples}

Matrix decomposition is motivated by several applications. We refer to \cite{2012agarwal:negahban:wainwright} and Chapter 7 of the book \cite{2015hastie:tibshirani:wainwright} for a precise discussion and further references. They consider the general framework: to estimate the pair 
$[\bfB^*,\bfGamma^*]\in(\mdR^p)^2$ given a noisy linear observation of its sum. Precisely, their model is
$
\bfY=\frX(\bfB^*+\bfGamma^*)+\bfXi
$
where the design $\frX:\mdR^p\rightarrow\re^{n\times m}$ takes values on matrices with $n$ iid rows and $\bfXi\in\re^{n\times m}$ is a noise matrix  independent of $\frX$ with $n$ centered iid rows. It is further required that $\bfB^*$ is low-rank, 
$\bfGamma^*$ is a sparse matrix and that a ``low-spikeness'' conditions holds. Three subproblems are analyzed in the framework of \cite{2012agarwal:negahban:wainwright}: factor analysis, robust covariance estimation and multi-task learning. The first two problems have identity designs ($\frX=I$). In multi-task learning, the design is invertible with high probability.
Trace regression with additive matrix decomposition corresponds to their model with $m=1$ and design 
$\frX_i(\bfB):=\llangle\bfX_i,\bfB\rrangle$. Alternatively, it corresponds to \eqref{equation:structural:equation} assuming the well-specified case, 
$\rank(\bfB^*)\le r$ 
and $\Vert\bfGamma^*\Vert_0\le s$ and 
$\btheta^*\equiv\mathbf{0}$. As discussed in item (d) of the introduction, the design property in \cite{2012agarwal:negahban:wainwright} is not guaranteed to hold for this problem. More broadly, we consider additive matrix decomposition in trace regression when labels are contaminated. 

To finish, we require the so called ``low-spikeness'' assumption: there exists $\sa^*>0$ such that for any potential parameter 
$\bfB$, 
$
\Vert\bfB\Vert_\infty\le\nicefrac{\sa^*}{\sqrt{n}}. 
$ 
For this problem, we consider estimator \eqref{equation:sorted:Huber:general} with tuning 
$\sa:=\sa^*/\sqrt{n}$, $\calR:=\Vert\cdot\Vert_N$ and $\calS:=\Vert\cdot\Vert_1$. Two well known particular submodels of RTRMD are:
\begin{itemize}
\item[a)] \emph{Sparse regression}  \cite{2009geer:buhlmann, BRT} is a particular case of model \eqref{equation:least-squares:regression} with $d_1=p$, $d_2=1$, 
$\bx:=\bfX\in\re^p$, $\bb^*:=\bfB^*\in\re^p$, 
$\bfGamma^*\equiv\bf0$ and $\Vert\bb^*\Vert_0\le s\ll p$. In case of label contamination ($0\neq\Vert\btheta^*\Vert_0\le o$), we consider estimators \eqref{equation:sorted:Huber} with $q\in\{1,2\}$ and 
$\calR:=\Vert\cdot\Vert_1$ or $\calR=\Vert\cdot\Vert_\sharp$.
\item[b)] \emph{Trace regression} (or matrix sensing) \cite{2010recht:fazel:parrilo,2011recht:xu:hassibi,2011rohde:tsybakov, 2011negahban:wainwright} correspond to model \eqref{equation:least-squares:regression} with 
$\bfB^*\in\mdR^{p}$ having rank $r\ll d_1\wedge d_2$ and 
$\bfGamma^*\equiv\bf0$. With contaminated labels, we consider estimators \eqref{equation:sorted:Huber} with $q\in\{1,2\}$ and 
$\calR:=\Vert\cdot\Vert_N$.
\end{itemize}

\section{Results for RTRMD and robust sparse/low-rank regression}\label{s:main:result:1}

We state in this section optimal guarantees for the estimators \eqref{equation:sorted:Huber:general}-\eqref{equation:sorted:Huber} of models of Section \ref{s:motivating:examples}. These are in fact particular consequences of our main results: Theorems \ref{thm:mult:process}-\ref{thm:product:process}, Proposition \ref{proposition:properties:subgaussian:designs}, Theorem \ref{thm:improved:rate} and Theorem \ref{thm:improved:rate:q=1'}. Being somewhat technical, we state them, respectively, in the later sections: Sections \ref{s:multiplier:process:main}, \ref{s:subgaussian:properties}, \ref{s:proof:main:paper} and \ref{s:proof:main:paper:q=1}. As a prelude to the proofs in Sections \ref{s:proof:main:paper}-\ref{s:proof:main:paper:q=1}, we include Section \ref{s:MP:in:M-estimation}. It discusses the role of the multiplier process inequality  (Theorem \ref{thm:mult:process} in Section \ref{s:multiplier:process:main}) within the framework of $M$-estimation with decomposable regularizers. Theorem \ref{thm:improved:rate}, the main result of the paper,  is a general deterministic result for estimator \eqref{equation:sorted:Huber:general} and RTRMD. It assumes specific design properties presented in Definition \ref{def:design:property} in Section \ref{s:subgaussian:properties}. Proposition \ref{proposition:properties:subgaussian:designs} in this section ensures these properties are satisfied with high probability. Its proof requires concentration inequalities stated in Theorems \ref{thm:mult:process}-\ref{thm:product:process} of Section \ref{s:multiplier:process:main}. Theorem \ref{thm:improved:rate:q=1'} in Section \ref{s:proof:main:paper:q=1} is a general deterministic result for estimator \eqref{equation:sorted:Huber} with $q=1$ and the problem of robust trace regression (with no matrix decomposition). This theorem ensures this estimator is near-optimal adaptively to the noise variance. The proof of Theorem \ref{thm:improved:rate:q=1'} is largely inspired by the proof of Theorem \ref{thm:improved:rate}. These points are explained in detail later. 

Next, we work with distributions satisfying the following assumption.
\begin{assumption}\label{assump:distribution:subgaussian}
$\bfX$ is centered and $L$-subgaussian for some $L\ge1$, that is, 
$
|\llangle\bfX,\bfV\rrangle|_{\psi_2}\le L\Vert\bfV\Vert_\Pi
$
for all $\bfV\in\mdR^{p}$. Additionally, $1\le|\xi|_{\psi_2}=:\sigma<\infty$.
\end{assumption}

Define
$$
\rho_1(\bfSigma):=\max_{j\in[p]}\bfSigma_{jj}^{1/2},\quad\mbox{ and }\quad
\rho_N(\bfSigma):=\sup_{\Vert[\bz,\bv]\Vert_2=1}\{\esp(\bz^\top\bfX\bv)^2\}^{1/2}.
$$
Throughout the paper, we let $\bfDelta^{\hat\btheta}:=\hat\btheta-\btheta^*$, 
$
\bfDelta_{\bfB}:= \widehat\bfB - \bfB
$
and 
$
\bfDelta_{\bfGamma}:= \widehat\bfGamma - \bfGamma,
$
where $[\hat\bfB,\hat\bfGamma,\hat\btheta]$ denote estimators and $\bfB,\bfGamma\in\mdR^p$ are given points. 

Next, we give guarantees for the estimator \eqref{equation:sorted:Huber:general} and RTRMD. Given $r,s\in\mathbb{N}$ and $\sa^*,\sc>0$, we define the class
\begin{align}
\calF(r,s,\sa^*,\sc)&:=\left\{[\bfB,\bfGamma]\in(\mdR^p)^2:
\begin{array}{c}
\rank(\bfB)\le r,\Vert\bfGamma\Vert_0\le s,\\
(\nicefrac{1}{n})\Vert\frX(\bfB+\bfGamma)-\boldf\Vert_2^2\le \sc\sigma^2,\\
\Vert\bfB\Vert_\infty\le\frac{\sa^*}{\sqrt{n}}
\end{array}
\right\}.
\end{align}
Let $\omega(\epsilon) := \epsilon\log(1/\epsilon)$. Given $C_1>0$, define
\begin{align}
r_{n,r,s,\delta}(\sa^*,C_1) := L\frac{1+\sqrt{\log(1/\delta)}}{\sqrt{n}}
+ L^2\left[\sqrt{\frac{r(d_1+d_2)}{n}} + \sqrt{\frac{s\log p}{n}}\right]
+\left(1+\frac{1}{C_1\sigma L}\right)\sa^*\sqrt{\frac{s}{n}}. 
\end{align}

\begin{theorem}[Robust trace regression with additive matrix decomposition]
\label{thm:tr:reg:matrix:decomp}
Grant Assumptions \ref{assump:label:contamination}-\ref{assump:distribution:subgaussian}, model \eqref{equation:structural:equation} and assume $\bfX$ is isotropic. Then there are absolute constants $\sc\in(0,1)$, $c_1\in(0,1/2)$ and $C_0,C_1\ge1$ such that the following holds. Suppose $C_1^2L^4\epsilon\log(1/\epsilon)\le c_1$. Given 
$\sa^*>0$, let $[\hat\bfB,\hat\bfGamma]$ be the solution of \eqref{equation:sorted:Huber:general} with $\calR:=\Vert\cdot\Vert_N$, 
$\calS:=\Vert\cdot\Vert_1$ and with tuning parameters $\sa:=\sa^*/\sqrt{n}$, 
$
\lambda\asymp \sigma L^2\sqrt{\nicefrac{d_1+d_2}{n}},
$
$
\chi\asymp \sigma L^2\sqrt{\nicefrac{\log p}{n}}+\nicefrac{\sa^*}{\sqrt{n}}
$
and 
$\tau\asymp C_1L^2\sigma/\sqrt{n}$. Assume that 
\begin{align}
n&\gtrsim \left[L^4\cdot r(d_1+d_2)\right]\bigvee
\left[\left(L^4\log p
+(\sa^*)^2\right)s\right].
\end{align}

Then, for any $\delta\in(0,1)$ such that
$
\delta\ge \exp\left(-\frac{n}{C_0L^4}\right),
$ 
on an event of probability $\ge1-\delta$, for all 
$[\bfB,\bfGamma]\in\calF(r,s,\sa^*,\sc)$,
\begin{align}
(\nicefrac{\lambda}{2})\Vert\bfDelta_{\bfB}\Vert_N
+ (\nicefrac{\chi}{2})\Vert\bfDelta_{\bfGamma}\Vert_1
+ \Vert\frX^{(n)}(\hat\bfB+\hat\bfGamma)-\boldf^{(n)}\Vert_2^2 &\le \left(1+\frac{\calO(1)}{C_1^2L^2}\right)\Vert\frX^{(n)}(\bfB+\bfGamma)-\boldf^{(n)}\Vert_2^2\\
&+ \calO(\sigma^2)r_{n,r,s,\delta}^2(\sa^*,C_1)+\calO(C_1^2\sigma^2L^6)\omega^2(\epsilon),\label{thm:tr:reg:matrix:decomp:eq1}
\end{align}
and also
\begin{align}
\Vert[\bfDelta_{\bfB},\bfDelta_{\bfGamma}]\Vert_{\Pi}&\le\left(\calO(1)+\frac{\calO(1)}{C_1L}\right)\Vert\frX^{(n)}(\bfB+\bfGamma)-\boldf^{(n)}\Vert_2
+ \calO(\sigma) r_{n,r,s,\delta}(\sa^*,C_1) + \calO(C_1\sigma L^3)\omega(\epsilon).\label{thm:tr:reg:matrix:decomp:eq2}
\end{align}
\end{theorem}

Theorem \ref{thm:tr:reg:matrix:decomp} states oracle inequalities of the form \eqref{oracle:inequality:eq}.\footnote{In case the approximation error is $\calO(\sigma) r_{n,r,s,\delta}(\sa^*,1)$ --- and assuming $L^2\epsilon\log(1/\epsilon)\le c_1$ and $\tau\asymp\sigma L/\sqrt{n}$ --- we can obtain slightly improved bounds for the estimator in Theorem \ref{thm:tr:reg:matrix:decomp}. Namely, the error coefficients $\calO(1)/C_1^2L^2$ in \eqref{thm:tr:reg:matrix:decomp:eq1} and $\calO(1)/C_1L$ in \eqref{thm:tr:reg:matrix:decomp:eq2} are improved to 
$\calO(\sigma^2)r_{n,r,s,\delta}^2(\sa^*,1)$ and 
$\calO(\sigma)r_{n,r,s,\delta}(\sa^*,1)$ respectively. Additionally, the corruption errors $\calO(C_1^2\sigma^2 L^6)\omega^2(\epsilon)$ in \eqref{thm:tr:reg:matrix:decomp:eq1} and $\calO(C_1\sigma L^3)\omega(\epsilon)$ in \eqref{thm:tr:reg:matrix:decomp:eq2} are improved to 
$\calO(\sigma^2 L^4)\omega^2(\epsilon)$ and $\calO(\sigma L^2)\omega(\epsilon)$ respectively.}
The next proposition ensures the rate in Theorem \ref{thm:tr:reg:matrix:decomp} is optimal up to a log factor.\footnote{By the general theory of \cite{2016chen:gao:ren}, the corruption term $\omega(\epsilon)$ is optimal (up to a log term). Thus, it is sufficient to give a lower bound for the non-corrupted model.} Its proof follows from similar arguments in \cite{2012agarwal:negahban:wainwright} for the noisy matrix decomposition problem with identity design. Define the class
\begin{align}
\calA(r,s,\sa^*)=\left\{\bfTheta^*:=[\bfB^*,\bfGamma^*]\in(\mdR^p)^2:\rank(\bfB^*)\le r, \Vert\bfGamma^*\Vert_0\le s, \Vert\bfB^*\Vert_\infty\le\frac{\sa^*}{\sqrt{n}}\right\}. 
\end{align}
For any $\bfTheta^*:=[\bfB^*,\bfGamma^*]\in(\mdR^p)^2$, let $\prob_{\bfTheta^*}$ denote the distribution of the data $\{y_i,\bfX_i\}_{i\in[n]}$ satisfying \eqref{equation:least-squares:regression} with parameters 
$[\bfB^*,\bfGamma^*]$. Finally, for some 
$\sigma>0$, let
\begin{align}
\Psi_n(r,s,\sa^*):=\sigma^2\left\{\frac{r(d_1+d_2)}{n}
+\frac{s}{n}\log\left(\frac{p-s}{s/2}\right)\right\}
+(\sa^*)^2\frac{s}{n}. 
\end{align}

\begin{proposition}\label{prop:lower:bound:tr:matrix:decomp}
Assume that $\{\xi_i\}_{i\in[n]}$ are iid $\calN(0,\sigma^2)$ independent of $\{\bfX_i\}_{i\in[n]}$, $\bfX$ is isotropic and $\Vert\bfB^*\Vert_\infty\le\sa^*/\sqrt{n}$. Assume $d_1,d_2\ge10$, $\sa^*\ge 32\sqrt{\log p}$ and $s<p$. 

Then there exists universal constants $c>0$ and $\beta\in(0,1)$ such that
\begin{align}
\inf_{\hat\bfTheta}\sup_{\bfTheta^*\in\calA(r,s,\sa^*)}\prob_{\bfTheta^*}
\left\{
\Vert[\bfDelta_{\bfB^*},\bfDelta_{\bfGamma^*}]\Vert_{\Pi}
\ge c\Psi_n^{\frac{1}{2}}(r,s,\sa^*)\right\}\ge\beta,
\end{align}
where the infimum is taken over all estimators 
$\hat\bfTheta=[\hat\bfB,\hat\bfGamma]$ constructed from the data $\{y_i,\bfX_i\}_{i\in[n]}$. 
\end{proposition}

Next, we state guarantees for the estimator \eqref{equation:sorted:Huber} with $q=2$ for three different parameter classes associated to the subproblems a)-b) in Section \ref{s:motivating:examples}. Let $\Vert\cdot\Vert$ denote either 
$\Vert\cdot\Vert_0$ or the rank operation 
$\rank(\cdot)$ and $\mbC_{\bfB}\subset\mdR^p$ be a cone parametrized by a point $\bfB\in\mdR^p$. Given $d,d_{\eff}\in\mathbb{N}$ and
$\rho,\sc,\sfC>0$, let
$
\calF(d,d_{\eff},\rho,\sfC):=\{\bfB\in\mdR^p:
\Vert\bfB\Vert\le d,
\sfC L^2\rho\mu\left(\mbC_{\bfB}\right)
\sqrt{\nicefrac{d_{\eff}}{n}}\le1\}
$
and 
\begin{align}
\calF(d,d_{\eff},\rho,\sc,\sfC):=\left\{\bfB\in \calF(d,d_{\eff},\rho,\sfC):
(\nicefrac{1}{n})\Vert\frX(\bfB)-\boldf\Vert_2^2\le \sc\sigma^2
\right\}.\label{def:class1}
\end{align}
Consider three cases:
\begin{itemize}
\item[\rm (i)] In sparse regression, take $\calR:=\Vert\cdot\Vert_1$ and 
$
\lambda\asymp L^2\sigma\rho_1(\bfSigma)\sqrt{\nicefrac{\log p}{n}}.
$
Set $\Vert\cdot\Vert:=\Vert\cdot\Vert_0$, $\rho:=\rho_1(\bfSigma)$, $d:=s$,  $d_{\eff}:=s\log p$, and $\mbC_{\bb}:=\calC_{\bb,\Vert\cdot\Vert_1}(6)$ for given $\bb$ --- see Section \ref{s:proof:main:paper} for the definition of this cone.
\item[\rm (ii)] In sparse regression, take
$\calR:=\Vert\cdot\Vert_\sharp$, the Slope norm in $\re^p$, and 
$
\lambda\asymp L^2\sigma\rho_1(\bfSigma)/\sqrt{n}.
$
Set $\Vert\cdot\Vert:=\Vert\cdot\Vert_0$, $\rho:=\rho_1(\bfSigma)$, $d:=s$,  $d_{\eff}:=s\log (ep/s)$, and the cone
$\mbC_{\bb}:=\overline\calC_s(6)$ for each $\bb$ such that 
$\Vert\bb\Vert_0\le s$ --- see Section 31
%\ref{s:proof:thm:response:sparse-low-rank:regression:ii} in the Supplement for the definition of this cone. 
\item[\rm (iii)] In trace regression, take $\calR:=\Vert\cdot\Vert_N$ and 
$
\lambda\asymp L^2\sigma\rho_N(\bfSigma)\sqrt{\nicefrac{d_1+d_2}{n}}.
$
Set $\Vert\cdot\Vert:=\rank(\cdot)$, $\rho:=\rho_N(\bfSigma)$, $d:=r$,  $d_{\eff}:=r(d_1+d_2)$ and the cone
$\mbC_{\bfB}:=\calC_{\bfB,\Vert\cdot\Vert_N}(6)$ for given $\bfB$ --- see Section \ref{s:proof:main:paper}.  
\end{itemize} 
Let us define 
\begin{align}
r_{n,d_{\eff},\delta}(\rho,\mu) := L\frac{1+\sqrt{\log(1/\delta)}}{\sqrt{n}}
+ L^2\rho\mu\sqrt{\frac{d_{\eff}}{n}}. 
\label{def:r:n}
\end{align}

\begin{theorem}[Robust sparse/low-rank regression]\label{thm:response:sparse-low-rank:regression}
Grant Assumptions \ref{assump:label:contamination}-\ref{assump:distribution:subgaussian} and model \eqref{equation:structural:equation}. Then there are absolute constants $\sc,c_1\in(0,1/2)$ and $\sfC,C_0,C_1\ge1$ such that the following holds. Suppose $C_1^2L^4\epsilon\log(1/\epsilon)\le c_1$ and take $\tau\asymp C_1L^2\sigma/\sqrt{n}$. Let $\hat\bfB$ be the solution of \eqref{equation:sorted:Huber} with $q=2$ --- correspondingly to each tuning in cases (i)-(iii). Consider the three different classes of type \eqref{def:class1} for each of the cases (i)-(iii). 

Then, for any $\delta\in(0,1)$ such that
$
\delta\ge \exp\left(-\frac{n}{C_0L^4}\right),
$
on an event of probability $\ge1-\delta$, for all 
$\bfB\in\calF(d,d_{\eff},\rho,\sc,\sfC)$,
\begin{align}
(\nicefrac{\lambda}{2})\calR(\bfDelta_{\bfB}) + \Vert\frX^{(n)}(\hat\bfB)-\boldf^{(n)}\Vert_2^2 &\le \left(1+\frac{\calO(1)}{C_1^2L^2}\right)\Vert\frX^{(n)}(\bfB)-\boldf^{(n)}\Vert_2^2\\
&+ \calO(\sigma^2)r_{n,d_{\eff},\delta}^2(\rho,\mu(\mbC_{\bfB}))
+\calO(C_1^2\sigma^2L^6) \omega^2(\epsilon),\label{thm:response:sparse-low-rank:regression:eq1}
\end{align}
and also
\begin{align}
\Vert\bfDelta_{\bfB}\Vert_{\Pi}&\le 
\left(\calO(1)+\frac{\calO(1)}{C_1L}\right)\Vert\frX^{(n)}(\bfB)-\boldf^{(n)}\Vert_2
+ \calO(\sigma) r_{n,d_{\eff},\delta}(\rho,\mu(\mbC_{\bfB}))
+ \calO(C_1\sigma L^3) \omega(\epsilon).
\label{thm:response:sparse-low-rank:regression:eq2}
\end{align}
\end{theorem} 

From the quadratic process inequality, we may replace $\rho_1(\bfSigma)$ with 
$\hat\rho_1:=\max_{j\in[p]}\Vert\mbX_{\bullet,j}\Vert_2$ and 
$\rho_N(\bfSigma)$ by its empirical counterpart $\hat\rho_N$. We now present estimation guarantees for the estimator \eqref{equation:sorted:Huber} with $q=1$. Consider three cases:
\begin{itemize}
\item[\rm(i')] Grant case (i) above but with
$
\lambda\asymp L\rho_1(\bfSigma)\sqrt{\nicefrac{\log p}{n}}.
$
\item[\rm(ii')] Grant case (ii) above but with
$
\lambda\asymp L\rho_1(\bfSigma)/\sqrt{n}. 
$
\item[\rm(iii')] Grant case (iii) above but with
$
\lambda\asymp L\rho_N(\bfSigma)\sqrt{\nicefrac{d_1+d_2}{n}}. 
$
\end{itemize}

\begin{theorem}[$\sigma$-adaptive robust sparse/low-rank regression]\label{thm:response:sparse-low-rank:regression:q=1}
Grant Assumptions \ref{assump:label:contamination}-\ref{assump:distribution:subgaussian} and model \eqref{equation:structural:equation}. Then there are absolute constants 
$c_1\in(0,1/2)$ and $\sfC,C_0\ge1$ such that the following holds. Suppose $L^2\epsilon\log(1/\epsilon)\le c_1$ and take $\tau\asymp L/\sqrt{n}$. Let 
$\hat\bfB$ be the solution of \eqref{equation:sorted:Huber} with $q=1$ --- correspondingly to each tuning in cases (i')-(iii'). Consider the three different classes of type \eqref{def:class1} for each of the cases (i')-(iii'). 

Let $\delta\in(0,1)$ such that
$
\delta\ge \exp\left(-\frac{n}{C_0(L^4\vee\sigma^2)}\right).
$ 
Let $\mu_*:=\sup_{\bfB\in\calF(d,d_{\eff},\rho,\sfC)}\mu(\mbC_{\bfB})$ and 
$\sc_{0,n}\asymp r_{n,d_{\eff},\delta}(\rho,\mu_*)$. Then, on an event of probability $\ge1-\delta$, it holds that, for all 
$\bfB\in\calF(d,d_{\eff},\rho,\sc_{0,n}^2,\sfC)$,
\begin{align}
(\nicefrac{\lambda}{2})\calR(\bfDelta_{\bfB}) + \Vert\frX^{(n)}(\hat\bfB)-\boldf^{(n)}\Vert_2^2 &\le \left(1+\calO(\sc_{0,n}^2)\right)\Vert\frX^{(n)}(\bfB)-\boldf^{(n)}\Vert_2^2\\
&+ \calO(\sigma^2)r_{n,d_{\eff},\delta}^2(\rho,\mu_*)
+\calO(\sigma^2L^4)\omega^2(\epsilon),\label{thm:response:sparse-low-rank:regression:q=1':eq1}
\end{align}
and also
\begin{align}
\Vert\bfDelta_{\bfB}\Vert_{\Pi}&\le 
\left(\calO(1)+\calO(\sc_{0,n})\right)\Vert\frX^{(n)}(\bfB)-\boldf^{(n)}\Vert_2
+ \calO(\sigma) r_{n,d_{\eff},\delta}(\rho,\mu_*)
+ \calO(\sigma L^2)\omega(\epsilon).\label{thm:response:sparse-low-rank:regression:q=1':eq2}
\end{align}
\end{theorem}

In all previous theorems, the correspondent estimator is adaptive to $\delta$ and the confidence level is $1-\delta$ across any $\delta\ge\exp(-cn)$ for a fixed constant $c>0$. The estimators are also adaptive to $(s,r,o)$, and, in Theorems \ref{thm:response:sparse-low-rank:regression}-\ref{thm:response:sparse-low-rank:regression:q=1}, adaptive to $\mu(\mbC_{\bfB})$. In case there is no matrix decomposition, estimator \eqref{equation:sorted:Huber} with $q=1$ achieves, up to constants, the same rate of estimator \eqref{equation:sorted:Huber} with $q=2$ with the advantage of being adaptive to $\sigma$. On the other hand, for $q=1$ the approximation must be $\calO(\sigma) r_{n,d_{\eff},\delta}(\rho,\mu_*)$ while for $q=2$ the approximation error is only required to be $\calO(\sigma)$.\footnote{In case the approximation error is $\calO(\sigma) r_{n,d_{\eff},\delta}(\rho,\mu_*)$ and assuming $L^2\epsilon\log(1/\epsilon)\le c_1$, the same bounds  \eqref{thm:response:sparse-low-rank:regression:q=1':eq1}-\eqref{thm:response:sparse-low-rank:regression:q=1':eq2} are valid for the estimator \eqref{equation:sorted:Huber} with $q=2$ --- with the tuning specified in cases (i)-(iii) and $\tau\asymp\sigma L/\sqrt{n}$.} Nothing is assumed beyond marginal subgaussianity of $(\bfX,\xi)$. In particular, the noise can depend arbitrarily on $\bfX$, be asymmetric and have zero mass around the origin. From \cite{2016chen:gao:ren}, the displayed rates are optimal up to the factor $\log(1/\epsilon)$. The approximately linear growth $\epsilon\mapsto\omega(\epsilon)$  --- in case the sorted Huber loss is used --- is confirmed in our numerical experiments (see Section \ref{s:simulation}).

\begin{remark}\label{rem:delta:adaptive:no:corruption}
Within the framework of $M$-estimation with decomposable regularizers, the first obtained oracle inequalities for sparse and trace regression give optimal rates in average \cite{BRT,2012negahban:ravikumar:wainwright:yu}. Still, as discussed in Section \ref{s:MP:in:M-estimation}, the rates in these seminal works are suboptimal in $\delta$. Additionally, their tuning assumes knowledge of $\delta$. \cite{2018bellec:lecue:tsybakov} was the first work to obtain the subgaussian rate with 
$\delta$-adaptive estimators for sparse linear regression. \cite{2018derumigny} later generalized the bounds of \cite{2018bellec:lecue:tsybakov} for the square-root Lasso estimator. Their proof strategy, however, fundamentally assumes the \emph{noise is independent of features} (see Section \ref{s:MP:in:M-estimation}). When $\epsilon=0$, a corollary of Theorems \ref{thm:response:sparse-low-rank:regression}-\ref{thm:response:sparse-low-rank:regression:q=1} is that the same estimators in \cite{2018bellec:lecue:tsybakov,2018derumigny} attain the subgaussian rate adaptively to 
$\delta$ assuming only $(\bfX,\xi)$ are marginally subgaussian. We refer to Section \ref{s:MP:in:M-estimation} and Remark \ref{rem:multiplier:process:mendelson} in Section \ref{s:multiplier:process:main} for an explanation on these technical issues. See also Section  \ref{ss:related:work:trace:regression:decomposition}. To finish, we remark that when $\epsilon=0$ we could prove ``sharp'' oracle inequalities --- that is, with constant $C=1$ in \eqref{oracle:inequality:eq}.
\end{remark}

\section{Related work and contributions}\label{s:related:work}

\subsection{Robust sparse regression}\label{ss:related:work:sparse:regression}
This model has been the subject of numerous works. From a methodological point of view, the $\ell_1$-penalized Huber's estimator has been considered in \cite{2001sardy:tseng:bruce, 2011she:owen, 2012lee:maceachern:jung}. Empirical evaluation for the choice of tuning parameters is comprehensively studied in these papers. For the adversarial model with Gaussian data, fast rates for such estimator have been obtained in \cite{2008candes:randall, 2009laska:davenport:baraniuk, 2009dalalyan:keriven, 2012dalalyan:chen, 2013nguyen:tran}. The average near optimal rate was shown only recently in \cite{2019dalalyan:thompson}. They obtain the rate 
$
\sqrt{s\log(p/\delta)/n}+\epsilon\log(n/\delta), 
$
with a breakdown point $\epsilon\le c/\log n$ for a constant $c>0$. Under different conditions, the same estimator was later shown in \cite{2019chinot} to attain the subgaussian rate with breakdown point $\epsilon\le c$ without the extra factor $\log(1/\epsilon)$ and allowing feature-dependent heavy-tailed noise. These two works are the most closely related to our particular Theorem \ref{thm:response:sparse-low-rank:regression} for robust sparse linear regression.

\begin{remark}[Comparison with \cite{2019chinot}]\label{rem:comparison:chinot}
The result in \cite{2019chinot} is interesting: the subgaussian rate is attainable with the standard Huber loss --- without the extra $\log (1/\epsilon)$. It also allows feature-dependent heavy-tailed noise --- assuming some extra mild conditions. Still, we argue that, in the subgaussian setting, this result is weaker than our Theorem \ref{thm:response:sparse-low-rank:regression} for a sparse parameter. We give two main reasons:
\begin{itemize}
\item[\rm (i)] Even though not explicitly stated, the proof in \cite{2019chinot} is specific to the \emph{oblivious model} --- a much weaker model than the adversarial one considered in this work. If $\calO$ denotes the index set of outliers, they fundamentally use that,\footnote{See equation (42) in page 3595 in \cite{2019chinot}.} for all $[\bb,\bb']$,
\begin{align}\label{rem:comparison:chinot:Hoeff}
\left|
\frac{1}{o}\sum_{i\in\calO}(|\langle\bx_i,\bb-\bb'\rangle|-\esp[|\langle\bx,\bb-\bb'\rangle|])
\right|_{\psi_2}
\lesssim \frac{\Vert\bb-\bb'\Vert_{\Pi}}{\sqrt{o}}.
\end{align}
This follows from Hoeffding's inequality, but only if $\{\bx_i\}_{i\in\calO}$ is iid for \emph{fixed} $\calO$. In the adversarial model, $\calO$ is an arbitrary random variable dependent on the data set. 
\item[\rm (ii)] \cite{2019chinot} attains the optimal rate for $\ell_1$-regularized Huber regression with penalization
\begin{align}\label{equation:tuning:chinot}
\lambda\asymp\sigma\left(\sqrt{\frac{\log p}{n}}\bigvee\mu(\mbC_{\bb^*})\sqrt{\frac{\log(1/\delta)}{s n}} \bigvee \mu(\mbC_{\bb^*})\frac{o}{\sqrt{s} n}\right),
\end{align}
and, in case the noise is  subgaussian, $\tau\asymp\sigma$. Our tuning $(\lambda,\tau)$ follows the very different scaling
$
\lambda\asymp \sigma\sqrt{\log p/n},
$
$ 
\tau\asymp\sigma/\sqrt{n}.
$
One notable difference is that our tuning is adaptive to $(s,o,\mu(\mbC_{\bb^*}),\delta)$, without resorting to Lepski's method. If we focus on  $(s,o,\mu(\mbC_{\bb^*}),\delta)$-adaptive estimators, our guarantees and simulation results are significantly in favor of sorted Huber-type losses instead of the standard Huber loss.
\end{itemize}

We argue that (i) and the different scaling \eqref{equation:tuning:chinot} follows from a different proof method. The proof in \cite{2019chinot} is based on \emph{``localization'' arguments} for regularized empirical risk minimization (ERM) \cite{2014mendelson,2018lecue:mendelson}.\footnote{This approach has a vast history. State-of-the art results were given in the seminal paper \cite{2014mendelson} --- introducing the ``small-ball method'' for ERM with the square loss. With it, proper localized control of the quadratic and multiplier processes entail optimal rates. \cite{2018lecue:mendelson} generalized this method to analyze regularized ERM. A key tool in this work is the so called ``sparsity equation''. This elegant method entails, in particular, optimality of Lasso, Slope and trace regression.} Being more precise, \cite{2019chinot} is able to show that regularized ERM with convex Lipschitz losses  \cite{2019alquier:cottet:lecue,  2020chinot:lecue:lerasle} is robust against contaminated labels\footnote{The works \cite{2019alquier:cottet:lecue,  2020chinot:lecue:lerasle} also use the ``sparsity equation'' but, unlike \cite{2014mendelson,2018lecue:mendelson}, do not use  explicit concentration of the quadratic/multiplier processes. For convex Lipschitz losses satisfying the ``Bernstein condition'', localized concentration of the empirical process suffices.}, assuming the \emph{oblivious} model. In this approach, one uses the fact that the loss based on Huber's function satisfies the so called ``Bernstein's condition'' --- under additional mild noise conditions.\footnote{\cite{2019chinot} uses Theorem 7 in \cite{2020chinot:lecue:lerasle} stating that, under subgaussian designs and noises with positive mass around the origin, the so called ``Bernstein's condition'' is satisfied by most convex Lipschitz losses of interest. This additional noise condition is not a serious restriction in many settings --- it also allows heavy-tailed noise. Still, it is unnecessary in the subgaussian setting --- giving some additional evidence that both proof methods are different.} Our proof method does not follow the ``localization'' literature but rather the literature on \emph{$M$-estimation with decomposable regularizers} \cite{BRT,2012negahban:ravikumar:wainwright:yu, 2018bellec:lecue:tsybakov}. In this approach, we do not use Lipschitz continuity of Huber-type losses. In fact, our analysis uses a loss based on the square cost and defined over an augmented variable --- see \eqref{equation:aug:slope:rob:estimator:general}-\eqref{equation:aug:slope:rob:estimator:q=2}. 
\end{remark}

\begin{remark}[Comparison with \cite{2019dalalyan:thompson}]\label{rem:comparison:dalalyan:thompson}
Granting reasons (i)-(ii) in Remark \ref{rem:comparison:chinot}, \cite{2019dalalyan:thompson} is the closest work to ours --- indeed, they consider the adversarial model and $(s,o,\mu(\mbC_{\bb^*}))$-adaptive estimators. Our most noted improvements in terms of rate guarantees are three-fold. First, we show that sorted Huber-type regression has improved bounds compared to Huber-regression: the corruption error $\epsilon\log n$ and breakdown point $c/\log n$ of the latter is replaced by $\epsilon\log(1/\epsilon)$ and breakdown point $c$. We give numerical evidence of the superiority of sorted Huber regression compared to standard Huber regression (see Figure \ref{fig.robust.linear.reg:HS}). Adaptations of Huber regression have been studied before. Still, they usually involve modifying the scaling of the tuning parameter. To our knowledge, our theoretical and empirical results for sorted Huber-type losses give new insights. Second, unlike the results in \cite{2019dalalyan:thompson}, our rates for sparse regression use $\delta$-adaptive estimators attaining the optimal $\delta$-subgaussian rate under weaker assumptions --- namely, subgaussian feature-dependent noise. See Remark \ref{rem:delta:adaptive:no:corruption},   Section \ref{ss:related:work:trace:regression:decomposition} and pointers therein. Thirdly, we give optimal guarantees for robust sparse regression with $\sigma$-adaptive estimators under the same set of assumptions. In Section \ref{s:proof:main:paper:q=1}, we explain that the proof of Theorem   \ref{thm:response:sparse-low-rank:regression:q=1} is an adaptation of the proof of Theorem \ref{thm:response:sparse-low-rank:regression}.

To finish, we remark that our improvements on \cite{2019dalalyan:thompson} are substantial in terms of proof techniques and structural properties. In fact, our main focus is the broader model RTRMD. See the next Section \ref{ss:related:work:trace:regression:decomposition} and pointers therein. 
\end{remark}

\begin{remark}[Further references in robust sparse regression]
We complement our review mentioning some literature analyzing different contamination models, e.g. dense bounded noise \cite{2010wright:ma, 2013li, 2013nguyen:tran-dense, 2014foygel:mackey, 2018adcock:bao:jakeman:narayan}. This setting is also studied in \cite{2018karmalkar:price} with the LAD-estimator \cite{2007wang:li:jiang}. Alternatively, a refined analysis of iterative thresholding methods were considered in \cite{2015bhatia:jain:kar,   2017consistent:robust:regression, 2019suggala:bhatia:ravikumar:jain,   2019mukhoty:gopakumar:jain:kar}. They obtain sharp breakdown points and consistency bounds for the oblivious model. Works on sparse linear regression with covariate contamination were considered early on by \cite{2013chen:caramanis:mannor} and, more recently, in \cite{2017balakrishnan:du:li:singh}, albeit with worst rates and breakdown points compared to the response contamination model. Works by \cite{2012loh:wainwright, 2012loh:wainwright-AoS} have also studied the optimality of sparse linear regression in models with error-in-variables and missing-data covariates. Although out of scope, we mention for completeness that tractable algorithms for linear regression with covariate contamination have been  intensively investigated in the low-dimensional scaling ($n\ge p$), with initial works by \cite{2019diakonikolas:kamath:kane:li:steinhardt:stewart,     2019diakonikolas:kong:stewart,  2018prasad:suggala:balakrishnan:ravikumar} and more recent ones in \cite{2020depersin, 2020cherapanamjeri:aras:tripuraneni:jordan:flammarion:bartlett, 2020pensia:jog:loh,2022oliveira:rico:thompson}.
\end{remark}

\subsection{Robust trace regression}\label{ss:related:work:trace:regression}
The first bounds on trace regression (with no corruption) were presented, e.g., in \cite{2011negahban:wainwright, 2011rohde:tsybakov, 2012negahban:ravikumar:wainwright:yu,2016xia:koltchinskii} following the framework of $M$-estimation with decomposable regularizers.\footnote{The complementary works \cite{2018lecue:mendelson, 2019alquier:cottet:lecue, 2020chinot:lecue:lerasle} also study trace regression with different techniques. See Remark \ref{rem:comparison:chinot}.} Trace regression with label-feature contamination is studied in detail in \cite{2020gao}. This paper is based on Tukey's depth, a hard computational problem in high dimensions. The recent papers \cite{2021fan:wang:zhu,2023shen:li:jian-feng:xia} focus on models with heavy-tailed noise. The interesting paper \cite{2023shen:li:jian-feng:xia} considers label contamination as well, but follows a methodology based on non-convex optimization, presenting bounds for a gradient-descent method. As such, it is hard to compare their results with ours, as they follow different set of assumptions. For instance, they assume the oblivious model --- a more restrictive model than the adversarial one. Other minor differences include the assumptions that the covariance matrix is invertible, the noise has positive density around the origin\footnote{They use similar assumptions as in \cite{2018elsener:geer,2019alquier:cottet:lecue, 2020chinot:lecue:lerasle,2019chinot}.} and that some conditions are satisfied to ensure good initialization.

Robust trace regression is not considered in \cite{2019chinot,2019dalalyan:thompson}. Still, their methods could be applied to this problem. In that case, the exact same comments in Remarks  \ref{rem:comparison:chinot}-\ref{rem:comparison:dalalyan:thompson} would still apply --- changing $(s,\log p)$ by $(r,d_1+d_2)$. As mentioned in Section \ref{ss:related:work:sparse:regression}, our improvements on guarantees and assumptions are in fact consequences of new proof techniques and structural design properties motivated to study the broader model RTRMD. We discuss this point in the next section. 

\subsection{Robust trace regression with additive matrix decomposition}\label{ss:related:work:trace:regression:decomposition}
The statistical theory for this model is the main concern of this work. In fact, the results in Sections \ref{ss:related:work:sparse:regression}-\ref{ss:related:work:trace:regression} are consequences of the techniques needed to analyze the broader model RTRMD. As mentioned in the introduction, additive matrix decomposition was extensively studied in \cite{2009wright:ganesh:rao:peng:ma, 2011chandrasekaran:sanghavi:parrilo:willsky, 2011candes:li:ma:wright, 2011xu:caramanis:sanghavi, 2011hsu:kakade:zhang, 2011mccoy:tropp, 2012agarwal:negahban:wainwright}. Still, there is currently no optimality theory for additive matrix decomposition in trace regression --- nor its extension with label contamination.
In this preliminary section, we briefly comment on three design properties needed to establish an optimality theory for RTRMD. A detailed discussion is referred to later sections.

When the parameter is the sum of a low-rank and sparse matrices, the random design in trace regression is singular with high probability. We identify a concentration inequality for the product process (see our Theorem \ref{thm:product:process}) as the sufficient property to prove restricted strong convexity for this model. In Definition \ref{def:design:property} in Section \ref{s:subgaussian:properties}, we denote this property by $\PP$. To the best of our knowledge, this is a novel application of product processes in high-dimensional statistics. This technical property is ``sharp'', in the sense that replacing it with other naive methods, e.g. dual-norm inequalities, fail to entail the optimal rate.  

$\PP$ is no longer sufficient in case of label contamination. To handle it, we take inspiration from \cite{2019dalalyan:thompson}. This work identified one design property sufficient to handle label contamination when the parameter is sparse. Termed ``incoherence principle'' ($\IP$), it is  derived from Chevet's inequality. \cite{2019dalalyan:thompson} is not concerned with matrix decomposition, and as such, $\PP$ is unnecessary. We show that a generalized version of $\IP$ (see Definition \ref{def:design:property} in Section \ref{s:subgaussian:properties}) and the new property $\PP$ are \emph{jointly} sufficient properties to ensure restricted strong convexity for RTRMD and to optimaly control the ``design-corruption interaction''. Again, $\PP$ and $\IP$ are ``sharp'': mere use of dual-norm inequalities fail to achieve optimality. See Remarks \ref{rem:design:properties:agarwal:negahban:wainwright}-\ref{rem:design:properties:dalalyan:thompson} in Section \ref{s:subgaussian:properties} and Remarks \ref{rem:relevance:PP+IP}-\ref{rem:comparison:dalalyan:thompson:2} in Section \ref{s:proof:main:paper}. 

The third design property we use enables us to achieve the optimal $\delta$-subgaussian rate with $\delta$-adaptive estimators, even when the noise is feature-dependent. This property, denoted by $\MP$ in Definition \ref{def:design:property} in Section \ref{s:subgaussian:properties}, follows from a concentration inequality for the multiplier process (see Theorem \ref{thm:mult:process} in Section \ref{s:multiplier:process:main}). In the framework of $M$-estimation with decomposable regularizers, $\MP$ is a classical property  used to control the ``design-noise interaction'' \cite{BRT,2012agarwal:negahban:wainwright,2018bellec:lecue:tsybakov,2019dalalyan:thompson}. The typical way to prove it is via the dual-norm inequality \cite{BRT,2012agarwal:negahban:wainwright,2019dalalyan:thompson}. This approach fails to entail the subgaussian rate and $\delta$-adaptivity. \cite{2018bellec:lecue:tsybakov} was the first to succeed on this point, using a suitable version of 
$\MP$  (see Remark  \ref{rem:delta:adaptive:no:corruption}). Still, they assume feature-independent noise --- in case the parameter is sparse  and there is no contamination. Our version of $\MP$ in Definition \ref{def:design:property} in Section \ref{s:subgaussian:properties} is more general than \cite{2018bellec:lecue:tsybakov} so to handle feature-dependent noise, additive matrix decomposition and label contamination.

To finish, as discussed in Section \ref{ss:related:work:sparse:regression},  the relevance of sorted Huber-type losses also applies to RTRMD. Using the standard Huber's loss, we would have rate $\omega(\epsilon)\asymp\epsilon\log n$ instead of $\omega(\epsilon)\asymp\epsilon\log(1/\epsilon)$, breakdown point $\epsilon\le c/\log n$ instead of $c$, for some constant $c\in(0,1/2)$. See also Figure \ref{fig.robust.trace-reg-MD}(b) in Section \ref{s:simulation}.

\section{The multiplier and product processes}
\label{s:multiplier:process:main}
In this section we present concentration inequalities for subgaussian Multiplier and Product processes with optimal dependence on $(d_{\eff},\delta)$. The notation in this section is independent of all previous sections. Throughout this section, $(B,\calB,\probn)$ is a probability space, $(\xi,X)$ is a random (possibly not independent) pair taking values 
on $\re\times B$ and $X$ has marginal distribution $\probn$. $\{(\xi_i, X_i)\}_{i\in[n]}$ will denote an iid copy of $(\xi,X)$ and $\hat\probn$ be denotes the empirical measure associated to $\{X_i\}_{i\in[n]}$. The \emph{multiplier process} over functions $f\in F$ is defined as
$$
M(f):=\frac{1}{n}\sum_{i\in[n]}(\xi_if(X_i)-\esp[\xi f(X)]).
$$ 
For instance, the \emph{empirical process} is a particular case when 
$\xi\equiv1$. The \emph{product process} is defined as
$$
A(f,g):=\frac{1}{n}\sum_{i\in[n]}
\bigg\{f(X_i)g(X_i)-\esp f(X_i)g(X_i) \bigg\},
$$
over two distinct classes $F$ and $G$ of measurable functions. When $F=G$, the correspondent process is often termed the \emph{quadratic process}. 

There is a large literature on concentration of these processes and its use in risk minimization. One pioneering idea is of ``generic chaining'', first developed by Talagrand for the empirical process \cite{2014talagrand}. This method was refined by Dirksen, Bednorz, Mendelson and collaborators, e.g., in \cite{2007mendelson:pajor:tomczak-jaegermann, 2015dirksen, 2014bednorz, 2016mendelson}. The following notion of complexity is used in generic chaining bounds.\footnote{The pioneering work by Talagrand presented the, somewhat mysterious, $\gamma_2$-functional as a measure of complexity of the class. The ``truncated''  
$\gamma_{2,p}$-functional was presented recently by Dirksen \cite{2015dirksen}.}

\begin{definition}[$\gamma_{2,p}$-functional]
Let $(T,\dist)$ be a pseudo-metric space. We say a sequence $(T_k)$ of subsets of $T$ is \emph{admissible} if $|T_0|=1$ and $|T_k|\le 2^{2^{k}}$ for $k\in\mathbb{N}$ and $\cup_{k\ge0} T_k$ is dense in $T$. Let $\calA$ denote the class of all such admissible subset sequences. Given $p\ge1$, the 
$\gamma_{2,p}$-functional with respect to $(T,\dist)$ is the quantity
\begin{align}
\gamma_{2,p}(T):=\inf_{(T_k)\in\calA}\sup_{t\in T}\sum_{k\ge\lfloor\log_2 p\rfloor}2^{k/2}\dist(t,T_k).
\end{align}
We will say that $(T_k)\in\calA$ is optimal if it achieves the infimum above. Set $\gamma_{2}(T):=\gamma_{2,1}(T)$.
\end{definition}
Let $L_{\psi_2}=L_{\psi_2}(\probn)$ be the family of measurable functions $f:B\rightarrow\re$ having finite $\psi_2$-norm 
$$
\|f\|_{\psi_2}:=|f(X)|_{\psi_2}:=\inf\{c>0:\esp[\psi_2(\nicefrac{f(X)}{c})]\le1\}
$$
where $\psi_2(t):=e^{t^2}-1$. We assume that the $\psi_2$-norm of $\xi$, denoted also by 
$\Vert\xi\Vert_{\psi_2}$, is finite. Given $f,g\in L_{\psi_2}$, we define the pseudo-distance
$
\dist(f,g):=\|f-g\|_{\psi_2}.
$
Given a subclass $F\subset L_{\psi_2}$, we let $\Delta(F):=\sup_{f,f'\in F}\dist(f,f')$ and $\bar\Delta(F):=\sup_{f\in F}\dist(f,0)$. We prove the following two results in Sections B 
%\ref{s:multiplier:process} 
and C
%\ref{s:product:process} in the supplement. 

\begin{theorem}[Multiplier process]\label{thm:mult:process}
There exists universal constant $c>0$, such that for all $f_0\in F$, $n\ge1$, $u\ge 1$ and $v\ge1$, with probability at least $1-ce^{-u/4}-ce^{-nv}$,
\begin{align}
\sup_{f\in F}|M(f)-M(f_0)|&\lesssim
\left(\sqrt{v}+1\right)\Vert\xi\Vert_{\psi_2}\frac{\gamma_2(F)}{\sqrt{n}}
+\left(\sqrt{\frac{2u}{n}}+\frac{u}{n}
+\sqrt{\frac{uv}{n}}\right)\Vert\xi\Vert_{\psi_2}\bar\Delta(F). 
\end{align}
\end{theorem}

\begin{theorem}[Product process]\label{thm:product:process}
Let $F,G$ be subclasses of $L_{\psi_2}$. There exist universal constants $c,C>0$, such that for all $n\ge1$ and $u\ge 1$, with probability at least $1-e^{-u}$,
\begin{align}
\sup_{(f,g)\in F\times G}\left|A(f,g)\right|&\le C\left[\frac{\gamma_{2}(F)\gamma_{2}(G)}{n}
+\bar\Delta(F)\frac{\gamma_{2}(G)}{\sqrt{n}}+\bar\Delta(G)\frac{\gamma_{2}(F)}{\sqrt{n}}\right]\\
&+c\sup_{(f,g)\in F\times G}\Vert fg-\probn fg\Vert_{\psi_1}\left(\sqrt{\frac{u}{n}}+\frac{u}{n}\right).
\end{align}
\end{theorem}

\begin{remark}[Confidence level \& complexity]\label{rem:multiplier:process:mendelson}
Mendelson \cite{2016mendelson} established impressive concentration inequalities for the multiplier and product processes. In fact, they hold for much more general $(\xi,X)$  having heavier tails (see Theorems 1.9, 1.13 and 4.4 in \cite{2016mendelson}). When specifying these bounds to  subgaussian classes and noise, however, the confidence parameter $u>0$ multiplies the complexities $\gamma_2(F)$ and 
$\gamma_2(G)$ --- unlike our Theorems \ref{thm:mult:process}-\ref{thm:product:process}.
For a related discussion regarding the empirical and quadratic processes, we refer to Remark 3.3(ii) and observations before Corollary 5.7 in Dirksen's paper \cite{2015dirksen}. This technical point is crucial in our proof to show that our class of estimators attain the $\delta$-subgaussian rate in the high-dimensional regime with $\delta$-adaptive estimators. We refer to Section \ref{s:MP:in:M-estimation} for a discussion on this topic. Note that we can take $v\asymp1$ for failure probability $\delta\ge e^{-c'n}$ for absolute constant $c'>0$. Our proofs are motivated by  Dirksen's method for the quadratic process \cite{2015dirksen} and Talagrand's proof for the empirical process \cite{2014talagrand}.\footnote{They are not corollaries of Dirksen's results. For instance, Theorem \ref{thm:product:process} cannot be derived from the quadratic process inequality and the parallelogram law. Indeed, we fundamentally need $F\neq Q$.}
\end{remark}

\section{Properties for subgaussian distributions}\label{s:subgaussian:properties}
In what follows,
$
\frM(\bfV,\bu):=\frX(\bfV)+\sqrt{n}\bu, 
$
$\calR$ and $\calS$ are norms on $\mdR^p$ and $\calQ$ is a norm on $\re^n$. Throughout this section, $(\bfX,\xi)\in\mdR^p\times\re$ satisfies Assumption \ref{assump:distribution:subgaussian}. Next, we define the empirical bilinear form
$$
\llangle\bfV,\bfW\rrangle_n:=\frac{1}{n}\sum_{i\in[n]}
\llangle\bfX_i,\bfV\rrangle\llangle\bfX_i,\bfW\rrangle = \langle\frX^{(n)}(\bfV),\frX^{(n)}(\bfW)\rangle. 
$$
In what follows, $\{\sa_i,\sb_i,\sc_i,\sd_i,\sf_i\}$ are fixed positive numbers.

\begin{definition}\label{def:design:property}
\begin{itemize}
\quad
\item[\rm (i)] $\frX$ satisfies 
$\TP_{\calR}(\sa_1,\sa_2)$ if for all $\bfV\in\mdR^p$, 
\begin{align}
\big\|\frX^{(n)}(\bfV)\big\|_2\ge \sa_1\Vert\bfV\Vert_\Pi-\sa_2\calR(\bfV).
\end{align}
\item[\rm (ii)] $\frX$ satisfies $\PP_{\calR,\calS}(\sc_1,\sc_2,\sc_3,\sc_4)$ if for all $[\bfV,\bfW]\in(\mdR^p)^2$, 
\begin{align}
\left|\llangle\bfV,\bfW\rrangle_n-\llangle\bfV,\bfW\rrangle_\Pi\right|&\le \sc_1\left\Vert\bfV\right\Vert_{\Pi}
\Vert\bfW\Vert_\Pi+\sc_2\calR(\bfV)\Vert\bfW\Vert_\Pi+\sc_3\left\Vert\bfV\right\Vert_\Pi\calS(\bfW)\\
&\quad+\sc_4\calR(\bfV)\calS(\bfW).
\end{align}
\item[\rm (iii)] $\frX$ satisfies 
$\IP_{\calR,\calS,\calQ}(\sb_1,\sb_2,\sb_3,\sb_4)$ if for all $[\bfV,\bfW,\bu]\in(\mdR^p)^2\times\re^n$,
\begin{align}
|\langle\bu,\frX^{(n)}(\bfV+\bfW)\rangle|&\le \sb_1\left\Vert[\bfV,\bfW]\right\Vert_{\Pi}
\Vert\bu\Vert_2 +\sb_2\calR(\bfV)\Vert\bu\Vert_2 +\sb_3\calS(\bfW)\Vert\bu\Vert_2\\
&+\sb_4\left\Vert[\bfV,\bfW]\right\Vert_\Pi\calQ(\bu).
\end{align}
\item[\rm (iv)] $\frX$ satisfies $\ATP_{\calR,\calS,\calQ}(\sd_1,\sd_2,\sd_3,\sd_4)$ if for all 
$[\bfV,\bfW,\bu]\in(\mdR^p)^2\times\re^n$,
\begin{align}
\left\{
\Vert\frM^{(n)}(\bfV+\bfW,\bu)\Vert_2^2 -2\llangle\bfV,\bfW\rrangle_{\Pi}
\right\}_+^{\frac{1}{2}}&\ge 
\sd_1\Vert[\bfV,\bfW,\bu]\Vert_\Pi-\sd_2\calR(\bfV)-\sd_3\calS(\bfW)-\sd_4\calQ(\bu). 
\end{align}
\item[\rm (v)] $(\frX,\bxi)$ satisfies $\MP_{\calR,\calS,\calQ}(\sf_1,\sf_2,\sf_3,\sf_4)$ if 
for all $[\bfV,\bfW,\bu]\in(\mdR^p)^2\times\re^n$,
\begin{align}
|\langle\bxi^{(n)},\frM^{(n)}(\bfV+\bfW,\bu)\rangle|\le 
\sf_1\Vert[\bfV,\bfW,\bu]\Vert_\Pi
+\sf_2\calR(\bfV)
+\sf_3\calS(\bfW)
+\sf_4\calQ(\bu).
\end{align}
\end{itemize}
\end{definition}

In the next lemmas, we show that $\TP$ and $\ATP$ are consequences of $\PP$ and $\IP$.
\begin{lemma}\label{lemma:TP}
Suppose $\frX$ satisfies $\PP_{\calR,\calR}(\alpha_1,\alpha_2,\alpha_3,\alpha_4)$ with $\alpha_1\in(0,1)$. Then $\TP_{\calR}(\sa_1,\sa_2)$ holds with constants
$
\sa_1:=\sqrt{(3/4)(1-\alpha_1)}
$
and
$
\sa_2:=\left\{
\frac{(\alpha_2+\alpha_3)^2}{(1-\alpha_1)} + \alpha_4
\right\}^{1/2}.
$
\end{lemma}

\begin{lemma}[Lemma 7 in \cite{2019dalalyan:thompson}]\label{lemma:ATP:calS=0}
Suppose $\frX$ satisfies $\TP_{\calR}(\sa_1,\sa_2)$ and $\IP_{\calR,0,\calQ}(\sb_1,\sb_2,0,\sb_4)$. Suppose further that $\sa_1\in(0,1)$ and 
$
\sb_1<3\sa_1^2/4. 
$
Then $\ATP_{\calR,0,\calQ}(\sd_1,\sd_2,0,\sd_4)$ holds with constants 
$\sd_1:=\sqrt{(\nicefrac{3}{4})\sa_1^2 - \sb_1}$, 
$
\sd_2:=\frac{2\sb_2}{\sa_1}+\sa_2, 
$
and $\sd_4:=\frac{2\sb_4}{\sa_1}$.
\end{lemma}

\begin{lemma}\label{lemma:ATP}
Suppose $\frX$ satisfies $\TP_{\calR}(\sa_1,\sa_2)$, $\TP_{\calS}(\bar\sa_1,\bar\sa_2)$, $\IP_{\calR,\calS,\calQ}(\sb_1,\sb_2,\sb_3,\sb_4)$ and 
$\PP_{\calR,\calS}(\sc_1,\sc_2,\sc_3,\sc_4)$ with $\sc_4=\gamma\sc_2\sc_3$ for some $\gamma>0$. Suppose further that $\sa_1,\bar\sa_1\in(0,1)$ and 
$
\sb_1+\sc_1<3(\sa_1\wedge\bar\sa_1)^2/4. 
$
Then $\ATP_{\calR,\calS,\calQ}(\sd_1,\sd_2,\sd_3,\sd_4)$ holds with constants $\sd_1:=\sqrt{(\nicefrac{3}{4})(\sa_1\wedge\bar\sa_1)^2 - (\sb_1+\sc_1)}$ and
\begin{align}
\sd_2:=\left\{\frac{8(\sb_2^2+\sc_2^2)}{(\sa_1\wedge\bar\sa_1)^2}+\gamma\sc_2^2\right\}^{\frac{1}{2}}+\sa_2,\quad
\sd_3:=\left\{\frac{8(\sb_3^2+\sc_3^2)}{(\sa_1\wedge\bar\sa_1)^2}
+\gamma\sc_3^2\right\}^{\frac{1}{2}}+\bar\sa_2,\quad
\sd_4:=\frac{2\sqrt{2}\sb_4}{(\sa_1\wedge\bar\sa_1)}.
\end{align}
\end{lemma}

In the following sections, we will not explicitly use $\PP$. From the previous lemmas, it should be clear that this property is implicitly used when we invoke $\TP$ and $\ATP$. The fundamental properties  we need to show for subgaussian designs and noises are $\PP$, $\IP$ and 
$\MP$. The next proposition ensures these properties hold with high probability.
\begin{proposition}\label{proposition:properties:subgaussian:designs}
Grant Assumption \ref{assump:distribution:subgaussian}. There is universal constant $C>0$ such that the following holds. Let 
$\delta\in(0,1)$. 
\begin{itemize}
\item[\rm(i)] With probability $\ge1-\delta$, $\PP_{\calR,\calS}(\sc_1,\sc_2,\sc_3,\sc_4)$ is satisfied with constants
\begin{align}
\sc_1 &= CL^2\left(
\frac{1+\log(1/\delta)}{n} + \frac{1+\sqrt{\log(1/\delta)}}{\sqrt{n}}
\right),\\
\sc_2 &= CL^2\left(\frac{1}{n}+\frac{1}{\sqrt{n}}\right)\mathscr G\left(\frS^{1/2}(\mbB_\calR)\right),\\
\sc_3 &= CL^2\left(\frac{1}{n}+\frac{1}{\sqrt{n}}\right)\mathscr G\left(\frS^{1/2}(\mbB_\calS)\right),\\
\sc_4 &= \frac{CL^2}{n}
\mathscr G\left(\frS^{1/2}(\mbB_\calR)\right)\cdot
\mathscr G\left(\frS^{1/2}(\mbB_\calS)\right). 
\end{align}

\item[\rm(ii)] Suppose that 
$\sc(\delta):=CL^2(\frac{1+\log(1/\delta)}{n} + \frac{1+\sqrt{\log(1/\delta)}}{\sqrt{n}})<1$. Then, with probability $\ge1-\delta$, 
$\TP_{\calR}(\sa_1,\sa_2)$ holds with constants
$
\sa_1 = \sqrt{(3/4)(1-\sc(\delta))}
$
and 
$$
\sa_2 = \left(\frac{2CL^2}{\sqrt{1-\sc(\delta)}}\left(\frac{1}{n}+\frac{1}{\sqrt{n}}\right) + \frac{CL}{\sqrt{n}}\right)\mathscr G\left(\frS^{1/2}(\mbB_\calR)\right).
$$
\item[\rm(iii)] With probability $\ge1-\delta$, 
$\IP_{\calR,\calS,\calQ}(\sb_1,\sb_2,\sb_3,\sb_4)$ holds with constants
$
\sb_1 = CL\frac{1+\sqrt{\log(1/\delta)}}{\sqrt{n}},
$ 
$
\sb_2 = CL\frac{\mathscr G(\frS^{1/2}(\mbB_\calR))}{\sqrt{n}},
$
$
\sb_3 = CL\frac{\mathscr G(\frS^{1/2}(\mbB_\calS))}{\sqrt{n}}
$
and 
$
\sb_4 = CL\frac{\mathscr G(\mbB_\calQ)}{\sqrt{n}}.
$
\item[\rm(iv)] Define the quantities
\begin{align}
\triangle_{n}(\delta)&:=(\nicefrac{1}{\sqrt{n}})[1 + \sqrt{\log(1/\delta)}]
+(\nicefrac{1}{n})[1 + \log(1/\delta) + \sqrt{\log(1/\delta)}],\\
\lozenge_n(\delta)&:=(\nicefrac{1}{\sqrt{n}})[1+(\nicefrac{1}{\sqrt{n}})\sqrt{\log(1/\delta)}] + (\nicefrac{1}{n}).
\end{align}
With probability $\ge1-\delta$, $\MP_{\calR,\calS,\calQ}(\sf_1,\sf_2,\sf_3,\sf_4)$ holds with constants $\sf_1 := C\sigma L\triangle_{n}(\delta)$, $\sf_2:=C\sigma L\lozenge_n(\delta)\cdot\mathscr
G(\frS^{1/2}(\mbB_\calR))$, $\sf_3:=C\sigma L\lozenge_n(\delta)\cdot\mathscr
G(\frS^{1/2}(\mbB_\calS))$ and $\sf_4:=\frac{\sigma}{\sqrt{n}}\mathscr G\left(\mbB_\calQ\right)$. 
\end{itemize} 
\end{proposition}
\begin{proof}[Proof sketch]
$\PP$ follows from the concentration inequality for the product process given in Theorem \ref{thm:product:process} and a two-parameter peeling argument. $\TP$ follows from $\PP$ and Lemma  \ref{lemma:TP}.\footnote{Alternatively, $\TP$ could be proved from Dirksen-Bednorz inequality for the quadratic process \cite{2015dirksen, 2014bednorz} and a one-parameter peeling lemma.}
To prove $\IP$, we invoke Chevet's inequality for subgaussian processes twice --- for each of the pair of norms $(\calR,\calQ)$ and $(\calS,\calQ)$ --- and a two-parameter peeling lemma.  To prove 
$\MP$, we invoke the multiplier process inequality of Theorem \ref{thm:mult:process} twice --- for each of the norms $(\calR,\calQ)$. We also concentrate the linear process 
$\bu\mapsto\langle\bu,\bxi\rangle$ using a symmetrization-comparison argument with a Gaussian linear process. From these three bounds and a one-parameter peeling lemma, $\MP$ follows. The proofs of these claims are referred to Sections \ref{s:prop:PP}, \ref{s:prop:TP}, \ref{s:prop:IP} and \ref{s:prop:MP} of the supplement.
\end{proof}

$\TP$ is the well known ``restricted strong convexity'' used to analyze regularized $M$-estimators with decomposable norms \cite{BRT, 2012loh:wainwright-AoS,2012negahban:ravikumar:wainwright:yu, 2018bellec:lecue:tsybakov}. $\MP$ over a single variable $\bfV$ and norm 
$\calR$ is also well known to be useful. We generalize this concept over the triple $[\bfV,\bfW,\bu]$ and norms $(\calR,\calS,\calQ)$. See the next Section \ref{s:MP:in:M-estimation} for a discussion on this notion. Similarly, $\IP$, $\PP$ and $\ATP$ --- an abbreviation for ``augmented'' restricted convexity --- are properties over the triplet $[\bfV,\bfW,\bu]$ which we show to be useful for RTRMD. 

\begin{remark}[Design properties in \cite{2012agarwal:negahban:wainwright}]\label{rem:design:properties:agarwal:negahban:wainwright}
The main design property in \cite{2012agarwal:negahban:wainwright} is stated in their Definition 2. In our terminology, it is equivalent to $\ATP_{\calR,\calS,0}(\sd_1,\sd_2,\sd_3,0)$, a particular version over the pair $[\bfV,\bfW]\in(\mdR^p)^2$. Their Theorem 1 states deterministic bounds assuming this property. Still, they guarantee this property holds only for two simple cases. The first is identity designs, for which $\ATP$ is trivially satisfied with $\sd_1=1$ and $\sd_2=\sd_3=0$. The second case is multi-task learning. Dirksen's inequality for the quadratic process implies the design is invertible with high probability.\footnote{The design components are $\frX_i(\bfB):=\bx_i^\top\bfB$. When $n\gtrsim d_1$, standard concentration inequalities imply
$
\frac{1}{n}\sum_{i=1}^n(\bx_i^\top\bb)^2\ge c\Vert\bb\Vert_\Pi^2
$
for all $\bb\in\re^{d_1}$ with high probability for some absolute constant $c\in(0,1)$. Thus, 
$
\frac{1}{n}\Vert\frX(\bfB)\Vert_F^2\ge c\Vert\bfB\Vert_\Pi^2
$
for all $\bfB\in\mdR^p$.} Even without label contamination, proving 
$\ATP_{\calR,\calS,0}(\sd_1,\sd_2,\sd_3,0)$ for trace regression with additive matrix decomposition requires a more refined argument using $\PP$.

The model in \cite{2012agarwal:negahban:wainwright} is not concerned with label contamination. Using our terminology, they do not need $\IP$ 
($\sb_i\equiv0$) and it is enough to set $\calQ\equiv0$, 
$\sc_4=\sd_4=\sf_4=0$. Additionally, they implicitly use $\MP_{\calR,\calS,0}(0,\sf_2,\sf_3,0)$ with 
$\sf_1=0$, resorting to the dual-norm inequality. As explained in the next section, this is the technical reason their bounds are optimal in average but sub-optimal in 
$\delta$. We remind that they assume the noise is independent of the features. 
\end{remark}

\begin{remark}[Design properties in \cite{2019dalalyan:thompson}]\label{rem:design:properties:dalalyan:thompson}
\cite{2019dalalyan:thompson} studies robust sparse regression with Huber's loss in the Gaussian setting. In this quest, they require the particular properties $\IP_{\Vert\cdot\Vert_1,0,\Vert\cdot\Vert_1}(\sb_1,\sb_2,0,\sb_4)$ and $\ATP_{\Vert\cdot\Vert_1,0,\Vert\cdot\Vert_1}(\sd_1,\sd_2,0,\sd_4)$ over the pair $[\bv,\bu]\in\re^p\times\re^n$.\footnote{They use the notation TP for $\TP$ and ATP for $\ATP$.} Explicitly:
\begin{align}
|\langle\bu,\frX^{(n)}(\bv)\rangle|
&\le \sb_1\left\Vert\bv\right\Vert_{\Pi}
\Vert\bu\Vert_2 
+\sb_2\Vert\bv\Vert_1\Vert\bu\Vert_2
+\sb_4\left\Vert\bv\right\Vert_\Pi\Vert\bu\Vert_1,\\
\Vert\frM^{(n)}(\bv,\bu)\Vert_2&\ge 
\sd_1\Vert[\bv,\bu]\Vert_\Pi
-\sd_2\Vert\bv\Vert_1-\sd_4\Vert\bu\Vert_1.
\end{align}
Their model is not concerned with matrix decomposition. Using our terminology, they do not need $\PP$ ($\sc_i\equiv0$) and it is enough to set $\calS\equiv0$, 
$\sb_3=\sd_3=\sf_3=0$. Our framework deals with the broader model RTRMD. By Lemma \ref{lemma:ATP} and Proposition \ref{proposition:properties:subgaussian:designs}, RTRMD requires a non trivial interplay between the product process inequality of Theorem \ref{thm:product:process} 
and Chevet's inequality (see Section \ref{s:prop:IP} in the supplement). Both inequalities are needed as they imply, respectively, $\PP$ and the general version of 
$\IP$ over the triple $[\bfV,\bfW,\bu]\in(\mdR^p)^2\times\re^n$ and norms $(\calR,\calS,\calQ)$. 

We now justify the relevance of regressing with sorted Huber-type losses. In properties $\IP$  and $\MP$ in Proposition    \ref{proposition:properties:subgaussian:designs}, 
$\sb_4\asymp\sf_4\asymp\sigma\mathscr G(\mbB_\calQ)
/\sqrt{n}$. In Huber regression 
($\calQ=\Vert\cdot\Vert_1$), one has 
$\sb_4\asymp\sf_4\asymp\sigma\{\log n/n\}^{1/2}$. With sorted Huber-type losses 
($\calQ=\Vert\cdot\Vert_\sharp$), 
$\sb_4\asymp\sf_4\asymp\sigma/\sqrt{n}$. Using this observation in Theorem \ref{thm:improved:rate}, we obtain the improvement on the corruption rate $\omega(\epsilon)$ and breakdown point by a factor $\log n$. For the details, see the proof of Theorem \ref{thm:tr:reg:matrix:decomp} in Section 28 
%\ref{proof:thm:tr:reg:matrix:decomp} 
in the supplement. 
Our simulations also show an improvement on the ``practical'' constant in the MSE (recall Figure  \ref{fig.robust.linear.reg:HS}).

To conclude, the analysis in \cite{2019dalalyan:thompson} implicitly uses the particular version $\MP_{\calR,0,\calQ}(0,\sf_2,0,\sf_4)$ with $\sf_1=0$ --- indeed, they resort to the dual-norm inequality. As in the previous remark, this approach leads to near-optimal bounds in $(n,s,p,\epsilon)$ in average but sub-optimal in $\delta$. We prove the general version of $\MP$ for RTRMD is defined over the triple $[\bfV,\bfW,\bu]\in(\mdR^p)^2\times\re^n$ and norms $(\calR,\calS,\Vert\cdot\Vert_\sharp)$. \cite{2019dalalyan:thompson} also assumes the noise is independent of the features.
\end{remark}

\section{$\MP$ in $M$-estimation with decomposable regularizers}\label{s:MP:in:M-estimation}
Let $\hat\bfB$ be the least-squares estimator with penalization 
$\lambda\calR$ --- when there is no contamination, 
the model is well-specified and $\bfGamma^*\equiv\bf0$. By the first-order condition, one gets
$$
\Vert\frX^{(n)}(\bfDelta_{\bfB^*})\Vert_2^2\le (\nicefrac{1}{n})\langle\bxi,\frX(\bfDelta_{\bfB^*})\rangle 
+ \lambda \big(\calR(\bfB^*) - \calR(\hat\bfB)\big).
$$
In the ``Lasso proof'', the standard argument is to properly upper bound the RHS and lower bound the LHS --- using the regularization effect of the decomposable norm 
$\calR$. For the upper bound, the typical way is to use the dual-norm inequality:
\begin{align}
(\nicefrac{1}{n})\langle\bxi,\frX(\bfDelta_{\bfB^*})\rangle\le (\nicefrac{1}{n})\calR^*(\frX^*(\bxi))\cdot\calR(\bfDelta_{\bfB^*}),\label{equation:dual-norm:inequality}
\end{align}
where $\calR^*$ denotes the dual-norm of 
$\calR$ and $\frX^*(\bxi):=\sum_{i=1}^n\xi_i\bfX_i$ is the adjoint operator of 
$\frX$. The displayed bound is 
$\MP_{\calR,0,0}(0,\sf_2,0,0)$ with 
$\sf_2=(1/n)\calR^*(\frX^*(\bxi))$. As well known, the choice $\lambda\gtrsim (1/n)\calR^*(\frX^*(\bxi))$ ensures 
$\bfDelta_{\bfB^*}$ lies in a dimension-reduction cone. $\TP$ can then be invoked for the lower bound --- indeed, it implies  strong-convexity over the dimension-reduction cone.

Consider trace regression with 
$\calR=\Vert\cdot\Vert_N$. As first shown in \cite{2011negahban:wainwright}, if we use \eqref{equation:dual-norm:inequality} we must take
$\lambda\asymp\sigma\rho_N(\bfSigma)(\sqrt{(d_1+d_2)/n}+\sqrt{\log(1/\delta)/n})$, implying the estimation rate
$$
\sigma\rho_N(\bfSigma)\sqrt{r(d_1+d_2)/n} + \sigma\rho_N(\bfSigma)\sqrt{r\log(1/\delta)/n}.
$$
This seminal result is optimal --- in average but subptimal in 
$\delta$. In case of sparse regression, the same approach leads to the rate
$
\sigma\rho_1(\bfSigma)\sqrt{s\log p/n} + \sigma\rho_1(\bfSigma)\sqrt{s\log(1/\delta)/n}.
$ 
To our knowledge, \cite{2018bellec:lecue:tsybakov} was the first to attain the $\delta$-subgaussian rate for sparse regression. See also \cite{2018derumigny} for extensions on their results for the square-root Lasso and Slope estimators. Let $\mbX$ denote the design matrix satisfying 
$\max_{j\in[p]}\Vert\mbX_{\bullet,j}\Vert_2\le1$ and assume that $\bxi$ is independent of $\mbX$. In their Theorem 9.1, they show that, with probability $\ge1-\delta$, for all $\bv\in\re^p$, 
\begin{align}
(\nicefrac{1}{n})\langle\bxi,\mbX\bv\rangle\le\tilde\sf_1\Vert\mbX^{(n)}\bv\Vert_2 + \tilde\sf_2\Vert\bv\Vert_\sharp, 
\label{equation:MP:bellec}
\end{align}
with $\tilde\sf_1\asymp\sigma(\nicefrac{1+\sqrt{\log(1/\delta)}}{\sqrt{n}})$ and $\tilde\sf_2\asymp\sigma/\sqrt{n}$. The above bound and an upper  bound on the quadratic process imply
$\MP_{\Vert\cdot\Vert_{\sharp},0,0}(\sf_1,\sf_2,0,0)$ with 
$\sf_1\asymp\tilde\sf_1$ and $\sf_2\asymp\tilde\sf_2$.\footnote{Dirksen's inequality implies, for suitable constants $\tilde\sa_1,\tilde\sa_2>0$, $\Vert\mbX^{(n)}\bv\Vert_2\le \tilde\sa_1\Vert\bv\Vert_\Pi + \tilde\sa_2\Vert\bv\Vert_\sharp$ for all $\bv$ with high probability.} It is a substantial improvement upon \eqref{equation:dual-norm:inequality} in the sense that $\tilde \sf_2$ does not depend on $\delta$ --- appearing only in the ``variance'' constant $\tilde\sf_1$. This allows us to take 
$\lambda\asymp\sigma/\sqrt{n}$, entailing the subgaussian rate 
$
\sigma\sqrt{s\log(p/s)/n} + \sigma\sqrt{\log(1/\delta)/n}.
$
In conclusion, showing $\MP_{\calR,0,0}(\sf_1,\sf_2,0,0)$ with finer arguments than \eqref{equation:dual-norm:inequality} --- with proper coefficient $\sf_1\neq0$ --- is the technical reason \cite{2018bellec:lecue:tsybakov} attains the $\delta$-subgaussian rate with $\delta$-adaptive estimators in the framework of $M$-estimation with decomposable regularizers.

Theorem 9.1 in \cite{2018bellec:lecue:tsybakov} fundamentally assumes 
$\bxi$ is independent of $\bfX$.\footnote{Indeed, define the random norm 
$\hat T(\bv):=\frac{\Vert\mbX^{(n)}\bv\Vert_2}{L}\vee\Vert\bv\Vert_\sharp$ with $L:=\sqrt{n/\log(1/\delta)}/\sigma$. In case $\mbX$ is fixed, they bound the multiplier process concentrating the linear process
$
\sup_{\bu\in \mbB_{\hat T}}\langle\bxi,\bu\rangle
$
--- see Proposition 9.2 in  \cite{2018bellec:lecue:tsybakov}. The proof of this elegant result follows from a simple application of a tail symmetrization-comparison argument and the gaussian concentration inequality. Peeling is not necessary --- as homogeneity of norms suffices.}
Without this assumption, Theorem \ref{thm:mult:process} and a peeling argument entail 
$\MP_{\Vert\cdot\Vert_\sharp,0,0}(\sf_1,\sf_2,0,0)$ with high probability and constants $\sf_1\asymp\sigma L(\nicefrac{1+\sqrt{\log(1/\delta)}}{\sqrt{n}})$ and $\sf_2\asymp\sigma L\rho_1(\bfSigma)/\sqrt{n}$ --- assuming $n\gtrsim 1 + \log(1/\delta)$. More generally, we prove $\MP_{\calR,\calS,\calQ}(\sf_1,\sf_2,\sf_3,\sf_4)$  with non-zero constants $\sf_3$ and $\sf_4\asymp\sigma L\mathscr{G}(\mbB_\calQ)/\sqrt{n}$ for general regularization norms 
$(\calR,\calS,\calQ)$. When 
$\calQ=\Vert\cdot\Vert_\sharp$, we show this property is useful to handle label contamination and/or additive matrix decomposition. These properties entail the $\delta$-subgaussian rate for the robust $\delta$-adaptive estimators \eqref{equation:sorted:Huber:general}-\eqref{equation:sorted:Huber} --- assuming just marginal subgaussianity of $\bfX$ and $\bxi$.

\section{First general theorem}\label{s:proof:main:paper}
The main result of this section is the general Theorem \ref{thm:improved:rate}. To stated it, we fix the positive constants $\{\sa_i\}$, $\{\sb_i\}$, $\{\sc_i\}$, $\{\sd_i\}$ and $\{\sf_i\}$ in Definition \ref{def:design:property}. Theorems \ref{thm:tr:reg:matrix:decomp}-\ref{thm:response:sparse-low-rank:regression} are consequences of Theorem \ref{thm:improved:rate} and Proposition  \ref{proposition:properties:subgaussian:designs} --- which ensure that the required design properties hold with high probability.

We start with some definitions. 
\begin{definition}[Decomposable norms \cite{2012negahban:ravikumar:wainwright:yu, 2011koltchinskii:lounici:tsybakov}]\label{def:decomposable:norm}
A norm $\calR$ over $\mdR^{p}$ is said to be decomposable if, for all $\bfB\in\mdR^{p}$, there exists linear map $\bfV\mapsto\calP_{\bfB}^\perp(\bfV)$ such that, for all $\bfV\in\mdR^{p}$, defining 
$\calP_{\bfB}(\bfV):=\bfV-\calP^\perp_{\bfB}(\bfV)$,
\begin{itemize}
\item $\calP_{\bfB}^\perp(\bfB)=0$,
\item $\llangle\calP_{\bfB}(\bfV),\calP^\perp_{\bfB}(\bfV)\rrangle=0$,
\item $\calR(\bfV)=\calR(\calP_{\bfB}(\bfV))+\calR(\calP^\perp_{\bfB}(\bfV))$.
\end{itemize}
\end{definition}
When $\calR=\Vert\cdot\Vert_1$, $\calP_{\bb}$ is the projection onto the support of vector $\bb$. When $\calR=\Vert\cdot\Vert_N$, $\calP_{\bfB}$ is the projection onto the ``low-rank'' support of the matrix $\bfB$ --- see Section \ref{s:decomposable:norms} in the supplement for a precise definition. In what follows, $\calR$ and $\calS$ will be decomposable norms on $\mdR^p$. We shall need the following definition. 
\begin{definition}[Dimension-reduction cone]\label{def:dim:red:cones}
Given $[\bfB,\bfGamma]\in(\mdR^p)^2$, let 
$\calP_{\bfB}$ and $\calP_{\bfGamma}$ be the projection maps  associated to $(\calR,\bfB)$ and $(\calS,\bfGamma)$ respectively. Fix $c_0,\gamma_{\calR},\gamma_{\calS},\eta\ge0$. Let $\calC_{\bfB,\bfGamma,\calR,\calS}(c_0,\gamma_{\calR},\gamma_{\calS},\eta)$ be the cone of points $[\bfV,\bfW,\bu]\in(\mdR^p)^2\times\re^n$ satisfying
$$
\gamma_{\calR}\calR(\calP_{\bfB}^\perp(\bfV))
+\gamma_{\calS}\calS(\calP_{\bfGamma}^\perp(\bfW))
+\sum_{i=o+1}^n\omega_i\bu_i^\sharp
\le 
c_0\left[\gamma_{\calR}\calR(\calP_{\bfB}(\bfV)) + \gamma_{\calS}\calS(\calP_{\bfGamma}(\bfW))
+\eta\Vert\bu\Vert_2
\right].
$$
We will sometimes omit some of the subscripts when they are clear in the context.
\end{definition}
The one-dimensional cone $\calC_{\bfB,\calR}(c_0):=\calC_{\bfB,\bf0,\calR,0}(c_0,1,0,0)$ is well known in high-dimensional statistics. In the analysis of RTRMD the three-dimensional cone $\calC_{\bfB,\bfGamma,\calR,\calS}$ is useful. Next, we will need some additional notation when dealing with contamination in miss-specified models.
\begin{definition}\label{def:rate:notation}
Let $\bfB,\bfGamma,\bfV,\bfW\in\mdR^p$ and non-negative numbers $(c_0,\alpha,\sf,r)$. Define the quantities $R_{\calR,c_0}(\bfV|\bfB):=\Psi_{\calR}(\calP_{\bfB}(\bfV))\mu(\calC_{\bfB,\calR}(2c_0))$ and $R_{\calS,c_0}(\bfW|\bfGamma):=\Psi_{\calS}(\calP_{\bfGamma}(\bfW))\mu(\calC_{\bfGamma,\calS}(2c_0))$. Additionally, set 
\begin{align}
r_{\lambda,\chi,\alpha,c_0}(\bfV,\bfW|\bfB,\bfGamma)&:=\{\lambda^2R_{\calR,c_0}^2(\bfV|\bfB) + \chi^2R_{\calS,c_0}^2(\bfW|\bfGamma) + \alpha^2\}^{1/2}.
\end{align}
Finally, set
$
\spadesuit_2(\sf,r):=(\nicefrac{1}{\sd_1^2})\left(
4\sf + 3r \right)^2
$
and 
$
\clubsuit_2(\sf,r):=(\nicefrac{1}{\sd_1^2})\left(
16\sf + 12r\right).
$
\end{definition}

\begin{theorem}[$q=2$ \& matrix decomposition]\label{thm:improved:rate}
Grant Assumption \ref{assump:label:contamination} and model \eqref{equation:structural:equation}. Consider the solution $[\hat\bfB,\hat\bfGamma]$ of  \eqref{equation:sorted:Huber:general} with hyper-parameters 
$\lambda,\chi,\tau>0$ and $\sa\in(0,\infty]$. Let $\sc_*,\sf_*\ge0$ be absolute constants and suppose:
\begin{itemize}\itemsep=0pt
\item[\rm (i)] $(\frX,\bxi)$ satisfies $\MP_{\calR,\calS,\Vert\cdot\Vert_\sharp}(\sf_1,\sf_2,\sf_3,\sf_4)$. 
\item[\rm{(ii)}] $\frX$ satisfies $\ATP_{\calR,\calS,\Vert\cdot\Vert_\sharp}(\sd_1,\sd_2,\sd_3,\sd_4)$.
\item[\rm(iii)] $\frX$ satisfies $\IP_{\calR,\calS,\Vert\cdot\Vert_\sharp}\left(\sb_1,\sb_2,\sb_3,\sb_4\right)$.
\item[\rm (iv)] The hyper-parameters $(\lambda,\chi,\tau)$ satisfy $\tau\ge 4[(\sigma\sd_4)\vee(2\sf_4)]$,
\begin{align}
\lambda \ge4[(\sigma\sd_2)\vee(2\sf_2+2\sc_*\sigma\sb_2)] \quad\mbox{ and }\quad
\chi\ge4[(\sigma\sd_3)\vee(2\sf_3+2\sf_*+2\sc_*\sigma\sb_3)].
\end{align}
\item[\rm (v)] $\hat\sf_1:=\sf_1 + \sb_1(\sc_*\sigma) + (\nicefrac{2\sb_4}{\tau})(\sc_*^2\sigma^2)\le\sigma\sd_1/2$. 
\end{itemize}

Let any $D\ge0$ and $[\bfB,\bfGamma]$ satisfying the constraints
\begin{align}
\Vert\frX^{(n)}(\bfB+\bfGamma)-\boldf^{(n)}\Vert_2 &= D,\label{cond4:general:norm}\\
\Vert\bfB\Vert_\infty &\le\sa,\label{cond5:general:norm}\\
|\llangle\bfDelta_{\bfB},\bfDelta_{\bfGamma}\rrangle_\Pi|&\le \sf_*\calS(\bfDelta_{\bfGamma}),\label{cond6:general:norm}\\
r := r_{\lambda,\chi,\tau\Omega,3}(\bfDelta_{\bfB},\bfDelta_{\bfGamma}|\bfB,\bfGamma)&<\sigma\sd_1, \label{cond1:general:norm}\\
D^2 + \spadesuit_2(\sf_1,r) &\le \sc_*^2\sigma^2,\label{cond8:general:norm}\\
\left[(\nicefrac{D^2}{\sigma\sd_1})\right]\bigvee\left[(\nicefrac{2\sqrt{2}}{\sd_1})D + \clubsuit_2(\sf_1,r)\right]&\le \sc_*\sigma.\label{cond9:general:norm}
\end{align}
Define the quantities
$\hat r:=r_{\lambda,\chi,0,3}(\bfDelta_{\bfB},\bfDelta_{\bfGamma}|\bfB,\bfGamma)$ and
\begin{align}
F&:=\sf_1 + \sb_1\left(
(\nicefrac{D^2}{\sigma\sd_1})\bigvee\left((\nicefrac{2\sqrt{2}}{\sd_1})D + \clubsuit_2(\sf_1,r)\right)
\right) + (2\sb_4/\tau)\left(
D^2 + \spadesuit_2(\sf_1,r)
\right).
\end{align}

Then, it holds that 
\begin{align}
&(\nicefrac{1}{2})\left(\lambda\calR(\bfDelta_{\bfB}) + \chi\calS(\bfDelta_{\bfGamma})\right) + \Vert\frX^{(n)}(\hat\bfB+\hat\bfGamma)-\boldf^{(n)}\Vert_2^2
\le D^2 + \spadesuit_2(F,\hat r).\label{thm:improved:rate:main:rate:eq1}\\
&\Vert[\bfDelta_{\bfB},\bfDelta_{\bfGamma}]\Vert_{\Pi} \le 
(\nicefrac{D^2}{\sigma\sd_1})\bigvee\left[(\nicefrac{2\sqrt{2}}{\sd_1})D + \clubsuit_2(F,\hat r)\right]. \label{thm:improved:rate:main:rate:eq2}
\end{align}
\end{theorem}

The conditions (i)-(iii) are the required design properties. Condition (iv) prescribes the ``optimal'' level of the hyper-parameters 
$(\lambda,\chi,\tau)$ in terms of the design constants, the noise level and the low-spikeness constant $\sf_*$. Notice that $(\lambda,\chi)$  also depend on the constant $\sc_*$ --- this constant is related with the constraints \eqref{cond4:general:norm} and \eqref{cond8:general:norm}-\eqref{cond9:general:norm}. As explained later, these constraints identify the effect of the corruption error $\bfDelta^{\hat\btheta}$ and miss-specification error on the choice of $(\lambda,\chi)$. The constraints \eqref{cond5:general:norm}-\eqref{cond6:general:norm} encode the low-spikeness assumption. Finally, condition (v) and constraint \eqref{cond1:general:norm} encode the minimal sample size and maximum breakdown point. Notice that the corruption and miss-specification errors also impact condition (v) via the constant $\sc_*$. 

The proof of Theorem \ref{thm:improved:rate} will be done via intermediate lemmas, stated next. These are proven in the supplement. We start with the next lemma, a consequence of the first order condition of  \eqref{equation:aug:slope:rob:estimator:general}. In the following, we grant Assumption \ref{assump:label:contamination} and model  \eqref{equation:structural:equation} and set $\bfDelta=\frX(\hat\bfB+\hat\bfGamma)-\boldf$.  

\begin{lemma}\label{lemma:recursion:1st:order:condition}
For all $[\bfB,\bfGamma]\in(\mdR^p)^2$ such that 
$\Vert\bfB\Vert_\infty\le\sa$,
\begin{align}
\langle \bfDelta^{(n)} +\bfDelta^{\hat\btheta}, \frM^{(n)}(\bfDelta_{\bfB}+\bfDelta_{\bfGamma},\bfDelta^{\hat\btheta})\rangle  
&\le \langle\bxi^{(n)},\frM^{(n)}(\bfDelta_{\bfB} + \bfDelta_{\bfGamma},\bfDelta^{\hat\btheta})\rangle\\
&+\lambda \big(\calR(\bfB) - \calR(\hat\bfB)\big) 
+\chi \big(\calS(\bfGamma) - \calS(\hat\bfGamma)\big)
+\tau\big(\|\btheta^*\|_\sharp -\|\hat\btheta\|_\sharp\big).
\label{lemma:recursion:1st:order:condition:eq}
\end{align}
\end{lemma}

Next, we upper (and lower) bound \eqref{lemma:recursion:1st:order:condition:eq} using $\MP$ (and $\ATP$). In case of additive matrix decomposition, we require an additional condition related to the spikeness assumption.
\begin{lemma}\label{lemma:MP+ATP}
Suppose conditions (i)-(ii) of Theorem \ref{thm:improved:rate} hold. For some $\sf_*\ge0$, let $[\bfB,\bfGamma]$ such that
$\Vert\bfB\Vert_\infty\le\sa$ and 
$
|\llangle\bfDelta_{\bfB},\bfDelta_{\bfGamma}\rrangle_{\Pi}|\le \sf_*\calS(\bfDelta_{\bfGamma}).
$
Define the quantities
\begin{align}
\blacktriangle & := ((\sigma\sd_2)\vee(2\sf_2))\calR(\bfDelta_{\bfB}) 
+ ((\sigma\sd_3)\vee(2\sf_3 + 2\sf_*))\calS(\bfDelta_{\bfGamma}) 
+ ((\sigma\sd_4)\vee(2\sf_4))\Vert\bfDelta^{\hat\btheta}\Vert_\sharp,\\
\blacktriangledown & := \lambda \big(\calR(\bfB) - \calR(\hat\bfB)\big) 
+\chi\big(\calS(\bfGamma) - \calS(\hat\bfGamma)\big)
+\tau\big(\|\btheta^*\|_\sharp -\|\hat\btheta\|_\sharp\big).
\end{align}

Then 
\begin{align}
\Vert\bfDelta^{(n)} +\bfDelta^{\hat\btheta}\Vert_2^2
+ \left(\sd_1\Vert[\bfDelta_{\bfB},\bfDelta_{\bfGamma},\bfDelta^{\hat\btheta}]\Vert_{\Pi} - (\nicefrac{\blacktriangle}{\sigma}) \right)_+^2&\le
\Vert\frX^{(n)}(\bfB+\bfGamma)-\boldf^{(n)}\Vert_2^2\\
&+ 2\sf_1\Vert[\bfDelta_{\bfB},\bfDelta_{\bfGamma},\bfDelta^{\hat\btheta}]\Vert_{\Pi}
+ \blacktriangle + 2\blacktriangledown.
\label{lemma:MP+ATP:eq}
\end{align}
\end{lemma}

To illustrate, consider well-specified trace regression with label contamination --- namely, $\bfB=\bfB^*$, 
$\bfGamma=\bfGamma^*=\bf0$ and $\frX^{(n)}(\bfB^*)-\boldf^{(n)}=\bf0$. In this case, $\chi=\sf_*=0$, $\sa=\infty$ and it is sufficient that $\MP$ and $\ATP$ hold with $\sd_3=\sf_3=0$ and $\calS\equiv0$. Using \eqref{lemma:recursion:1st:order:condition:eq}, a similar proof of \eqref{lemma:MP+ATP:eq} entails
\begin{align}
\Vert\frM^{(n)}(\bfDelta_{\bfB^*},\bfDelta^{\hat\btheta})\Vert_2^2
&\le \frac{\sf_1}{\sc_1}\Vert\frM^{(n)}(\bfDelta_{\bfB^*},\bfDelta^{\hat\btheta})\Vert_2
+ \left(\sf_2+\frac{\sf_1\sc_2}{\sc_1}\right)\calR(\bfDelta_{\bfB^*})
+ \left(\sf_3+\frac{\sf_1\sc_3}{\sc_1}\right)\Vert\bfDelta^{\hat\btheta}\Vert_\sharp\\
&+\lambda \big(\calR(\bfB) - \calR(\hat\bfB)\big) 
+\tau\big(\|\btheta^*\|_\sharp -\|\hat\btheta\|_\sharp\big). 
\label{lemma:MP+ATP:illustration:eq}
\end{align}
In case $\sf_1=0$, the above bound and decomposability of norms can be used to show that $[\bfDelta_{\bfB^*},\bf0,\bfDelta^{\hat\btheta}]\in\calC_{\bfB^*,\bf0}(c_0,\gamma,0,\Omega)$ for some $c_0>0$ and $\gamma:=\lambda/\tau$  --- provided the penalization $(\lambda,\tau)$ is large enough. This would be the approach using the dual-norm inequality. Instead, we resort to Theorem \ref{thm:mult:process} to obtain $\MP$ with $\sf_1\neq0$ --- enabling us to obtain the optimal rate in 
$\delta$. $\ATP$ can be further used to lower bound \eqref{lemma:MP+ATP:illustration:eq} --- as it implies restricted strong convexity over $\calC_{\bfB^*,\bf0}(c_0,\gamma,0,\Omega)$. Inequality \eqref{lemma:MP+ATP:eq} is a non trivial generalization of \eqref{lemma:MP+ATP:illustration:eq} to handle miss-specification and additive matrix decomposition.\footnote{More precisely, \eqref{lemma:MP+ATP:eq} gives a recursion in the variable $\Vert[\bfDelta_{\bfB},\bfDelta_{\bfGamma},\bfDelta^{\hat\btheta}]\Vert_{\Pi}$ --- instead of 
$\Vert\frM^{(n)}(\bfDelta^{\hat\bfB}+\bfDelta^{\hat\bfGamma},\bfDelta^{\hat\btheta})\Vert_2$. This technical point is needed in case of matrix decomposition, accounting for the ``bias'' 
$
|\llangle\bfDelta_{\bfB},\bfDelta_{\bfGamma}\rrangle_{\Pi}|\le \sf_*\calS(\bfDelta_{\bfGamma}).
$}

\begin{remark}[Relevance of $\PP$ and $\IP$]\label{rem:relevance:PP+IP}
One should not take for granted the fact that $\PP_{\calR,\calS}(\sc_1,\sc_2,\sc_3,\sc_4)$ and $\IP_{\calR,\calS,\Vert\cdot\Vert_\sharp}\left(\sb_1,\sb_2,\sb_3,\sb_4\right)$ over the triplet $[\bfV,\bfW,\bu]$ are \emph{implicitly} invoked in Lemma \ref{lemma:MP+ATP}. Indeed, by Lemmas \ref{lemma:TP} and  \ref{lemma:ATP}, both properties entail 
$\ATP_{\calR,\calS,\Vert\cdot\Vert_\sharp}(\sd_1,\sd_2,\sd_3,\sd_4)$. The optimal bound for $\Vert[\bfDelta_{\bfB},\bfDelta_{\bfGamma},\bfDelta^{\hat\btheta}]\Vert_{\Pi}$ is obtained invoking $\ATP$ with sharp constants (as stated in Proposition \ref{proposition:properties:subgaussian:designs}). One could argue if a more ``direct'' approach could lead to the optimal rate, for instance, one without resorting to such technical definitions. It turns out that the mere use of dual-norm inequalities is suboptimal.  
\end{remark}

Before proceeding, we state the next lemma stating the useful bounds \eqref{lemma:aug:rest:convexity:eq2}-\eqref{lemma:aug:rest:convexity:eq3} for points in $\calC_{\bfB,\bfGamma}$. For convenience, given $[\bfV,\bfW,\bfB,\bfGamma,\bu]$,  we define
\begin{align}
\triangle_{\lambda,\chi,\tau}(\bfV,\bfW,\bu|\bfB,\bfW)&:=(\nicefrac{3\lambda}{2})(\calR\circ\calP_{\bfB})(\bfV) 
-(\nicefrac{\lambda}{2})(\calR\circ\calP_{\bfB}^\perp)(\bfV)\\
&+(\nicefrac{3\chi}{2})(\calS\circ\calP_{\bfGamma})(\bfW) - (\nicefrac{\chi}{2})(\calS\circ\calP_{\bfGamma}^\perp)(\bfW)\\
&+(\nicefrac{3\tau\Omega}{2})\Vert\bu\Vert_2 -(\nicefrac{\tau}{2})\sum_{i=o+1}^n\omega_i\bu_i^\sharp.
\end{align}
\begin{lemma}\label{lemma:aug:rest:convexity}
Define $\gamma_{\calR}:=\lambda/\tau$ and $\gamma_{\calS}:=\chi/\tau$ and let $c_0,\eta>0$. Then, for any $[\bfB,\bfGamma]$ and $[\bfV,\bfW,\bu]\in\calC_{\bfB,\bfGamma}(c_0,\gamma_{\calR},\gamma_{\calS},\eta)$, 
\begin{align}
\triangle_{\lambda,\chi,\tau}(\bfV,\bfW,\bu|\bfB,\bfGamma)&\le(\nicefrac{3}{2}) r_{\lambda,\chi,\tau\eta,c_0}(\bfV,\bfW|\bfB,\bfGamma)\cdot\Vert[\bfV,\bfW,\bu]\Vert_\Pi,\label{lemma:aug:rest:convexity:eq2}\\
\lambda\calR(\bfV) + \chi\calS(\bfW)+\tau\big\|\bu\big\|_\sharp&\le 
2(c_0+1)\cdot r_{\lambda,\chi,\tau\eta,c_0}(\bfV,\bfW|\bfB,\bfGamma)\cdot
\Vert[\bfV,\bfW,\bu]\Vert_\Pi.\label{lemma:aug:rest:convexity:eq3}
\end{align}
\end{lemma}

The previous lemmas entail the next proposition.  
\begin{proposition}\label{prop:suboptimal:rate}
Suppose the conditions (i)-(ii) of Theorem \ref{thm:improved:rate} hold and, additionally, 
\begin{itemize}
\item[\rm (iii')] 
$
\lambda\ge4[(\sigma\sd_2)\vee(2\sf_2)], 
$
$\chi\ge4[(\sigma\sd_3)\vee(2\sf_3+2\sf_*)]$
and 
$
\tau\ge 4[(\sigma\sd_4)\vee(2\sf_4)].
$
\item[\rm (iv')] $2\sf_1\le\sigma\sd_1$. 
\end{itemize}
For any $D\ge0$ and $[\bfB,\bfGamma]$ satisfying the constraints
\eqref{cond4:general:norm}, \eqref{cond5:general:norm}, \eqref{cond6:general:norm} and \eqref{cond1:general:norm},
\begin{align}
(\nicefrac{1}{2})(\lambda\calR(\bfDelta_{\bfB}) + \chi\calS(\bfDelta_{\bfGamma}) + \tau\Vert\bfDelta^{\hat\btheta}\Vert_\sharp)
+ \Vert\bfDelta^{(n)} + \bfDelta^{\hat\btheta}\Vert_2^2
&\le D^2 + \spadesuit_2(\sf_1,r), \label{prop:suboptimal:rate:eq1}
\end{align}
where $r := r_{\lambda,\chi,\tau\Omega,3}(\bfDelta_{\bfB},\bfDelta_{\bfGamma}|\bfB,\bfGamma)$. 
Additionally, 
\begin{align}
\Vert[\bfDelta_{\bfB},\bfDelta_{\bfGamma},\bfDelta^{\hat\btheta}]\Vert_{\Pi} &\le
\left[(\nicefrac{D^2}{\sigma\sd_1})\right]\bigvee\left[(\nicefrac{2\sqrt{2}}{\sd_1})D + \clubsuit_2(\sf_1,r)\right]. 
\label{prop:suboptimal:rate:eq2}
\end{align}
\end{proposition}

Using Proposition \ref{proposition:properties:subgaussian:designs}, we can show the bound in \eqref{prop:suboptimal:rate:eq1} is a near-optimal oracle inequality for the triplet 
$[\bfDelta_{\bfB},\bfDelta_{\bfGamma},\bfDelta^{\hat\btheta}]$. Still, it is suboptimal for 
$[\bfDelta_{\bfB},\bfDelta_{\bfGamma}]$. Next, we show that the bounds for $\bfDelta^{\hat\btheta}$, implied by Proposition \ref{prop:suboptimal:rate}, are enough to obtain the near-optimal rate for 
$[\bfDelta_{\bfB},\bfDelta_{\bfGamma}]$. First, we prove the following lemma --- an easy consequence of the first-order condition for fixed
$\btheta\equiv\hat\btheta$. 
\begin{lemma}\label{lemma:recursion:delta:bb:general:norm}
For all $[\bfB,\bfGamma]\in(\mdR^p)^2$ such that $\Vert\bfB\Vert_\infty\le\sa$,
\begin{align}
\langle \bfDelta^{(n)},\frX^{(n)}(\bfDelta_{\bfB}+\bfDelta_{\bfGamma})\rangle &\le 
\langle\bxi^{(n)}-\bfDelta^{\hat\btheta},\frX^{(n)}(\bfDelta_{\bfB}+\bfDelta_{\bfGamma})\rangle\\
&+\lambda(\calR(\bfB)-\calR(\hat\bfB))
+\chi(\calS(\bfGamma)-\calS(\hat\bfGamma)).
\label{lemma:recursion:delta:bb:general:norm:eq}
\end{align}
\end{lemma}

In case of label contamination and/or additive matrix decomposition, a key difference with the standard ``Lasso proof'' is that $\MP_{\calR,0,0}(\sf_1,\sf_2,0,0)$ is not sufficient to upper bound \eqref{lemma:recursion:delta:bb:general:norm:eq}. Indeed, the noise 
$\bxi^{(n)}$ is shifted by $-\bfDelta^{\hat\btheta}$ and the multiplier process depends on the decomposed error $\bfDelta_{\bfB}+\bfDelta_{\bfGamma}$. Suppose $\MP_{\calR,\calS,\calQ}(\sf_1,\sf_2,\sf_3,\sf_4)$ and $\IP_{\calR,\calS,\calQ}(\sb_1,\sb_2,\sb_3,\sb_4)$ hold. The next lemma states that, if the nuisance error $\bfDelta^{\hat\btheta}$ is ``not too large'' compared to the noise, then the ``perturbed multiplier process'' 
$$
[\bfDelta_{\bfB},\bfDelta_{\bfGamma},\bfDelta^{\hat\btheta}]\mapsto\langle\bxi^{(n)}-\bfDelta^{\hat\btheta},\frX^{(n)}(\bfDelta_{\bfB}+\bfDelta_{\bfGamma})\rangle
$$ 
in \eqref{lemma:recursion:delta:bb:general:norm:eq} can be effectively upper bounded. For the lower bound, we assume 
$\ATP$ and the low-spikeness condition hold.

\begin{lemma}\label{lemma:MP+IP+ATP}
Suppose conditions (i)-(iii) of Theorem \ref{thm:improved:rate} hold. For some $\sf_*\ge0$, let $[\bfB,\bfGamma]$ such that
$\Vert\bfB\Vert_\infty\le\sa$ and 
$
|\langle\bfDelta_{\bfB},\bfDelta_{\bfGamma}\rangle|\le \sf_*\calS(\bfDelta_{\bfGamma}).
$
Define the quantities
\begin{align}
\hat\blacktriangle & := \left((\sigma\sd_2)\vee\left(2\sf_2 
+ 2\sb_2\Vert\bfDelta^{\hat\btheta}\Vert_2\right)\right)\calR(\bfDelta_{\bfB}) 
+ \left((\sigma\sd_3)\vee\left(2\sf_3 + 2\sf_* + 2\sb_3\Vert\bfDelta^{\hat\btheta}\Vert_2\right)\right)\calS(\bfDelta_{\bfGamma}),\\
\hat\blacktriangledown & := \lambda \big(\calR(\bfB) - \calR(\hat\bfB)\big) 
+\chi\big(\calS(\bfGamma) - \calS(\hat\bfGamma)\big). 
\end{align}

Then 
\begin{align}
\Vert\bfDelta^{(n)}\Vert_2^2
+ \left(\sd_1\Vert[\bfDelta_{\bfB},\bfDelta_{\bfGamma}]\Vert_{\Pi} - (\nicefrac{\hat\blacktriangle}{\sigma}) \right)_+^2&\le
\Vert\frX^{(n)}(\bfB+\bfGamma)-\boldf^{(n)}\Vert_2^2\\
& + \left(2\sf_1 + 2\sb_1\Vert\bfDelta^{\hat\btheta}\Vert_2 + 2\sb_4\Vert\bfDelta^{\hat\btheta}\Vert_\sharp\right)\Vert[\bfDelta_{\bfB},\bfDelta_{\bfGamma}]\Vert_{\Pi}\\
&+ \hat\blacktriangle + 2\hat\blacktriangledown.
\label{lemma:MP+IP+ATP:eq}
\end{align}
\end{lemma}

With no contamination ($\bfDelta^{\hat\btheta}=\bf0$, $\tau=0$), Lemmas \ref{lemma:MP+ATP} and \ref{lemma:MP+IP+ATP} coincide. In that case, Proposition \ref{prop:suboptimal:rate} implies the optimal bound for $[\bfDelta_{\bfB},\bfDelta_{\bfGamma}]$. Lemma \ref{lemma:MP+IP+ATP} improves upon  Lemma \ref{lemma:MP+ATP} in case of label contamination. Next, we discuss how this lemma is used in the proof of Theorem \ref{thm:improved:rate}. The complete proof is presented in the supplement. 

Lemma \ref{lemma:MP+IP+ATP} gives a recursion on the parameter error $[\bfDelta_{\bfB},\bfDelta_{\bfGamma}]$ --- instead of 
$[\bfDelta_{\bfB},\bfDelta_{\bfGamma},\bfDelta^{\hat\btheta}]$ as in Lemma \ref{lemma:MP+ATP}. Instead of a parameter estimate, $\bfDelta^{\hat\btheta}$ is seen as a nuisance error perturbing the noise levels $(\sf_1,\sf_2,\sf_3)$ and $(\sigma\sd_1,\sigma\sd_2,\sigma\sd_3)$. As expected, this perturbation affects the tuning of the hyper-parameters 
$(\lambda,\chi)$ and the corresponding rate for 
$[\bfDelta_{\bfB},\bfDelta_{\bfGamma}]$. For \eqref{lemma:MP+IP+ATP:eq} to be meaningful, 
$\bfDelta^{\hat\btheta}$ must be small enough compared to the noise. Precisely, the auxiliary Proposition \ref{prop:suboptimal:rate} implies that 
$\|\bfDelta^{\hat\btheta}\|_2\le\sc_*\sigma$ and 
$\tau\|\bfDelta^{\hat\btheta}\|_\sharp\le\sc_*^2\sigma^2$,  in case we assume \eqref{cond8:general:norm}-\eqref{cond9:general:norm}. Using Proposition \ref{proposition:properties:subgaussian:designs}, we can show that these conditions hold with high probability --- including miss-specified models satisfying 
$(\nicefrac{1}{n})\Vert\frX(\bfB+\bfGamma)-\boldf\Vert_2^2\lesssim\sigma^2$. 

\begin{remark}[The relevance of $\IP$]\label{rem:relevance:IP}
$\IP$ over the triplet $[\bfV,\bfW,\bu]$ is a major tool in the proof of Lemma \ref{lemma:MP+IP+ATP}. Examining this lemma, we see that the thresholds for $(\lambda,\chi)$ are perturbed, respectively, by terms of order $\sb_2\Vert\bfDelta^{\hat\btheta}\Vert_2$ and 
$\sb_3\Vert\bfDelta^{\hat\btheta}\Vert_2$ and the coefficient of $\Vert[\bfDelta_{\bfB},\bfDelta_{\bfGamma}]\Vert_{\Pi}$ in inequality \eqref{lemma:MP+IP+ATP:eq} is perturbed by a term of order $\sb_1\Vert\bfDelta^{\hat\btheta}\Vert_2+\sb_4\Vert\bfDelta^{\hat\btheta}\Vert_\sharp$. The rate optimality of estimator \eqref{equation:aug:slope:rob:estimator:general} follows from these precise expressions and the bounds 
$\|\bfDelta^{\hat\btheta}\|_2\le\sc_*\sigma$ and 
$\tau\|\bfDelta^{\hat\btheta}\|_\sharp\le\sc_*^2\sigma^2$. Indeed, Proposition \ref{proposition:properties:subgaussian:designs} reveals that the constants $(\sb_1,\sb_2,\sb_3,\sb_4)$ are sharp --- it can be shown that the mere use of dual-norm inequalities instead of $\IP$ do not entail the optimal rate for $[\bfDelta_{\bfB},\bfDelta_{\bfGamma}]$.
\end{remark}

\begin{remark}[Proofs in \cite{2019dalalyan:thompson}]\label{rem:comparison:dalalyan:thompson:2}
\cite{2019dalalyan:thompson} studies well-specified robust sparse regression with Huber's loss and Gaussian distributions. There are two ``high level ideas'' used in \cite{2019dalalyan:thompson} which we borrow. Their first idea is to identify $\IP_{\Vert\cdot\Vert_1,0,\Vert\cdot\Vert_1}(\sb_1,\sb_2,0,\sb_4)$ over the pair $[\bv,\bu]\in\re^p\times\re^n$ as the sufficient design property to handle label contamination with Huber's loss --- in case the parameter is known to be sparse and well-specified. To prove this property they use  Chevet's inequality for gaussian processes. Their second idea is to treat $\bfDelta^{\hat\btheta}$ as a nuisance parameter, using a ``two-stage'' proof. The first stage establishes the optimal bound for the pair 
$[\bfDelta_{\bb^*},\bfDelta^{\hat\btheta}]$. The second establishes the optimal bound for $\bfDelta_{\bb^*}$, using the nuisance bound for $\bfDelta^{\hat\btheta}$. 

In this work, we study a broader model: miss-specified RTRMD. As such, the arguments in the proof of Theorem \ref{thm:improved:rate} have substantial changes, both on technical details and structural design properties. On a fundamental level, we give three  contributions when compared to the analysis in \cite{2019dalalyan:thompson}. The first is to identify the new design property $\PP_{\calR,\calS}(\sc_1,\sc_2,\sc_3,\sc_4)$ over the pair $[\bfV,\bfW]\in(\mdR^p)^2$ as the sufficient property to handle additive matrix decomposition in trace regression. 

The second is to identify the more general version $\IP_{\calR,\calS,\Vert\cdot\Vert_\sharp}(\sb_1,\sb_2,\sb_3,\sb_4)$ over the triplet $[\bfV,\bfW,\bu]\in(\mdR^p)^2\times\re^n$, and its relation with $\PP$, as the sufficient design properties to handle, simultaneously, label contamination and additive matrix decomposition. For this, we fundamentally need to use Chevet's inequality and the product process inequality of Theorem \ref{thm:product:process}. We are not aware of a similar application of product processes in high-dimensional estimation. 

The proof in \cite{2019dalalyan:thompson} handles the design-noise interaction in the simplest way: invoking $\MP_{\Vert\cdot\Vert_1,0,\Vert\cdot\Vert_1}(0,\sf_2,0,\sf_4)$ over the pair $[\bv,\bu]$ via the dual-norm inequality. Consequently, they do not attain the subgaussian rate. Our third contribution is to derive a  multiplier process inequality (Theorem \ref{thm:mult:process}) and obtain the general version  $\MP_{\calR,\calS,\Vert\cdot\Vert_\sharp}(\sf_1,\sf_2,\sf_3,\sf_4)$ over the triplet $[\bfV,\bfW,\bu]$. It leads to sharp constants in terms of $\delta$ when the noise is feature-dependent --- recall the discussion in Section \ref{s:MP:in:M-estimation} and comparison with \cite{2018bellec:lecue:tsybakov}. The $\MP$ property with constant $\sf_1\neq0$ is fundamental to achieve the optimal subgaussian rate in $\delta$ with $\delta$-adaptive estimators.
\end{remark}

\section{Second general theorem}\label{s:proof:main:paper:q=1}
In what follows, $\calR$ is a decomposable norm in 
$\mdR^p$. The main result of this section is Theorem \ref{thm:improved:rate:q=1'}. To stated it, we fix the positive constants $\{\sa_i\}$, $\{\sb_i\}$, $\{\sd_i\}$ and $\{\sf_i\}$ in Definition \ref{def:design:property}. Theorem \ref{thm:response:sparse-low-rank:regression:q=1} will follow from Theorem \ref{thm:improved:rate:q=1'} and Proposition \ref{proposition:properties:subgaussian:designs}. Before stating Theorem \ref{thm:improved:rate:q=1'}, we simplify some of the notation in Definition \ref{def:rate:notation}. Given $\bfB,\bfV\in\mdR^p$ and $\alpha,c_0>0$, we define
$
r_{\lambda,\alpha,c_0}(\bfV|\bfB):=\{\lambda^2R_{\calR,c_0}^2(\bfV|\bfB) + \alpha^2\}^{1/2}.
$
Given $\sf,\sc>0$, we let 
$
\spadesuit_1(\sf,\sc):=(\nicefrac{2}{\sd_1^2})\left(
2\sf + \sc \right)^2
$
and 
$
\clubsuit_1(\sf,\sc):=(\nicefrac{4}{\sd_1})\left(
2\sf + \sc\right).
$

\begin{theorem}[$q=1$ \& no matrix decomposition]\label{thm:improved:rate:q=1'}
Grant Assumption \ref{assump:label:contamination} and model \eqref{equation:structural:equation}. Consider the solution $\hat\bfB$ of  \eqref{equation:sorted:Huber} with $q=1$ and hyper-parameters $\lambda,\tau>0$. Let $\hat\sigma:=\Vert\bxi^{(n)}\Vert_2$ and constants $\sc_*>0$ and $\sc_n\ge0$ such that $\sc_*+\sc_n\in[0,1/4)$. Assume that:
\begin{itemize}\itemsep=0pt
\item[\rm (i)] $(\frX,\bxi)$ satisfies $\MP_{\calR,0,\Vert\cdot\Vert_\sharp}(\sf_1,\sf_2,0,\sf_4)$. 
\item[\rm{(ii)}] $\frX$ satisfies $\ATP_{\calR,0,\Vert\cdot\Vert_\sharp}(\sd_1,\sd_2,0,\sd_4)$.
\item[\rm(iii)] $\frX$ satisfies $\IP_{\calR,0,\Vert\cdot\Vert_\sharp}\left(\sb_1,\sb_2,0,\sb_4\right)$.
\item[\rm (iv)] The hyper-parameters $(\lambda,\tau)$ satisfy
\begin{align}
(1-4(\sc_n+\sc_*))\hat\sigma\lambda &\ge4\left[\sf_2 + \frac{\sf_1\sd_2}{\sd_1}
+\left(
\sb_2 + \frac{\sb_1\sd_2}{\sd_1}
\right)\sc_*\hat\sigma
+2(\nicefrac{\sb_4}{\tau})\frac{\sd_2}{\sd_1}\sc_*^2\hat\sigma^2
\right],\\
(1-4\sc_n)\hat\sigma\tau&\ge 4\left[\sf_4 + \frac{\sf_1\sd_4}{\sd_1}\right].
\end{align}
\item[\rm (v)] For 
$
\hat\sf_1:=\sf_1 + \sb_1(\sc_*\sigma) + 
2(\nicefrac{\sb_4}{\hat\sigma\tau})(\sc_*^2\sigma^2)
$
one has 
$56[(\nicefrac{\hat\sf_1}{\sd_1})+\hat\sigma\sc_n]\le3\hat\sigma$. 
\end{itemize}

Let any $D\ge0$ and $\bfB$ satisfying the constraints
\begin{align}
\Vert\frX^{(n)}(\bfB)-\boldf^{(n)}\Vert_2 &= D\le\hat\sigma\sc_n,\label{cond4:general:norm:q=1'}\\
r:=r_{\hat\sigma\lambda,\hat\sigma\tau\Omega,6}(\bfDelta_{\bfB}|\bfB)&\le\left\{\frac{1}{112}\bigvee \frac{1}{28[(\nicefrac{\sd_2}{\lambda})\vee(\nicefrac{\sd_4}{\tau})]}\right\}\hat\sigma\sd_1, \label{cond1:general:norm:q=1'}\\
D^2 + \spadesuit_1\left(\sf_1,R\right)
&\le\sc_*^2\hat\sigma^2,\label{cond8:general:norm:q=1'}\\
(\nicefrac{4D}{\sd_1}) + (\nicefrac{2}{\sd_1})\clubsuit_1
\left(\sf_1,R\right)
&\le\sc_*\hat\sigma,\label{cond9:general:norm:q=1'}
\end{align}
where $R:=(\hat\sigma\sc_n)\vee(3r)$. 

Define the quantities
$\hat r:=r_{\hat\sigma\lambda,0,6}(\bfDelta_{\bfB}|\bfB)$,  $\hat R:=(\hat\sigma\sc_n)\vee(3\hat r)$ and
\begin{align}
F&:= \sf_1 + \sb_1\left[
(\nicefrac{4D}{\sd_1}) + (\nicefrac{2}{\sd_1})\clubsuit_1
\left(\sf_1,R\right)\right]
 + 2(\nicefrac{\sb_4}{\hat\sigma\tau})\left[
D^2 + \spadesuit_1\left(\sf_1,R\right)
\right].
\end{align}

Then 
\begin{align}
&(\nicefrac{\hat\sigma\lambda}{2})\calR(\bfDelta_{\bfB}) + \Vert\frX^{(n)}(\hat\bfB)-\boldf^{(n)}\Vert_2^2
\le D^2 + \spadesuit_1(F,\hat R),\label{thm:improved:rate:q=1':eq1}\\
&\Vert\frX^{(n)}(\bfDelta_{\bfB})\Vert_2
\le 2D + \clubsuit_1(F,\hat R), \label{thm:improved:rate:q=1':eq2}\\
&\Vert\bfDelta_{\bfB}\Vert_\Pi\le (\nicefrac{1}{\sd_1})
[2D + \clubsuit_1(F,\hat R)] 
+ (\nicefrac{2\sd_2}{\sd_1\hat\sigma\lambda})[D^2 + \spadesuit_1(F,\hat R)]. 
\label{thm:improved:rate:q=1':eq3}
\end{align}
\end{theorem}

Theorem \ref{thm:improved:rate:q=1'} has a similar format to Theorem \ref{thm:improved:rate} but with a few differences. The first obvious one is that constraints \eqref{cond5:general:norm}-\eqref{cond6:general:norm} are excluded. Indeed, Theorem \ref{thm:improved:rate:q=1'} is not concerned with matrix decomposition ($\sa=\infty$, $\sf_3=\sd_3=\sb_3=\sf_*=0$, $\calS\equiv0$). The main difference is that $(\hat\sigma\lambda,\hat\sigma\tau)$ in condition (iv) of Theorem \ref{thm:improved:rate:q=1'} substitutes $(\lambda,\tau)$ in condition (iv) of Theorem \ref{thm:improved:rate}. By Bernstein's inequality, $\sigma/2\le\hat\sigma\le3\sigma/2$ with high probability --- assuming $n\gtrsim\sigma^2(1+\log(1/\delta))$. This justifies why the tuning 
$(\lambda,\tau)$ in Theorem \ref{thm:response:sparse-low-rank:regression:q=1} is adaptive to $\sigma$. Another difference is that the constraints \eqref{cond4:general:norm:q=1'} and \eqref{cond8:general:norm:q=1'}-\eqref{cond9:general:norm:q=1'} require the constants $(\sc_n,\sc_*)$ to be $\calO(1) r_{n,d_{\eff},\delta}(\rho,\mu_*)$ --- recall \eqref{def:r:n}. This justifies why the approximation error in Theorem \ref{thm:response:sparse-low-rank:regression:q=1} must be 
$\calO(\sigma) r_{n,d_{\eff},\delta}(\rho,\mu_*)$ instead of $\calO(\sigma)$ as in Theorem \ref{thm:response:sparse-low-rank:regression}. 

From the previous discussion, it is not surprising that the road map of the proofs of Theorem \ref{thm:improved:rate:q=1'} and Theorem \ref{thm:response:sparse-low-rank:regression:q=1} are, respectively, an adaptation of the proofs of Theorem \ref{thm:improved:rate} and Theorem \ref{thm:response:sparse-low-rank:regression}. The details are left to Section 33
%\ref{s:proof:thm:improved:rate:q=1'} 
in the supplement. For completeness, we present a brief description of the proof of Theorem \ref{thm:improved:rate:q=1'}. To facilitate, we use as reference the proof of Theorem \ref{thm:improved:rate} in Section \ref{s:proof:main:paper} and give pointers to the changes in the supplement. 
\begin{enumerate}
%%%%%%%%%
\item Lemma 31
%\ref{lemma:recursion:1st:order:condition:q=1'} 
in Section 33
%\ref{s:proof:thm:improved:rate:q=1'} 
is a variation of Lemma \ref{lemma:recursion:1st:order:condition}. Inequality (90)
%\eqref{lemma:recursion:1st:order:condition:q=1':eq2} 
in Lemma 31
%\ref{lemma:recursion:1st:order:condition:q=1'} 
is similar to \eqref{lemma:recursion:1st:order:condition:eq} but with the addition of the error term
\begin{align}
\left(\Vert\frM^{(n)}(\bfDelta_{\bfB},\bfDelta^{\hat\btheta})\Vert_2+\Vert\calE_{\bfB}^{(n)}\Vert_2\right)\left(
\lambda\calR(\bfDelta_{\bfB})+\tau\Vert\bfDelta^{\hat\btheta}\Vert_\sharp
\right).\label{eq:error:term:augmented}  
\end{align}
In above, $\calE_{\bfB}:=\frX(\bfB)-\boldf$ denotes the miss-specification error. The term above appears because we use the adaptive loss 
$\calL_{\tau\bomega,1}$. To address this term, Lemma 31
%\ref{lemma:recursion:1st:order:condition:q=1'} 
states the auxiliary inequality (89)
%\eqref{lemma:recursion:1st:order:condition:q=1':eq1}
. 
%%%%%%%%
\item Lemma 32
%\ref{lemma:MP+ATP:q=1'} 
in Section 33
%\ref{s:proof:thm:improved:rate:q=1'} 
is a variation of Lemma \ref{lemma:MP+ATP}. Inequality (92)
%\eqref{lemma:MP+ATP:q=1':eq2} 
in Lemma 32
%\ref{lemma:MP+ATP:q=1'} 
is similar to \eqref{lemma:MP+ATP:eq} but with the addition of the error term \eqref{eq:error:term:augmented}.\footnote{Because it handles matrix decomposition, inequality \eqref{lemma:MP+ATP:eq} states a recursion in the variable $\Vert[\bfDelta_{\bfB},\bfDelta_{\bfGamma},\bfDelta^{\hat\btheta}]\Vert_{\Pi}$. The presence of the error term \eqref{eq:error:term:augmented} justifies why inequality (92)
%\eqref{lemma:MP+ATP:q=1':eq2} 
states a recursion in the variable $\Vert\frM^{(n)}(\bfDelta^{\hat\bfB},\bfDelta^{\hat\btheta})\Vert_2$.} Examining Lemma 32
%\ref{lemma:MP+ATP:q=1'}
, we note that the miss-specification errors 
$\lambda\Vert\calE_{\bfB}^{(n)}\Vert_2$ and $\tau\Vert\calE_{\bfB}^{(n)}\Vert_2$ impact the choice of $(\lambda,\tau)$. This fact implies the constant $\sc_n$ to be of the order of the minimax rate. Like Lemma 31
%\ref{lemma:recursion:1st:order:condition:q=1'}
, Lemma 32
%\ref{lemma:MP+ATP:q=1'} 
states the auxiliary inequality (91)
%\eqref{lemma:MP+ATP:q=1':eq1}
.
%%%%%%%%
\item Lemma 32
%\ref{lemma:MP+ATP:q=1'} 
entails Proposition 8
%\ref{prop:suboptimal:rate:q=1'} 
in Section 33
%\ref{s:proof:thm:improved:rate:q=1'}
, a variation of Proposition \ref{prop:suboptimal:rate}.
%%%%%%%%%
\item Lemma 33
%\ref{lemma:recursion:delta:bb:general:norm:q=1'} 
in Section 33
%\ref{s:proof:thm:improved:rate:q=1'} 
is a variation of Lemma 19
%\ref{lemma:recursion:delta:bb:general:norm}
. Inequality (102)
%\eqref{lemma:recursion:delta:bb:general:norm:q=1':eq2} 
in Lemma 33
%\ref{lemma:recursion:delta:bb:general:norm:q=1'} 
is analogous to \eqref{lemma:recursion:delta:bb:general:norm:eq} 
with the addition of the error term
\begin{align}
\left(\Vert\frX^{(n)}(\bfDelta_{\bfB})+\Vert\bfDelta^{\hat\btheta}\Vert_2
+\Vert\calE_{\bfB}^{(n)}\Vert_2\right)\lambda\calR(\bfDelta_{\bfB}).\label{eq:error:term}  
\end{align}
To address it, Lemma 33
%\ref{lemma:recursion:delta:bb:general:norm:q=1'} 
states the auxiliary inequality (101)
%\eqref{lemma:recursion:delta:bb:general:norm:q=1':eq1}
.
%%%%%%%%%
\item Finally, Lemma 34
%\ref{lemma:MP+IP+ATP:q=1'} 
in Section 33
%\ref{s:proof:thm:improved:rate:q=1'} 
is a variation of Lemma \ref{lemma:MP+IP+ATP}. 
The major difference is the addition of the term \eqref{eq:error:term} in inequality (105)
%\eqref{lemma:MP+IP+ATP:q=1':eq2} 
of Lemma 34
%\ref{lemma:MP+IP+ATP:q=1'}
.\footnote{Inequality \eqref{lemma:MP+IP+ATP:eq} states a recursion in the variable $\Vert[\bfDelta_{\bfB},\bfDelta_{\bfGamma}]\Vert_{\Pi}$ while (105)
%\eqref{lemma:MP+IP+ATP:q=1':eq2} 
states a recursion in the variable $\Vert\frX^{(n)}(\bfDelta^{\hat\bfB})\Vert_2$.} Consequently, the corruption error 
$\lambda\Vert\bfDelta^{\hat\btheta}\Vert_2$ and the miss-specification error 
$\lambda\Vert\calE_{\bfB}^{(n)}\Vert_2$ affect the choice of $\lambda$. This is why the constants $(\sc_n,\sc_*)$ need to be of the order of the minimax rate. To address \eqref{eq:error:term}, Lemma 34
%\ref{lemma:MP+IP+ATP:q=1'} 
states the auxiliary inequality (104)
%\eqref{lemma:MP+IP+ATP:q=1':eq1}
.
\end{enumerate}
The proof of Theorem \ref{thm:improved:rate:q=1'} follows from Lemma 34
%\ref{lemma:MP+IP+ATP:q=1'} 
and Proposition 8
%\ref{prop:suboptimal:rate:q=1'}
 in the supplement. Proposition 8
 %\ref{prop:suboptimal:rate:q=1'} 
 is invoked to give precise bounds on the nuisance error $\bfDelta^{\hat\btheta}$.

\section{Simulation results}\label{s:simulation}
We report simulation results in \texttt{R} with synthetic data demonstrating agreement between theory and practice. The code is in \href{https://github.com/philipthomp/Outlier-robust-regression}{https://github.com/philipthomp/Outlier-robust-regression}. We simulate with standard gaussian design entries and gaussian noise. In experiments j) and k) below, $\sigma=0.1$; in the others,we use 
$\sigma=1$. The purpose of our experiments is threefold. The first is to identify empirically the guarantees of Theorems \ref{thm:tr:reg:matrix:decomp} and \ref{thm:response:sparse-low-rank:regression} for the estimators \eqref{equation:aug:slope:rob:estimator:general} and  \eqref{equation:aug:slope:rob:estimator:q=2} with $q=2$. Namely, the linear growth of the mean standard error 
($\sqrt{\texttt{MSE}}$) with respect to $\epsilon$ and of the mean squared error ($\texttt{MSE}$) with respect to $(r,s)$.  Secondly, we wish to verify empirically the superiority of ``sorted'' Huber regression to classical Huber regression. Thirdly, we compare  robust estimators with non-robust regularized estimators. Our solvers use a batch version of an alternated proximal gradient method on the separable variables $[\bfB,\bfGamma,\btheta]$.\footnote{The proximal map of the Slope or $\ell_1$ norms are computed with the function \texttt{sortedL1Prox()} of the \texttt{SLOPE} R package \cite{2015bogdan:berg:sabatti:su:candes}. The 
($\ell_\infty$-constrained) proximal map of the nuclear norm is computed via ($\ell_\infty$-constrained) soft-thresholding of the singular value decomposition.} When using the Slope norm, we always set $A = \bar A = 10$ in the sequences $\bomega$ and $\bar\bomega$. ``Huber'' and ``Sorted Huber'' denote Huber regression and Sorted-Huber regression respectively. Due to lack of space, we leave to future work numerical results for the robust estimator \eqref{equation:aug:slope:rob:estimator:q=1} regarding adaptation to $\sigma$. We implement the following experiments:

\begin{enumerate}
\item[a)] We conduct Sorted-Huber sparse linear regression with estimator \eqref{equation:aug:slope:rob:estimator:q=2} ($q=2$) with $p=100$, $n=1000$, $s = \{15, 25, 35\}$ and 100 repetitions. We take $\calR$ to be the Slope norm in $\re^p$. Respectively, $\bb^*$ and $\btheta^*$ are set with the first $s$ and $o$ entries equal to $10$ and zero otherwise. See Figure \ref{fig.robust.sparse.linear.reg}(a).
\item[b)] Same set-up of a) but with $s=25$ and parameter coordinates with modulus $50$. This time, we compare ``Huber'' and ``Sorted Huber''. See Figure \ref{fig.robust.linear.reg:HS}. The first 25 entries estimated by ``Huber'' and ``Sorted Huber'' fluctuated around $40$ and $48$ respectively.
\item[c)] Same set-up of b) but with $\bb^*$ and $\btheta^*$ having non-zero coordinates set to $1$ and $1000$ respectively. We compare \eqref{equation:aug:slope:rob:estimator:q=2} ($q=2$) with sparse regression with Slope norm regularization (``Slope reg''). See Figure  \ref{fig.robust.sparse.linear.reg}(b).
\item[d)] We conduct Sorted-Huber low-rank trace regression with estimator \eqref{equation:aug:slope:rob:estimator:q=2} ($q=2$) with $d_1=d_2=10$, $n=1000$, $r = \{1, 2, 3\}$ and 50 repetitions. We take $\calR$ to be the nuclear norm. The low-rank parameter is generated by randomly choosing the spaces of left and right singular vectors with all nonzero singular values equal to $B:=10$. The corruption vector $\btheta^*$ is set to have the first $o$ entries equal to $M:=10$ and zero otherwise. See Figure \ref{fig.robust.low-rank.linear.reg}(a).
\item[e)] Same set-up of d) but with $r=5$, $B=M=100$. This time, we compare ``Huber'' and ``Sorted Huber''. See Figure \ref{fig.robust.low-rank.linear.reg}(b).
\item[f)] Same set-up of e) but with $B=100$ and $M=1000$. We compare estimator \eqref{equation:aug:slope:rob:estimator:q=2} ($q=2$) with trace regression with nuclear norm regularization (``Nuclear norm''). See Figure  \ref{fig.robust.low-rank.linear.reg}(c).
\item[g)] We conduct non-corrupted trace regression with additive matrix decomposition ($\epsilon=0$) using estimator \eqref{equation:aug:slope:rob:estimator:general} with $(\tau,\btheta)=(0,\mathbf{0})$, $d_1=d_2=10$, $n=1000$, $s = \{5, 80\}$ and 20 repetitions for varying $r$. We take $\calR$ and $\calS$ to be the nuclear norm and $\ell_1$-norm respectively. The low-rank parameter is generated by randomly choosing the spaces of left and right singular vectors such that $\sa^*=1$. The sparse matrix parameter is simulated with the non-zero entries of value $G:=10$ chosen uniformly at random. See Figure \ref{fig.trace-reg-MD}(a).
\item[h)] Similar set-up of h) but for $r=5$ and varying $s$. See Figure \ref{fig.trace-reg-MD}(b).
\item[i)] We conduct Sorted-Huber trace regression with additive decomposition with estimator  \eqref{equation:aug:slope:rob:estimator:general} for varying $\epsilon$ and three different cases: 
$(r,s)=(1,5)$, $(r,s)=(3,5)$ and $(r,s)=(5,5)$. The corruption entries are set to $M=1$. The other configurations, except for $(\tau,\btheta)\neq(0,\mathbf{0})$,  are the same as in g). See Figure \ref{fig.robust.trace-reg-MD}(a).
\item[j)] Same set-up of i) but with $\sigma=0.1$, $M=0.5$ and $(r,s)=(1,5)$. This time we compare ``Huber'' and ``Sorted Huber''. See Figure \ref{fig.robust.trace-reg-MD}(b).
\item[k)] Same set-up of j). We compare estimator \eqref{equation:aug:slope:rob:estimator:general} with trace regression regularized with nuclear norm plus the $\ell_1$-norm (``Low-rank + sparse reg''). See Figure \ref{fig.robust.trace-reg-MD}(c).
\end{enumerate}
The plots identify the linear growth expected from the theory and the superiority of Sorted Huber regression in comparison to Huber regression and non-robust methods. 

\begin{figure}
\hspace*{\fill}
\subcaptionbox{}{\includegraphics[scale=0.27]{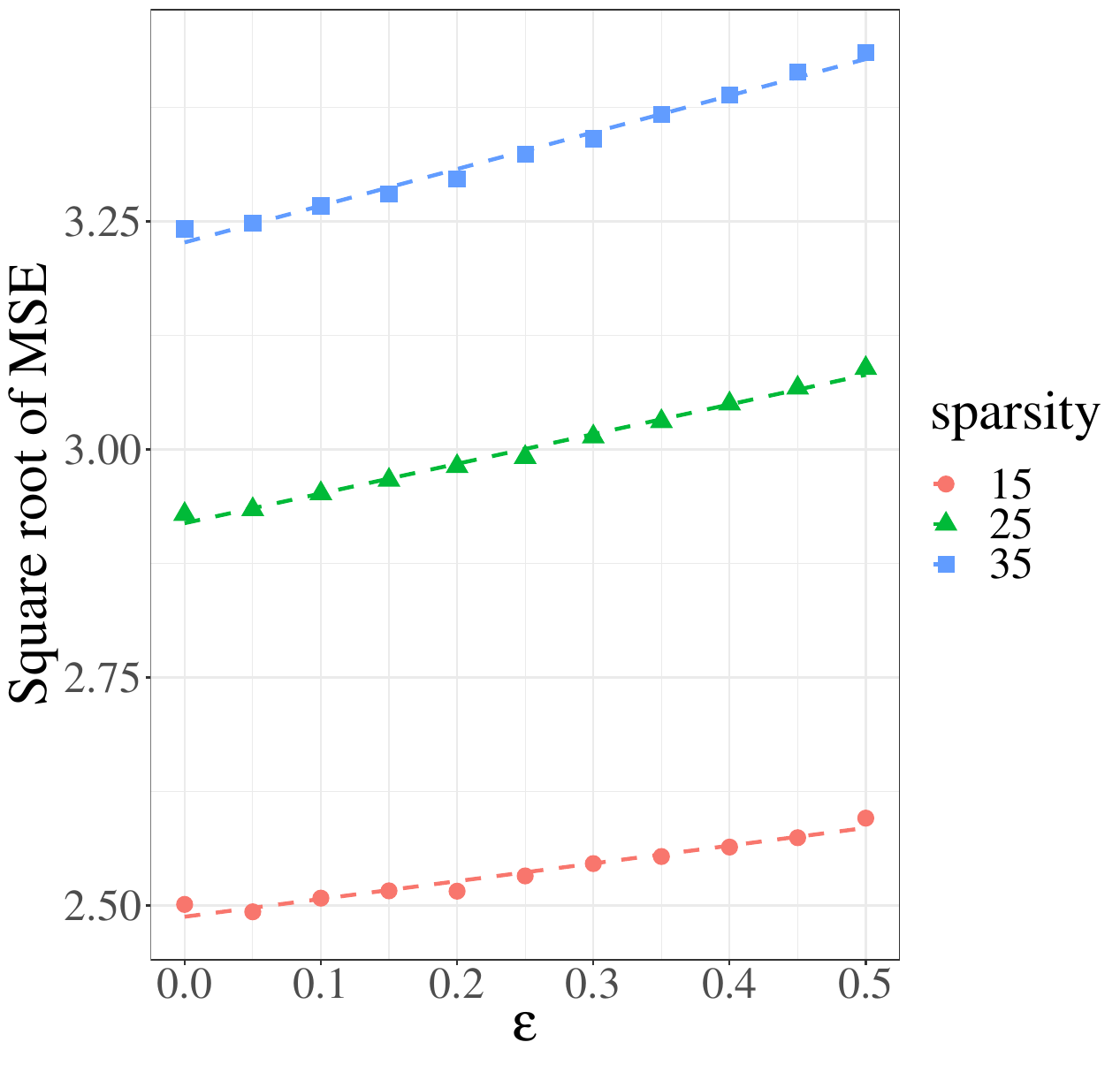}}
\hfill
\subcaptionbox{}{\includegraphics[scale=0.32]{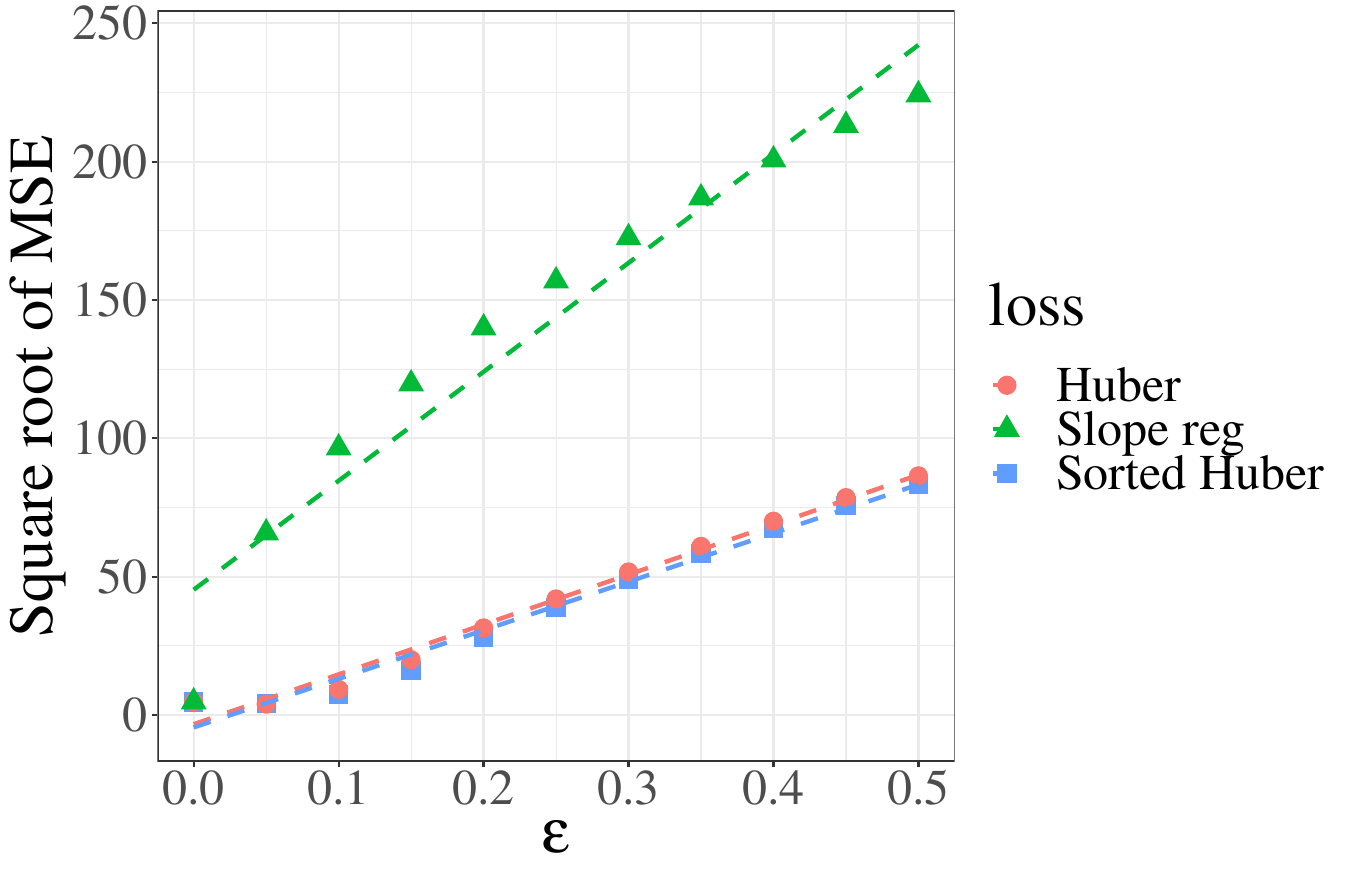}}
\hspace*{\fill}
\caption{Robust sparse linear regression:
different levels of sparsity (a) and comparisons between methods (b).}\label{fig.robust.sparse.linear.reg}
\end{figure}

\begin{figure}
%\hspace*{\fill}
\subcaptionbox{}{\includegraphics[scale=0.22]{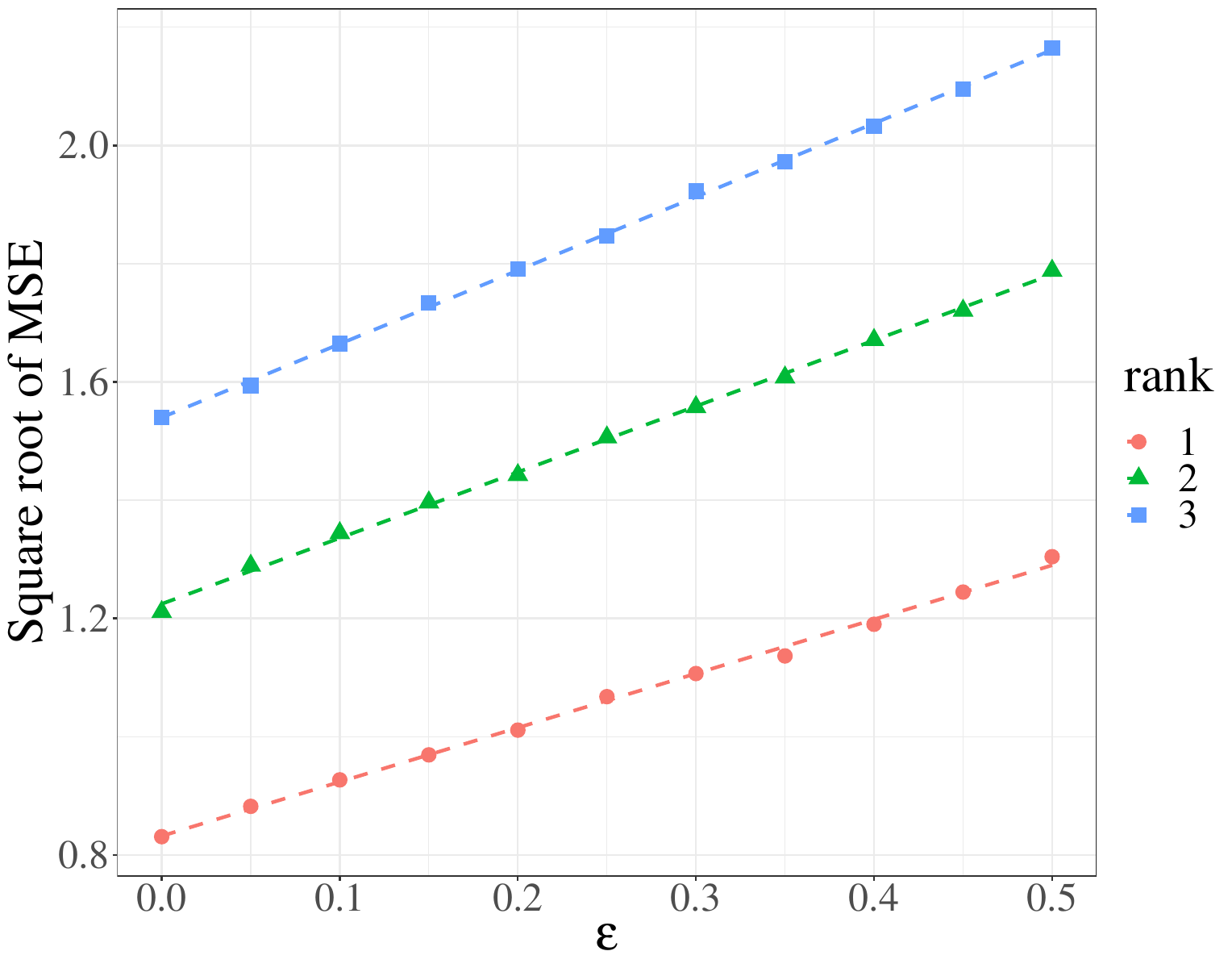}}
%\hfill
\subcaptionbox{}{\includegraphics[scale=0.3]{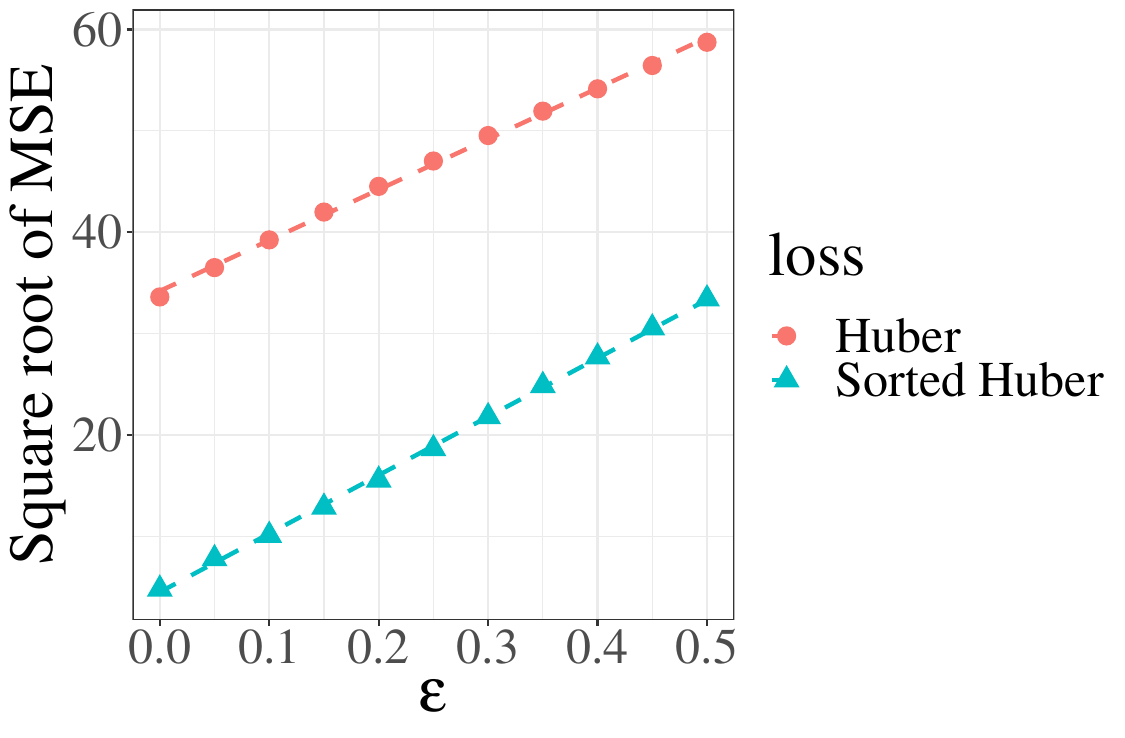}}
%\hfill
\subcaptionbox{}{\includegraphics[scale=0.28]{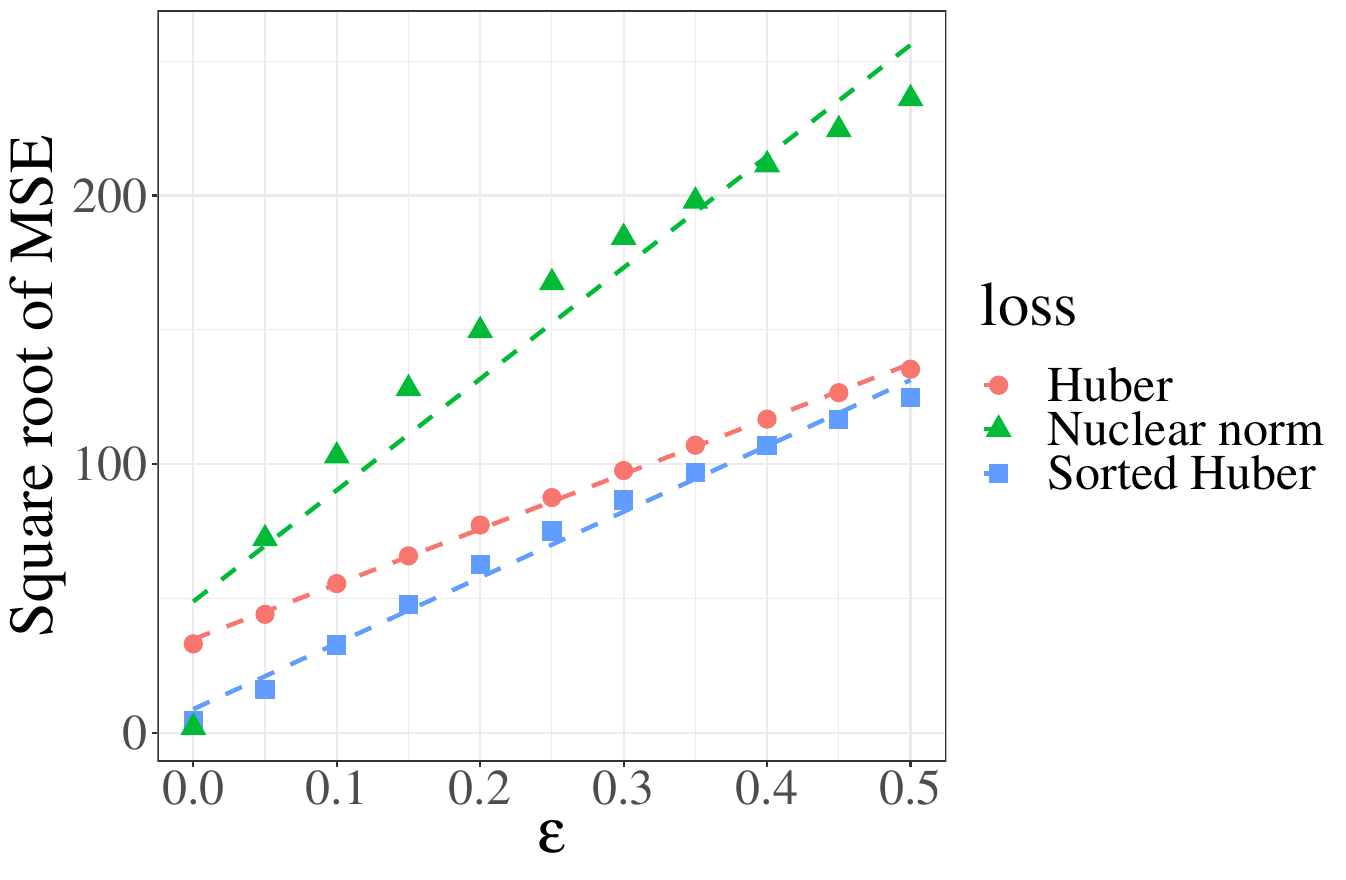}}
%\hspace*{\fill}
\caption{Robust low-rank trace regression: different values of rank (a), and comparisons between methods (b,c).}\label{fig.robust.low-rank.linear.reg}
\end{figure}

\begin{figure}
\hspace*{\fill}%
\subcaptionbox{Fixed sparsity.}{\includegraphics[scale=0.27]{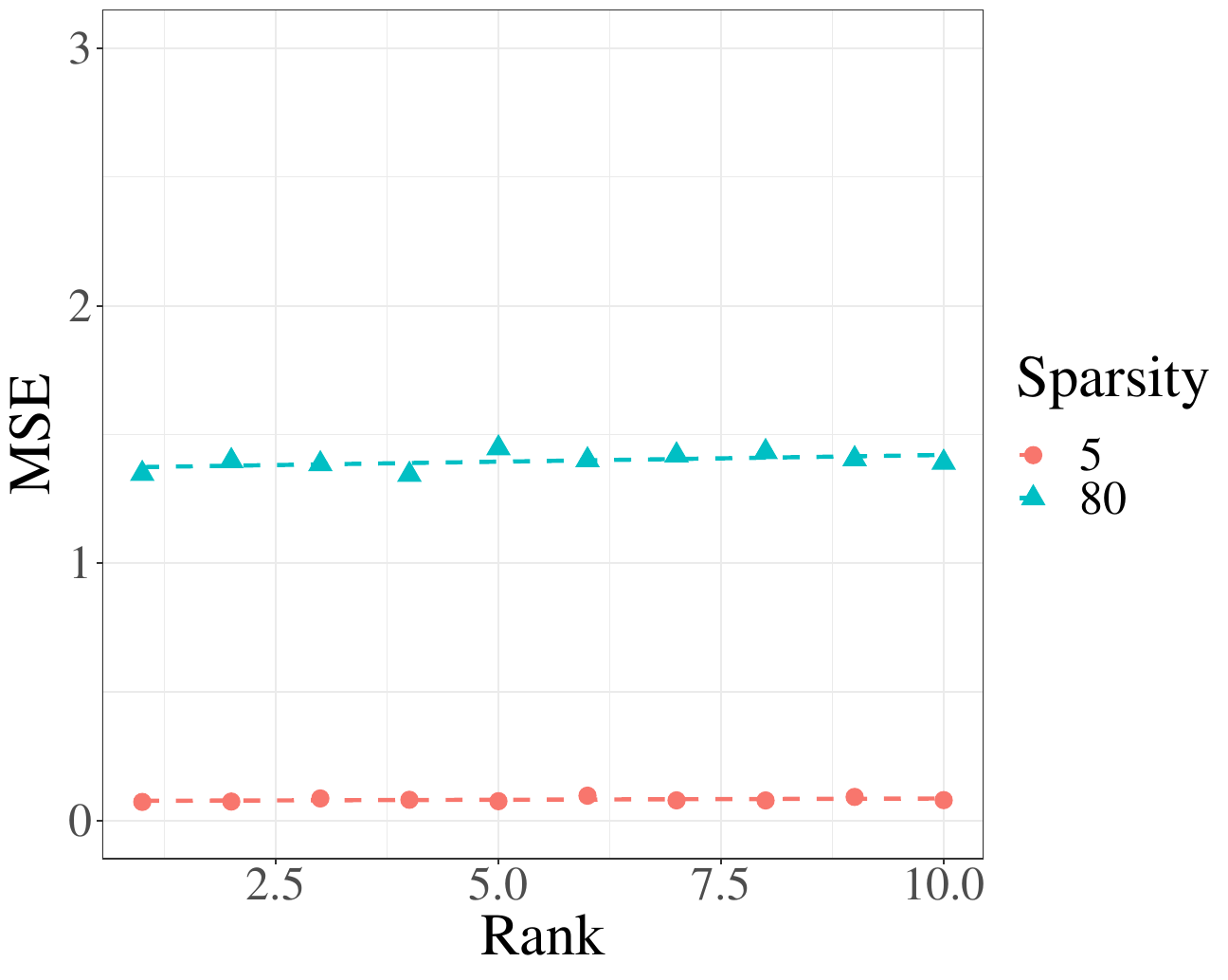}}\hfill%
\subcaptionbox{Fixed rank.}{\includegraphics[scale=0.3]{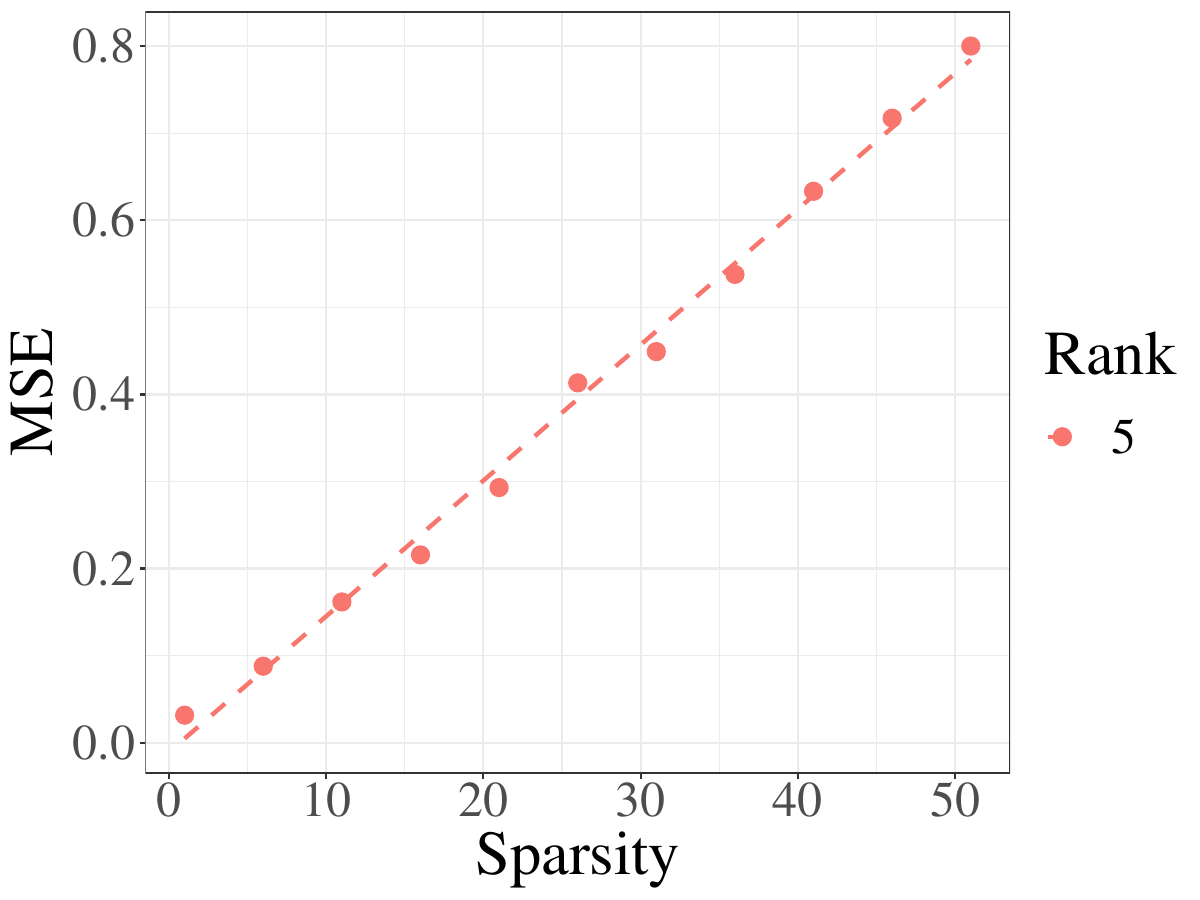}}
\hspace*{\fill}%
\caption{Trace regression with additive matrix decomposition with $\epsilon=0$}\label{fig.trace-reg-MD}
\end{figure}

\begin{figure}
%\hspace*{\fill}
\subcaptionbox{}{\includegraphics[scale=0.22]{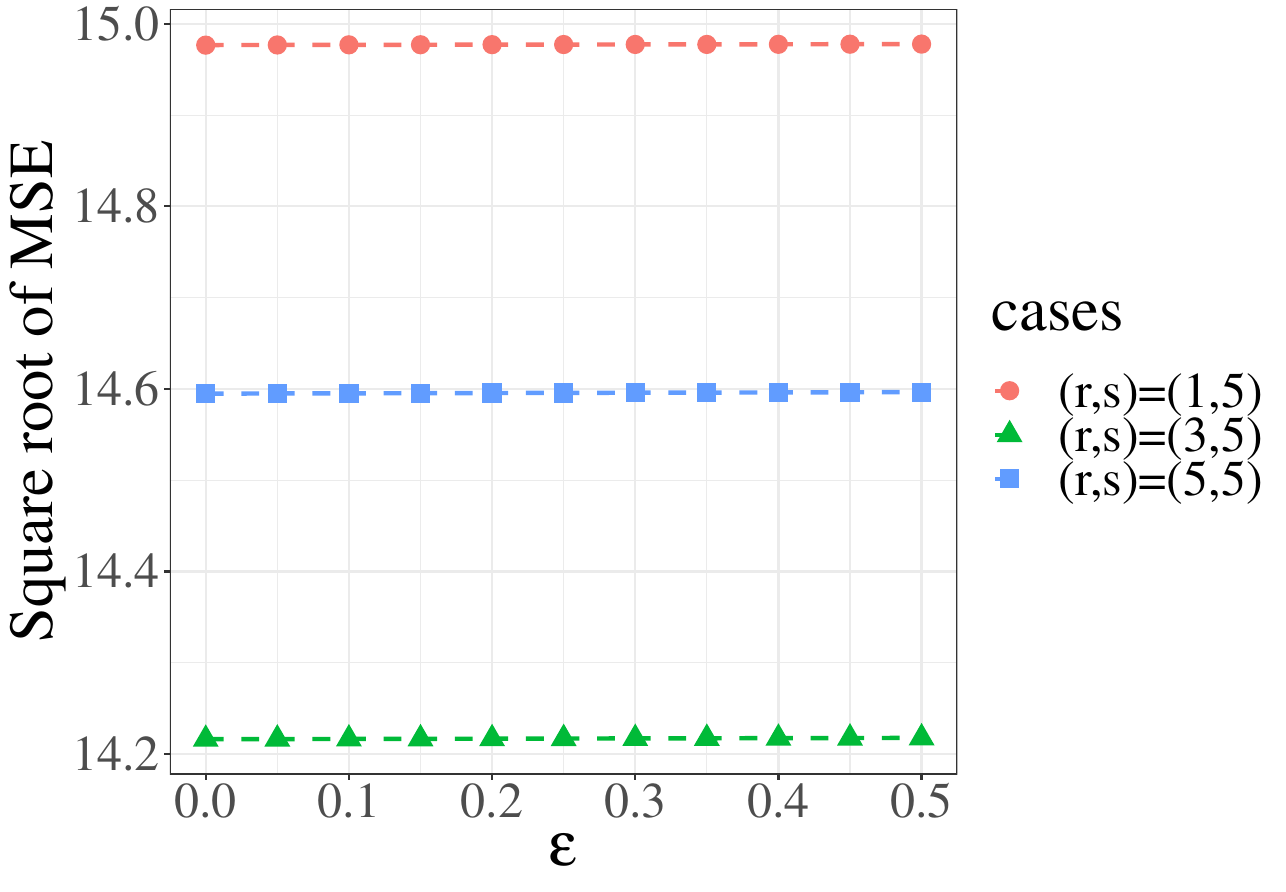}}
%\hfill
\subcaptionbox{}{\includegraphics[scale=0.3]{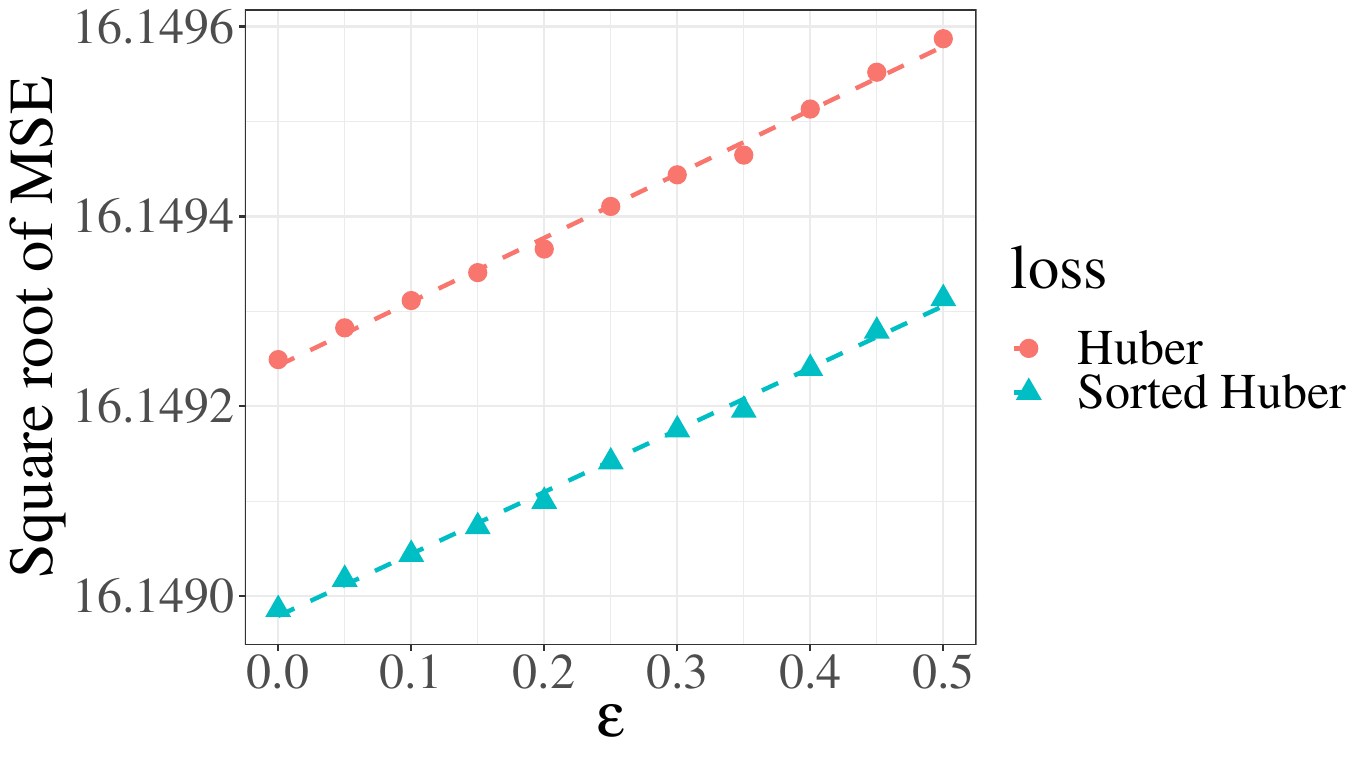}}
%\hfill
\subcaptionbox{}{\includegraphics[scale=0.28]{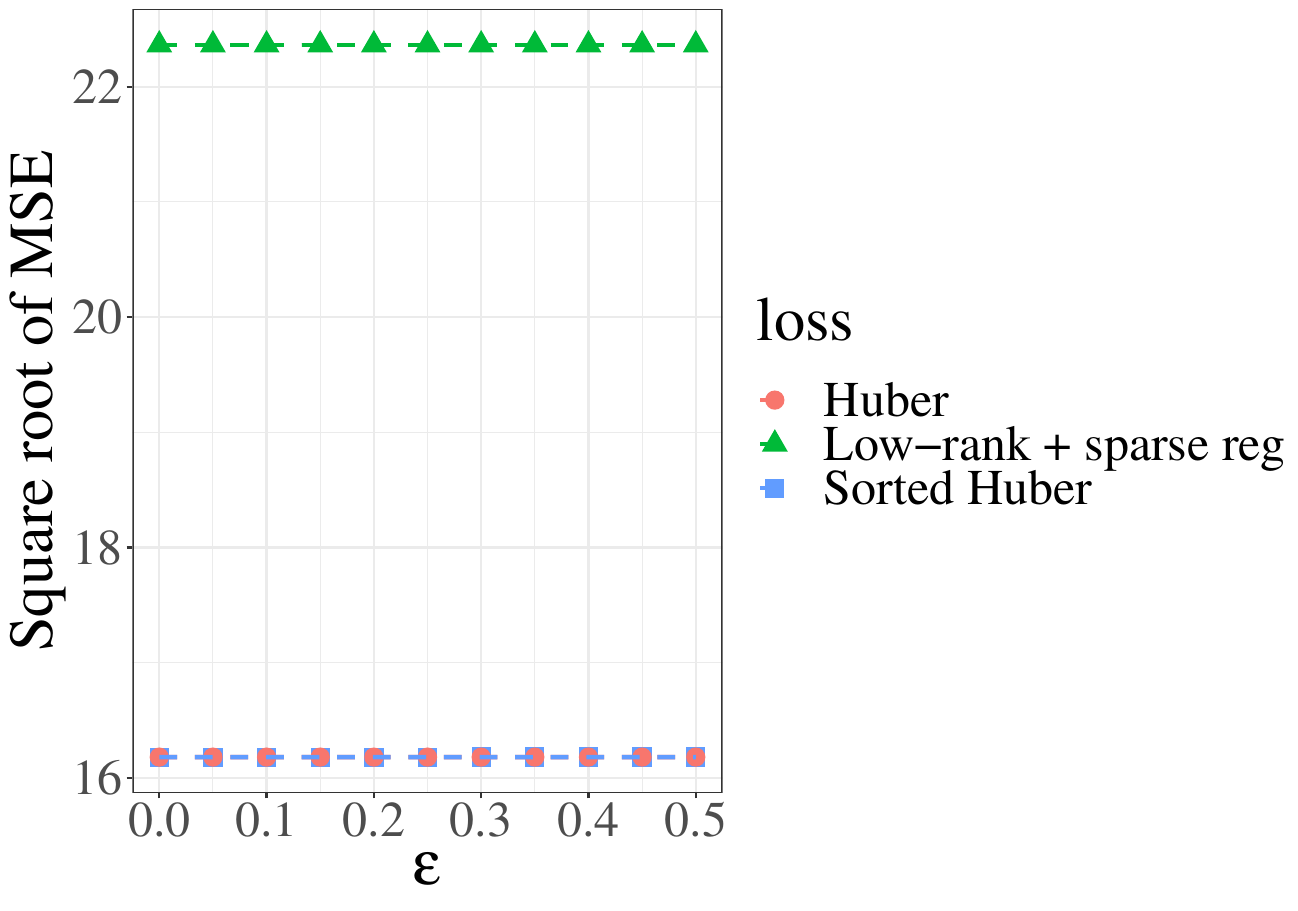}}
%\hspace*{\fill}
\caption{Robust trace regression with additive decomposition: different values of rank/sparsity (a), and comparisons between methods (b,c).}\label{fig.robust.trace-reg-MD}
\end{figure}

\section{Discussion}\label{s:discussion}
We present theoretical and empirical improvements  resorting to ``sorted'' variations of the classical Huber's loss. Instead of the $\ell_1$ norm, the loss in Definition \ref{def:sorted:Huber:loss} is based on a variational formula regularized by the slope norm --- treating each label outlier magnitude individually. This use of the slope norm refines an existing robust loss and should be contrasted to \cite{2015bogdan:berg:sabatti:su:candes,2018bellec:lecue:tsybakov}. There, the slope norm is used as a refined regularization norm when estimating the sparse parameter.\footnote{While the estimator \eqref{equation:aug:slope:rob:estimator:q=2} of $[\bb^*,\btheta^*]$ can be written as a slope estimator with the augmented design $[\mbX^{(n)} -\bfI_n]$, this point of view is not enough to entail the optimal rate. Indeed, 
$\IP$ is crucially needed to entail optimality when estimating only $\bb^*$. See Remarks \ref{rem:relevance:PP+IP}-\ref{rem:relevance:IP} and comments after Proposition \ref{prop:suboptimal:rate}.} More generally, this work identifies three design properties, 
$\PP$, $\IP$ and $\MP$, which jointly entail that the robust estimator \eqref{equation:sorted:Huber:general} is near-optimal in terms of dimension, $\epsilon$ and $\delta$ --- when dealing simultaneously with matrix decomposition, label contamination, featured-dependent noise and miss-specification. These properties are based on sharp concentration inequalities stated in Section \ref{s:multiplier:process:main}. We believe these could be useful elsewhere, e.g., non-parametric least squares regression \cite{2019han:wellner,2022kuchibhotla:patra} and compressive sensing theory \cite{2015dirksen}. We reemphasize that there seems to be no prior estimation theory for RTRMD --- with $\PP$ being a new application of the product process inequality. Under the incoherence assumption, it would be interesting to investigate further when exact recovery is possible or how nonconvex optimization approaches behave on the model RTRMD \cite{2020chen:fan:ma:yan-nonconvex}.

{\footnotesize
\bibliography{references_AoS.bib}}

\renewcommand*\contentsname{Summary of the paper}
\tableofcontents

\newpage

\begin{tcolorbox}
\large\textbf{Supplementary Material} 
\end{tcolorbox}

In this supplement we give detailed proofs of the results stated in the main text. It is organized as follows:

\begin{itemize}
\item The supplement starts proving Lemmas \ref{lemma:TP} and \ref{lemma:ATP} and Proposition \ref{proposition:properties:subgaussian:designs} in Section 
\ref{s:subgaussian:properties} of the main text --- see Sections \ref{proof:lemma:TP} to \ref{s:prop:MP}. Auxiliary lemmas are recalled in Section \ref{s:decomposable:norms}. 
\item We then pass to the proof of Theorem \ref{thm:improved:rate} in Section 
\ref{s:proof:main:paper} of the main text. The auxiliary lemmas are proved in order from Sections \ref{s:proof:lemma:recursion:1st:order:condition} to \ref{s:proof:lemma:MP+IP+ATP}. The proof of Theorem \ref{thm:improved:rate} is concluded in Section \ref{s:proof:thm:improved:rate:supp}. 
\item The proof of Theorem \ref{thm:tr:reg:matrix:decomp} in Section
\ref{s:main:result:1} of the main text is given Section \ref{proof:thm:tr:reg:matrix:decomp}. It follows from Theorem \ref{thm:improved:rate} and Proposition \ref{proposition:properties:subgaussian:designs}. The lower bound --- Proposition \ref{prop:lower:bound:tr:matrix:decomp} is proved in Section \ref{s:proof:prop:lower:bound:tr:matrix:decomp}.
\item The proof of Theorem \ref{thm:response:sparse-low-rank:regression} in Section \ref{s:main:result:1} in the main text is given in Sections \ref{s:proof:thm:response:sparse-low-rank:regression:i} to \ref{s:thm:response:sparse-low-rank:regression:iii}. It also follows Theorem \ref{thm:improved:rate} and Proposition \ref{proposition:properties:subgaussian:designs}. 
\item The proof of Theorem \ref{thm:improved:rate:q=1'} in Section \ref{s:proof:main:paper:q=1} of the main text is presented in Section \ref{s:proof:thm:improved:rate:q=1'}. This theorem and Proposition \ref{proposition:properties:subgaussian:designs} entail Theorem \ref{thm:response:sparse-low-rank:regression:q=1} in Section \ref{s:main:result:1} of the main text --- its proof is given in Sections \ref{s:proof:thm:response:sparse-low-rank:regression:q=1:i'} to \ref{s:proof:thm:response:sparse-low-rank:regression:q=1:iii'}.  
\item The proofs of Theorems \ref{thm:mult:process} and \ref{thm:product:process} in Section \ref{s:multiplier:process:main} of the main text and needed peeling lemmas are presented in the Appendix.
\end{itemize}

Unless otherwise stated, $C>0$ and $c\in(0,1)$ will denote universal constants that may change within the text.

\section{Proof of Lemma \ref{lemma:TP}}\label{proof:lemma:TP}
From $\PP_{\calR,\calR}(\alpha_1,\alpha_2,\alpha_3,\alpha_4)$, 
\begin{align}
|\Vert\frX^{(n)}(\bfV)\Vert_2^2-\Vert\bfV\Vert_\Pi^2|
&\le \alpha_1\Vert\bfV\Vert_\Pi^2 
+(\alpha_2+\alpha_3)\calR(\bfV)\Vert\bfV\Vert_\Pi
+\alpha_4\calR^2(\bfV)\\
&\le \Vert\bfV\Vert_\Pi^2\left(\alpha_1 + (\nicefrac{\alpha^2}{2})\right)
+\calR^2(\bfV)\left(
\alpha_4 + \frac{(\alpha_2+\alpha_3)^2}{2\alpha^2}
\right).
\end{align} 
Let $0<\alpha<\sqrt{2(1-\alpha_1)}$. Rearranging and taking the square-root, it follows that $\TP_{\calR}(\sa_1,\sa_2)$ holds with constants
$
\sa_1:=\{1-\alpha_1-(\nicefrac{\alpha^2}{2})\}^{1/2}
$
and
$
\sa_2:=\{
\frac{(\alpha_2+\alpha_3)^2}{2\alpha^2} + \alpha_4
\}^{1/2}.
$
The proof is finished once we set $\alpha^2:=(1-\alpha_1)/2$.

\section{Proof of Lemma \ref{lemma:ATP}}
Let $\alpha:=(\sa_1\wedge\bar\sa_1)/2\sqrt{2}$ and 
$\beta:=\sb_1+\sc_1$. By assumption $b:=\beta+2\alpha^2<(\sa_1\wedge\bar\sa_1)^2$ and 
$\sd_1^2=(\sa_1\wedge\bar\sa_1)^2-b$. From $\TP$, 
\begin{align}
\sd_1^2\Vert[\bfV,\bfW,\bu]\Vert_{\Pi}^2 &\le
\sa_1^2\Vert\bfV\Vert_{\Pi}^2 + \bar\sa_1^2\Vert\bfW\Vert_{\Pi}^2 + \Vert\bu\Vert_2^2
-b\Vert[\bfV,\bfW,\bu]\Vert_{\Pi}^2\\
&\le \left(\Vert\frX^{(n)}(\bfV)\Vert_2+\sa_2\calR(\bfV)\right)^2
+ \left(\Vert\frX^{(n)}(\bfW)\Vert_2+\bar\sa_2\calS(\bfW)\right)^2 +\Vert\bu\Vert_2^2\\
&-b\Vert[\bfV,\bfW,\bu]\Vert_{\Pi}^2,  
\end{align}
so, after taking the square-root,\footnote{Here we used the relation 
$
\sqrt{(A+B)^2+(\bar A+\bar B)^2+C^2}
\le \sqrt{A^2+\bar A^2+C^2} + B + \bar B
$
for positive numbers $(A,\bar A,B, \bar B, C)$. 
} 
\begin{align}
\sd_1\Vert[\bfV,\bfW,\bu]\Vert_{\Pi}&\le
\left\{
\Vert\frX^{(n)}(\bfV)\Vert_2^2 + \Vert\frX^{(n)}(\bfW)\Vert_2^2
+ \left(\Vert\bu\Vert_2^2 - b\Vert[\bfV,\bfW,\bu]\Vert_{\Pi}^2\right)_+
\right\}^{\frac{1}{2}} \\
&+\sa_2\calR(\bfV)+\bar\sa_2\calS(\bfW). 
\label{lemma:ATP:eq1}
\end{align}
In what follows, it is enough to consider the case 
$
\Vert\bu\Vert_2^2 \ge b\Vert[\bfV,\bfW,\bu]\Vert_{\Pi}^2.
$

Taking the squares, 
\begin{align}
\Vert\frX^{(n)}(\bfV)\Vert_2^2 + \Vert\frX^{(n)}(\bfW)\Vert_2^2 + \Vert\bu\Vert_2^2 & = \Vert\frX^{(n)}(\bfV+\bfW) + \bu\Vert_2^2\\
&- 2\langle\frX^{(n)}(\bfV),\frX^{(n)}(\bfW)\rangle 
- 2\langle\frX^{(n)}(\bfV+\bfW),\bu\rangle.
\end{align}

By $\IP$ and Young's inequality,
\begin{align}
T_1:=-2\langle\frX^{(n)}(\bfV+\bfW),\bu\rangle &\le 
2\sb_1\left\Vert[\bfV,\bfW]\right\Vert_{\Pi}
\Vert\bu\Vert_2
+2\sb_2\calR(\bfV)\Vert\bu\Vert_2\\
&+2\sb_3\calS(\bfV)\Vert\bu\Vert_2
+2\sb_4\left\Vert[\bfV,\bfW]\right\Vert_\Pi\calQ(\bu)\\
&\le (\sb_1+\alpha^2)(\Vert[\bfV,\bfW]\Vert_{\Pi}^2+\Vert\bu\Vert_2^2)
+\frac{\sb_2^2}{\alpha^2}\calR^2(\bfV)\\
&+\frac{\sb_3^2}{\alpha^2}\calS^2(\bfW)
+\frac{\sb_4^2}{\alpha^2}\calQ^2(\bu).
\end{align}

Let $T_2:=-2\langle\frX^{(n)}(\bfV),\frX^{(n)}(\bfW)\rangle$. By $\PP$ and Young's inequality, 
\begin{align}
T_2 + 2\llangle\bfV,\bfW\rrangle_{\Pi} &\le 
2\sc_1\left\Vert\bfV\right\Vert_{\Pi}\Vert\bfW\Vert_\Pi
+2\sc_2\calR(\bfV)\Vert\bfW\Vert_\Pi
+2\sc_3\left\Vert\bfV\right\Vert_\Pi\calS(\bfW)\\
&+2\gamma\sc_2\sc_3\calR(\bfV)\calS(\bfW)\\
&\le (\sc_1+\alpha^2)(\Vert\bfV\Vert_{\Pi}^2+\Vert\bfW\Vert_{\Pi}^2)\\
&+\left(\frac{\sc_2^2}{\alpha^2}+\gamma\sc_2^2\right)\calR^2(\bfV)
+\left(\frac{\sc_3^2}{\alpha^2}+\gamma\sc_3^2\right)\calS^2(\bfW).
\end{align}

Let $T_0:=\Vert\frX^{(n)}(\bfV)\Vert_2^2 + \Vert\frX^{(n)}(\bfW)\Vert_2^2 + \Vert\bu\Vert_2^2$. We thus conclude that 
\begin{align}
T_0 &\le \Vert\frM^{(n)}(\bfV+\bfW,\bu)\Vert_2^2 -2\llangle\bfV,\bfW\rrangle_{\Pi}\\
&+(\sb_1+\sc_1+2\alpha^2)\Vert[\bfV,\bfW]\Vert_{\Pi}^2
+(\sb_1+\alpha^2)\Vert\bu\Vert_2^2\\
&+\left(\frac{\sb_2^2}{\alpha^2}+\frac{\sc_2^2}{\alpha^2}+\gamma\sc_2^2\right)\calR^2(\bfV)
+\left(\frac{\sb_3^2}{\alpha^2}+\frac{\sc_3^2}{\alpha^2}+\gamma\sc_3^2\right)\calS^2(\bfW)
+\frac{\sb_4^2}{\alpha^2}\calQ^2(\bu). 
\label{lemma:ATP:eq2}
\end{align}

From \eqref{lemma:ATP:eq1}-\eqref{lemma:ATP:eq2} and definitions of 
$(\beta,b)$, 
\begin{align}
\sd_1\Vert[\bfV,\bfW,\bu]\Vert_{\Pi}&\le
\left\{
T_0 - b\Vert[\bfV,\bfW,\bu]\Vert_{\Pi}^2
\right\}^{\frac{1}{2}}
+\sa_2\calR(\bfV)+\bar\sa_2\calS(\bfW)\\
&\le\left\{
\Vert\frM^{(n)}(\bfV+\bfW,\bu)\Vert_2^2 -2\llangle\bfV,\bfW\rrangle_{\Pi}
\right\}_+^{\frac{1}{2}} \\
&+ \left(
\left\{\frac{\sb_2^2}{\alpha^2}+\frac{\sc_2^2}{\alpha^2}+\gamma\sc_2^2\right\}^{\frac{1}{2}}
+\sa_2
\right)\calR(\bfV)
+\left(
\left\{\frac{\sb_3^2}{\alpha^2}+\frac{\sc_3^2}{\alpha^2}+\gamma\sc_3^2\right\}^{\frac{1}{2}}
+\bar\sa_2
\right)\calS(\bfW)\\
&+\frac{\sb_4}{\alpha}\calQ(\bu).
\end{align}

\section{Proof of Proposition \ref{proposition:properties:subgaussian:designs}, item (i)}
\label{s:prop:PP}

We start with the following lemma. 
\begin{lemma}\label{lemma:prod:process}
Suppose that $\bfX$ is $L$-subgaussian. Let $\calB_1$ and $\calB_2$ be bounded subsets of 
$\mbB_\Pi$. For any $n\ge1$ and $t\ge1$, with probability at least $1-e^{-t}$, it holds that
\begin{align}
\sup_{[\bfV,\bfW]\in \calB_1\times\calB_2}\left|\llangle\bfV,\bfW\rrangle_n-\llangle\bfV,\bfW\rrangle_\Pi\right|&\le \frac{CL^2}{n}\mathscr{G}\big(\frS^{1/2}(\calB_1))\mathscr{G}\big(\frS^{1/2}(\calB_2))\\
&+ \frac{CL^2}{\sqrt{n}}\left[\mathscr{G}\big(\frS^{1/2}(\calB_1)+\mathscr{G}\big(\frS^{1/2}(\calB_2)\right]\\
&+CL^2\left(\frac{t}{n}+\sqrt{\frac{t}{n}}\right).
\end{align}
\end{lemma}
Lemma \ref{lemma:prod:process} is immediate from Theorem \ref{thm:product:process} applied to the classes $F:=\{\bfV\in\calB_1:\llangle\cdot,\bfV\rrangle\}$ and $G:=\{\bfW\in\calB_2:\llangle\cdot,\bfW\rrangle\}$ and Talagrand's majorizing theorems \cite{2014talagrand}.\footnote{By these theorems, 
$\gamma_2(F)\asymp L\mathscr{G}\big(\frS^{1/2}(\calB_1))$. We also note that 
$\bar\Delta(F)\lesssim L$.} The next proposition is a restatement of item (i) of Proposition \ref{proposition:properties:subgaussian:designs}. We prove it using Lemma \ref{lemma:prod:process} and the peeling Lemma \ref{lemma:peeling:product:process} in Appendix \ref{Peeling lemmas}. 
\begin{proposition}[$\PP$]\label{prop:PP}
Suppose that $\bfX$ is $L$-subgaussian. For all $\delta\in(0,1)$
and $n\in\mathbb{N}$, with probability at least
$1-\delta$, the following property holds: for all
$[\bfV,\bfW]\in(\mdR^p)^2$,
\begin{align}
\left|\llangle\bfW,\bfV\rrangle_n-\llangle\bfW,\bfV\rrangle_\Pi\right|&\le 
CL^2\left(
\frac{1+\log(1/\delta)}{n} + \frac{1+\sqrt{\log(1/\delta)}}{\sqrt{n}}
\right)\Vert\bfV\Vert_\Pi\Vert\bfW\Vert_\Pi\\
&+CL^2\left(\frac{1}{n}+\frac{1}{\sqrt{n}}\right)\mathscr G\left(\calR(\bfV)\frS^{1/2}(\mbB_\calR)\cap\Vert\bfV\Vert_\Pi\mbB_F\right)
\Vert\bfW\Vert_\Pi\\
&+CL^2\left(\frac{1}{n}+\frac{1}{\sqrt{n}}\right)\mathscr G\left(\calS(\bfW)\frS^{1/2}(\mbB_\calS)\cap\Vert\bfW\Vert_\Pi\mbB_F\right)\big\|\bfV\big\|_\Pi
\\
&+\frac{CL^2}{n}
\mathscr G\left(\calR(\bfV)\frS^{1/2}(\mbB_\calR)\cap\Vert\bfV\Vert_\Pi\mbB_F\right)
\cdot
\mathscr G\left(\calS(\bfW)\frS^{1/2}(\mbB_\calS)\cap\Vert\bfW\Vert_\Pi\mbB_F\right).
\end{align}
\end{proposition}
\begin{proof}
Let $r,\bar r>0$ and define the sets
\begin{align}
V_1:=\{\bfV:\Vert\bfV\Vert_\Pi\le 1,\calR(\bfV)\le r\},\quad
V_{2}:=\{\bfW:\Vert\bfW\Vert_\Pi\le1,\calS(\bfW)\le \bar r\}.
\end{align}
Note that, 
\begin{align}
\mathscr{G}\left(\frS^{1/2}(V_1)\right)&\le r\mathscr G\left(\frS^{1/2}(\mbB_\calR)\cap r^{-1}\mbB_F\right)=:g(r),\\
\mathscr{G}\left(\frS^{1/2}(V_2)\right)&\le \bar r\mathscr G\left(\frS^{1/2}(\mbB_\calS)\cap \bar r^{-1}\mbB_F\right)=:\bar g(\bar r).
\end{align}

By Lemma \ref{lemma:prod:process}, for any $r,\bar r>0$ and $\delta\in(0,1/c]$, we have with probability at least $1-c\delta$,
\begin{align}
\sup_{[\bfV,\bfW]\in V_1\times V_2}|\llangle\bfW,\bfV\rrangle_n-\llangle\bfW,\bfV\rrangle_\Pi|
&\le \frac{CL^2}{n}\cdot g(r)\bar g(\bar r) + \frac{CL^2}{\sqrt{n}}[g(r) + \bar g(\bar r)]\\
&+CL^2\left(\sqrt{\frac{\log(1/\delta)}{n}}+\frac{\log(1/\delta)}{n}\right).
\end{align}

We now invoke Lemma \ref{lemma:peeling:product:process} with the set $V:=\mbB_\Pi\times\mbB_\Pi$, functions
\begin{align}
M(\bfV,\bfW)&:=-\left|\llangle\bfW,\bfV\rrangle_n-\llangle\bfW,\bfV\rrangle_\Pi\right|, \\
h(\bfV,\bfW)&:=\calR(\bfV)\\
\bar h(\bfV,\bfW)&:=\calS(\bfW), 
\end{align}
functions $g$ and $\bar g$ as stated above and constant $b:=CL^2$. The claim follows from such lemma and the homogeneity of norms.
\end{proof}

\section{Proof of Proposition \ref{proposition:properties:subgaussian:designs}, item (ii)}
\label{s:prop:TP}

By item (i) of Proposition \ref{proposition:properties:subgaussian:designs},    with probability$\ge1-\delta$, $\PP_{\calR,\calR}$ holds with constants
\begin{align}
\sc_1 &= CL^2\left(
\frac{1+\log(1/\delta)}{n} + \frac{1+\sqrt{\log(1/\delta)}}{\sqrt{n}}
\right),\\
\sc_2=\sc_3 &= CL^2\left(\frac{1}{n}+\frac{1}{\sqrt{n}}\right)\mathscr G\left(\frS^{1/2}(\mbB_\calR)\right),\\
\sc_4 &= \frac{CL^2}{n}
\mathscr G^2\left(\frS^{1/2}(\mbB_\calR)\right). 
\end{align}
The claim follows from this and Lemma \ref{lemma:TP} with constant 
$\alpha^2=(1-\sc_1)/2$ --- noting that, by assumption $\sc_1\in(0,1)$.

\section{Proof of Proposition \ref{proposition:properties:subgaussian:designs}, item (iii)}
\label{s:prop:IP}

We start with the following lemma, stating a high-probability version of Chevet's inequality. This result is suggested as an exercise in Vershynin \cite{2018vershynin}. We give a proof for completeness.
\begin{lemma}\label{lemma:chevet}
Suppose that $\bfX$ is $L$-subgaussian. Let $V$ be any bounded subset of 
$\mbB_\Pi\times\mbB_2^{n}$ Define $V_1 := \{\bfV:\exists\, \bu \text{ s.t. } (\bfV,\bu)\in V\}$
and $V_2 := \{\bu:\exists\, \bfV \text{ s.t. } (\bfV,\bu)\in V\}$.

Then, for any $n\ge1$ and $t>0$, with probability at least $1-2\exp(-t^2)$,
$$
\sup_{[\bfV,\bu]\in V}\langle\bu,\frX(\bfV)\rangle \le  CL[\mathscr G\big(\frS^{1/2}(V_1)) + \mathscr G\big(V_2\big) + t].
$$
\end{lemma}
\begin{proof}
For each $(\bfV,\bu)\in V$, we define
\begin{align}
Z_{\bfV,\bu}&:=
\langle\bu,\frX(\bfV)\rangle=\sum_{i\in[n]}\bu_i\llangle\bfX_i,\bfV\rrangle,\qquad
W_{\bfV,\bu}:= L(\llangle\bfV,\frS^{1/2}(\bfXi)\rrangle+\langle\bu,\bxi\rangle),	
\end{align}
where $\bfXi\in\mdR^{p}$ and $\bxi\in\re^n$ are independent each one having iid $\calN(0,1)$ entries. Therefore,
$(\bfV,\bu)\mapsto W_{\bfV,\bu}$ defines a centered Gaussian process indexed by $V$.

We may easily bound the $\psi_2$-norm of the increments using rotation invariance of sub-Gaussian random variables. Indeed,  using that $\{\bfX_i\}$ is an iid sequence and Proposition 2.6.1  in \cite{2018vershynin}, given $[\bfV,\bu]$ and $[\bfV',\bu']$ in $V$, 
\begin{align}
|Z_{\bfV,\bu}-Z_{\bfV',\bu'}|_{\psi_2}^2&=\left|\sum_{i\in[n]}\llangle\bfX_i,\bu_i\bfV-\bu'_i\bfV'\rrangle\right|_{\psi_2}^2\\
&\le C\sum_{i\in[n]}\left|\llangle\bfX_i,\bu_i\bfV-\bu_i'\bfV'\rrangle\right|_{\psi_2}^2\\
&\le2C\sum_{i\in[n]}\left|\llangle\bfX_i,(\bu_i-\bu_i')\bfV\rrangle\right|^2_{\psi_2}+2C\sum_{i\in[n]}\left|\llangle\bfX_i,\bu_i'(\bfV-\bfV')\rrangle\right|_{\psi_2}^2\\
&\le2CL^2\Vert\bu-\bu'\Vert_2^2\Vert\bfV\Vert_\Pi^2
+2CL^2\Vert\bu'\Vert_2^2\Vert\bfV-\bfV'\Vert_\Pi^2\le2CL^2\dist([\bfV,\bu],[\bfV',\bu']),
\label{lemma:aux2:eq1}
\end{align}
with the pseudo-metric $\dist([\bfV,\bu],[\bfV',\bu']):=\sqrt{\Vert\bu-\bu'\Vert_2^2+\Vert\bfV-\bfV'\Vert_\Pi^2}$, using that 
$\Vert\bfV\Vert_\Pi\le1$ and $\Vert\bu'\Vert_2\le1$.  On the other hand, by definition of the process $W$ it is easy to check that 
\begin{align}
\esp[(W_{\bfV,\bu}-W_{\bfV',\bu'})^2]=L^2(\Vert \bfV-\bfV'\Vert_\Pi^2+\Vert \bu-\bu'\Vert_2^2).
\label{lemma:aux2:eq2}
\end{align}	

From \eqref{lemma:aux2:eq1},\eqref{lemma:aux2:eq2}, we conclude that the processes $W$ and $Z$ satisfy the conditions of Talagrand's majoration and minoration generic chaining bounds for sub-Gaussian processes (e.g. Theorems 8.5.5 and 8.6.1 in \cite{2018vershynin}). Hence, for any $t\ge0$, with probability at least $1-2e^{-t^2}$,
\begin{align}
\sup_{[\bfV,\bu]\in V}|Z_{\bfV,\bu}|\le CL\left\{\esp\left[\sup_{[\bfV,\bu]\in V}W_{\bfV,\bu}\right]+t\right\}.\label{lemma:aux2:eq3}
\end{align}
In above we used that $Z_{\bfV_0,\bu_0}=0$ at $[\bfV_0,\bu_0]=0$ and that the diameter of $V\subset\mbB_\Pi^{m_1\times m_2}\times\mbB_2^{n}$ under the metric 
$\dist$ is less than $2\sqrt{2}$. We also have
\begin{align}
\esp\bigg[\sup_{[\bfV,\bu]\in V}
W_{\bfV,\bu}\bigg]\le \esp\bigg[\sup_{\bfV\in V_1}\llangle\bfXi,\frS^{1/2}(\bfV)\rrangle\bigg]+
\esp\bigg[\sup_{\bu\in V_2}\langle\bu,\bxi\rangle\bigg]=
\mathscr G\big(\frS^{1/2}(V_1))+\mathscr G(V_2).
\end{align}
Joining the two previous inequalities complete the proof of the claimed inequality. 
\end{proof}

The next proposition is a restatement of item (iii) of Proposition \ref{proposition:properties:subgaussian:designs}. We prove it using Lemma \ref{lemma:chevet} and the peeling Lemma \ref{lemma:peeling:chevet} in Appendix \ref{Peeling lemmas}.
\begin{proposition}[$\IP$]\label{prop:gen:IP}
Suppose that $\bfX$ is $L$-subgaussian. For all 
$\delta\in(0,1)$ and $n\in\mathbb{N}$, with probability at least
$1-\delta$, the following property holds: for all
$[\bfV,\bfW,\bu]\in(\mdR^p)^2\times \re^n$,
\begin{align}
\langle\bu,\frX^{(n)}(\bfV+\bfW)\rangle &\le CL
\frac{1+\sqrt{\log(1/\delta)}}{\sqrt{n}}\|[\bfV,\bfW]\|_\Pi
\Vert\bu\Vert_2\\
&+CL\frac{\mathscr
G\big(\calR(\bfV)\frS^{1/2}(\mbB_\calR)\cap\|[\bfV,\bfW]\|_\Pi\mbB_F\big)}{\sqrt{n}}\Vert\bu\Vert_2\\
&+CL\frac{\mathscr
G\big(\calS(\bfW)\frS^{1/2}(\mbB_\calS)\cap\|[\bfV,\bfW]\|_\Pi\mbB_F\big)}{\sqrt{n}}\Vert\bu\Vert_2\\
&+CL\frac{\mathscr G\big(\calQ(\bu)\mathbb B_\calQ\cap \Vert\bu\Vert_2\mathbb B_2^n\big)}{\sqrt{n}}\|[\bfV,\bfW]\|_\Pi.
\end{align}
\end{proposition}
\begin{proof}
Let $R_1,R_2,R_3>0$ and define the sets
\begin{align}
V_1&:=\{\bfV\in\mdR^{p}:\Vert\bfV\Vert_\Pi\le 1,\calR(\bfV)\le R_1\},\label{prop:gen:IP:eq1}\\
V_2&:=\{\bu\in\re^n:\Vert\bu\Vert_2\le1,\calQ(\bu)\le R_2\},\label{prop:gen:IP:eq2}\\
V_3&:=\{\bfW\in\mdR^{p}:\Vert\bfW\Vert_\Pi\le 1,\calS(\bfV)\le R_3\}.\label{prop:gen:IP:eq3}
\end{align}

We note that
\begin{align}
\mathscr{G}\left(\frS^{1/2}(V_1)\right)&\le R_1\mathscr G\left(\frS^{1/2}(\mbB_\calR)\cap R_1^{-1}\mbB_F\right)=:g_1(R_1),\label{prop:gen:IP:eq4}\\
\mathscr{G}(V_{2})&\le R_2\mathscr{G}\left(\mathbb{B}_\calQ\cap R_2^{-1}\mathbb{B}_2^n\right)=:g_2(R_2),\label{prop:gen:IP:eq5}\\
\mathscr{G}\left(\frS^{1/2}(V_3)\right)&\le R_3\mathscr G\left(\frS^{1/2}(\mbB_\calS)\cap R_3^{-1}\mbB_F\right)=:g_3(R_3).
\label{prop:gen:IP:eq6}
\end{align}

By Lemma \ref{lemma:chevet}, we have that, for any $R_1,R_2>0$ and 
$\delta\in(0,1)$, with probability at least $1-\delta$, the following inequality holds:
\begin{align}
\sup_{[\bfV,\bu]\in V_1\times V_2}\langle\bu,\frX^{(n)}(\bfV)\rangle 
&\le  \frac{CL}{\sqrt{n}}g_1(R_1) + \frac{CL}{\sqrt{n}}g_2(R_2)+\frac{CL}{\sqrt{n}}\sqrt{\log(2/\delta)}.
\end{align}
Next, we invoke Lemma \ref{lemma:peeling:chevet} with the set $V:=\mbB_\Pi\times\mbB_2^n$, functions
\begin{align}
M(\bfV,\bu)&:=-\langle\bu,\frX^{(n)}(\bfV)\rangle,\\
h(\bfV,\bu)&:=\calR(\bfV)\\
\bar h(\bfV,\bu)&:=\calQ(\bu), 
\end{align}
functions $g:=g_1$ and $\bar g:=g_2$ and constant $b:=CL$. By this lemma, given $\delta\in(0,1)$, with probability at least $1-\delta$, for all $[\bfV,\bu]\in \mbB_\Pi\times\mbB_2^n$, 
\begin{align}
\langle\bu,\frX^{(n)}(\bfV)\rangle 
&\le  \frac{CL}{\sqrt{n}}\mathscr
G\left(\calR(\bfV)\frS^{1/2}(\mbB_\calR)\cap\mbB_F\right)\\
&+ \frac{CL}{\sqrt{n}}\mathscr
G\left(\calQ(\bu)\mbB_\calQ\cap\mbB_2^n\right)\\
&+\frac{CL}{\sqrt{n}}(1+\sqrt{\log(1/\delta)}).
\end{align}

Similarly, we will invoke Lemma \ref{lemma:chevet} with set $V_3\times V_2$ and Lemma \ref{lemma:peeling:chevet} with set $V:=\mbB_\Pi\times\mbB_2^n$, functions
\begin{align}
M(\bfW,\bu)&:=-\langle\bu,\frX^{(n)}(\bfW)\rangle,\\
h(\bfW,\bu)&:=\calS(\bfV)\\
\bar h(\bfW,\bu)&:=\calQ(\bu), 
\end{align}
functions $g:=g_3$ and $\bar g:=g_2$ and constant $b:=CL$. We obtain that, for any $\delta\in(0,1)$, with probability at least $1-\delta$, for all $[\bfW,\bu]\in \mbB_\Pi\times\mbB_2^n$, 
\begin{align}
\langle\bu,\frX^{(n)}(\bfW)\rangle 
&\le  \frac{CL}{\sqrt{n}}\mathscr
G\left(\calS(\bfW)\frS^{1/2}(\mbB_\calS)\cap\mbB_F\right)\\
&+ \frac{CL}{\sqrt{n}}\mathscr
G\left(\calQ(\bu)\mbB_\calQ\cap\mbB_2^n\right)\\
&+\frac{CL}{\sqrt{n}}(1+\sqrt{\log(1/\delta)}).
\end{align}

By an union bound, we obtain that, for every $\delta\in(0,1)$, with probability at least $1-\delta$, for all $[\bfV,\bfW,\bu]\in \mbB_\Pi\times\mbB_\Pi\times\mbB_2^n$, 
\begin{align}
\langle\bu,\frX^{(n)}(\bfV + \bfW)\rangle&\le 
\frac{CL}{\sqrt{n}}\mathscr
G\left(\calR(\bfV)\frS^{1/2}(\mbB_\calR)\cap\mbB_F\right)
+\frac{CL}{\sqrt{n}}\mathscr
G\left(\calS(\bfW)\frS^{1/2}(\mbB_\calS)\cap\mbB_F\right)\\
&+ \frac{2CL}{\sqrt{n}}\mathscr
G\left(\calQ(\bu)\mbB_\calQ\cap\mbB_2^n\right)\\
&+\frac{2CL}{\sqrt{n}}(1+\sqrt{\log(1/\delta)}).
\end{align}
To finish, we use that, for any $[\bfV,\bfW,\bu]$ with non-zero coordinates, 
$
\left[
\frac{\bfV}{\Vert[\bfV,\bfW]\Vert_\Pi}, 
\frac{\bfW}{\Vert[\bfV,\bfW]\Vert_\Pi}, 
\frac{\bu}{\Vert\bu\Vert_2}
\right]
$
belongs $\mbB_\Pi\times\mbB_\Pi\times\mbB_2^n$ and use homogeneity of norms. 
\end{proof}

\section{Proof of Proposition \ref{proposition:properties:subgaussian:designs}, item (iv)}
\label{s:prop:MP}

We start by stating two auxiliary lemmas. The next result follows from a tail symmetrization-contraction argument and the Gaussian concentration inequality. 
\begin{lemma}[Proposition 9.2 in \cite{2018bellec:lecue:tsybakov}]\label{lemma:noise:concentration}
Assume $\sigma:=|\xi|_{\psi_2}<\infty$. Let $U$ be any bounded subset of 
$\mbB^n_2$. For any $n\ge1$ and $t>0$, with probability at least $1-\exp(-t^2/2)$,
$$
\sup_{\bu\in U}\langle\bxi,\bu\rangle \le  C\sigma\left[\mathscr G\big(U\big) + t\right].
$$
\end{lemma}

Next, we state the following lemma. 
\begin{lemma}\label{lemma:mult:process}
Suppose that $\bfX$ is $L$-subgaussian. Let $V$ be a bounded subset of $\mbB_\Pi$. There exists universal constant $c>0$, such that for all $n\ge1$, $u,v\ge 1$, with probability at least $1-ce^{-u/4}-ce^{-nv}$,
\begin{align}
\sup_{\bfV\in V}\langle\bxi^{(n)},\frX^{(n)}(\bfV)\rangle &\le C
\left(\sqrt{v}+1\right)\frac{\sigma L}{\sqrt{n}}\mathscr G\big(\frS^{1/2}(V))
+C\sigma L\left(\sqrt{\frac{2u}{n}}+\frac{u}{n}
+\sqrt{\frac{uv}{n}}\right). 
\end{align}
\end{lemma}
The previous lemma is immediate from Theorem \ref{thm:mult:process} applied to the class $F:=\{\bfV\in\calB:\llangle\cdot,\bfV\rrangle\}$ and Talagrand's majorizing theorems.

Finally, the next proposition is a restatement of item (iv) of Proposition \ref{proposition:properties:subgaussian:designs}. We prove it using Lemmas \ref{lemma:noise:concentration} and \ref{lemma:mult:process} and two peeling lemmas: Lemma \ref{lemma:peeling:chevet} (with $\bar g=\bar h\equiv0$) and Lemma \ref{lemma:peeling:multiplier:process} in Appendix \ref{Peeling lemmas}. We recall the following definition:
\begin{align}
\triangle_{n}(\delta)&:=(\nicefrac{1}{\sqrt{n}})[1 + \sqrt{\log(1/\delta)}]
+(\nicefrac{1}{n})[1 + \log(1/\delta) + \sqrt{\log(1/\delta)}].  
\end{align}
Next, we also define the functions
\begin{align}
g_{\calR}(\bfV,\bfW,\bu)&:=\mathscr
G\left(\calR(\bfV)\frS^{1/2}(\mbB_\calR)\cap\Vert[\bfV,\bfW,\bu]\Vert_{\Pi}\mbB_F\right),\\
g_{\calS}(\bfV,\bfW,\bu)&:=\mathscr
G\left(\calS(\bfW)\frS^{1/2}(\mbB_\calS)\cap\Vert[\bfV,\bfW,\bu]\Vert_{\Pi}\mbB_F\right),\\
g_{\calQ}(\bfV,\bfW,\bu)&:=\mathscr
G\left(\calQ(\bu)\mbB_\calQ\cap\Vert[\bfV,\bfW,\bu]\Vert_\Pi\mbB_2^n\right).
\end{align}

\begin{proposition}[$\MP$]\label{prop:gen:MP}
For all $n\in\mathbb{N}$ and all $\delta\in(0,1)$, with probability at least $1-\delta$, for all $[\bfV,\bfW,\bu]\in \mbB_\Pi\times\mbB_\Pi\times\mbB_2^n$, 
\begin{align}
\langle\bxi^{(n)},\frM^{(n)}(\bfV+\bfW,\bu)\rangle 
&\le C\sigma L\cdot\triangle_n(\delta)\cdot\Vert[\bfV,\bfW,\bu]\Vert_\Pi\\
&+C\sigma L\left\{[1+(\nicefrac{1}{\sqrt{n}})\sqrt{\log(1/\delta)}]\frac{1}{\sqrt{n}}
+\frac{1}{n}\right\}g_{\calR}(\bfV,\bfW,\bu)\\
& +C\sigma L\left\{[1+(\nicefrac{1}{\sqrt{n}})\sqrt{\log(1/\delta)}]\frac{1}{\sqrt{n}}+\frac{1}{n}\right\}g_{\calS}(\bfV,\bfW,\bu)\\
&+C\frac{\sigma}{\sqrt{n}}g_{\calQ}(\bfV,\bfW,\bu).
\end{align}
\end{proposition}

\begin{proof}
Given $R_1,R_2,R_3>0$, we recall the definitions of the sets $V_1$, $V_2$ and $V_3$ in \eqref{prop:gen:IP:eq1}, \eqref{prop:gen:IP:eq2} and \eqref{prop:gen:IP:eq3} and functions $g_1$, $g_2$ and $g_3$ in \eqref{prop:gen:IP:eq4}, \eqref{prop:gen:IP:eq5} and \eqref{prop:gen:IP:eq6}.  

By Lemma \ref{lemma:mult:process}, there is constant $c\ge1$, such that, for any $R_1>0$ and 
$\delta\in(0,1/c]$, with probability at least $1-2c\delta$, 
\begin{align}
\sup_{\bfV\in V_1}\langle\bxi^{(n)},\frX^{(n)}(\bfV)\rangle 
&\le  C\left(\sqrt{\frac{\log(1/\delta)}{n}}+1\right)\frac{\sigma L}{\sqrt{n}}g_1(R_1)\\
& +C\sigma L\left(\sqrt{\frac{\log(1/\delta)}{n}}+\frac{\log(1/\delta)}{n}\right).
\end{align}
Next, we invoke Lemma \ref{lemma:peeling:multiplier:process} with the set $V:=\mbB_\Pi$, functions
\begin{align}
M(\bfV)&:=-\langle\bxi^{(n)},\frX^{(n)}(\bfV)\rangle,\\
h(\bfV)&:=\calR(\bfV), 
\end{align}
function $g:=g_1$ and constant $b:=C\sigma L$. By this lemma, given 
$\delta\in(0,1/2c]$, with probability at least $1-2c\delta$, for all $\bfV\in\mbB_\Pi$, 
\begin{align}
\langle\bxi^{(n)},\frX^{(n)}(\bfV)\rangle 
&\le C\left\{[1+(\nicefrac{1}{\sqrt{n}})\sqrt{\log(1/\delta)}]\frac{\sigma L}{\sqrt{n}}
+\frac{\sigma L}{n}\right\}\mathscr
G\left(\calR(\bfV)\frS^{1/2}(\mbB_\calR)\cap\mbB_F\right)\\
&+C(\nicefrac{\sigma L}{\sqrt{n}})[1 + \sqrt{\log(1/\delta)}]
+C(\nicefrac{\sigma L}{n})[\log(1/\delta) + \sqrt{\log(1/\delta)}].  
\end{align}

Proceeding exactly like above but with set $V_3$, norm $\calS$ and function $g:=g_3$, we get that for all $\delta\in(0,1/2c]$, with probability at least $1-2c\delta$, for all $\bfW\in\mbB_\Pi$, 
\begin{align}
\langle\bxi^{(n)},\frX^{(n)}(\bfW)\rangle 
&\le C\left\{[1+(\nicefrac{1}{\sqrt{n}})\sqrt{\log(1/\delta)}]\frac{\sigma L}{\sqrt{n}}
+\frac{\sigma L}{n}\right\}\mathscr
G\left(\calS(\bfW)\frS^{1/2}(\mbB_\calS)\cap\mbB_F\right)\\
&+C(\nicefrac{\sigma L}{\sqrt{n}})[1 + \sqrt{\log(1/\delta)}]
+C(\nicefrac{\sigma L}{n})[\log(1/\delta) + \sqrt{\log(1/\delta)}].  
\end{align}

Finally, by Lemma \ref{lemma:noise:concentration}, for any $R_2>0$ and $\delta\in(0,1)$, with probability at least $1-\delta$,
\begin{align}
\sup_{\bu\in V_2}\langle\bxi^{(n)},\bu\rangle \le  C\frac{\sigma}{\sqrt{n}}\left[g_2(R_2) + \sqrt{\log(1/\delta)}\right].
\end{align}
We now invoke Lemma \ref{lemma:peeling:chevet} with set $V:=\mbB_2^n$, functions
\begin{align}
M(\bu)&:=-\langle\bxi^{(n)},\bu\rangle,\\
h(\bu)&:=\calQ(\bu), 
\end{align}
function $g:=g_2$ (and $\bar g=\bar h\equiv0$) and constant $b:=C\sigma$. We get that, for any $\delta\in(0,1)$, with probability at least $1-\delta$, for all $\bu\in\mbB_2^n$, 
\begin{align}
\langle\bxi^{(n)},\bu\rangle \le  C\frac{\sigma}{\sqrt{n}}\mathscr
G\left(\calQ(\bu)\mbB_\calQ\cap\mbB_2^n\right) + C\frac{\sigma}{\sqrt{n}}\left[1 + \sqrt{\log(1/\delta)}\right].
\end{align}

By an union bound, we obtain that, for every $\delta\in(0,1/(4c+1)]$, with probability at least $1-(4c+1)\delta$, for all $[\bfV,\bfW,\bu]\in \mbB_\Pi\times\mbB_\Pi\times\mbB_2^n$, 
\begin{align}
\langle\bxi^{(n)},\frM^{(n)}(\bfV+\bfW,\bu)\rangle &=
\langle\bxi^{(n)},\frX^{(n)}(\bfV)\rangle
+\langle\bxi^{(n)},\frX^{(n)}(\bfW)\rangle
+\langle\bxi^{(n)},\bu\rangle\\
&\le C\left\{[1+(\nicefrac{1}{\sqrt{n}})\sqrt{\log(1/\delta)}]\frac{\sigma L}{\sqrt{n}}
+\frac{\sigma L}{n}\right\}\mathscr
G\left(\calR(\bfV)\frS^{1/2}(\mbB_\calR)\cap\mbB_F\right)\\
& +C\left\{[1+(\nicefrac{1}{\sqrt{n}})\sqrt{\log(1/\delta)}]\frac{\sigma L}{\sqrt{n}}+\frac{\sigma L}{n}\right\}\mathscr
G\left(\calS(\bfV)\frS^{1/2}(\mbB_\calS)\cap\mbB_F\right)\\
&+C\frac{\sigma}{\sqrt{n}}\mathscr
G\left(\calQ(\bu)\mbB_\calQ\cap\mbB_2^n\right)\\
&+3C(\nicefrac{\sigma L}{\sqrt{n}})[1 + \sqrt{\log(1/\delta)}]
+2C(\nicefrac{\sigma L}{n})[\log(1/\delta) + \sqrt{\log(1/\delta)}],
\end{align}
where we used that $L\ge1$. 

To finish, we use that, for any $[\bfV,\bfW,\bu]$ with non-zero coordinates, the vector
$$
\left[
\frac{\bfV}{\Vert[\bfV,\bfW,\bu]\Vert_\Pi}, 
\frac{\bfW}{\Vert[\bfV,\bfW,\bu]\Vert_\Pi}, 
\frac{\bu}{\Vert[\bfV,\bfW,\bu]\Vert_\Pi}
\right]
$$
belongs to $\mbB_\Pi\times\mbB_\Pi\times\mbS_2^n$ and use homogeneity of norms.
\end{proof}

\section{Lemmas for decomposable norms}\label{s:decomposable:norms}

Recall Definition \ref{def:decomposable:norm} in Section \ref{s:proof:main:paper}. We first remind the reader that the $\ell_1$ and nuclear norms are decomposable.
\begin{example}[$\ell_1$-norm]
Given $\bfB\in\mdR^{p}$ with \emph{sparsity support} 
$\mathscr S(\bfB):=\{[j,k]:\bfB_{j,k}\neq0\}$, the $\ell_1$-norm in 
$\mdR^{p}$ satisfies the above decomposability condition with the map
$
\bfV\mapsto\calP^\perp_{\bfB}(\bfV):=\bfV_{\calS(\bfB)^c}
$
where $\bfV_{\calS(\bfB)^c}$ denotes the $d_1\times d_2$ matrix whose entries are zero at indexes in $\mathscr S(\bfB)$. 
\end{example}
\begin{example}[Nuclear norm]
Let $\bfB\in\mdR^{p}$ with rank $r:=\rank(\bfB)$, singular values $\{\sigma_j\}_{j\in[r]}$ and singular vector decomposition $\bfB=\sum_{j\in[r]}\sigma_j\bu_j\bv_j^\top$. Here  $\{\bu_j\}_{j\in[r]}$ are the left singular vectors spanning the subspace $\calU$ and $\{\bv_j\}_{j\in[r]}$ are the right singular vectors spanning the subspace $\calV$. The pair $(\calU,\calV)$ is sometimes referred as the \emph{low-rank support} of $\bfB$. Given subspace $S\subset\re^\ell$ let  
$\bfP_{S^\perp}$ denote the matrix defining the orthogonal projection onto $S^\perp$. Then, the map
$
\bfV\mapsto\calP^\perp_{\bfB}(\bfV):=\bfP_{\calU^\perp}\bfV\bfP_{\calV^\perp}^\top
$
satisfy the decomposability condition for the nuclear norm $\Vert\cdot\Vert_N$. 
\end{example}

In the framework of regularized least-squares regression, decomposability is mostly useful because of the following lemma.  
\begin{lemma}[\cite{2012negahban:ravikumar:wainwright:yu}]\label{lemma:A1:B} 
Let $\calR$ be a decomposable norm over $\mdR^p$. Let $\bfB,\hat\bfB\in\mdR^{p}$ and 
$\bfV:=\hat\bfB-\bfB$. Then, for any $\nu\in[0,1]$, 
\begin{align}
\nu\calR(\bfV)+\calR(\bfB) - \calR(\hat\bfB)\le
(1+\nu)\calR(\calP_\bfB(\bfV)) -(1-\nu)\calR(\calP_\bfB^\perp(\bfV)).
\end{align}
\end{lemma}

Next, we state a well known lemma for the Slope norm that improves upon the previous lemma --- when comparing it with the $\ell_1$-norm. 
\begin{lemma}[\cite{2018bellec:lecue:tsybakov}]\label{lemma:A1}
Let $o\in[n]$, $\btheta,\hat\btheta\in\re^n$ such that $\Vert\btheta\Vert_0\le o$. Set 
$\bu:=\hat\btheta-\btheta$. Then
$
\|\btheta\|_\sharp-\|\hat\btheta\|_\sharp
\le\sum_{i=1}^o\omega_i\bu_i^\sharp-\sum_{i=o+1}^n\omega_i\bu_i^\sharp.
$
In particular, for any $\nu\in[0,1]$, 
\begin{align}
\nu\|\bu\|_\sharp+\|\btheta\|_\sharp - \|\hat\btheta\|_\sharp \le
(1+\nu)\sum_{i=1}^o\omega_i\bu_i^\sharp -(1-\nu)\sum_{i=o+1}^n\omega_i\bu_i^\sharp.
\end{align}
\end{lemma}

\section{Proof of Lemma \ref{lemma:recursion:1st:order:condition}}\label{s:proof:lemma:recursion:1st:order:condition}
The first order condition of \eqref{equation:aug:slope:rob:estimator:general} at 
$[\hat\bfB,\hat\bfGamma,\hat\btheta]$ is equivalent to the statement: there exist $\bfV\in\partial\calR(\hat\bfB)$, $\bfW\in\partial\calS(\hat\bfGamma)$ and $\bu\in\partial\Vert\hat\btheta\Vert_\sharp$ such that for all $[\bfB,\bfGamma,\btheta]$ such that $\Vert\bfB\Vert_\infty\le\sa$,
\begin{align}
\sum_{i\in[n]}\left[y_i^{(n)}-\frX^{(n)}_i(\widehat\bfB + \widehat\bfGamma)
-\hat\btheta_i\right]\llangle\bfX^{(n)}_i,\hat\bfB-\bfB\rrangle&\ge\lambda\llangle\bfV,\hat\bfB-\bfB\rrangle,\\
\sum_{i\in[n]}\left[y_i^{(n)}-\frX^{(n)}_i(\widehat\bfB + \widehat\bfGamma)
-\hat\btheta_i\right]\llangle\bfX^{(n)}_i,\hat\bfGamma-\bfGamma\rrangle&\ge\chi\llangle\bfW,\hat\bfGamma-\bfGamma\rrangle,\\
\langle\by^{(n)}-\frX^{(n)}(\hat\bfB+\hat\bfGamma)-\hat\btheta,\hat\btheta-\btheta\rangle&\ge\tau\langle\bu,\hat\btheta-\btheta\rangle.\label{equation:first:order:condition}
\end{align}
Setting $\btheta=\btheta^*$ and using that 
$
\by^{(n)} = \boldf^{(n)} + \btheta^* + \bxi^{(n)},
$
we obtain, for $[\bfB,\bfGamma]$ such that $\Vert\bfB\Vert_\infty\le\sa$,
\begin{align}
\sum_{i\in[n]}\left[ \bfDelta_i^{(n)} +\bfDelta_i^{\hat\btheta}\right]\llangle\bfX^{(n)}_i,\bfDelta_{\bfB}\rrangle&\le\sum_{i\in[n]}\xi_i^{(n)}\llangle\bfX_i^{(n)},\bfDelta_{\bfB}\rrangle-\lambda\llangle\bfV,\bfDelta_{\bfB}\rrangle,\\
\sum_{i\in[n]}\left[ \bfDelta_i^{(n)} +\bfDelta_i^{\hat\btheta}\right]\llangle\bfX^{(n)}_i,\bfDelta_{\bfGamma}\rrangle&\le\sum_{i\in[n]}\xi_i^{(n)}\llangle\bfX_i^{(n)},\bfDelta_{\bfGamma}\rrangle-\chi\llangle\bfV,\bfDelta_{\bfGamma}\rrangle,\\
\left\langle \bfDelta^{(n)} +\bfDelta^{\hat\btheta},\bfDelta^{\hat\btheta}\right\rangle
&\le\langle\bxi^{(n)},\bfDelta^{\hat\btheta}\rangle - \tau\langle\bu,\bfDelta^{\hat\btheta}\rangle.
\end{align}
Summing the above inequalities, 
\begin{align}
\langle \bfDelta^{(n)} +\bfDelta^{\hat\btheta}, \frM^{(n)}(\bfDelta_{\bfB}+\bfDelta_{\bfGamma},\bfDelta^{\hat\btheta})\rangle &\le
\langle\bxi^{(n)},\frM^{(n)}(\bfDelta_{\bfB}+\bfDelta_{\bfGamma},\bfDelta^{\hat\btheta})\rangle\\
&-\lambda\llangle\bfV,\bfDelta_{\bfB}\rrangle -\chi\llangle\bfV,\bfDelta_{\bfGamma}\rrangle - \tau\langle\bu,\bfDelta^{\hat\btheta}\rangle\\
&\le\langle\bxi^{(n)},\frM^{(n)}(\bfDelta_{\bfB} + \bfDelta_{\bfGamma},\bfDelta^{\hat\btheta})\rangle\\
&+\lambda \big(\calR(\bfB) - \calR(\hat\bfB)\big) 
+\chi\big(\calS(\bfGamma) - \calS(\hat\bfGamma)\big)
+ \tau\big(\|\btheta^*\|_\sharp -\|\hat\btheta\|_\sharp\big),
\label{lemma:dim:reduction:eq0} 
\end{align}
where we used that\footnote{By the definition of the subdifferential of $\calR$ at $\hat\bfB$, there is $\bfV$ such that 
$\calR^*(\bfV)\le 1$ and $\llangle\bfV,\hat\bfB\rrangle = \calR(\hat\bfB)$. Hence, 
$
-\llangle\bfDelta_{\bfB},\bfV\rrangle = \llangle\bfB-\hat\bfB,\bfV\rrangle =
\llangle\bfB,\bfV\rrangle -\calR(\hat\bfB)\le 
\calR(\bfB)-\calR(\hat\bfB).
$}
$-\llangle\bfDelta_{\bfB},\bfV\rrangle\le\calR(\bfB)-\calR(\hat\bfB)$, $-\llangle\bfDelta_{\bfGamma},\bfW\rrangle\le\calS(\bfGamma)-\calS(\hat\bfGamma)$ and $-\langle\bfDelta^{\hat\btheta},\bu\rangle \le \|\btheta^*\|_\sharp-\|\hat\btheta\|_\sharp$.

\section{Proof of Lemma \ref{lemma:MP+ATP}}
By the parallelogram law,
\begin{align}
&\langle \bfDelta^{(n)} +\bfDelta^{\hat\btheta}, \frM^{(n)}(\bfDelta_{\bfB}+\bfDelta_{\bfGamma},\bfDelta^{\hat\btheta})\rangle =\\
&=\frac{1}{2}\Vert\bfDelta^{(n)} + \bfDelta^{\hat\btheta}\Vert_2^2
+ \frac{1}{2}\Vert\frM^{(n)}(\bfDelta_{\bfB}+\bfDelta_{\bfGamma},\bfDelta^{\hat\btheta})\Vert_2^2
-\frac{1}{2}\Vert\frX^{(n)}(\bfB+\bfGamma)-\boldf^{(n)}\Vert_2^2.
\end{align}

$\ATP$ implies in particular that
\begin{align}
\Vert\frM^{(n)}(\bfDelta_{\bfB}+\bfDelta_{\bfGamma},\bfDelta^{\hat\btheta})\Vert_2^2 &\ge\left( \sd_1\Vert[\bfDelta_{\bfB},\bfDelta_{\bfGamma},\bfDelta^{\hat\btheta}]\Vert_{\Pi}
-\sd_2\calR(\bfDelta_{\bfB}) 
-\sd_3\calS(\bfDelta_{\bfGamma}) 
-\sd_4\Vert\bfDelta^{\hat\btheta}\Vert_\sharp
\right)_+^2\\
&- 2|\llangle\bfDelta_{\bfB},\bfDelta_{\bfGamma}\rrangle_\Pi|, 
\end{align}
noticing that, by assumption,
$$
|\llangle\bfDelta_{\bfB},\bfDelta_{\bfGamma}\rrangle_\Pi|\le\sf_*\calS(\bfDelta_{\bfGamma}).
$$ 

$\MP$ implies that
\begin{align}
\langle\bxi^{(n)},\frM^{(n)}(\bfDelta_{\bfB} + \bfDelta_{\bfGamma},\bfDelta^{\hat\btheta})\rangle &\le
\sf_1\Vert[\bfDelta_{\bfB},\bfDelta_{\bfGamma},\bfDelta^{\hat\btheta}]\Vert_{\Pi} 
+\sf_2\calR(\bfDelta_{\bfB}) 
+\sf_3\calS(\bfDelta_{\bfGamma}) 
+\sf_4\Vert\bfDelta^{\hat\btheta}\Vert_\sharp. 
\end{align}

The claim of the lemma follows from the four previous displays and Lemma \ref{lemma:recursion:1st:order:condition}.

\section{Proof of Lemma \ref{lemma:aug:rest:convexity}}
In what follows, $R_{\bfB}:=R_{\calR}(\bfV|\bfB)=\Psi_{\calR}(\calP_{\bfB}(\bfV))\mu(\calC_{\bfB}(2c_0))$. Similarly, $R_{\bfGamma}:=R_{\calS}(\bfW|\bfGamma)$. By Cauchy-Schwarz,
\begin{align}
	\triangle_{\lambda,\chi,\tau}(\bfV,\bfW,\bu)&\le(\nicefrac{3}{2})\left(
	\lambda\calR(\calP_{\bfB}(\bfV)) + \chi\calS(\calP_{\bfGamma}(\bfW)) + \eta\|\bu\|_2\right)\\
	&\le (\nicefrac{3}{2})\{\lambda^2R_{\bfB}^2 + \chi^2R_{\bfGamma}^2+ \tau^2\eta^2\}^{1/2}\Vert[\bfV,\bfW,\bu]\Vert_\Pi,
\end{align}
that is, \eqref{lemma:aug:rest:convexity:eq2}. We now split our arguments in four cases.  
\begin{description}
\item[Case 1:] $\bfV\in\calC_{\bfB}(2c_0)$ and $\bfW\in\calC_{\bfGamma}(2c_0)$.

Decomposability of $(\calR,\calS)$ and $[\bfV,\bfW,\bu]\in\calC_{\bfB,\bfGamma}(c_0,\gamma_{\calR},\gamma_{\calS},\eta)$ and Cauchy-Schwarz further imply
\begin{align}
\lambda\calR(\bfV) +\chi\calR(\bfW) +\tau \|\bu\|_\sharp 
&\le (c_0+1)(\lambda\calR(\calP_{\bfB}(\bfV)) +\chi\calR(\calP_{\bfGamma}(\bfW)) + \tau\eta\|\bu\|_2)\\
&\le (c_0+1)\{\lambda^2R_{\bfB}^2 + \chi^2R_{\bfGamma}^2 +\tau^2\eta^2\}^{1/2}\Vert[\bfV,\bfW,\bu]\Vert_\Pi.
\end{align}

\item[Case 2:] $\bfV\notin\calC_{\bfB}(2c_0)$ and 
$\bfW\in\calC_{\bfGamma}(2c_0)$.

As $[\bfV,\bfW,\bu]\in\calC_{\bfB,\bfGamma}(c_0,\gamma_{\calR},\gamma_{\calS},\eta)$, we get 
\begin{align} 
c_0\gamma_{\calR}\calR(\calP_{\bfB}(\bfV))\le
c_0\left[\gamma_{\calS}\calS(\calP_{\bfGamma}(\bfW)) + \eta\Vert\bu\Vert_2\right].
\end{align}
Hence, 
\begin{align}
\lambda\calR(\bfV) +\chi\calS(\bfW)+\tau\big\|\bu\big\|_\sharp &\le
(c_0+1)(\lambda\calR(\calP_{\bfB}(\bfV)) +\chi\calR(\calP_{\bfGamma}(\bfW)) + \tau\eta\|\bu\|_2)\\
&\le 2(c_0+1)(\chi\calR(\calP_{\bfGamma}(\bfW)) + \tau\eta\|\bu\|_2)\\
&\le 2(c_0+1)\{\chi^2R_{\bfGamma}^2 +\tau^2\eta^2\}^{1/2}
\Vert[\bfW,\bu]\Vert_\Pi. 
\end{align}

\item[Case 3:] $\bfV\in\calC_{\bfB}(2c_0)$ and 
$\bfW\notin\calC_{\bfGamma}(2c_0)$.

This case follows very similarly to Case 2, exchanging the roles between 
$(\bfV,\calR)$ and $(\bfW,\calS)$. This leads to the bounds:
\begin{align}
\lambda\calR(\bfV) +\chi\calS(\bfW)+\tau\big\|\bu\big\|_\sharp &\le 
2(c_0+1)\{\lambda^2R_{\bfB}^2 +\tau^2\eta^2\}^{1/2}
\Vert[\bfV,\bu]\Vert_\Pi.
\end{align}

\item[Case 4:] $\bfV\notin\calC_{\bfB}(2c_0)$ and 
$\bfW\notin\calC_{\bfGamma}(2c_0)$.

As $[\bfV,\bfW,\bu]\in\calC_{\bfB,\bfGamma}(c_0,\gamma_{\calR},\gamma_{\calS},\eta)$, we get 
\begin{align} 
c_0\left[\gamma_{\calR}\calR(\calP_{\bfB}(\bfV)) + \gamma_{\calS}\calS(\calP_{\bfGamma}(\bfW))
\right]\le c_0\eta\Vert\bu\Vert_2.
\end{align}
Hence, 
\begin{align}
\lambda\calR(\bfV) +\chi\calS(\bfW)+\tau\big\|\bu\big\|_\sharp &\le
(c_0+1)(\lambda\calR(\calP_{\bfB}(\bfV)) +\chi\calR(\calP_{\bfGamma}(\bfW)) + \tau\eta|\bu\|_2)\\
&\le 2(c_0+1)\tau\eta\|\bu\|_2. 
\end{align}

\end{description}
Relation \eqref{lemma:aug:rest:convexity:eq3} follows by taking the largest bounds among all four cases.

\section{Proof of Proposition \ref{prop:suboptimal:rate}}
Let
$\blacksquare:=(\nicefrac{\lambda}{4})\calR(\bfDelta_{\bfB}) +
(\nicefrac{\chi}{4})\calS(\bfDelta_{\bfGamma}) + 
(\nicefrac{\tau}{4})\|\bfDelta^{\hat\btheta}\|_\sharp$. By Lemma \ref{lemma:MP+ATP}, 
\begin{align}
2\blacksquare +\Vert\bfDelta^{(n)} +\bfDelta^{\hat\btheta}\Vert_2^2
+ \left(\sd_1\Vert[\bfV,\bfW,\bu]\Vert_{\Pi} - (\nicefrac{1}{\sigma})\blacktriangle \right)_+^2&\le
\Vert\frX^{(n)}(\bfB+\bfGamma)-\boldf^{(n)}\Vert_2^2\\
&+ \left(2\sf_1\Vert[\bfV,\bfW,\bu]\Vert_{\Pi} - \blacktriangle\right)
+ 2(\blacktriangle + \blacksquare +\blacktriangledown). 
\end{align}

By Lemmas \ref{lemma:A1:B} and \ref{lemma:A1} (with $\nu=1/2$) and condition (iii'),   
\begin{align}
\blacktriangle + \blacksquare +\blacktriangledown & =
\left[((\sigma\sd_2)\vee(2\sf_2)) + (\nicefrac{\lambda}{4})\right]\calR(\bfDelta_{\bfB}) + \lambda \big(\calR(\bfB) - \calR(\hat\bfB)\big)\\
&+\left[((\sigma\sd_3)\vee(2\sf_3 + 2\sf_*)) + (\nicefrac{\chi}{4})\right]\calS(\bfDelta_{\bfGamma}) +\chi\big(\calS(\bfGamma) - \calS(\hat\bfGamma)\big)\\
&+\left[((\sigma\sd_4)\vee(2\sf_4)) + (\nicefrac{\tau}{4})\right]
\|\bfDelta^{\hat\btheta}\|_\sharp+\tau\big(\|\btheta^*\|_\sharp -\|\hat\btheta\|_\sharp\big)\\
&\le(\nicefrac{\lambda}{2})\calR(\bfDelta_{\bfB}) + \lambda \big(\calR(\bfB) - \calR(\hat\bfB)\big)\\
&+(\nicefrac{\chi}{2})\calS(\bfDelta_{\bfGamma})+\chi \big(\calS(\bfGamma) - \calS(\hat\bfGamma)\big)\\
&+(\nicefrac{\tau}{2})\|\bfDelta^{\hat\btheta}\|_\sharp+\tau\big(\|\btheta^*\|_\sharp -\|\hat\btheta\|_\sharp\big)\\
&\le \triangle_{\lambda,\chi,\tau}(\bfDelta_{\bfB},\bfDelta_{\bfGamma},\bfDelta^{\hat\btheta}|\bfB,\bfGamma).
\end{align}

Next, we will define some local variables for convenience of notation. Let 
$G:=\Vert[\bfDelta_{\bfB},\bfDelta_{\bfGamma},\bfDelta^{\hat\btheta}]\Vert_{\Pi}$, 
$
D:=\Vert\frX^{(n)}(\bfB+\bfGamma)-\boldf^{(n)}\Vert_2, 
$
$
x:=\Vert\bfDelta^{(n)} + \bfDelta^{\hat\btheta}\Vert_2,
$
and 
$r:=r_{\lambda,\chi,\tau\Omega,3}(\bfDelta_{\bfB},\bfDelta_{\bfGamma}|\bfB,\bfGamma)$. Define also $\triangle:=\triangle_{\lambda,\chi,\tau}(\bfDelta_{\bfB},\bfDelta_{\bfGamma},\bfDelta^{\hat\btheta}|\bfB,\bfGamma)$ and 
\begin{align}
H&:=(\nicefrac{3\lambda}{2})(\calR\circ\calP_{\bfB})(\bfDelta_{\bfB})
+(\nicefrac{3\chi}{2})(\calS\circ\calP_{\bfGamma})(\bfDelta_{\bfGamma})
+(\nicefrac{3\tau\Omega}{2})\Vert\bfDelta^{\hat\btheta}\Vert_2,\\
I&:=(\nicefrac{\lambda}{2})(\calR\circ\calP_{\bfB}^\perp)(\bfDelta_{\bfB})
+(\nicefrac{\chi}{2})(\calS\circ\calP_{\bfGamma}^\perp)(\bfDelta_{\bfGamma})
+(\nicefrac{\tau}{2})\sum_{i=o+1}^n\omega_i(\bfDelta^{\hat\btheta})_i^\sharp.
\end{align}
In particular, 
$
\triangle = H - I. 
$

The previous two bounds entail 
\begin{align}
2\blacksquare + x^2 + (\sd_1 G - (\nicefrac{\blacktriangle}{\sigma}))_+^2
&\le D^2 + (2\sf_1G - \blacktriangle) +2\triangle.\label{proof:prop:suboptimal:rate:eq0}
\end{align}
We split our argument in two cases. 
\begin{description}
\item[Case 1:] $\sd_1G\le\blacktriangle/2\sigma$. We next show that 
$\triangle\le0$. If $\triangle>0$ then 
$[\bfDelta_{\bfB},\bfDelta_{\bfGamma},\bfDelta^{\hat\btheta}]\in\calC_{\bfB,\bfGamma}(3,\gamma_{\calR},\gamma_{\calS},\Omega)$. By Lemma \ref{lemma:aug:rest:convexity} and (iii'), 
$
\blacktriangle\le(\nicefrac{1}{4})[\lambda\calR(\bfDelta_{\bfB}) + \chi\calS(\bfDelta_{\bfGamma}) +\tau\big\|\bfDelta^{\hat\btheta}\big\|_\sharp]
\le 2rG<2\sigma\sd_1 G, 
$
where we have used condition \eqref{cond1:general:norm}. This is a contradiction, showing that $\triangle\le0$.

Now, condition (iv') implies that $(2\sf_1G - \blacktriangle)\le\sigma(\sd_1G-(\nicefrac{\blacktriangle}{\sigma}))\le -\blacktriangle/2$. We thus obtain from \eqref{proof:prop:suboptimal:rate:eq0} that 
\begin{align}
2\blacksquare + x^2 + \frac{\blacktriangle}{2} \le D^2. \label{proof:prop:suboptimal:rate:eq1}
\end{align}
In particular, 
\begin{align}
\sd_1G \le \frac{D^2}{\sigma}. \label{proof:prop:suboptimal:rate:eq2}
\end{align}

\item[Case 2:] $\sd_1G\ge\blacktriangle/2\sigma$. In particular, from \eqref{proof:prop:suboptimal:rate:eq0},
\begin{align}
2\blacksquare + x^2 + \frac{\sd_1^2 G^2}{4} &\le D^2 + 2\sf_1G +2\triangle.\label{proof:prop:suboptimal:rate:eq3}
\end{align}
We next consider two cases.
\begin{description}
\item[Case 2.1:] $\sf_1G\ge H$. Hence, $\triangle\le H\le \sf_1G$. From \eqref{proof:prop:suboptimal:rate:eq3}, 
\begin{align}
2\blacksquare + x^2 + \frac{\sd_1^2 G^2}{4} &\le D^2 + 4\sf_1G.
\end{align}
From $4\sf_1G\le\frac{16\sf_1^2}{\sd_1^2} + \frac{\sd_1^2G^2}{4}$, we get 
\begin{align}
2\blacksquare + x^2 &\le D^2 + \frac{16\sf_1^2}{\sd_1^2}.\label{proof:prop:suboptimal:rate:eq4} 
\end{align}
If we use instead $4\sf_1G\le\frac{32\sf_1^2}{\sd_1^2} + \frac{\sd_1^2G^2}{8}$, we get 
\begin{align}
\frac{\sd_1^2 G^2}{8} &\le D^2 + \frac{32\sf_1^2}{\sd_1^2}.\label{proof:prop:suboptimal:rate:eq5}
\end{align}

\item[Case 2.2:] $\sf_1G\le H$. Suppose first $\triangle\le0$. From $2\sf_1G\le\frac{4\sf_1^2}{\sd_1^2} + \frac{\sd_1^2G^2}{4}$ and \eqref{proof:prop:suboptimal:rate:eq3}, we get 
\begin{align}
2\blacksquare + x^2 &\le D^2 + \frac{4\sf_1^2}{\sd_1^2}.\label{proof:prop:suboptimal:rate:eq6}
\end{align}
If instead we use $2\sf_1G\le\frac{8\sf_1^2}{\sd_1^2} + \frac{\sd_1^2G^2}{8}$, we obtain 
\begin{align}
\frac{\sd_1^2 G^2}{8} &\le D^2 + \frac{8\sf_1^2}{\sd_1^2}.\label{proof:prop:suboptimal:rate:eq7}
\end{align}

Suppose now $\triangle\ge0$. Hence $[\bfDelta_{\bfB},\bfDelta_{\bfGamma},\bfDelta^{\hat\btheta}]\in\calC_{\bfB,\bfGamma}(3,\gamma_{\calR},\gamma_{\calS},\Omega)$. By Lemma \ref{lemma:aug:rest:convexity}, $\triangle\le 3rG/2$. From \eqref{proof:prop:suboptimal:rate:eq3}, 
\begin{align}
2\blacksquare + x^2 + \frac{\sd_1^2 G^2}{4} &\le D^2 + (2\sf_1 + 3r)G.
\end{align}
Proceeding similarly before, we obtain from the displayed bound that 
\begin{align}
2\blacksquare + x^2 &\le D^2 + \frac{(2\sf_1 + 3r)^2}{\sd_1^2},\label{proof:prop:suboptimal:rate:eq8}\\
\frac{\sd_1^2 G^2}{8} &\le D^2 + \frac{2(2\sf_1 + 3r)^2}{\sd_1^2}.\label{proof:prop:suboptimal:rate:eq9}
\end{align}
\end{description}
\end{description}
The proof of \eqref{prop:suboptimal:rate:eq1} follows by taking the largest of the bounds in \eqref{proof:prop:suboptimal:rate:eq1}, \eqref{proof:prop:suboptimal:rate:eq4}, \eqref{proof:prop:suboptimal:rate:eq6} and \eqref{proof:prop:suboptimal:rate:eq8}. The proof of \eqref{prop:suboptimal:rate:eq2} follows by taking the largest of the bounds in \eqref{proof:prop:suboptimal:rate:eq2}, \eqref{proof:prop:suboptimal:rate:eq5}, \eqref{proof:prop:suboptimal:rate:eq7} and \eqref{proof:prop:suboptimal:rate:eq9}.

\section{Proof of Lemma \ref{lemma:recursion:delta:bb:general:norm}}
As $[\hat\bfB,\hat\bfGamma,\hat\btheta]$ is the minimizer of \eqref{equation:aug:slope:rob:estimator:general}, in particular
\begin{align}
[\hat\bfB,\hat\bfGamma]\in 
&\argmin_{[\bfB,\bfGamma]:\Vert\bfB\Vert_\infty\le\sa}\left\{\frac{1}{2}\Vert\by^{(n)}-\frX^{(n)}(\bfB+\bfGamma)-\hat\btheta\Vert_2^2+\lambda\calR(\bfB)+\chi\calS(\bfGamma)\right\}.
\end{align}
By the first order condition, there exist $\bfV,\bfW\in\mdR^p$ such that 
$\calR^*(\bfV)\le1$, $\llangle\bfV,\hat\bfB\rrangle=\calR(\hat\bfB)$, $\calS^*(\bfW)\le1$, $\llangle\bfW,\hat\bfGamma\rrangle=\calS(\hat\bfGamma)$, such that, for all $[\bfB,\bfGamma]$ with $\Vert\bfB\Vert_\infty\le\sa$,
\begin{align}
0&\le\sum_{i\in[n]}\left[\frX^{(n)}_i(\hat\bfB+\hat\bfGamma)+\hat\btheta_i-y_i^{(n)}\right]\llangle\bfX^{(n)}_i,\bfB-\hat\bfB\rrangle+\lambda\llangle\bfV,\bfB-\hat\bfB\rrangle,\\
0&\le\sum_{i\in[n]}\left[\frX^{(n)}_i(\hat\bfB+\hat\bfGamma)+\hat\btheta_i-y_i^{(n)}\right]\llangle\bfX^{(n)}_i,\bfGamma-\hat\bfGamma\rrangle+\chi\llangle\bfW,\bfGamma-\hat\bfGamma\rrangle. 
\end{align}
Summing the previous inequalities and using that
$
\by^{(n)} = \boldf^{(n)} + \btheta^* + \bxi^{(n)}
$
we obtain, for all $[\bfB,\bfGamma]$ with $\Vert\bfB\Vert_\infty\le\sa$,
\begin{align}
\left\langle \bfDelta^{(n)},\frX^{(n)}(\bfDelta_{\bfB}+\bfDelta_{\bfGamma})\right\rangle &\le\left\langle\frX^{(n)}(\bfDelta_{\bfB}+\bfDelta_{\bfGamma}),\bxi^{(n)}
-\bfDelta^{\hat\btheta}\right\rangle-\lambda\llangle\bfDelta_{\bfB},\bfV\rrangle-\chi\llangle\bfDelta_{\bfGamma},\bfW\rrangle.
\end{align}

Moreover, 
$-\llangle\bfDelta_{\bfB},\bfV\rrangle\le\calR(\bfB)-\calR(\hat\bfB)$. Similarly,  $-\llangle\bfDelta_{\bfGamma},\bfW\rrangle\le\calS(\bfGamma)-\calS(\hat\bfGamma)$. Combining these bounds with the previous displays finishes the proof.

\section{Proof of Lemma \ref{lemma:MP+IP+ATP}}\label{s:proof:lemma:MP+IP+ATP}
By the parallelogram law,
\begin{align}
&\langle \bfDelta^{(n)}, \frX^{(n)}(\bfDelta_{\bfB}+\bfDelta_{\bfGamma})\rangle =\\
&=\frac{1}{2}\Vert\bfDelta^{(n)}\Vert_2^2
+ \frac{1}{2}\Vert\frX^{(n)}(\bfDelta_{\bfB}+\bfDelta_{\bfGamma})\Vert_2^2
- \frac{1}{2}\Vert\frX^{(n)}(\bfB+\bfGamma)-\boldf^{(n)}\Vert_2^2.
\end{align}

By $\ATP$ (with variable $\bu=\bf0$), 
\begin{align}
\Vert\frX^{(n)}(\bfDelta_{\bfB}+\bfDelta_{\bfGamma})\Vert_2^2 &\ge\left( \sd_1\Vert[\bfDelta_{\bfB},\bfDelta_{\bfGamma}]\Vert_{\Pi}
-\sd_2\calR(\bfDelta_{\bfB}) 
-\sd_3\calS(\bfDelta_{\bfGamma})
\right)_+^2 - 2|\langle\bfDelta_{\bfB},\bfDelta_{\bfGamma}\rangle|.
\end{align}
Also, by assumption, 
$$
|\langle\bfDelta_{\bfB},\bfDelta_{\bfGamma}\rangle|\le\sf_*\calS(\bfDelta_{\bfGamma}).
$$ 

The previous three displays and \eqref{lemma:recursion:delta:bb:general:norm:eq} imply
\begin{align}
\Vert\bfDelta^{(n)}\Vert_2^2 + \left( \sd_1\Vert[\bfDelta_{\bfB},\bfDelta_{\bfGamma}]\Vert_{\Pi}
-\sd_2\calR(\bfDelta_{\bfB}) 
-\sd_3\calS(\bfDelta_{\bfGamma})
\right)_+^2 &\le \Vert\frX^{(n)}(\bfB+\bfGamma)-\boldf^{(n)}\Vert_2^2\\
&+ 2\langle\bxi^{(n)}-\bfDelta^{\hat\btheta},\frX^{(n)}(\bfDelta_{\bfB}+\bfDelta_{\bfGamma})\rangle
+ 2\sf_*\calS(\bfDelta_{\bfGamma})\\
&+2\lambda(\calR(\bfB)-\calR(\hat\bfB)) 
+ 2\chi(\calS(\bfGamma)-\calS(\hat\bfGamma)). 
\label{proof:lemma:MP+IP+PP+ATP:eq1}
\end{align}

By $\IP$,
\begin{align}
\langle-\bfDelta^{\hat\btheta},\frX^{(n)}(\bfDelta_{\bfB}+\bfDelta_{\bfGamma})\rangle 
&\le \sb_1\Vert[\bfDelta_{\bfB},\bfDelta_{\bfGamma}]\Vert_{\Pi}\Vert\bfDelta^{\hat\btheta}\Vert_2 
+\sb_2\calR(\bfDelta_{\bfB})\Vert\bfDelta^{\hat\btheta}\Vert_2 
+\sb_3\calS(\bfDelta_{\bfGamma})\Vert\bfDelta^{\hat\btheta}\Vert_2\\
&+\sb_4\Vert[\bfDelta_{\bfB},\bfDelta_{\bfGamma}]\Vert_{\Pi}\Vert\bfDelta^{\hat\btheta}\Vert_\sharp.  
\end{align}

By $\MP$ (with variable $\bu=\bf0$),
\begin{align}
\langle\bxi^{(n)},\frX^{(n)}(\bfDelta_{\bfB} + \bfDelta_{\bfGamma})\rangle &\le
\sf_1\Vert[\bfDelta_{\bfB},\bfDelta_{\bfGamma}]\Vert_{\Pi} 
+\sf_2\calR(\bfDelta_{\bfB}) 
+\sf_3\calS(\bfDelta_{\bfGamma}).
\end{align}

The two previous displays imply
\begin{align}
&\langle\bxi^{(n)}-\bfDelta^{\hat\btheta},\frX^{(n)}(\bfDelta_{\bfB}+\bfDelta_{\bfGamma})\rangle\\
&\le\Vert[\bfDelta_{\bfB},\bfDelta_{\bfGamma}]\Vert_{\Pi}\left(
\sb_1\Vert\bfDelta^{\hat\btheta}\Vert_2 + \sb_4\Vert\bfDelta^{\hat\btheta}\Vert_\sharp + \sf_1
\right) + \calR(\bfDelta_{\bfB})(\sb_2\|\bfDelta^{\hat\btheta}\|_2 + \sf_2)
+\calS(\bfDelta_{\bfGamma})(\sb_3\|\bfDelta^{\hat\btheta}\|_2 + \sf_3). 
\end{align}

The proof of \eqref{lemma:MP+IP+ATP:eq} follows from the previous display and \eqref{proof:lemma:MP+IP+PP+ATP:eq1}. 

\section{Proof of Theorem \ref{thm:improved:rate}}\label{s:proof:thm:improved:rate:supp}
Let $\hat\blacksquare:=(\nicefrac{\lambda}{4})\calR(\bfDelta_{\bfB}) +
(\nicefrac{\chi}{4})\calS(\bfDelta_{\bfGamma})$. By Lemma \ref{lemma:MP+IP+ATP}, 
\begin{align}
2\hat\blacksquare +\Vert\bfDelta^{(n)}\Vert_2^2
+ \left(\sd_1\Vert[\bfDelta_{\bfB},\bfDelta_{\bfGamma}]\Vert_{\Pi} - (\nicefrac{\hat\blacktriangle}{\sigma}) \right)_+^2&\le
\Vert\frX^{(n)}(\bfB+\bfGamma)-\boldf^{(n)}\Vert_2^2\\
& + \left(2\sf_1 + 2\sb_1\Vert\bfDelta^{\hat\btheta}\Vert_2 + 2\sb_4\Vert\bfDelta^{\hat\btheta}\Vert_\sharp\right)\Vert[\bfDelta_{\bfB},\bfDelta_{\bfGamma}]\Vert_{\Pi} - \hat\blacktriangle\\
&+ 2(\hat\blacktriangle + \hat\blacksquare +\hat\blacktriangledown).\label{proof:thm:improved:rate:eq0'}
\end{align}

Additionally, all conditions of Proposition \ref{prop:suboptimal:rate} hold. Hence, 
\begin{align}
\|\bfDelta^{\hat\btheta}\|_2&\le
(\nicefrac{D^2}{\sigma\sd_1})\bigvee\left((\nicefrac{2\sqrt{2}}{\sd_1})D + \clubsuit_2(\sf_1,r)\right) \le\sc_*\sigma,\label{proof:thm:improved:rate:eq:club}\\ 
(\nicefrac{\tau}{2})\|\bfDelta^{\hat\btheta}\|_\sharp &\le D^2 + \spadesuit_2(\sf_1,r)
\le\sc_*^2\sigma^2, \label{proof:thm:improved:rate:eq:spade}
\end{align}
where we have used conditions \eqref{cond8:general:norm}-\eqref{cond9:general:norm}. In particular, by (iv), 
\begin{align}
(\sigma\sd_2)\vee(2\sf_2+2\sb_2\|\bfDelta^{\hat\btheta}\|_2)&\le\frac{\lambda}{4},\label{proof:thm:improved:rate:eq0''}\\
(\sigma\sd_3)\vee(2\sf_3 + 2\sf_* + 2\sb_3\|\bfDelta^{\hat\btheta}\|_2)&\le\frac{\chi}{4}.\label{proof:thm:improved:rate:eq0'''}
\end{align} 

By Lemma \ref{lemma:A1:B} (with $\nu=1/2$) and \eqref{proof:thm:improved:rate:eq0''}-\eqref{proof:thm:improved:rate:eq0'''},  
\begin{align}
\hat\blacktriangle + \hat\blacksquare + \hat\blacktriangledown & =
\left[((\sigma\sd_2)\vee(2\sf_2+2\sb_2\|\bfDelta^{\hat\btheta}\|_2)) + (\nicefrac{\lambda}{4})\right]\calR(\bfDelta_{\bfB}) + \lambda \big(\calR(\bfB) - \calR(\hat\bfB)\big)\\
&+\left[((\sigma\sd_3)\vee(2\sf_3 + 2\sf_* + 2\sb_3\|\bfDelta^{\hat\btheta}\|_2)) + (\nicefrac{\chi}{4})\right]\calS(\bfDelta_{\bfGamma}) +\chi\big(\calS(\bfGamma) - \calS(\hat\bfGamma)\big)\\
&\le(\nicefrac{\lambda}{2})\calR(\bfDelta_{\bfB}) + \lambda \big(\calR(\bfB) - \calR(\hat\bfB)\big)
+(\nicefrac{\chi}{2})\calS(\bfDelta_{\bfGamma})+\chi \big(\calS(\bfGamma) - \calS(\hat\bfGamma)\big)\\
&\le \triangle_{\lambda,\chi,0}(\bfDelta_{\bfB},\bfDelta_{\bfGamma},\bf0|\bfB,\bfGamma)=:\triangle.\label{proof:thm:improved:rate:eq0''''}
\end{align}

Next, we define the local variables
\begin{align}
H&:=(\nicefrac{3\lambda}{2})(\calR\circ\calP_{\bfB})(\bfDelta_{\bfB})
+(\nicefrac{3\chi}{2})(\calS\circ\calP_{\bfGamma})(\bfDelta_{\bfGamma}),\\
I&:=(\nicefrac{\lambda}{2})(\calR\circ\calP_{\bfB}^\perp)(\bfDelta_{\bfB})
+(\nicefrac{\chi}{2})(\calS\circ\calP_{\bfGamma}^\perp)(\bfDelta_{\bfGamma}), 
\end{align}
noting that $\triangle = H - I$. Recall 
$
D=\Vert\frX^{(n)}(\bfB+\bfGamma)-\boldf^{(n)}\Vert_2.
$
For convenience, we define the additional local variables 
$G:=\Vert[\bfDelta_{\bfB},\bfDelta_{\bfGamma}]\Vert_{\Pi}$, 
$
x:=\Vert\bfDelta^{(n)}\Vert_2,
$
and 
$\hat r:=r_{\lambda,\chi,0,3}(\bfDelta_{\bfB},\bfDelta_{\bfGamma}|\bfB,\bfGamma)$.  
Finally, let us define the auxiliary variables
\begin{align}
F&:=\sf_1 + \sb_1\left(
(\nicefrac{D^2}{\sigma\sd_1})\bigvee\left((\nicefrac{2\sqrt{2}}{\sd_1})D + \clubsuit_2(\sf_1,r)\right)
\right) + (2\sb_4/\tau)\left(
D^2 + \spadesuit_2(\sf_1,r)
\right),\\
\hat\sf_1&:=\sf_1 + \sb_1 (\sc_*\sigma)+ (2\sb_4/\tau)(\sc_*^2\sigma^2). 
\end{align}
Note that, from \eqref{proof:thm:improved:rate:eq:club}-\eqref{proof:thm:improved:rate:eq:spade} and condition (v), $F\le\hat\sf_1\le\sigma\sd_1/2$. 

From \eqref{proof:thm:improved:rate:eq0'},  \eqref{proof:thm:improved:rate:eq:club}-\eqref{proof:thm:improved:rate:eq:spade}, 
and \eqref{proof:thm:improved:rate:eq0''''}, 
\begin{align}
2\hat\blacksquare + x^2 + (\sd_1 G - (\nicefrac{\hat\blacktriangle}{\sigma}))_+^2
&\le D^2 + (2F G - \hat\blacktriangle) + 2\triangle.\label{proof:thm:improved:rate:eq0}
\end{align}
The rest of the proof uses similar arguments used in the proof of Proposition \ref{prop:suboptimal:rate}. 

We split our argument in two cases. 
\begin{description}
\item[Case 1:] $\sd_1 G\le\hat\blacktriangle/2\sigma$. We next show that 
$\triangle\le0$. If $\triangle>0$ then 
$[\bfDelta_{\bfB},\bfDelta_{\bfGamma},\bold{0}]\in\calC_{\bfB,\bfGamma}(3,\gamma_{\calR},\gamma_{\calS},0)$. By Lemma \ref{lemma:aug:rest:convexity} and \eqref{proof:thm:improved:rate:eq0''}-\eqref{proof:thm:improved:rate:eq0'''},
$
\hat \blacktriangle\le(\nicefrac{1}{4})[\lambda\calR(\bfDelta_{\bfB}) + \chi\calS(\bfDelta_{\bfGamma})]
\le 2\hat rG<2\sigma\sd_1 G, 
$
where we have used condition \eqref{cond1:general:norm}. This is a contradiction, hence $\triangle\le0$.

From \eqref{proof:thm:improved:rate:eq:club}-\eqref{proof:thm:improved:rate:eq:spade} and condition (v), $(2FG - \hat\blacktriangle)\le(2\hat\sf_1G - \hat\blacktriangle)\le\sigma(\sd_1\hat G-(\nicefrac{\hat\blacktriangle}{\sigma}))\le -\hat\blacktriangle/2$. We thus obtain from \eqref{proof:thm:improved:rate:eq0} that 
\begin{align}
2\hat\blacksquare + x^2 + \frac{\hat\blacktriangle}{2} \le D^2. \label{proof:thm:improved:rate:eq1}
\end{align}
In particular, 
\begin{align}
\sd_1G \le \frac{D^2}{\sigma}. \label{proof:thm:improved:rate:eq2}
\end{align}

\item[Case 2:] $\sd_1G\ge\hat\blacktriangle/2\sigma$. In particular, from \eqref{proof:thm:improved:rate:eq0},
\begin{align}
2\hat\blacksquare + x^2 + \frac{\sd_1^2 G^2}{4} &\le D^2 + 2FG +2\triangle.\label{proof:thm:improved:rate:eq3}
\end{align}
We next consider two cases.
\begin{description}
\item[Case 2.1:] $FG\ge H$. Hence, $\triangle\le H\le FG$. From \eqref{proof:thm:improved:rate:eq3}, 
\begin{align}
2\hat\blacksquare + x^2 + \frac{\sd_1^2 G^2}{4} &\le D^2 + 4FG.
\end{align}
From $4FG\le\frac{16F^2}{\sd_1^2} + \frac{\sd_1^2G^2}{4}$, we get 
\begin{align}
2\hat\blacksquare + x^2 &\le D^2 + \frac{16F^2}{\sd_1^2}.\label{proof:thm:improved:rate:eq4} 
\end{align}
If we use instead $4FG\le\frac{32F^2}{\sd_1^2} + \frac{\sd_1^2G^2}{8}$, we get 
\begin{align}
\frac{\sd_1^2 G^2}{8} &\le D^2 + \frac{32F^2}{\sd_1^2}.\label{proof:thm:improved:rate:eq5}
\end{align}

\item[Case 2.2:] $FG\le H$. Suppose first $\triangle\le0$. From $2FG\le\frac{4F^2}{\sd_1^2} + \frac{\sd_1^2G^2}{4}$ and \eqref{proof:thm:improved:rate:eq3}, we get 
\begin{align}
2\hat\blacksquare + x^2 &\le D^2 + \frac{4F^2}{\sd_1^2}.\label{proof:thm:improved:rate:eq6}
\end{align}
If instead we use $2FG\le\frac{8F^2}{\sd_1^2} + \frac{\sd_1^2G^2}{8}$, we obtain 
\begin{align}
\frac{\sd_1^2 G^2}{8} &\le D^2 + \frac{8F^2}{\sd_1^2}.\label{proof:thm:improved:rate:eq7}
\end{align}

Suppose now $\triangle\ge0$. Hence $[\bfDelta_{\bfB},\bfDelta_{\bfGamma},\bold{0}]\in\calC_{\bfB,\bfGamma}(3,\gamma_{\calR},\gamma_{\calS},0)$. By Lemma \ref{lemma:aug:rest:convexity}, $\triangle\le 3\hat rG/2$. From \eqref{proof:thm:improved:rate:eq3}, 
\begin{align}
2\hat\blacksquare + x^2 + \frac{\sd_1^2 G^2}{4} &\le D^2 + (2F + 3\hat r)G.
\end{align}
Proceeding similarly as before, we obtain from the displayed bound that 
\begin{align}
2\hat\blacksquare + x^2 &\le D^2 + \frac{(2F+3\hat r)^2}{\sd_1^2},\label{proof:thm:improved:rate:eq8}\\
\frac{\sd_1^2 G^2}{8} &\le D^2 + \frac{2(2F+3\hat r)^2}{\sd_1^2}.\label{proof:thm:improved:rate:eq9}
\end{align}
\end{description}
\end{description}
The proof of \eqref{thm:improved:rate:main:rate:eq1} follows by taking the largest of the bounds in \eqref{proof:thm:improved:rate:eq1}, \eqref{proof:thm:improved:rate:eq4}, \eqref{proof:thm:improved:rate:eq6} and \eqref{proof:thm:improved:rate:eq8}. The proof of \eqref{thm:improved:rate:main:rate:eq2} follows by taking the largest of the bounds in \eqref{proof:thm:improved:rate:eq2}, \eqref{proof:thm:improved:rate:eq5}, \eqref{proof:thm:improved:rate:eq7} and \eqref{proof:thm:improved:rate:eq9}.

\section{Proof of Theorem \ref{thm:tr:reg:matrix:decomp}}\label{proof:thm:tr:reg:matrix:decomp}

In the following $\calR:=\Vert\cdot\Vert_N$, 
$\calS:=\Vert\cdot\Vert_1$, $\sa=\sa^*/\sqrt{n}$ and $\sf_*=2\sa^*/\sqrt{n}$. Recall that $\bfSigma$ is the identity matrix. We know that 
$\mathscr{G}(\bfSigma^{1/2}\mbB_{\Vert\cdot\Vert_N})\lesssim\sqrt{d_1+d_2}$, $\mathscr{G}(\bfSigma^{1/2}\mbB_1^p)\lesssim\sqrt{\log p}$ and 
$\mathscr{G}(\mbB_\sharp)\lesssim1$. Next, we assume that $n\ge C_0L^4(1+\log(1/\delta))$ for an absolute constant to be determined next. We will also use that $L\ge1$. In the following, $C>0$ is the universal constant stated in Proposition \ref{proposition:properties:subgaussian:designs}. Without loss on generality, we assume the constant $\sc_*$ in Theorem \ref{thm:improved:rate} is $\ge1$. No effort is made to optimize the numerical constants.

By Proposition \ref{proposition:properties:subgaussian:designs}(i) and taking $C_0\ge1$ large enough, we get that, on an event 
$\calE_1$ of probability $\ge1-\delta$, $\PP_{\Vert\cdot\Vert_N,\Vert\cdot\Vert_1}(\sc_1(\delta),\sc_2,\sc_3,\sc_4)$ holds with constants
\begin{align}
\sc_1(\delta) \asymp CL^2\frac{1+\sqrt{\log(1/\delta)}}{\sqrt{n}},
\quad
\sc_2\asymp CL^2\sqrt{\frac{d_1+d_2}{n}}, \\
\sc_3\asymp CL^2\sqrt{\frac{\log p}{n}},
\quad
\sc_4 \asymp CL^2\sqrt{\frac{d_1+d_2}{n}}\cdot\sqrt{\frac{\log p}{n}}.
\end{align}
By Proposition \ref{proposition:properties:subgaussian:designs}(ii) and taking $C_0\gtrsim C^2$ large enough, we get that, on an event $\calE_2$ of probability $\ge1-2\delta$, $\TP_{\Vert\cdot\Vert_N}(\sa_1,\sa_2)$ and 
$\TP_{\Vert\cdot\Vert_1}(\bar\sa_1,\bar\sa_2)$ hold with constants 
$\sa_1=\bar\sa_1\in (0,1)$ and 
\begin{align}
\sa_2 \asymp CL^2\sqrt{\frac{d_1+d_2}{n}},\quad
\bar\sa_2\asymp CL^2\sqrt{\frac{\log p}{n}}.
\end{align}
By Proposition \ref{proposition:properties:subgaussian:designs}(iii), on an event 
$\calE_3$ of probability $\ge1-\delta$, $\IP_{\Vert\cdot\Vert_N,\Vert\cdot\Vert_1,\Vert\cdot\Vert_\sharp}(\sb_1(\delta),\sb_2,\sb_3,\sb_4)$ holds with constants
$\sb_4\asymp\frac{CL}{\sqrt{n}}$, 
\begin{align}
\sb_1(\delta) \asymp CL\frac{1+\sqrt{\log(1/\delta)}}{\sqrt{n}},
\quad
\sb_2\asymp CL\sqrt{\frac{d_1+d_2}{n}}, 
\quad
\sb_3\asymp CL\sqrt{\frac{\log p}{n}}.
\end{align}
By Proposition \ref{proposition:properties:subgaussian:designs}(iv) and $C_0\ge1$ large enough, we have that, on an event $\calE_4$ of probability $\ge1-\delta$, $\MP_{\Vert\cdot\Vert_N,\Vert\cdot\Vert_1,\Vert\cdot\Vert_\sharp}(\sf_1(\delta),\sf_2,\sf_3,\sf_4)$ holds with constants $\sf_4\asymp\frac{C\sigma}{\sqrt{n}}$, 
\begin{align}
\sf_1(\delta) \asymp C\sigma L\frac{2+3\sqrt{\log(1/\delta)}}{\sqrt{n}},
\quad
\sf_2\asymp C\sigma L\sqrt{\frac{d_1+d_2}{n}},
\quad
\sf_3\asymp C\sigma L\sqrt{\frac{\log p}{n}}.
\end{align}
Finally, by Lemma \ref{lemma:ATP} and taking $C_0\gtrsim C^2$ large enough, we have that, on the event $\calE_1\cap\calE_2\cap\calE_3\cap\calE_4$ of probability $\ge1-5\delta$, all the previous stated properties hold and 
$\ATP_{\Vert\cdot\Vert_N,\Vert\cdot\Vert_1,\Vert\cdot\Vert_\sharp}(\sd_1,\sd_2,\sd_3,\sd_4)$ holds with constants $\sd_1\in(0,1)$, 
$\sd_4\asymp\frac{CL}{\sqrt{n}}$, 
\begin{align}
\sd_2\asymp CL^2\sqrt{\frac{d_1+d_2}{n}}, \quad
\sd_3\asymp CL^2\sqrt{\frac{\log p}{n}}.
\end{align}

The rest of the proof will happen on the event 
$\calE_1\cap\calE_2\cap\calE_3\cap\calE_4$. We will invoke the general Theorem \ref{thm:improved:rate}. Conditions (i)-(iii) were shown previously. We now verify conditions (iv)-(v). Note that
\begin{align}
4[(\sigma\sd_2)\vee(2\sf_2+2\sc_*\sigma\sb_2)]
\lesssim (1+\sc_*)C\sigma L^2\sqrt{\frac{d_1+d_2}{n}}\asymp\lambda.
\end{align}
Also, 
\begin{align}
4[(\sigma\sd_3)\vee(2\sf_3+2\sf_*+2\sc_*\sigma\sb_3)]
\lesssim (1+\sc_*)C\sigma L^2\sqrt{\frac{\log p}{n}}+\frac{\sa^*}{\sqrt{n}}\asymp\chi.
\end{align}
Next, we choose $\tau\asymp C_1\sc_*^2C\sigma L^2/\sqrt{n}$ for some absolute constant $C_1\ge1$. We have 
$
4[(\sigma\sd_4)\vee(2\sf_4)]\lesssim C\sigma L/\sqrt{n}\lesssim\tau. 
$
This shows (iv). Finally, by the choice of $\tau$, 
\begin{align}
\hat\sf_1 =\sf_1 + \sc_*(\sigma\sb_1) + 2\sc_*^2\sigma^2(\nicefrac{\sb_4}{\tau})
\lesssim (1+\sc_*)C\sigma L\frac{1+\sqrt{\log(1/\delta)}}{\sqrt{n}}
+ \frac{1}{C_1L}\sigma, 
\end{align}
which can be set strictly less than $\sigma\sd_1/2\asymp\sigma$ taking $C_0\gtrsim C^2(1+\sc_*)^2$ and $C_1\ge1$ large enough. Hence, (v) holds.

In what follows, we verify the conditions \eqref{cond4:general:norm}-\eqref{cond9:general:norm}. Let $[\bfB,\bfGamma]$ with $D=\Vert\frX^{(n)}(\bfB+\bfGamma)-\boldf^{(n)}\Vert_2$, $\rank(\bfB)\le r$, 
$\Vert\bfGamma\Vert_0\le s$ and $\Vert\bfB\Vert_\infty\le\sa^*/\sqrt{n}$. Conditions \eqref{cond4:general:norm}-\eqref{cond5:general:norm} hold. Condition \eqref{cond6:general:norm} also holds since, by the dual-norm inequality, isotropy and the facts that $\Vert\hat\bfB\Vert_\infty\le\sa^*/\sqrt{n}$ and $\Vert\bfB\Vert_\infty\le\sa^*/\sqrt{n}$,
$$
|\langle\bfDelta_{\bfB},\bfDelta_{\bfGamma}\rangle_\Pi|\le
\Vert\bfDelta_{\bfB}\Vert_\infty\Vert\bfDelta_{\bfGamma}\Vert_1
\le\frac{2\sa^*}{\sqrt{n}}\Vert\bfDelta_{\bfGamma}\Vert_1.
$$
We have that $\Psi_{\Vert\cdot\Vert_N}(\calP_{\bfB}(\bfDelta_{\bfB}))\le\sqrt{r}$. By isotropy $\mu(\bfB):=\mu\left(\calC_{\bfB,\Vert\cdot\Vert_N}(6)\right)=1$. Similarly, $\Psi_{\Vert\cdot\Vert_1}(\calP_{\bfGamma}(\bfDelta_{\bfGamma}))\le\sqrt{s}$ and
$\mu(\bfGamma):=\mu\left(\calC_{\bfGamma,\Vert\cdot\Vert_1}(6)\right)=1$. By the choice of $(\lambda,\chi,\tau)$,
\begin{align}
r^2 &\asymp (1+\sc_*)^2C^2\sigma^2 L^4\cdot\frac{r(d_1+d_2)}{n} + (1+\sc_*)^2C^2\sigma^2 L^4\cdot\frac{s\log p}{n} + \frac{(\sa^*)^2s}{n}\\
&\quad\quad + C_1^2\sc_*^4C^2\sigma^2L^2\epsilon\log(e/\epsilon),
\end{align}
where we used that $\Omega^2\le 2o\log(en/o)$. We have 
$r<\sigma\sd_1\asymp\sigma$ if
\begin{align}
C_2(1+\sc_*)CL^2\cdot
\sqrt{\frac{r(d_1+d_2)}{n}}&<1,\label{thm:response:matrix:decomp:eq1}\\
C_2(1+\sc_*)CL^2\cdot
\sqrt{\frac{s\log p}{n}} + \sa^*\sqrt{\frac{s}{n}}
&<1,\label{thm:response:matrix:decomp:eq2}\\
C_1\sc_*^2CL^2\sqrt{\epsilon\log(e/\epsilon)}<c_1,\label{thm:response:matrix:decomp:eq3}
\end{align}
for some absolute constants $C_2\ge1$ and $c_1\in(0,1)$.

Next, we show that $D^2+\spadesuit_2(\sf_1,r)\le\sc_*^2\sigma^2$. For this to be true it is sufficient that $D\le\frac{\sc_*\sigma}{3}$ and 
$\spadesuit_2^{1/2}(\sf_1,r)\le\frac{3\sc_*\sigma}{4}$. Note that
$$
\spadesuit_2^{1/2}(\sf_1,r) =\frac{4\sf_1+3r}{\sd_1}
\lesssim C\sigma L\frac{1+\sqrt{\log(1/\delta)}}{\sqrt{n}}
+r,
$$
which is not greater than $\frac{3\sc_*\sigma}{4}$ if
\begin{align}
\calO(1)C L\frac{1+\sqrt{\log(1/\delta)}}{\sqrt{n}}&\le\frac{3\sc_*}{16}\\
\calO(1)(1+\sc_*)CL^2\cdot
\sqrt{\frac{r(d_1+d_2)}{n}}&\le\frac{3\sc_*}{16},\\
\calO(1)(1+\sc_*)CL^2\cdot
\sqrt{\frac{s\log p}{n}} + \calO(1)\sa^*\sqrt{\frac{s}{n}}&\le\frac{3\sc_*}{16}, \\
\calO(1)C_1\sc_*^2CL^2\sqrt{\epsilon\log(e/\epsilon)}&\le \frac{3\sc_*}{16}.
\end{align}

Next, we show that 
$
\left[(\nicefrac{D^2}{\sigma\sd_1})\right]\bigvee\left[(\nicefrac{2\sqrt{2}}{\sd_1})D + \clubsuit_2(\sf_1,r)\right]\le \sc_*\sigma.
$
For that to happen it is sufficient that 
$
D\le [(\sqrt{\sd_1 \sc_*})\wedge(\sd_1\sc_*/6\sqrt{2})]\sigma
$
and $\clubsuit_2(\sf_1,r)\le\frac{3\sc_*\sigma}{4}$. Since $\sd_1\asymp1$ and $\sc_*\ge1$, the first relation is guaranteed if $D\lesssim \sqrt{\sc_*}\sigma$. We have 
$$
\clubsuit_2(\sf_1,r) =\frac{16\sf_1+12r}{\sd_1^2}
\lesssim C\sigma L\frac{1+\sqrt{\log(1/\delta)}}{\sqrt{n}}
+ r,
$$
which is not greater than $\frac{3\sc_*\sigma}{4}$ if all the four conditions in the previous display are met --- up to enlarging the constants if necessary. 

Optimizing the previous inequalities, we conclude that for conditions \eqref{cond4:general:norm}-\eqref{cond9:general:norm} to hold it is sufficient to take $\sc_*\asymp1$ large enough, $C_0\gtrsim C^2$ large enough and that \eqref{thm:response:matrix:decomp:eq1}, \eqref{thm:response:matrix:decomp:eq2}, \eqref{thm:response:matrix:decomp:eq3} hold (possibly enlarging $C_2$ if necessary) and $D\lesssim\sigma$  holds (possibly decreasing the constant if necessary).

In conclusion, assume $n\ge C_0L^4(1+\log(1/\delta))$ with $C_0\gtrsim C^2$ large enough and set the hyper-parameters to be
\begin{align}
\lambda\asymp C\sigma L^2\sqrt{\frac{d_1+d_2}{n}},\quad
\chi\asymp C\sigma L^2\sqrt{\frac{\log p}{n}}+\frac{\sa^*}{\sqrt{n}},\quad
\tau\asymp\frac{C_1C\sigma L^2}{\sqrt{n}}.
\end{align}
Assume \eqref{thm:response:matrix:decomp:eq1}, \eqref{thm:response:matrix:decomp:eq2}, \eqref{thm:response:matrix:decomp:eq3} hold. Let $[\bfB,\bfGamma]$ with $D=\Vert\frX^{(n)}(\bfB+\bfGamma)-\boldf^{(n)}\Vert_2\lesssim\sigma$, 
$\rank(\bfB)\le r$, $\Vert\bfGamma\Vert_0\le s$ and 
$\Vert\bfB\Vert_\infty\le\sa^*/\sqrt{n}$. Then conditions (i)-(v) and conditions \eqref{cond4:general:norm}-\eqref{cond9:general:norm} of Theorem \ref{thm:improved:rate} are satisfied, implying \eqref{thm:improved:rate:main:rate:eq1}-\eqref{thm:improved:rate:main:rate:eq2} for such $[\bfB,\bfGamma]$. We now verify the statement of Theorem \ref{thm:tr:reg:matrix:decomp}. 

First, note that
\begin{align}
\hat r^2 &\lesssim 
 C^2\sigma^2 L^4\cdot\frac{r(d_1+d_2)}{n}
+ C^2\sigma^2 L^4\cdot\frac{s\log p}{n} + \frac{(\sa^*)^2s}{n}\\
r^2&\lesssim \hat r^2 + C_1^2C^2\sigma^2L^4\epsilon\log(e/\epsilon),\\
\end{align}
Next, using that $D\lesssim\sigma$, $\sd_1\asymp1$, $n\ge C_0L^4(1+\log(1/\delta))$ and $\sb_1r\le \frac{\sigma\sb_1^2}{2}+\frac{r^2}{2\sigma}$,
\begin{align}
F &\lesssim \sf_1 + \sb_1(\sigma + \sf_1 + r ) +\frac{1}{C_1L\sigma}(D^2+(\sf_1+r)^2)\\
&\lesssim (1+\sb_1)\sf_1 + \sigma\sb_1 + \frac{\sigma\sb_1^2}{2}+\frac{r^2}{2\sigma}
+\frac{D}{C_1 L} + \frac{\sf_1^2}{C_1L\sigma} +\frac{r^2}{C_1L\sigma}\\
&\lesssim \frac{D}{C_1 L} + \sf_1 + \sigma\sb_1 + \frac{r^2}{C_1L\sigma}.
\end{align}

From \eqref{thm:response:matrix:decomp:eq1}-\eqref{thm:response:matrix:decomp:eq3}, we can write
\begin{align}
\frac{r^4}{C_1^2L^2\sigma^2} &\lesssim 
(\nicefrac{C}{C_1})^2\sigma^2 L^2\cdot\frac{r(d_1+d_2)}{n} + 
(\nicefrac{C}{C_1})^2\sigma^2 L^2\cdot\frac{s\log p}{n} 
+ \frac{1}{C_1^2L^2\sigma^2}\cdot\frac{(\sa^*)^2s}{n}\\
&+ C_1^2C^4\sigma^2 L^6\epsilon^2\log^2(e/\epsilon). 
\end{align}

We have that
\begin{align}
\spadesuit_2(F,\hat r)  = \frac{1}{\sd_1^2}(4F+3\hat r)^2\
\lesssim F^2 + \hat r^2 
\lesssim  \frac{D^2}{C_1^2 L^2} 
+ \sf_1^2 + \sigma^2\sb_1^2 + \hat r^2 + \frac{r^4}{C_1^2L^2\sigma^2}.
\end{align}
It follows that
\begin{align}
D^2 + \spadesuit_2(F,\hat r) & \le D^2\left(1+\frac{\calO(1)}{C_1^2L^2}\right)\\
&+ \calO(1)C^2\cdot\sigma^2L^2\frac{1+\log(1/\delta)}{n}\\
&+ \calO(1)C^2\cdot\sigma^2 L^4\cdot\frac{r(d_1+d_2)}{n}
+ \calO(1)C^2\cdot\sigma^2 L^4\cdot\frac{s\log p}{n}
+ \left(1+\frac{\calO(1)}{C_1^2\sigma^2L^2}\right)\frac{(\sa^*)^2s}{n}\\
&+\calO(1)C^4\cdot C_1^2\sigma^2 L^6\epsilon^2\log^2(e/\epsilon). 
\end{align}
This establishes \eqref{thm:tr:reg:matrix:decomp:eq1} in Theorem \ref{thm:tr:reg:matrix:decomp} using \eqref{thm:improved:rate:main:rate:eq1} in Theorem \ref{thm:improved:rate}.

Note that, since $D\lesssim \sigma$ and 
$\sd_1\asymp1$, 
$$
(\nicefrac{D^2}{\sigma\sd_1})\bigvee\left[(\nicefrac{2\sqrt{2}}{\sd_1})D + \clubsuit_2(F,\hat r)\right]
\lesssim D + \clubsuit_2(F,\hat r).
$$
Also, $\clubsuit_2(F,\hat r)\lesssim\spadesuit_2^{1/2}(F,\hat r)$. Thus
\begin{align}
(\nicefrac{D^2}{\sigma\sd_1})\bigvee\left[(\nicefrac{2\sqrt{2}}{\sd_1})D + \clubsuit_2(F,\hat r)\right]&\le D\left(\calO(1)+\frac{\calO(1)}{C_1L}\right)\\
&+ \calO(1)C\cdot\sigma L\frac{1+\sqrt{\log(1/\delta)}}{\sqrt{n}}\\
&+ \calO(1)C\cdot\sigma L^2\cdot\sqrt{\frac{r(d_1+d_2)}{n}}\\
&+ \calO(1)C\cdot\sigma L^2\cdot\sqrt{\frac{s\log p}{n}}
+ \left(1+\frac{\calO(1)}{C_1\sigma L}\right)\sa^*\sqrt{\frac{s}{n}}\\
&+\calO(1)C^2\cdot C_1\sigma L^3\epsilon\log(e/\epsilon).
\end{align}
This establishes \eqref{thm:tr:reg:matrix:decomp:eq2} in Theorem \ref{thm:tr:reg:matrix:decomp} using \eqref{thm:improved:rate:main:rate:eq2} in Theorem \ref{thm:improved:rate}.

\section{Proof of Proposition \ref{prop:lower:bound:tr:matrix:decomp}}\label{s:proof:prop:lower:bound:tr:matrix:decomp}

In trace-regression with matrix decomposition, the design is random. Up to conditioning on the feature data, the proof of Proposition \ref{prop:lower:bound:tr:matrix:decomp} follows almost identical arguments in the proof of Theorem 2 in \cite{2012agarwal:negahban:wainwright} for the matrix decomposition problem with identity design $(\frX=I)$. We present a sketch here for completeness.  

We prepare the ground so to apply Fano's method. For any 
$\bfTheta:=[\bfB,\bfGamma]\in(\mdR^p)^2$, we define for convenience
$
\Vert\bfTheta\Vert_F^2:=\Vert\bfB\Vert_F^2
+\Vert\bfGamma\Vert_F^2.
$
Given $\eta>0$ and $M\in\mathbb{N}$, a $\eta$-packing of $\calA(r,s,\sa^*)$ of size $M$ is a finite subset 
$\calA=\{\bfTheta_1,\ldots,\bfTheta_M\}$ of 
$\calA(r,s,\sa^*)$ satisfying 
$
\Vert\bfTheta_\ell-\bfTheta_k\Vert_F\ge\eta
$
for all $\ell\neq k$. For model \eqref{equation:least-squares:regression}, let $y_1^n:=\{y_i\}_{i\in[n]}$ and $\bfX_1^n:=\{\bfX_i\}_{i\in[n]}$. 
For any $k\in[M]$ and $i\in[n]$, we will denote by 
$
P^k_{y_1^n|\bfX_1^n}
$
(or $P^k_{y_i|\bfX_1^n}$)
the conditional distribution of $y_1^n$ (or $y_i$) given 
$\bfX_1^n$ corresponding to the model \eqref{equation:least-squares:regression} with parameters $\bfTheta_k=[\bfB_k,\bfGamma_k]$ belonging to the packing $\calA$. Being normal distributions, they are mutually absolute continuous. We denote by $\KL(\probn\Vert\bfQ)$ the Kullback-Leibler divergence between probability measures $\probn$ and $\bfQ$. In that setting, Fano's method assures that
\begin{align}
\inf_{\hat\bfTheta}\sup_{\bfTheta^*\in\calA(r,s,\sa^*)}\prob_{\bfTheta^*}
\left\{\Vert\hat\bfTheta-\bfTheta^*\Vert_F^2
\ge\eta^2\right\}\ge
1-\frac{\frac{1}{\binom{M}{2}}\sum_{k,\ell=1}^M\esp_{\bfX_1^n}\KL\left(P^k_{y_1^n|\bfX_1^n}\big\Vert P^\ell_{y_1^n|\bfX_1^n}\right)+\log 2}{\log M}.
\label{equation:Fano}
\end{align}
The proof follows from an union bound on the following two separate lower bounds. 

\emph{Lower bound on the low-spikeness bias}. It is sufficient to give a lower bound on $\calA(1,s,\sa^*)$. We define the subset $\calA$ of $\calA(1,s,\sa^*)$ with size $M=4$ by
\begin{align}
\calA:=\{[\bfB^*,-\bfB^*],[-\bfB^*,\bfB^*],
(\nicefrac{1}{2})[\bfB^*,-\bfB^*],
[\bf0,\bf0]\}
\end{align}
using the matrix $\bfB^*\in\mdR^p$ defined by
$$
\bfB^*:=\frac{\sa^*}{\sqrt{n}}
\left[\begin{array}{c}
1\\
0\\
\vdots\\
0
\end{array}\right]
\underbrace{\left[\begin{array}{ccccccc}
1 &
1 &
\cdots &
1 &
0 &
\cdots &
0 
\end{array}\right]^\top}_{\boldf^\top},
$$
where $\boldf\in\re^{d_2}$ has $s$ unit coordinates. It is easy to check that $\calA$ is a $\eta$-packing of 
$\calA(1,s,\sa^*)$ with $\eta=c_0\sa^*\sqrt{\frac{s}{n}}$ for some constant $c_0>0$. For any element $[\bfB_k,\bfGamma_k]$ of $\calA$, $\bfB_k+\bfGamma_k=0$ implying $P^k_{y_1^n|\bfX_1^n}\sim\calN_n(0,n\sigma^2\bfI)$. Hence, for any $k\neq\ell$, 
$$
\KL\left(P^k_{y_1^n|\bfX_1^n}\big\Vert P^\ell_{y_1^n|\bfX_1^n}\right)=0.
$$
From \eqref{equation:Fano}, one obtains a lower bound with rate of order $\sa^*\sqrt{\frac{s}{n}}$ with positive probability. 

\emph{Lower bound on the estimation error}. From the packing constructions in Lemmas 5 and 6 in \cite{2012agarwal:negahban:wainwright}, one may show that, for $d_1,d_2\ge10$, $\sa^*\ge 32\sqrt{\log p}$, $s<p$ and any $\eta>0$, there exists a $\eta$-packing $\calA=\{\bfTheta_k\}_{k\in[M]}$ of $\calA(r,s,\sa^*)$ with size 
\begin{align}
M\ge \frac{1}{4}\exp\left\{\frac{s}{2}\log\frac{p-s}{s/2}+\frac{r(d_1+d_2)}{256}\right\}
,\label{equation:M:lower:bound}
\end{align}
satisfying $\Vert\bfTheta_k\Vert_F\le3\eta$ for any $k\in[M]$. By independence of $\{\xi_i\}_{i\in[n]}$ and $\bfX_1^n$ and isotropy,
\begin{align}
\esp_{\bfX_1^n}\KL\left(P^k_{y_1^n|\bfX_1^n}\big\Vert P^\ell_{y_1^n|\bfX_1^n}\right)=\sum_{i\in[n]}\esp_{\bfX_1^n}\KL\left(P^k_{y_i|\bfX_i}\big\Vert P^\ell_{y_i|\bfX_i}\right)
=\frac{n\Vert\bfTheta_k-\bfTheta_k\Vert_F^2}{2\sigma^2}\le \frac{18n}{\sigma^2}\eta^2.
\end{align}
From \eqref{equation:Fano} and \eqref{equation:M:lower:bound}, one then checks that tacking
\begin{align}
\eta^2:=c_o\sigma^2\left\{\frac{r(d_1+d_2)}{n}
+\frac{s}{n}\log\left(\frac{p-s}{s/2}\right)\right\},
\end{align}
for some constant $c_0>0$, one obtains a lower bound with rate of order $\eta^2$ with positive probability.

\section{Proof of Theorem \ref{thm:response:sparse-low-rank:regression}, case (i)}\label{s:proof:thm:response:sparse-low-rank:regression:i}
In the following $\calR:=\Vert\cdot\Vert_1$ and $\calS\equiv0$, 
$\hat\bfGamma=\bfGamma=\bold{0}$, $\sa=\infty$ and $\sf_*=\chi=0$. A standard Gaussian maximal inequality implies 
$\mathscr{G}(\bfSigma^{1/2}\mbB_1^p)\lesssim\rho_1(\bfSigma)\sqrt{\log p}$ and Proposition E.2 in  \cite{2018bellec:lecue:tsybakov} implies 
$\mathscr{G}(\mbB_\sharp)\lesssim1$. Next, we assume that $n\ge C_0L^4(1+\log(1/\delta))$ for an absolute constant to be determined next. We will also use that $L\ge1$. In the following, $C>0$ is the universal constant stated in Proposition \ref{proposition:properties:subgaussian:designs}. Without loss on generality, we assume the constant $\sc_*$ in Theorem \ref{thm:improved:rate} is $\ge1$. No effort is made to optimize the numerical constants. 

By Proposition \ref{proposition:properties:subgaussian:designs}(ii) and taking $C_0\gtrsim C^2$ large enough, we get that, on an event 
$\calE_1$ of probability $\ge1-\delta$, $\TP_{\Vert\cdot\Vert_1}(\sa_1,\sa_2)$ holds with constants $\sa_1\in(0,1)$ and 
\begin{align}
\sa_2 \asymp CL^2\rho_1(\bfSigma)\sqrt{\frac{\log p}{n}}.
\end{align}
By Proposition \ref{proposition:properties:subgaussian:designs}(iii), on an event 
$\calE_2$ of probability $\ge1-\delta$, $\IP_{\Vert\cdot\Vert_1,0,\Vert\cdot\Vert_\sharp}(\sb_1(\delta),\sb_2,0,\sb_4)$ holds with constants
$\sb_4\asymp\frac{CL}{\sqrt{n}}$, 
\begin{align}
\sb_1(\delta) \asymp CL\frac{1+\sqrt{\log(1/\delta)}}{\sqrt{n}},
\quad\mbox{and}\quad
\sb_2\asymp CL\rho_1(\bfSigma)\sqrt{\frac{\log p}{n}}.
\end{align}
By Proposition \ref{proposition:properties:subgaussian:designs}(iv) and $C_0\ge1$ large enough, we have that, on an event $\calE_3$ of probability $\ge1-\delta$, $\MP_{\Vert\cdot\Vert_1,0,\Vert\cdot\Vert_\sharp}(\sf_1(\delta),\sf_2,0,\sf_4)$ holds with constants $\sf_4\asymp\frac{C\sigma}{\sqrt{n}}$, 
\begin{align}
\sf_1(\delta) \asymp C\sigma L\frac{1+\sqrt{\log(1/\delta)}}{\sqrt{n}},
\quad\mbox{and}\quad
\sf_2\asymp C\sigma L\rho_1(\bfSigma)\sqrt{\frac{\log p}{n}}.
\end{align}
Finally, from Lemma \ref{lemma:ATP:calS=0} and taking $C_0\gtrsim C^2$ large enough, we have that, on the event $\calE_1\cap\calE_2\cap\calE_3$ of probability $\ge1-3\delta$, all the previous stated properties hold and 
$\ATP_{\Vert\cdot\Vert_1,0,\Vert\cdot\Vert_\sharp}(\sd_1,\sd_2,0,\sd_4)$ holds\footnote{Since there is no matrix decomposition, the proof of Theorem \ref{thm:response:sparse-low-rank:regression} only needs a particular version of the mentioned properties --- with $\calS\equiv0$ and $\bfW\equiv\bf0$. We give a proof derived from the unified Theorem \ref{thm:improved:rate} --- which handles matrix decomposition.} with constants $\sd_1\in(0,1)$, $\sd_4\asymp\frac{CL}{\sqrt{n}}$, 
\begin{align}
\sd_2\asymp CL^2\rho_1(\bfSigma)\sqrt{\frac{\log p}{n}}.
\end{align}

The rest of the proof will happen on the event 
$\calE_1\cap\calE_2\cap\calE_3$. We will invoke the general Theorem \ref{thm:improved:rate}. Conditions (i)-(iii) were shown previously. We now verify conditions (iv)-(v). Note that
\begin{align}
4[(\sigma\sd_2)\vee(2\sf_2+2\sc_*\sigma\sb_2)]
\lesssim (1+\sc_*)C\sigma L^2\rho_1(\bfSigma)\sqrt{\frac{\log p}{n}}\asymp\lambda.
\end{align}
Next, we choose $\tau\asymp C_1\sc_*^2C\sigma L^2/\sqrt{n}$ for some absolute constant $C_1\ge1$. We have 
$
4[(\sigma\sd_4)\vee(2\sf_4)]\lesssim C\sigma L/\sqrt{n}\lesssim\tau. 
$
This shows (iv). Finally, by the choice of $\tau$, 
\begin{align}
\hat\sf_1 =\sf_1 + \sc_*(\sigma\sb_1) + 2\sc_*^2\sigma^2(\nicefrac{\sb_4}{\tau})
\lesssim (1+\sc_*)C\sigma L\frac{1+\sqrt{\log(1/\delta)}}{\sqrt{n}}
+ \frac{1}{C_1L}\sigma, 
\end{align}
which can be set strictly less than $\sigma\sd_1/2\asymp\sigma$ taking $C_0\gtrsim C^2(1+\sc_*)^2$ and $C_1\ge1$ large enough. Hence, (v) holds.

In what follows, we verify the conditions \eqref{cond4:general:norm}-\eqref{cond9:general:norm}. Let $\bb$ with $D=\Vert\frX^{(n)}(\bb)-\boldf^{(n)}\Vert_2$ and $\Vert\bb\Vert_0\le s$. Note that conditions \eqref{cond5:general:norm}-\eqref{cond6:general:norm} are trivially satisfied. We have that $\Psi_{\Vert\cdot\Vert_1}(\calP_{\bb}(\bfDelta_{\bb}))\le\sqrt{s}$.  Let us denote $\mu(\bb):=\mu\left(\calC_{\bb,\Vert\cdot\Vert_1}(6)\right)$.
By the choice of $(\lambda,\tau)$,
\begin{align}
r^2 &\asymp (1+\sc_*)^2C^2\sigma^2 L^4\rho_1^2(\bfSigma)\mu^2(\bb)
\cdot\frac{s\log p}{n}\\
& + C_1^2\sc_*^4C^2\sigma^2L^4\epsilon\log(e/\epsilon),
\end{align}
where we used that $\Omega^2\le 2o\log(en/o)$. We have 
$r<\sigma\sd_1\asymp\sigma$ if
\begin{align}
C_2(1+\sc_*)CL^2\rho_1(\bfSigma)
\mu(\bb)\cdot
\sqrt{\frac{s\log p}{n}}&<1,\label{thm:response:sparse-low-rank:reg:eq1}\\
C_1\sc_*^2CL^2\sqrt{\epsilon\log(e/\epsilon)}<c_1,\label{thm:response:sparse-low-rank:reg:eq2}
\end{align}
for some absolute constants $C_2\ge1$ and $c_1\in(0,1)$. 

Next, we show that $D^2+\spadesuit_2(\sf_1,r)\le\sc_*^2\sigma^2$. For this to be true it is sufficient that $D\le\frac{\sc_*\sigma}{3}$ and 
$\spadesuit_2^{1/2}(\sf_1,r)\le\frac{3\sc_*\sigma}{4}$. Note that
$$
\spadesuit_2^{1/2}(\sf_1,r) =\frac{4\sf_1+3r}{\sd_1}
\lesssim C\sigma L\frac{1+\sqrt{\log(1/\delta)}}{\sqrt{n}}
+r,
$$
which is not greater than $\frac{3\sc_*\sigma}{4}$ if
\begin{align}
\calO(1)C L\frac{1+\sqrt{\log(1/\delta)}}{\sqrt{n}}&\le\frac{\sc_*}{4}\\
\calO(1)(1+\sc_*)CL^2\rho_1(\bfSigma)
\mu(\bb)\cdot
\sqrt{\frac{s\log p}{n}}&\le\frac{\sc_*}{4},\\
\calO(1)C_1\sc_*^2CL^2\sqrt{\epsilon\log(e/\epsilon)}&\le \frac{\sc_*}{4}.
\end{align}

Next, we show that 
$
\left[(\nicefrac{D^2}{\sigma\sd_1})\right]\bigvee\left[(\nicefrac{2\sqrt{2}}{\sd_1})D + \clubsuit_2(\sf_1,r)\right]\le \sc_*\sigma.
$
For that to happen it is sufficient that 
$
D\le [(\sqrt{\sd_1 \sc_*})\wedge(\sd_1\sc_*/6\sqrt{2})]\sigma
$
and $\clubsuit_2(\sf_1,r)\le\frac{3\sc_*\sigma}{4}$. Since $\sd_1\asymp1$ and $\sc_*\ge1$, the first relation is guaranteed if $D\lesssim \sqrt{\sc_*}\sigma$. We have 
$$
\clubsuit_2(\sf_1,r) =\frac{16\sf_1+12r}{\sd_1^2}
\lesssim C\sigma L\frac{1+\sqrt{\log(1/\delta)}}{\sqrt{n}}
+ r,
$$
which is not greater than $\frac{3\sc_*\sigma}{4}$ if the previous three conditions in display are met --- up to enlarging the constants if necessary.  

Optimizing the previous inequalities, we conclude that for conditions \eqref{cond4:general:norm}-\eqref{cond9:general:norm} to hold it is sufficient to take $\sc_*\asymp1$ large enough, $C_0\gtrsim C^2$ large enough and that \eqref{thm:response:sparse-low-rank:reg:eq1}-\eqref{thm:response:sparse-low-rank:reg:eq2} and $D\lesssim\sigma$ (possibly with a small enough constant) hold.

In conclusion, assume $n\ge C_0L^4(1+\log(1/\delta))$ with $C_0\gtrsim C^2$ large enough and set the hyper-parameters to be
\begin{align}
\lambda\asymp C\sigma L^2\rho_1(\bfSigma)\sqrt{\frac{\log p}{n}},\quad
\tau\asymp\frac{C_1C\sigma L^2}{\sqrt{n}}.
\end{align}
Assume \eqref{thm:response:sparse-low-rank:reg:eq2} holds. Let any 
$\bb\in\re^p$ such that $D=\Vert\frX^{(n)}(\bb)-\boldf^{(n)}\Vert_2\lesssim\sigma$ and \eqref{thm:response:sparse-low-rank:reg:eq1} all hold. Then conditions (i)-(v) and conditions \eqref{cond4:general:norm}-\eqref{cond9:general:norm} of Theorem \ref{thm:improved:rate} are satisfied, implying \eqref{thm:improved:rate:main:rate:eq1}-\eqref{thm:improved:rate:main:rate:eq2} for such $\bb$. We now verify the statement of Theorem \ref{thm:response:sparse-low-rank:regression}, case (i). 

First, note that
\begin{align}
\hat r^2 = s\lambda^2\mu^2(\bb) &\lesssim 
 C^2\sigma^2 L^4\rho_1^2(\bfSigma)\mu^2(\bb)
\cdot\frac{s\log p}{n}\\
r^2&\lesssim \hat r^2 + C_1^2C^2\sigma^2L^4\epsilon\log(e/\epsilon),\\
\end{align}
Next, using that $D\lesssim\sigma$, $\sd_1\asymp1$, $n\ge C_0L^4(1+\log(1/\delta))$ and $\sb_1r\le \frac{\sigma\sb_1^2}{2}+\frac{r^2}{2\sigma}$,
\begin{align}
F &\lesssim \sf_1 + \sb_1(\sigma + \sf_1 + r ) +\frac{1}{C_1\sigma L}(D^2+(\sf_1+r)^2)\\
&\lesssim (1+\sb_1)\sf_1 + \sigma\sb_1 + \frac{\sigma\sb_1^2}{2}+\frac{r^2}{2\sigma}
+\frac{D}{C_1L} + \frac{\sf_1^2}{C_1\sigma L} +\frac{r^2}{C_1\sigma L}\\
&\lesssim \frac{D}{C_1 L} + \sf_1 + \sigma\sb_1 + \frac{r^2}{C_1\sigma L}.
\end{align}

From \eqref{thm:response:sparse-low-rank:reg:eq1}-\eqref{thm:response:sparse-low-rank:reg:eq2}, we can write
\begin{align}
\frac{r^4}{C_1^2\sigma^2L^2} &\lesssim 
(\nicefrac{C}{C_1})^2\sigma^2 L^2\rho_1^2(\bfSigma)\mu^2(\bb)
\cdot\frac{s\log p}{n} 
+ C_1^2C^4\sigma^2 L^6\epsilon^2\log^2(e/\epsilon). 
\end{align}

We have that
\begin{align}
\spadesuit_2(F,\hat r) & = \frac{1}{\sd_1^2}(4F+3\hat r)^2\\
&\lesssim F^2 + \hat r^2 \\
&\lesssim  \frac{D^2}{C_1^2L^2} 
+ \sf_1^2 + \sigma^2\sb_1^2 + \hat r^2 + \frac{r^4}{C_1^2\sigma^2L^2}.
\end{align}
It follows that
\begin{align}
D^2 + \spadesuit_2(F,\hat r) & \le D^2\left(1+\frac{\calO(1)}{C_1^2L^2}\right)\\
&+ \calO(1)C^2\cdot\sigma^2L^2\frac{1+\log(1/\delta)}{n}
+ \calO(1)C^2\cdot\sigma^2 L^4\rho_1^2(\bfSigma)\mu^2(\bb)
\cdot\frac{s\log p}{n} \\
&+\calO(1)C^4\cdot C_1^2\sigma^2 L^6\epsilon^2\log^2(e/\epsilon). 
\end{align}
This establishes \eqref{thm:response:sparse-low-rank:regression:eq1} in Theorem \ref{thm:response:sparse-low-rank:regression}, case (i) using \eqref{thm:improved:rate:main:rate:eq1} in Theorem \ref{thm:improved:rate}.

Note that, since $D\lesssim \sigma$ and 
$\sd_1\asymp1$, 
$$
(\nicefrac{D^2}{\sigma\sd_1})\bigvee\left[(\nicefrac{2\sqrt{2}}{\sd_1})D + \clubsuit_2(F,\hat r)\right]
\lesssim D + \clubsuit_2(F,\hat r).
$$
Also, $\clubsuit_2(F,\hat r)\lesssim\spadesuit_2^{1/2}(F,\hat r)$. Thus
\begin{align}
(\nicefrac{D^2}{\sigma\sd_1})\bigvee\left[(\nicefrac{2\sqrt{2}}{\sd_1})D + \clubsuit_2(F,\hat r)\right]&\le D\left(\calO(1)+\frac{\calO(1)}{C_1L}\right)\\
&+ \calO(1)C\cdot\sigma L\frac{1+\sqrt{\log(1/\delta)}}{\sqrt{n}}
+ \calO(1)C\cdot\sigma L^2\rho_1(\bfSigma)\mu(\bb)
\cdot\sqrt{\frac{s\log p}{n}}\\
&+\calO(1)C^2\cdot C_1\sigma L^{3}\epsilon\log(e/\epsilon).
\end{align}
This establishes \eqref{thm:response:sparse-low-rank:regression:eq2} in Theorem \ref{thm:response:sparse-low-rank:regression}, case (i) using \eqref{thm:improved:rate:main:rate:eq2} in Theorem \ref{thm:improved:rate}.

\section{Proof of Theorem \ref{thm:response:sparse-low-rank:regression}, case (ii)}\label{s:proof:thm:response:sparse-low-rank:regression:ii}

We need to consider different cones. Recall that 
$\Omega:=\sqrt{\sum_{i=1}^o\omega_i^2}$ and let $\bar\Omega_s:=\sqrt{\sum_{j=1}^s\bar\omega_j^2}$. 
\begin{definition}\label{def:cones:slope:norm}
For $c_0,\gamma>0$, let 
\begin{align}
\overline\calC_s(c_0)&:=\left\{\bv\in\re^p
:\sum_{j=s+1}^p\bar\omega_j\bv_j^\sharp\le c_0\bar\Omega_s\Vert\bv\Vert_2\right\},\\
\overline\calC_s(c_0,\gamma)&:=\left\{[\bv,\bu]\in\re^p\times\re^n: \gamma\sum_{j=s+1}^p\bar\omega_j\bv_j^\sharp
+\sum_{i=o+1}^n\omega_i\bu_i^\sharp\le c_0\left[\gamma\bar\Omega_s\Vert\bv\Vert_2+\Omega\Vert\bu\Vert_2\right]\right\}.
\end{align}
\end{definition}

\begin{definition}\label{def:r:slope:norm}
Given $[\bv,\bu]\in\re^p\times\re^n$, $s\in[n]$ and $c_0,\alpha>0$, define 
\begin{align}
r_{\lambda,\alpha,c_0}(s):=\{\lambda^2\bar\Omega_s^2\mu^2(\overline\calC_s(2c_0))+\alpha^2\}^{1/2},
\end{align}
and 
\begin{align}
\triangle_{\lambda,\tau}(\bv,\bu)&:=(\nicefrac{3\lambda\bar\Omega_s}{2})\Vert\bv\Vert_2
-(\nicefrac{\lambda}{2})\sum_{j=s+1}^p\bar\omega_j\bv_j^\sharp
+(\nicefrac{3\tau\Omega}{2})\Vert\bu\Vert_2 -(\nicefrac{\tau}{2})\sum_{i=o+1}^n\omega_i\bu_i^\sharp.
\end{align}
\end{definition}

The proof with $\calR=\Vert\cdot\Vert_\sharp$, the Slope norm in $\re^p$, follows a similar path to the proof of Theorem \ref{thm:response:sparse-low-rank:regression}, case (i). First, we claim that a very similar theorem to Theorem \ref{thm:improved:rate} holds but with the minor changes:
\begin{quote}
We set $\calS\equiv0$, $\hat\bfGamma=\bfGamma=\bold{0}$, $\sa=\infty$ and 
$\sf_*=\chi=0$. As \eqref{cond5:general:norm}-\eqref{cond6:general:norm} are trivially satisfied they can be removed. Also, we replace $r:=r_{\lambda,\chi,\tau\Omega,3}(\bfDelta_{\bfB},\bfDelta_{\bfGamma}|\bfB,\bfGamma)$ with $r:=r_{\lambda,\tau\Omega,3}(s)$ and $\hat r:=r_{\lambda,\chi,0,3}(\bfDelta_{\bfB},\bfDelta_{\bfGamma}|\bfB,\bfGamma)$ with $\hat r:=r_{\lambda,0,3}(s)$.
\end{quote}
Let us call it Theorem \ref{thm:improved:rate}'. Using this theorem, setting $\mu(\bb):=\mu(\overline\calC_s(6))$ for given $\bb$ with 
$\Vert\bb\Vert_0\le s$, using the bound $\mathscr{G}(\bfSigma^{1/2}\mbB_\sharp)\lesssim\rho_1(\bfSigma)$ --- which follows from Proposition E.2 in  \cite{2018bellec:lecue:tsybakov} --- and the fact that 
$\bar\Omega_s\le2s\log(ep/s)$, the proof of Theorem \ref{thm:response:sparse-low-rank:regression}, case (ii) follows very similar arguments to the case (i). 

Next, we highlight the minor changes in the proof of Theorem \ref{thm:improved:rate}'. Lemmas \ref{lemma:recursion:1st:order:condition}-\ref{lemma:MP+ATP} are unchanged. First, we obtain a variation of Lemma \ref{lemma:aug:rest:convexity} --- with a similar proof.\footnote{In fact, we only need to consider the cases $\bv\in\overline\calC_s(2c_0)$ and 
$\bv\notin\overline\calC_s(2c_0)$.}. 

\begin{lemma}\label{lemma:aug:rest:convexity:slope}
Define $\gamma:=\lambda/\tau$ and let $c_0>0$. Then, for any $[\bv,\bu]\in\overline\calC_{s}(c_0,\gamma)$, 
\begin{align}
\triangle_{\lambda,\tau}(\bv,\bu)&\le(\nicefrac{3}{2})\cdot r_{\lambda,\tau\Omega,c_0}(s)\cdot\Vert[\bv,\bu]\Vert_\Pi,\label{lemma:aug:rest:convexity:slope:eq2}\\
\lambda\Vert\bv\Vert_\sharp + \tau\big\|\bu\big\|_\sharp&\le 
2(c_0+1)\cdot r_{\lambda,\tau\Omega,c_0}(s)\cdot
\Vert[\bv,\bu]\Vert_\Pi.\label{lemma:aug:rest:convexity:slope:eq3}
\end{align}
\end{lemma}

Using this lemma, we obtain a variation of Proposition \ref{prop:suboptimal:rate} --- again with a similar proof, using 
$\triangle:=\triangle_{\lambda,\tau}(\bfDelta_{\bb},\bfDelta^{\hat\btheta})$ instead of $\triangle:=\triangle_{\lambda,\chi,\tau}(\bfDelta_{\bfB},\bfDelta_{\bfGamma},\bfDelta^{\hat\btheta}|\bfB,\bfGamma)$ and $r:=r_{\lambda,\tau\Omega,3}(s)$ instead of $r:=r_{\lambda,\chi,\tau\Omega,3}(\bfDelta_{\bfB},\bfDelta_{\bfGamma}|\bfB,\bfGamma)$. 
\begin{proposition}\label{prop:suboptimal:rate:slope}
Suppose the conditions (i)-(ii) of Theorem \ref{thm:improved:rate}' hold and, additionally,
\begin{itemize}
\item[\rm (iii')] 
$
\lambda\ge4[(\sigma\sd_2)\vee(2\sf_2)],
$ 
and
$
\tau\ge 4[(\sigma\sd_4)\vee(2\sf_4)].
$
\item[\rm (iv')] $2\sf_1\le\sigma\sd_1$. 
\end{itemize}
For any $D\ge0$ and $\bb$ satisfying the constraints
\eqref{cond4:general:norm} (with $\bfGamma=\bold{0}$) and \eqref{cond1:general:norm} (with $r:=r_{\lambda,\tau\Omega,3}(s)$),
\begin{align}
(\nicefrac{1}{2})(\lambda\Vert\bfDelta_{\bb}\Vert_\sharp + \tau\Vert\bfDelta^{\hat\btheta}\Vert_\sharp)
+ \Vert\bfDelta^{(n)} + \bfDelta^{\hat\btheta}\Vert_2^2
&\le D^2 + \spadesuit_2(\sf_1,r), \label{prop:suboptimal:rate:slope:eq1}
\end{align}
where $r := r_{\lambda,\tau\Omega,3}(s)$. 
Moreover, 
\begin{align}
\Vert[\bfDelta_{\bb},\bfDelta^{\hat\btheta}]\Vert_{\Pi} &\le
\left[(\nicefrac{D^2}{\sigma\sd_1})\right]\bigvee\left[(\nicefrac{2\sqrt{2}}{\sd_1})D + \clubsuit_2(\sf_1,r)\right]. 
\label{prop:suboptimal:rate:slope:eq2}
\end{align}
\end{proposition}

Next, Lemmas \ref{lemma:recursion:delta:bb:general:norm}-\ref{lemma:MP+IP+ATP} are unchanged. Using the previous auxiliary results, we claim that the proof of Theorem \ref{thm:improved:rate}' follows the same arguments in the proof of  Theorem \ref{thm:improved:rate} --- using 
$\triangle:=\triangle_{\lambda,0}(\bfDelta_{\bb},\bold{0})$ instead of 
$\triangle:=\triangle_{\lambda,\chi,0}(\bfDelta_{\bfB},\bfDelta_{\bfGamma},\bold{0}|\bfB,\bfGamma)$, $r:=r_{\lambda,0,3}(s)$ instead of $r:=r_{\lambda,\chi,0,3}(\bfDelta_{\bfB},\bfDelta_{\bfGamma}|\bfB,\bfGamma)$ and $\overline\calC_s(3)$ instead of $\calC_{\bfB,\bfGamma}(3,\gamma_{\calR},\gamma_{\calS},0)$. 

\section{Proof of Theorem \ref{thm:response:sparse-low-rank:regression}, case (iii)}\label{s:thm:response:sparse-low-rank:regression:iii}

Invoking Theorem \ref{thm:improved:rate}, the proof follows exact the same guidelines as in the Proof of Theorem \ref{thm:response:sparse-low-rank:regression}, case (i) but using the nuclear norm $\calR:=\Vert\cdot\Vert_N$. We highlight next the minor changes. First, $\mathscr{G}(\frS^{1/2}(\mbB_{\Vert\cdot\Vert_N}))\le\rho_N(\bfSigma)(\sqrt{d_1}+\sqrt{d_2})$ by Lemma H.1 in \cite{2011negahban:wainwright} --- up to an absolute constant. As a consequence, in the constants $\{\sa_i\}$, $\{\sb_i\}$, $\{\sc_i\}$, $\{\sd_i\}$ and $\{\sf_i\}$, we have to replace the factor $\rho_1(\bfSigma)\sqrt{\frac{\log p}{n}}$ by $\rho_N(\bfSigma)\sqrt{\frac{d_1+d_2}{n}}$. Second, we have $\Psi_{\Vert\cdot\Vert_N}(\calP_{\bfB}(\bfDelta_{\bfB}))\le\sqrt{r}$ --- so that we replace $\mu(\bb)=\mu\left(\calC_{\bb,\Vert\cdot\Vert_1}(6)\right)$ by $\mu(\bfB)=\mu\left(\calC_{\bfB,\Vert\cdot\Vert_N}(6)\right)$. 

\section{Proof of Theorem \ref{thm:improved:rate:q=1'}}\label{s:proof:thm:improved:rate:q=1'}
Throughout this section $\calR$ is a decomposable norm on 
$\mdR^p$ and we grant Assumption \ref{assump:label:contamination} and model \eqref{equation:structural:equation}. We set 
$\hat\bxi:=\by-\frX(\hat\bfB)-\sqrt{n}\hat\btheta$, 
$\bfDelta:=\frX(\hat\bfB)-\boldf$ and $\hat\sigma:=\Vert\bxi^{(n)}\Vert_2$. Given $\bfB$, we define the quantities
$
\bxi_{\bfB} := \by-\frX(\bfB)-\sqrt{n}\btheta^*
$
and
$
\calE_{\bfB} := \frX(\bfB) - \boldf.
$
Note that $\bxi=\bxi_{\bfB}+\calE_{\bfB}$. 

\begin{lemma}\label{lemma:recursion:1st:order:condition:q=1'}
For all $\bfB\in\mdR^p$, it holds that
\begin{align}
0&\le \langle\bxi^{(n)}_{\bfB},\frM^{(n)}(\bfDelta_{\bfB},\bfDelta^{\hat\btheta})\rangle
+(\hat\sigma\lambda)\big(\calR(\bfB) - \calR(\hat\bfB)\big) +  
(\hat\sigma\tau)\big(\|\btheta^*\|_\sharp -\|\hat\btheta\|_\sharp\big)\\
&+\Vert\calE_{\bfB}^{(n)}\Vert_2\left(\lambda\calR(\bfDelta_{\bfB})
+ \tau\Vert\bfDelta^{\hat\btheta}\Vert_\sharp\right).
\label{lemma:recursion:1st:order:condition:q=1':eq1}
\end{align}
and also
\begin{align}
\langle \bfDelta^{(n)} +\bfDelta^{\hat\btheta}, \frM^{(n)}(\bfDelta_{\bfB},\bfDelta^{\hat\btheta})\rangle  
&\le \langle\bxi^{(n)},\frM^{(n)}(\bfDelta_{\bfB},\bfDelta^{\hat\btheta})\rangle\\
&+(\hat\sigma\lambda)\big(\calR(\bfB) - \calR(\hat\bfB)\big) 
+(\hat\sigma\tau)\big(\|\btheta^*\|_\sharp -\|\hat\btheta\|_\sharp\big)\\
&+\left(\Vert\frM^{(n)}(\bfDelta_{\bfB},\bfDelta^{\hat\btheta})\Vert_2+\Vert\calE_{\bfB}^{(n)}\Vert_2\right)\left(
\lambda\calR(\bfDelta_{\bfB})+\tau\Vert\bfDelta^{\hat\btheta}\Vert_\sharp
\right).
\label{lemma:recursion:1st:order:condition:q=1':eq2}
\end{align}
\end{lemma}
\begin{proof}
We first prove \eqref{lemma:recursion:1st:order:condition:q=1':eq1}. It is enough to assume 
$\bxi_{\bfB}\neq\bf0$. Comparing the minimality of 
$[\hat\bfB,\hat\btheta]$ with $[\bfB,\btheta^*]$ and by convexity of 
$\bz\mapsto\Vert\bz\Vert_2$,
\begin{align}
0&\le \Vert\bxi^{(n)}_{\bfB}\Vert_2 - \Vert\hat\bxi^{(n)}\Vert_2
+\lambda \big(\calR(\bfB) - \calR(\hat\bfB)\big) +  
\tau\big(\|\btheta^*\|_\sharp -\|\hat\btheta\|_\sharp\big)\\
&\le -\left\langle\frac{\bxi^{(n)}_{\bfB}}{\Vert\bxi^{(n)}_{\bfB}\Vert_2},\hat\bxi^{(n)}-\bxi^{(n)}_{\bfB}\right\rangle
+\lambda \big(\calR(\bfB) - \calR(\hat\bfB)\big) +  
\tau\big(\|\btheta^*\|_\sharp -\|\hat\btheta\|_\sharp\big).
\end{align}
Using $\bxi^{(n)}_{\bfB}=\hat\bxi^{(n)}+\frM^{(n)}(\bfDelta_{\bfB},\bfDelta^{\hat\btheta})$ we get  
\begin{align}
0\le \langle\bxi^{(n)}_{\bfB},\frM^{(n)}(\bfDelta_{\bfB},\bfDelta^{\hat\btheta})\rangle
+\Vert\bxi^{(n)}_{\bfB}\Vert_2\lambda\big(\calR(\bfB) - \calR(\hat\bfB)\big) +  
\Vert\bxi^{(n)}_{\bfB}\Vert_2\tau\big(\|\btheta^*\|_\sharp -\|\hat\btheta\|_\sharp\big).
\end{align}

We now write
\begin{align}
&\Vert\bxi^{(n)}_{\bfB}\Vert_2\lambda\big(\calR(\bfB) - \calR(\hat\bfB)\big) + \Vert\bxi^{(n)}_{\bfB}\Vert_2\tau\big(\|\btheta^*\|_\sharp -\|\hat\btheta\|_\sharp\big)\\
&=\Vert\bxi^{(n)}\Vert_2\lambda\big(\calR(\bfB) - \calR(\hat\bfB)\big) +  
\Vert\bxi^{(n)}\Vert_2\tau\big(\|\btheta^*\|_\sharp -\|\hat\btheta\|_\sharp\big)\\
&+(\Vert\bxi^{(n)}_{\bfB}\Vert_2-\Vert\bxi^{(n)}\Vert_2)\lambda\big(\calR(\bfB) - \calR(\hat\bfB)\big) +  
(\Vert\bxi^{(n)}_{\bfB}\Vert_2-\Vert\bxi^{(n)}\Vert_2)\tau\big(\|\btheta^*\|_\sharp -\|\hat\btheta\|_\sharp\big).
\label{proof:lemma:recursion:1st:order:condition:q=1':eq1}
\end{align}
By convexity of $\bfW\mapsto\calR(\bfW)$ and 
$\btheta\mapsto\Vert\btheta\Vert_\sharp$, there exist $\bfV_{\bfB}\in\calR(\bfB)$ and $\bu_{\btheta^*}\in\partial\Vert\btheta^*\Vert_\sharp$ such that
\begin{align}
&(\Vert\bxi^{(n)}_{\bfB}\Vert_2-\Vert\bxi^{(n)}\Vert_2)\lambda\big(\calR(\bfB) - \calR(\hat\bfB)\big) +  
(\Vert\bxi^{(n)}_{\bfB}\Vert_2-\Vert\bxi^{(n)}\Vert_2)\tau\big(\|\btheta^*\|_\sharp -\|\hat\btheta\|_\sharp\big)\\
&\le (\Vert\bxi^{(n)}_{\bfB}\Vert_2-\Vert\bxi^{(n)}\Vert_2)\lambda
\left(
-\llangle\bfV_{\bfB},\hat\bfB-\bfB\rrangle
\right) +  
(\Vert\bxi^{(n)}_{\bfB}\Vert_2-\Vert\bxi^{(n)}\Vert_2)\tau
\left(
-\llangle\bu_{\btheta^*},\hat\btheta-\btheta^*\rrangle
\right)\\
&\le \Vert\calE_{\bfB}^{(n)}\Vert_2\lambda\calR(\bfDelta_{\bfB})
+ \Vert\calE_{\bfB}^{(n)}\Vert_2\tau\Vert\bfDelta^{\hat\btheta}\Vert_\sharp, 
\label{proof:lemma:recursion:1st:order:condition:q=1':eq2}
\end{align}
where we used that
$
\bxi_{\bfB}=\bxi - \calE_{\bfB}, 
$ 
$|\llangle\bfV_{\bfB},\bfDelta_{\bfB}\rrangle|\le\calR(\bfDelta_{\bfB})$ and 
$|\langle\bu_{\btheta^*},\bfDelta^{\hat\btheta}\rangle|\le\Vert\bfDelta^{\hat\btheta}\Vert_\sharp$. Equation \eqref{lemma:recursion:1st:order:condition:q=1':eq1} follow from the three previous displays. 

We now show \eqref{lemma:recursion:1st:order:condition:q=1':eq2}. The first order condition of \eqref{equation:aug:slope:rob:estimator:q=1} at $[\hat\bfB,\hat\btheta]$ is equivalent to the statement: there exist 
$\bfV\in\partial\calR(\hat\bfB)$ and $\bu\in\partial\Vert\hat\btheta\Vert_\sharp$ such that for all $[\bfB,\btheta]$,
\begin{align}
\sum_{i\in[n]}\left[y_i^{(n)}-\frX^{(n)}_i(\widehat\bfB)-\hat\btheta_i\right]\llangle\bfX^{(n)}_i,\hat\bfB-\bfB\rrangle&\ge\lambda\Vert\hat\bxi^{(n)}\Vert_2\llangle\bfV,\hat\bfB-\bfB\rrangle,\\
\langle\by^{(n)}-\frX^{(n)}(\hat\bfB)-\hat\btheta,\hat\btheta-\btheta\rangle&\ge\tau\Vert\hat\bxi^{(n)}\Vert_2\langle\bu,\hat\btheta-\btheta\rangle.
\end{align}
Setting $\btheta=\btheta^*$ and using that 
$
\by^{(n)}=\boldf^{(n)}+\btheta^*+\bxi^{(n)}
$
we obtain that for all $\bfB$, 
\begin{align}
\sum_{i\in[n]}\left[ \bfDelta_i^{(n)} +\bfDelta_i^{\hat\btheta}\right]\llangle\bfX^{(n)}_i,\bfDelta_{\bfB}\rrangle&\le\sum_{i\in[n]}\xi_i^{(n)}\llangle\bfX_i^{(n)},\bfDelta_{\bfB}\rrangle-\lambda\Vert\hat\bxi^{(n)}\Vert_2\llangle\bfV,\bfDelta_{\bfB}\rrangle,\\
\left\langle \bfDelta^{(n)} +\bfDelta^{\hat\btheta},\bfDelta^{\hat\btheta}\right\rangle
&\le\langle\bxi^{(n)},\bfDelta^{\hat\btheta}\rangle - \tau\Vert\hat\bxi^{(n)}\Vert_2\langle\bu,\bfDelta^{\hat\btheta}\rangle.
\end{align}
Summing the above inequalities,
\begin{align}
\langle \bfDelta^{(n)} +\bfDelta^{\hat\btheta}, \frM^{(n)}(\bfDelta_{\bfB},\bfDelta^{\hat\btheta})\rangle &\le
\langle\bxi^{(n)},\frM^{(n)}(\bfDelta_{\bfB},\bfDelta^{\hat\btheta})\rangle\\
&-\lambda\Vert\hat\bxi^{(n)}\Vert_2\llangle\bfV,\bfDelta_{\bfB}\rrangle - \tau\Vert\hat\bxi^{(n)}\Vert_2\langle\bu,\bfDelta^{\hat\btheta}\rangle. 
\end{align}
We now write 
\begin{align}
&-\lambda\Vert\hat\bxi^{(n)}\Vert_2\llangle\bfV,\bfDelta_{\bfB}\rrangle - \tau\Vert\hat\bxi^{(n)}\Vert_2\langle\bu,\bfDelta^{\hat\btheta}\rangle\\
&=-\lambda\Vert\bxi^{(n)}\Vert_2\llangle\bfV,\bfDelta_{\bfB}\rrangle - \tau\Vert\bxi^{(n)}\Vert_2\langle\bu,\bfDelta^{\hat\btheta}\rangle\\
&+\lambda(\Vert\bxi^{(n)}\Vert_2-\Vert\hat\bxi^{(n)}\Vert_2)\llangle\bfV,\bfDelta_{\bfB}\rrangle 
+\tau(\Vert\bxi^{(n)}\Vert_2-\Vert\hat\bxi^{(n)}\Vert_2)\langle\bu,\bfDelta^{\hat\btheta}\rangle
\end{align}

Using that $-\llangle\bfDelta_{\bfB},\bfV\rrangle\le\calR(\bfB)-\calR(\hat\bfB)$ and $-\langle\bfDelta^{\hat\btheta},\bu\rangle \le \|\btheta^*\|_\sharp-\|\hat\btheta\|_\sharp$, we get that 
\begin{align}
-\lambda\Vert\bxi^{(n)}\Vert_2\llangle\bfV,\bfDelta_{\bfB}\rrangle - \tau\Vert\bxi^{(n)}\Vert_2\langle\bu,\bfDelta^{\hat\btheta}\rangle
\le (\hat\sigma\lambda)\big(\calR(\bfB) - \calR(\hat\bfB)\big) +  
(\hat\sigma\tau)\big(\|\btheta^*\|_\sharp -\|\hat\btheta\|_\sharp\big). 
\end{align}

Finally, using the facts that $\hat\bxi^{(n)}=\bxi^{(n)}-\calE_{\bfB}^{(n)}-\frM^{(n)}(\bfDelta_{\bfB},\bfDelta^{\hat\btheta})$, 
$|\llangle\bfV,\bfDelta_{\bfB}\rrangle|\le\calR(\bfDelta_{\bfB})$ and 
$|\langle\bu,\bfDelta^{\hat\btheta}\rangle|\le\Vert\bfDelta^{\hat\btheta}\Vert_\sharp$
we get 
\begin{align}
&\lambda(\Vert\bxi^{(n)}\Vert_2-\Vert\hat\bxi^{(n)}\Vert_2)\llangle\bfV,\bfDelta_{\bfB}\rrangle 
+\tau(\Vert\bxi^{(n)}\Vert_2-\Vert\hat\bxi^{(n)}\Vert_2)\langle\bu,\bfDelta^{\hat\btheta}\rangle\\
&\le \left(\Vert\frM^{(n)}(\bfDelta_{\bfB},\bfDelta^{\hat\btheta})\Vert_2+\Vert\calE_{\bfB}^{(n)}\Vert_2\right)\left(
\lambda\calR(\bfDelta_{\bfB})+\tau\Vert\bfDelta^{\hat\btheta}\Vert_\sharp
\right). 
\end{align}
Equation \eqref{lemma:recursion:1st:order:condition:q=1':eq2} follows from the four previous displays. 
\end{proof}

Next, we upper (and lower) bound \eqref{lemma:recursion:1st:order:condition:q=1':eq2} using $\MP$ (and $\ATP$). 
\begin{lemma}\label{lemma:MP+ATP:q=1'}
Suppose conditions (i)-(ii) of Theorem \ref{thm:improved:rate:q=1'} hold. Let $\bfB\in\mdR^p$ and
define the quantities
\begin{align}
\blacktriangle_1 & := \left(\sf_2 + \frac{\sf_1\sd_2}{\sd_1}+\Vert\calE_{\bfB}^{(n)}\Vert_2\lambda\right)\calR(\bfDelta_{\bfB}) 
+ \left(\sf_4 + \frac{\sf_1\sd_4}{\sd_1}+\Vert\calE_{\bfB}^{(n)}\Vert_2\tau\right)\Vert\bfDelta^{\hat\btheta}\Vert_\sharp,\\
\blacktriangledown_1 & :=(\hat\sigma\lambda) \big(\calR(\bfB) - \calR(\hat\bfB)\big) 
+(\hat\sigma\tau)\big(\|\btheta^*\|_\sharp -\|\hat\btheta\|_\sharp\big).
\end{align}

Then 
\begin{align}
0\le \left(\frac{\sf_1}{\sd_1}+\Vert\calE_{\bfB}^{(n)}\Vert_2\right)\Vert\frM^{(n)}(\bfDelta_{\bfB},\bfDelta^{\hat\btheta})\Vert_2
+\blacktriangle_1 + \blacktriangledown_1,\label{lemma:MP+ATP:q=1':eq1}
\end{align}
and also
\begin{align}
\Vert\bfDelta^{(n)} +\bfDelta^{\hat\btheta}\Vert_2^2
+ \Vert\frM^{(n)}(\bfDelta_{\bfB},\bfDelta^{\hat\btheta})\Vert_2^2&\le
\Vert\frX^{(n)}(\bfB)-\boldf^{(n)}\Vert_2^2\\
& +\frac{2\sf_1}{\sd_1}\Vert\frM^{(n)}(\bfDelta_{\bfB},\bfDelta^{\hat\btheta})\Vert_2
+ 2(\blacktriangle_1 + \blacktriangledown_1)\\
&+2\Vert\frM^{(n)}(\bfDelta_{\bfB},\bfDelta^{\hat\btheta})\Vert_2\left(
\lambda\calR(\bfDelta_{\bfB})+\tau\Vert\bfDelta^{\hat\btheta}\Vert_\sharp
\right).
\label{lemma:MP+ATP:q=1':eq2}
\end{align}
\end{lemma}
\begin{proof}
By the parallelogram law,
\begin{align}
&\langle \bfDelta^{(n)} +\bfDelta^{\hat\btheta}, \frM^{(n)}(\bfDelta_{\bfB},\bfDelta^{\hat\btheta})\rangle =\\
&=\frac{1}{2}\Vert\bfDelta^{(n)} + \bfDelta^{\hat\btheta}\Vert_2^2
+ \frac{1}{2}\Vert\frM^{(n)}(\bfDelta_{\bfB},\bfDelta^{\hat\btheta})\Vert_2^2
-\frac{1}{2}\Vert\frX^{(n)}(\bfB)-\boldf^{(n)}\Vert_2^2.
\end{align}

$\ATP$ implies in particular that
\begin{align}
\Vert\frM^{(n)}(\bfDelta_{\bfB},\bfDelta^{\hat\btheta})\Vert_2 &\ge\sd_1\Vert[\bfDelta_{\bfB},\bfDelta^{\hat\btheta}]\Vert_{\Pi}
-\sd_2\calR(\bfDelta_{\bfB}) 
-\sd_4\Vert\bfDelta^{\hat\btheta}\Vert_\sharp.
\end{align}

$\MP$ implies that
\begin{align}
\langle\bxi^{(n)},\frM^{(n)}(\bfDelta_{\bfB},\bfDelta^{\hat\btheta})\rangle &\le
\sf_1\Vert[\bfDelta_{\bfB},\bfDelta^{\hat\btheta}]\Vert_{\Pi} 
+\sf_2\calR(\bfDelta_{\bfB}) 
+\sf_4\Vert\bfDelta^{\hat\btheta}\Vert_\sharp. 
\end{align}

The proof of \eqref{lemma:MP+ATP:q=1':eq2} follows from the three previous displays and inequality \eqref{lemma:recursion:1st:order:condition:q=1':eq2} of Lemma \ref{lemma:recursion:1st:order:condition:q=1'}. 

The proof of \eqref{lemma:MP+ATP:q=1':eq1} follows from the two previous displays, inequality \eqref{lemma:recursion:1st:order:condition:q=1':eq1} and the fact that 
\begin{align}
\langle\bxi^{(n)}_{\bfB},\frM^{(n)}(\bfDelta_{\bfB},\bfDelta^{\hat\btheta})\rangle & =
\langle\bxi^{(n)},\frM^{(n)}(\bfDelta_{\bfB},\bfDelta^{\hat\btheta})\rangle
-\langle\calE^{(n)}_{\bfB},\frM^{(n)}(\bfDelta_{\bfB},\bfDelta^{\hat\btheta})\rangle\\
& \le 
\langle\bxi^{(n)},\frM^{(n)}(\bfDelta_{\bfB},\bfDelta^{\hat\btheta})\rangle
+\Vert\calE^{(n)}_{\bfB}\Vert_2\Vert\frM^{(n)}(\bfDelta_{\bfB},\bfDelta^{\hat\btheta})\Vert_2. 
\end{align}
\end{proof}

The previous lemmas entail the next proposition. For convenience, given $[\bfV,\bfB,\bu]$,  we define
\begin{align}
\triangle_{\lambda,\tau}(\bfV,\bu|\bfB)&:=(\nicefrac{3\lambda}{2})(\calR\circ\calP_{\bfB})(\bfV) 
-(\nicefrac{\lambda}{2})(\calR\circ\calP_{\bfB}^\perp)(\bfV)\\
&+(\nicefrac{3\tau\Omega}{2})\Vert\bu\Vert_2 -(\nicefrac{\tau}{2})\sum_{i=o+1}^n\omega_i\bu_i^\sharp.
\end{align}
\begin{proposition}\label{prop:suboptimal:rate:q=1'}
Suppose the conditions (i)-(ii) of Theorem \ref{thm:improved:rate:q=1'} hold and, additionally, for some $\sc_n\in[0,1/4)$, 
\begin{itemize}
\item[\rm (iii')] 
$
(1-4\sc_n)\hat\sigma\lambda\ge4[\sf_2 + (\nicefrac{\sf_1\sd_2}{\sd_1})]
$
and 
$
(1-4\sc_n)\hat\sigma\tau\ge 4[\sf_4 + (\nicefrac{\sf_1\sd_4}{\sd_1})].
$
\item[\rm (iv')] 
$
56\left(\nicefrac{\sf_1}{\sd_1}+\hat\sigma\sc_n\right)\le3\hat\sigma.
$
\end{itemize}
For any $D\ge0$ and $\bfB$ satisfying the constraints
\eqref{cond4:general:norm:q=1'}-\eqref{cond1:general:norm:q=1'},  
\begin{align}
(\nicefrac{\hat\sigma}{2})(\lambda\calR(\bfDelta_{\bfB}) + \tau\Vert\bfDelta^{\hat\btheta}\Vert_\sharp)
+ \Vert\bfDelta^{(n)} + \bfDelta^{\hat\btheta}\Vert_2^2
&\le D^2 + \spadesuit_1\left(\sf_1,R\right), \label{prop:suboptimal:rate:q=1':eq1}
\end{align}
where $r:=r_{\hat\sigma\lambda,\hat\sigma\tau\Omega,6}(\bfDelta_{\bfB}|\bfB)$ and $R:=(\hat\sigma\sc_n)\vee(3r)$. Moreover, 
\begin{align}
\Vert\frM^{(n)}(\bfDelta_{\bfB},\bfDelta^{\hat\btheta})\Vert_2
\le 2D + \clubsuit_1\left(\sf_1,R\right).
\label{prop:suboptimal:rate:q=1':eq2}
\end{align}
\end{proposition}
\begin{proof}
Let
$\blacksquare_1:=(\nicefrac{\hat\sigma\lambda}{4})\calR(\bfDelta_{\bfB}) +
(\nicefrac{\hat\sigma\tau}{4})\|\bfDelta^{\hat\btheta}\|_\sharp$.
By Lemma \ref{lemma:MP+ATP:q=1'} and $\Vert\calE_{\bfB}^{(n)}\Vert_2\le\hat\sigma\sc_n$, we have 
\begin{align}
0\le \left(\frac{\sf_1}{\sd_1}+\hat\sigma\sc_n\right)\Vert\frM^{(n)}(\bfDelta_{\bfB},\bfDelta^{\hat\btheta})\Vert_2
+\blacktriangle_1+\blacktriangledown_1,
\end{align}
and also
\begin{align}
2\blacksquare_1 +\Vert\bfDelta^{(n)} +\bfDelta^{\hat\btheta}\Vert_2^2
+ \Vert\frM^{(n)}(\bfDelta_{\bfB},\bfDelta^{\hat\btheta})\Vert_2^2&\le
\Vert\frX^{(n)}(\bfB)-\boldf^{(n)}\Vert_2^2\\
&+ (\nicefrac{2\sf_1}{\sd_1})\Vert\frM^{(n)}(\bfDelta_{\bfB},\bfDelta^{\hat\btheta})\Vert_2
+ 2(\blacktriangle_1+\blacksquare_1 +\blacktriangledown_1)\\
&+2\Vert\frM^{(n)}(\bfDelta_{\bfB},\bfDelta^{\hat\btheta})\Vert_2\left(
\lambda\calR(\bfDelta_{\bfB})+\tau\Vert\bfDelta^{\hat\btheta}\Vert_\sharp
\right).
\end{align}

By Lemmas \ref{lemma:A1:B} and \ref{lemma:A1} (with 
$\nu=1/2$) and conditions (iii') and $\Vert\calE_{\bfB}\Vert_2\le\hat\sigma\sc_n$,
\begin{align}
\blacktriangle_1 + \blacksquare_1 +\blacktriangledown_1 & \le
\left[\sf_2 + \frac{\sf_1\sd_2}{\sd_1} + \hat\sigma\lambda\sc_n + (\nicefrac{\hat\sigma\lambda}{4})\right]\calR(\bfDelta_{\bfB}) + (\hat\sigma\lambda) \big(\calR(\bfB) - \calR(\hat\bfB)\big)\\
&+\left[\sf_4 + \frac{\sf_1\sd_4}{\sd_1} + \hat\sigma\tau\sc_n + (\nicefrac{\hat\sigma\tau}{4})\right]
\|\bfDelta^{\hat\btheta}\|_\sharp+(\hat\sigma\tau)\big(\|\btheta^*\|_\sharp -\|\hat\btheta\|_\sharp\big)\\
&\le(\nicefrac{\hat\sigma\lambda}{2})\calR(\bfDelta_{\bfB}) + (\hat\sigma\lambda) \big(\calR(\bfB) - \calR(\hat\bfB)\big)\\
&+(\nicefrac{\hat\sigma\tau}{2})\|\bfDelta^{\hat\btheta}\|_\sharp+(\hat\sigma\tau)\big(\|\btheta^*\|_\sharp -\|\hat\btheta\|_\sharp\big)\\
&\le \triangle_{\hat\sigma\lambda,\hat\sigma\tau}(\bfDelta_{\bfB},\bfDelta^{\hat\btheta}|\bfB).
\end{align}

Next, we will define some local variables for convenience of notation. Let 
$G:=\Vert\frM^{(n)}(\bfDelta_{\bfB},\bfDelta^{\hat\btheta})\Vert_2$, 
$
D:=\Vert\frX^{(n)}(\bfB)-\boldf^{(n)}\Vert_2, 
$
$
x:=\Vert\bfDelta^{(n)} + \bfDelta^{\hat\btheta}\Vert_2,
$
and 
$r:=r_{\hat\sigma\lambda,\hat\sigma\tau\Omega,6}(\bfDelta_{\bfB}|\bfB)$. Define also 
$\triangle:=\triangle_{\hat\sigma\lambda,\hat\sigma\tau}(\bfDelta_{\bfB},\bfDelta^{\hat\btheta}|\bfB)$ and 
\begin{align}
H&:=(\nicefrac{3\hat\sigma\lambda}{2})(\calR\circ\calP_{\bfB})(\bfDelta_{\bfB})
+(\nicefrac{3\hat\sigma\tau\Omega}{2})\Vert\bfDelta^{\hat\btheta}\Vert_2,\\
I&:=(\nicefrac{\hat\sigma\lambda}{2})(\calR\circ\calP_{\bfB}^\perp)(\bfDelta_{\bfB})
+(\nicefrac{\hat\sigma\tau}{2})\sum_{i=o+1}^n\omega_i(\bfDelta^{\hat\btheta})_i^\sharp.
\end{align}
In particular, 
$
\triangle = H - I. 
$

The previous three bounds entail the two inequalities:
\begin{align}
0&\le \left(\nicefrac{\sf_1}{\sd_1}+\hat\sigma\sc_n\right)G
+\triangle,\label{proof:prop:suboptimal:rate:q=1':eq0'}\\
2\blacksquare_1 + x^2 + G^2
&\le D^2 + \left(\nicefrac{2\sf_1}{\sd_1}\right)G +2\left(\lambda\calR(\bfDelta_{\bfB})+\tau\Vert\bfDelta^{\hat\btheta}\Vert_\sharp\right)G
+ 2\triangle.\label{proof:prop:suboptimal:rate:q=1':eq0}
\end{align}
We split our argument in two cases. 

\begin{description}
\item[Case 1:] $\left(\nicefrac{\sf_1}{\sd_1}+\hat\sigma\sc_n\right)G\ge H$. Hence, $\triangle\le H\le \left(\nicefrac{\sf_1}{\sd_1}+\hat\sigma\sc_n\right)G$. From \eqref{proof:prop:suboptimal:rate:q=1':eq0'}, 
$
I\le \left(\nicefrac{\sf_1}{\sd_1}+\hat\sigma\sc_n\right)G + H
\le 2\left(\nicefrac{\sf_1}{\sd_1}+\hat\sigma\sc_n\right)G.
$
This fact and decomposability imply 
\begin{align}
\hat\sigma\left(\lambda\calR(\bfDelta_{\bfB})+\tau\Vert\bfDelta^{\hat\btheta}\Vert_\sharp\right)
\le 2I+\frac{2H}{3} 
\le \frac{14}{3}\left(\nicefrac{\sf_1}{\sd_1}+\hat\sigma\sc_n\right)G.
\end{align}
From \eqref{proof:prop:suboptimal:rate:q=1':eq0}, we get 
\begin{align}
2\blacksquare_1 + x^2 + G^2
&\le D^2 + \left(\nicefrac{2\sf_1}{\sd_1}\right)G +
\frac{28}{3\hat\sigma}\left(\nicefrac{\sf_1}{\sd_1}+\hat\sigma\sc_n\right)G^2 + 2\triangle\\
&\le D^2 + 2\left(\nicefrac{2\sf_1}{\sd_1}+\hat\sigma\sc_n\right)G +
\frac{28}{3\hat\sigma}\left(\nicefrac{\sf_1}{\sd_1}+\hat\sigma\sc_n\right)G^2, 
\end{align}
which, together with condition (iv'), implies 
\begin{align}
2\blacksquare_1 + x^2 + \frac{G^2}{2}
&\le D^2 + 2\left(\nicefrac{2\sf_1}{\sd_1}+\hat\sigma\sc_n\right)G. 
\end{align}

From $2\left(\nicefrac{2\sf_1}{\sd_1}+\hat\sigma\sc_n\right)G\le2\left(\nicefrac{2\sf_1}{\sd_1}+\hat\sigma\sc_n\right)^2 + \frac{G^2}{2}$, we get 
\begin{align}
2\blacksquare_1 + x^2 &\le D^2 + 2\left(\nicefrac{2\sf_1}{\sd_1}+\hat\sigma\sc_n\right)^2.\label{proof:prop:suboptimal:rate:q=1':eq4} 
\end{align}
If we use instead $2\left(\nicefrac{2\sf_1}{\sd_1}+\hat\sigma\sc_n\right)G\le4\left(\nicefrac{2\sf_1}{\sd_1}+\hat\sigma\sc_n\right)^2 + \frac{G^2}{4}$, we get 
\begin{align}
\frac{G^2}{4} &\le D^2 + 4\left(\nicefrac{2\sf_1}{\sd_1}+\hat\sigma\sc_n\right)^2.\label{proof:prop:suboptimal:rate:q=1':eq5}
\end{align}

\item[Case 2:] $\left(\nicefrac{\sf_1}{\sd_1}+\hat\sigma\sc_n\right)G\le H$. From \eqref{proof:prop:suboptimal:rate:q=1':eq0'}, $0\le H+\triangle=2H-I$ so that $[\bfDelta_{\bfB},\bold{0}, \bfDelta^{\hat\btheta}]\in\calC_{\bfB,\bold{0}}(6,\gamma,0,\Omega)$, where we defined $\gamma:=\lambda/\tau$. By Lemma \ref{lemma:aug:rest:convexity}, 
\begin{align}
\triangle&\le(\nicefrac{3}{2})r\Vert[\bfDelta_{\bfB},\bfDelta^{\hat\btheta}]\Vert_{\Pi}, \\
(\hat\sigma\lambda)\calR(\bfDelta_{\bfB}) + (\hat\sigma\tau)\big\|\bfDelta^{\hat\btheta}\big\|_\sharp &\le 14r \Vert[\bfDelta_{\bfB},\bfDelta^{\hat\btheta}]\Vert_\Pi.
\end{align}
The second inequality above and $\ATP$ imply that
\begin{align}
\sd_1\Vert[\bfDelta_{\bfB},\bfDelta^{\hat\btheta}]\Vert_{\Pi} &\le G
+[(\nicefrac{\sd_2}{\lambda})\vee(\nicefrac{\sd_4}{\tau})](\lambda\calR(\bfDelta_{\bfB}) + \tau\big\|\bfDelta^{\hat\btheta}\big\|_\sharp)\\
&\le G + [(\nicefrac{\sd_2}{\lambda})\vee(\nicefrac{\sd_4}{\tau})]
14(\nicefrac{r}{\hat\sigma})\Vert[\bfDelta_{\bfB},\bfDelta^{\hat\btheta}]\Vert_{\Pi}.
\end{align}
By condition \eqref{cond1:general:norm:q=1'}, $14[(\nicefrac{\sd_2}{\lambda})\vee(\nicefrac{\sd_4}{\tau})](\nicefrac{r}{\hat\sigma})\le\sd_1/2$, implying $(\sd_1/2)\Vert[\bfDelta_{\bfB},\bfDelta^{\hat\btheta}]\Vert_{\Pi}\le G$. 
We conclude that 
\begin{align}
\triangle&\le(\nicefrac{3r}{\sd_1})G,\\
\hat\sigma\lambda\calR(\bfDelta_{\bfB}) + \hat\sigma\tau\big\|\bfDelta^{\hat\btheta}\big\|_\sharp &\le (\nicefrac{28r}{\sd_1})G. 
\end{align}

From \eqref{proof:prop:suboptimal:rate:q=1':eq0}, 
\begin{align}
2\blacksquare_1 + x^2 + G^2 &\le D^2 + \left(\nicefrac{2\sf_1}{\sd_1}\right)G + (\nicefrac{56r}{\hat\sigma\sd_1})G^2
+ (\nicefrac{6r}{\sd_1})G\\
&= D^2 + \left(\frac{2\sf_1 + 6r}{\sd_1}\right)G + (\nicefrac{56r}{\hat\sigma\sd_1})G^2, 
\end{align}
which, together with $56r\le\hat\sigma\sd_1/2$ --- as stated in condition \eqref{cond1:general:norm:q=1'} ---, entails
\begin{align}
2\blacksquare_1 + x^2 + \frac{G^2}{2} &\le  D^2 + 2\left(\frac{\sf_1 + 3r}{\sd_1}\right)G. 
\end{align}

Proceeding similarly as before, we obtain from the displayed bound that 
\begin{align}
2\blacksquare_1 + x^2 &\le D^2 + 2\left(\frac{\sf_1 + 3r}{\sd_1}\right)^2,\label{proof:prop:suboptimal:rate:q=1':eq8}\\
\frac{G^2}{4} &\le D^2 + 4\left(\frac{\sf_1 + 3r}{\sd_1}\right)^2.\label{proof:prop:suboptimal:rate:q=1':eq9}
\end{align}
\end{description}
The proof of \eqref{prop:suboptimal:rate:q=1':eq1} follows by taking the largest of the bounds in  \eqref{proof:prop:suboptimal:rate:q=1':eq4} and \eqref{proof:prop:suboptimal:rate:q=1':eq8}. The proof of \eqref{prop:suboptimal:rate:q=1':eq2} follows by taking the largest of the bounds in \eqref{proof:prop:suboptimal:rate:q=1':eq5} and \eqref{proof:prop:suboptimal:rate:q=1':eq9}.
\end{proof}

We will later invoke Proposition \ref{prop:suboptimal:rate:q=1'} in the proof of Theorem \ref{thm:improved:rate:q=1'} when bounding the nuisance error 
$\bfDelta^{\hat\btheta}$. Next, we prove the following lemma --- an easy consequence of the first order condition when fixing 
$\btheta\equiv\hat\btheta$. 

\begin{lemma}\label{lemma:recursion:delta:bb:general:norm:q=1'}
For all $\bfB\in\mdR^p$, 
\begin{align}
0&\le \langle\bxi^{(n)}_{\bfB}-\bfDelta^{\hat\btheta},\frX^{(n)}(\bfDelta_{\bfB})\rangle
+(\hat\sigma\lambda)\big(\calR(\bfB) - \calR(\hat\bfB)\big)
+ \left(\Vert\calE_{\bfB}^{(n)}\Vert_2 + \Vert\bfDelta^{\hat\btheta}\Vert_2\right)\lambda\calR(\bfDelta_{\bfB}). \label{lemma:recursion:delta:bb:general:norm:q=1':eq1}
\end{align}
and also
\begin{align}
\langle \bfDelta^{(n)},\frX^{(n)}(\bfDelta_{\bfB})\rangle &\le 
\langle\bxi^{(n)}-\bfDelta^{\hat\btheta},\frX^{(n)}(\bfDelta_{\bfB})\rangle+(\hat\sigma\lambda)(\calR(\bfB)-\calR(\hat\bfB))\\
&+\left(\Vert\frM^{(n)}(\bfDelta_{\bfB},\bfDelta^{\hat\btheta})\Vert_2
+ \Vert\calE_{\bfB}^{(n)}\Vert_2\right)
\lambda\calR(\bfDelta_{\bfB}).
\label{lemma:recursion:delta:bb:general:norm:q=1':eq2}
\end{align}
\end{lemma}
\begin{proof}
Observe that 
\begin{align}
\hat\bfB\in 
&\argmin_{\bfB}\left\{\Vert\by^{(n)}-\frX^{(n)}(\bfB)-\hat\btheta\Vert_2+\lambda\calR(\bfB)\right\}.
\end{align}
Let us define $\hat\bxi_{\bfB}:=\by-\frX(\bfB)-\sqrt{n}\hat\btheta$. We first claim that 
\begin{align}
0\le \langle\hat\bxi_{\bfB}^{(n)},\frX^{(n)}(\bfDelta_{\bfB})\rangle
+\Vert\hat\bxi_{\bfB}^{(n)}\Vert_2\lambda\big(\calR(\bfB) - \calR(\hat\bfB)\big).\label{proof:lemma:recursion:delta:bb:general:norm:q=1':eq1}
\end{align}
It is enough to assume $\hat\bxi_{\bfB}\neq\bf0$. Comparing the minimality of $\hat\bfB$ with $\bfB$ and by convexity of 
$\bu\mapsto\Vert\bu\Vert_2$,
\begin{align}
0&\le \Vert\hat\bxi^{(n)}_{\bfB}\Vert_2 - \Vert\hat\bxi^{(n)}\Vert_2
+\lambda \big(\calR(\bfB) - \calR(\hat\bfB)\big) \\
&\le -\left\langle\frac{\hat\bxi^{(n)}_{\bfB}}{\Vert\hat\bxi^{(n)}_{\bfB}\Vert_2},\hat\bxi^{(n)}-\bxi^{(n)}_{\bfB}\right\rangle
+\lambda \big(\calR(\bfB) - \calR(\hat\bfB)\big), 
\end{align}
implying \eqref{proof:lemma:recursion:delta:bb:general:norm:q=1':eq1} since 
$\hat\bxi^{(n)}_{\bfB}=\hat\bxi^{(n)}+\frX^{(n)}(\bfDelta_{\bfB})$. Note also, that $\hat\bxi_{\bfB}^{(n)}=\bxi_{\bfB}^{(n)}-\bfDelta^{\hat\btheta}$. 

We now write
\begin{align}
\Vert\hat\bxi^{(n)}_{\bfB}\Vert_2\lambda\big(\calR(\bfB) - \calR(\hat\bfB)\big)&=\Vert\bxi^{(n)}\Vert_2\lambda\big(\calR(\bfB) - \calR(\hat\bfB)\big)\\
&+(\Vert\hat\bxi^{(n)}_{\bfB}\Vert_2-\Vert\bxi^{(n)}\Vert_2)\lambda\big(\calR(\bfB) - \calR(\hat\bfB)\big).
\label{proof:lemma:recursion:delta:bb:general:norm:q=1':eq2}
\end{align}
By convexity of $\bfW\mapsto\calR(\bfW)$, there exist $\bfV_{\bfB}\in\calR(\bfB)$ such that
\begin{align}
(\Vert\hat\bxi^{(n)}_{\bfB}\Vert_2-\Vert\bxi^{(n)}\Vert_2)\lambda\big(\calR(\bfB) - \calR(\hat\bfB)\big)
&\le (\Vert\hat\bxi^{(n)}_{\bfB}\Vert_2-\Vert\bxi^{(n)}\Vert_2)\lambda
\left(
-\llangle\bfV_{\bfB},\hat\bfB-\bfB\rrangle
\right)\\
&\le (\Vert\calE_{\bfB}^{(n)}\Vert_2 + \Vert\bfDelta^{\hat\btheta}\Vert_2)\lambda\calR(\bfDelta_{\bfB}), 
\label{proof:lemma:recursion:delta:bb:general:norm:q=1':eq3}
\end{align}
where we used that
$
\hat\bxi_{\bfB}^{(n)} = \bxi^{(n)} - \calE_{\bfB}^{(n)} - \bfDelta^{\hat\btheta}
$
and 
$|\llangle\bfV_{\bfB},\bfDelta_{\bfB}\rrangle|\le\calR(\bfDelta_{\bfB})$. Equation \eqref{lemma:recursion:delta:bb:general:norm:q=1':eq1} follow from the two previous displays and \eqref{proof:lemma:recursion:delta:bb:general:norm:q=1':eq1}.

We now prove \eqref{lemma:recursion:delta:bb:general:norm:q=1':eq2}. The KKT conditions imply that there exists $\bfV\in\mdR^p$ with 
$\calR^*(\bfV)\le1$ and  $\llangle\bfV,\hat\bfB\rrangle=\calR(\hat\bfB)$ such that, for all $\bfB\in\mdR^p$,
\begin{align}
0&\le\sum_{i\in[n]}\left[\frX^{(n)}_i(\widehat\bfB)+\widehat\btheta_i-y_i^{(n)}\right]\llangle\bfX^{(n)}_i,\bfB-\hat\bfB\rrangle+\lambda\Vert\hat\bxi^{(n)}\Vert_2\llangle\bfV,\bfB-\hat\bfB\rrangle.
\end{align}
Using that $\by^{(n)}=\boldf^{(n)}+\btheta^*+\bxi^{(n)}$, we arrive at
\begin{align}
0&\le\left\langle\bfDelta^{(n)},-\frX^{(n)}(\bfDelta_{\bfB})\right\rangle + \langle\bxi^{(n)}-\bfDelta^{\hat\btheta},\frX^{(n)}(\bfDelta_{\bfB})\rangle
-\lambda\Vert\hat\bxi^{(n)}\Vert_2\llangle\bfDelta_{\bfB},\bfV\rrangle.
\end{align}

We now write 
\begin{align}
-\lambda\Vert\hat\bxi^{(n)}\Vert_2\llangle\bfV,\bfDelta_{\bfB}\rrangle
&=-\lambda\Vert\bxi^{(n)}\Vert_2\llangle\bfV,\bfDelta_{\bfB}\rrangle\\
&+\lambda(\Vert\bxi^{(n)}\Vert_2-\Vert\hat\bxi^{(n)}\Vert_2)\llangle\bfV,\bfDelta_{\bfB}\rrangle.  
\end{align}

Using that $-\llangle\bfDelta_{\bfB},\bfV\rrangle\le\calR(\bfB)-\calR(\hat\bfB)$, we get that 
\begin{align}
-\lambda\Vert\bxi^{(n)}\Vert_2\llangle\bfV,\bfDelta_{\bfB}\rrangle
\le (\hat\sigma\lambda)\big(\calR(\bfB) - \calR(\hat\bfB)\big). 
\end{align}

Finally, using the facts that $\hat\bxi^{(n)}=\bxi^{(n)}-\calE_{\bfB}^{(n)}-\frM^{(n)}(\bfDelta_{\bfB},\bfDelta^{\hat\btheta})$ and
$|\llangle\bfV,\bfDelta_{\bfB}\rrangle|\le\calR(\bfDelta_{\bfB})$
we get 
\begin{align}
\lambda(\Vert\bxi^{(n)}\Vert_2-\Vert\hat\bxi^{(n)}\Vert_2)\llangle\bfV,\bfDelta_{\bfB}\rrangle 
&\le \left(\Vert\frM^{(n)}(\bfDelta_{\bfB},\bfDelta^{\hat\btheta})\Vert_2+\Vert\calE_{\bfB}^{(n)}\Vert_2\right)
\lambda\calR(\bfDelta_{\bfB}). 
\end{align}
Equation \eqref{lemma:recursion:delta:bb:general:norm:q=1':eq2} follows from the four previous displays.
\end{proof}

Using the previous lemma and $\IP$ --- in addition to $\MP$ and $\ATP$ ---, we obtain the following  lemma. 
\begin{lemma}\label{lemma:MP+IP+ATP:q=1'}
Suppose conditions (i)-(iii) of Theorem \ref{thm:improved:rate} hold. Define the quantities
\begin{align}
\hat\blacktriangle_1 & := \left[
\sf_2 +\frac{\sf_1\sd_2}{\sd_1}
+\left(\sb_2 + \lambda + \frac{\sb_1\sd_2}{\sd_1}\right)\|\bfDelta^{\hat\btheta}\|_2 + \frac{\sb_4\sd_2}{\sd_1}\|\bfDelta^{\hat\btheta}\|_\sharp
+\lambda\Vert\calE_{\bfB}^{(n)}\Vert_2
\right]\calR(\bfDelta_{\bfB}),\\
\hat\blacktriangledown_1 & := (\hat\sigma\lambda) \big(\calR(\bfB) - \calR(\hat\bfB)\big).
\end{align}

Then 
\begin{align}
0\le \left[(\nicefrac{\sf_1}{\sd_1})
+(\nicefrac{\sb_1}{\sd_1})\|\bfDelta^{\hat\btheta}\|_2
+(\nicefrac{\sb_4}{\sd_1})\|\bfDelta^{\hat\btheta}\|_\sharp+\Vert\calE_{\bfB}^{(n)}\Vert_2\right]\Vert\frX^{(n)}(\bfDelta_{\bfB})\Vert_2
+\hat\blacktriangle_1 + \hat\blacktriangledown_1,\label{lemma:MP+IP+ATP:q=1':eq1}
\end{align}
and also
\begin{align}
\Vert\bfDelta^{(n)}\Vert_2^2 + \Vert\frX^{(n)}(\bfDelta_{\bfB})\Vert_2^2 &\le \Vert\frX^{(n)}(\bfB)-\boldf^{(n)}\Vert_2^2\\
& + \left[
(\nicefrac{2\sf_1}{\sd_1})
+(\nicefrac{2\sb_1}{\sd_1})\|\bfDelta^{\hat\btheta}\|_2
+(\nicefrac{2\sb_4}{\sd_1})\|\bfDelta^{\hat\btheta}\|_\sharp
\right]\Vert\frX^{(n)}(\bfDelta_{\bfB})\Vert_2\\
&+ 2(\hat\blacktriangle_1 + \hat\blacktriangledown_1) 
+2\Vert\frX^{(n)}(\bfDelta_{\bfB})\Vert_2\lambda\calR(\bfDelta_{\bfB}). 
\label{lemma:MP+IP+ATP:q=1':eq2}
\end{align}
\end{lemma}
\begin{proof}
By the parallelogram law,
\begin{align}
&\langle \bfDelta^{(n)}, \frX^{(n)}(\bfDelta_{\bfB})\rangle =
\frac{1}{2}\Vert\bfDelta^{(n)}\Vert_2^2
+ \frac{1}{2}\Vert\frX^{(n)}(\bfDelta_{\bfB})\Vert_2^2
- \frac{1}{2}\Vert\frX^{(n)}(\bfB)-\boldf^{(n)}\Vert_2^2.
\end{align}

The previous display, \eqref{lemma:recursion:delta:bb:general:norm:q=1':eq2} and 
$\Vert\frM^{(n)}(\bfDelta_{\bfB},\bfDelta^{\hat\btheta})\Vert_2\le 
\Vert\frX^{(n)}(\bfDelta_{\bfB})\Vert_2+\Vert\bfDelta^{\hat\btheta}\Vert_2$ imply
\begin{align}
\Vert\bfDelta^{(n)}\Vert_2^2 + \Vert\frX^{(n)}(\bfDelta_{\bfB})\Vert_2^2 &\le \Vert\frX^{(n)}(\bfB)-\boldf^{(n)}\Vert_2^2\\
&+ 2\langle\bxi^{(n)}-\bfDelta^{\hat\btheta},\frX^{(n)}(\bfDelta_{\bfB})\rangle\\
&+2(\hat\sigma\lambda)(\calR(\bfB)-\calR(\hat\bfB))\\
&+2\Vert\frX^{(n)}(\bfDelta_{\bfB})\Vert_2\lambda\calR(\bfDelta_{\bfB}) + 2(\Vert\calE_{\bfB}^{(n)}\Vert_2+\Vert\bfDelta^{\hat\btheta}\Vert_2)\lambda\calR(\bfDelta_{\bfB}).
\label{proof:lemma:MP+IP+PP+ATP:q=1':eq1}
\end{align}

By $\IP$ (with variable $\bfW=\bold{0}$),
\begin{align}
\langle-\bfDelta^{\hat\btheta},\frX^{(n)}(\bfDelta_{\bfB})\rangle 
&\le \sb_1\Vert\bfDelta_{\bfB}\Vert_{\Pi}\Vert\bfDelta^{\hat\btheta}\Vert_2 
+\sb_2\calR(\bfDelta_{\bfB})\Vert\bfDelta^{\hat\btheta}\Vert_2
+\sb_4\Vert\bfDelta_{\bfB}\Vert_{\Pi}\Vert\bfDelta^{\hat\btheta}\Vert_\sharp.  
\end{align}

By $\MP$ (with variables $\bfW=\bold{0}$ and $\bu=\bold{0}$),
\begin{align}
\langle\bxi^{(n)},\frX^{(n)}(\bfDelta_{\bfB})\rangle &\le
\sf_1\Vert\bfDelta_{\bfB}\Vert_{\Pi} 
+\sf_2\calR(\bfDelta_{\bfB}).
\end{align}

By $\ATP$ (with variables $\bfW=\bold{0}$ and $\bu=\bold{0}$), 
\begin{align}
\Vert\frX^{(n)}(\bfDelta_{\bfB})\Vert_2 &\ge \sd_1\Vert\bfDelta_{\bfB}\Vert_{\Pi}-\sd_2\calR(\bfDelta_{\bfB}).
\end{align}

The three previous displays imply
\begin{align}
\langle\bxi^{(n)}-\bfDelta^{\hat\btheta},\frX^{(n)}(\bfDelta_{\bfB})\rangle
&\le\Vert\bfDelta_{\bfB}\Vert_{\Pi}\left(
\sb_1\Vert\bfDelta^{\hat\btheta}\Vert_2 + \sb_4\Vert\bfDelta^{\hat\btheta}\Vert_\sharp + \sf_1
\right) + \calR(\bfDelta_{\bfB})(\sb_2\|\bfDelta^{\hat\btheta}\|_2 + \sf_2)\\
&\le \Vert\frX^{(n)}(\bfDelta_{\bfB})\Vert_2
\left[
(\nicefrac{\sb_1}{\sd_1})\|\bfDelta^{\hat\btheta}\|_2
+(\nicefrac{\sb_4}{\sd_1})\|\bfDelta^{\hat\btheta}\|_\sharp
+(\nicefrac{\sf_1}{\sd_1})
\right]\\
&+ \calR(\bfDelta_{\bfB})\left[
(\nicefrac{\sd_2}{\sd_1})\left(
\sb_1\Vert\bfDelta^{\hat\btheta}\Vert_2 + \sb_4\Vert\bfDelta^{\hat\btheta}\Vert_\sharp + \sf_1
\right)
+\sb_2\|\bfDelta^{\hat\btheta}\|_2 + \sf_2
\right]. 
\end{align}
The proof of \eqref{lemma:MP+IP+ATP:q=1':eq2} follows from the previous display and \eqref{proof:lemma:MP+IP+PP+ATP:q=1':eq1}. 

The proof of \eqref{lemma:MP+IP+ATP:q=1':eq1} follows from the previous display, inequality \eqref{lemma:recursion:delta:bb:general:norm:q=1':eq1} and the fact that 
\begin{align}
\langle\bxi^{(n)}_{\bfB}-\bfDelta^{\hat\btheta},\frX^{(n)}(\bfDelta_{\bfB})\rangle & =
\langle\bxi^{(n)}-\bfDelta^{\hat\btheta},\frX^{(n)}(\bfDelta_{\bfB})\rangle
-\langle\calE^{(n)}_{\bfB},\frX^{(n)}(\bfDelta_{\bfB})\rangle\\
& \le 
\langle\bxi^{(n)}-\bfDelta^{\hat\btheta},\frX^{(n)}(\bfDelta_{\bfB})\rangle
+\Vert\calE^{(n)}_{\bfB}\Vert_2\Vert\frX^{(n)}(\bfDelta_{\bfB})\Vert_2. 
\end{align}
\end{proof}

We conclude with the proof of Theorem \ref{thm:improved:rate:q=1'}. It uses Proposition \ref{prop:suboptimal:rate:q=1'} and Lemma \ref{lemma:MP+IP+ATP:q=1'}. 
\begin{proof}[Proof of Theorem \ref{thm:improved:rate:q=1'}]
Let
$\hat\blacksquare_1:=(\nicefrac{\hat\sigma\lambda}{4})\calR(\bfDelta_{\bfB})$.
By Lemma \ref{lemma:MP+IP+ATP:q=1'} and $\Vert\calE_{\bfB}^{(n)}\Vert_2\le\hat\sigma\sc_n$, we have 
\begin{align}
0\le \left[(\nicefrac{\sf_1}{\sd_1})
+(\nicefrac{\sb_1}{\sd_1})\|\bfDelta^{\hat\btheta}\|_2
+(\nicefrac{\sb_4}{\sd_1})\|\bfDelta^{\hat\btheta}\|_\sharp + \hat\sigma\sc_n\right]\Vert\frX^{(n)}(\bfDelta_{\bfB})\Vert_2
+\hat\blacktriangle_1 + \hat\blacktriangledown_1,\label{proof:thm:improved:rate:q=1':ineq1} 
\end{align}
and also
\begin{align}
2\hat\blacksquare_1 + \Vert\bfDelta^{(n)}\Vert_2^2 + \Vert\frX^{(n)}(\bfDelta_{\bfB})\Vert_2^2 &\le \Vert\frX^{(n)}(\bfB)-\boldf^{(n)}\Vert_2^2\\
& + \left[
(\nicefrac{2\sf_1}{\sd_1})
+(\nicefrac{2\sb_1}{\sd_1})\|\bfDelta^{\hat\btheta}\|_2
+(\nicefrac{2\sb_4}{\sd_1})\|\bfDelta^{\hat\btheta}\|_\sharp
\right]\Vert\frX^{(n)}(\bfDelta_{\bfB})\Vert_2\\
&+ 2(\hat\blacktriangle_1 + \hat\blacksquare_1 + \hat\blacktriangledown_1) 
+2\Vert\frX^{(n)}(\bfDelta_{\bfB})\Vert_2\lambda\calR(\bfDelta_{\bfB}).\label{proof:thm:improved:rate:q=1':ineq2}
\end{align}

For convenience, let $R:=(\hat\sigma\sc_n)\vee(3r)$. All conditions of Proposition \ref{prop:suboptimal:rate:q=1'} hold. Hence, 
\begin{align}
\|\bfDelta^{\hat\btheta}\|_2&\le
(\nicefrac{4D}{\sd_1}) + (\nicefrac{2}{\sd_1})\clubsuit_1
\left(\sf_1,R\right)
 \le\sc_*\hat\sigma,\label{proof:thm:improved:rate:q=1':eq:club}\\ 
(\nicefrac{\hat\sigma\tau}{2})\|\bfDelta^{\hat\btheta}\|_\sharp &\le D^2 + \spadesuit_1\left(\sf_1,R\right)
\le\sc_*^2\hat\sigma^2, \label{proof:thm:improved:rate:q=1':eq:spade}
\end{align}
where we have used conditions \eqref{cond8:general:norm:q=1'}-\eqref{cond9:general:norm:q=1'}. In particular, by the condition on $\lambda$ in (iv), 
\begin{align}
\sf_2 + \frac{\sf_1\sd_2}{\sd_1} 
+ \left(\sb_2 + \lambda + \frac{\sb_1\sd_2}{\sd_1}\right)\|\bfDelta^{\hat\btheta}\|_2 + \frac{\sb_4\sd_2}{\sd_1}\|\bfDelta^{\hat\btheta}\|_\sharp
+ \hat\sigma\lambda\sc_n &\le\frac{\hat\sigma\lambda}{4}.\label{proof:thm:improved:rate:q=1':eq0''}
\end{align}

By Lemmas \ref{lemma:A1:B} (with $\nu=1/2$), \eqref{proof:thm:improved:rate:q=1':eq0''} and $\Vert\calE_{\bfB}\Vert_2\le\hat\sigma\sc_n$,
\begin{align}
\hat\blacktriangle_1 + \hat\blacksquare_1 + \hat\blacktriangledown_1 & \le
\left[\sf_2 + \frac{\sf_1\sd_2}{\sd_1} 
+ \left(\sb_2 + \lambda + \frac{\sb_1\sd_2}{\sd_1}\right)\|\bfDelta^{\hat\btheta}\|_2 + \frac{\sb_4\sd_2}{\sd_1}\|\bfDelta^{\hat\btheta}\|_\sharp
+ \hat\sigma\lambda\sc_n + (\nicefrac{\hat\sigma\lambda}{4})\right]\calR(\bfDelta_{\bfB})\\
& + (\hat\sigma\lambda) \big(\calR(\bfB) - \calR(\hat\bfB)\big)\\
&\le(\nicefrac{\hat\sigma\lambda}{2})\calR(\bfDelta_{\bfB}) + (\hat\sigma\lambda) \big(\calR(\bfB) - \calR(\hat\bfB)\big)\\
&\le \triangle_{\hat\sigma\lambda,0}(\bfDelta_{\bfB},\bold{0}|\bfB)=:\triangle.\label{proof:thm:improved:rate:q=1':eq:triangle}
\end{align}

Next, we define the local variables
\begin{align}
H&:=(\nicefrac{3\hat\sigma\lambda}{2})(\calR\circ\calP_{\bfB})(\bfDelta_{\bfB}),\\
I&:=(\nicefrac{\hat\sigma\lambda}{2})(\calR\circ\calP_{\bfB}^\perp)(\bfDelta_{\bfB}).
\end{align}
In particular, 
$
\triangle = H - I. 
$
Recall 
$
D:=\Vert\frX^{(n)}(\bfB)-\boldf^{(n)}\Vert_2. 
$
Define also $G:=\Vert\frX^{(n)}(\bfDelta_{\bfB})\Vert_2$, 
$
x:=\Vert\bfDelta^{(n)}\Vert_2,
$
and 
$\hat r:=r_{\hat\sigma\lambda,0,6}(\bfDelta_{\bfB}|\bfB)$. Finally, let us define the auxiliary variables
\begin{align}
F&:= \sf_1 + \sb_1\left[
(\nicefrac{4D}{\sd_1}) + (\nicefrac{2}{\sd_1})\clubsuit_1
\left(\sf_1,R\right)\right]
 + 2(\nicefrac{\sb_4}{\hat\sigma\tau})\left[
D^2 + \spadesuit_1\left(\sf_1,R\right)
\right],\\
\hat\sf_1&:=\sf_1 + \sb_1(\sc_*\sigma) + 
2(\nicefrac{\sb_4}{\hat\sigma\tau})(\sc_*^2\sigma^2).
\end{align}
Note that, from \eqref{proof:thm:improved:rate:q=1':eq:club}-\eqref{proof:thm:improved:rate:q=1':eq:spade} and $F\le\hat\sf_1$.

From \eqref{proof:thm:improved:rate:q=1':ineq1}-\eqref{proof:thm:improved:rate:q=1':ineq2}, \eqref{proof:thm:improved:rate:q=1':eq:club}-\eqref{proof:thm:improved:rate:q=1':eq:spade} and \eqref{proof:thm:improved:rate:q=1':eq:triangle}, 
\begin{align}
0&\le \left(\nicefrac{F}{\sd_1}+\hat\sigma\sc_n\right)G
+\triangle,\label{proof:thm:improved:rate:q=1':eq0'}\\
2\hat\blacksquare_1 + x^2 + G^2
&\le D^2 + 2(\nicefrac{F}{\sd_1})G +2\lambda\calR(\bfDelta_{\bfB})G
+ 2\triangle.\label{proof:thm:improved:rate:q=1':eq0}
\end{align}
The rest of the proof is similar to the proof of Proposition \ref{prop:suboptimal:rate:q=1'}. 

We split our argument in two cases. 
\begin{description}
\item[Case 1:] $\left(\nicefrac{F}{\sd_1}+\hat\sigma\sc_n\right)G\ge H$. Hence, $\triangle\le H\le \left(\nicefrac{F}{\sd_1}+\hat\sigma\sc_n\right)G$. From \eqref{proof:thm:improved:rate:q=1':eq0'}, 
$
I\le \left(\nicefrac{F}{\sd_1}+\hat\sigma\sc_n\right)G + H
\le 2\left(\nicefrac{F}{\sd_1}+\hat\sigma\sc_n\right)G.
$
This fact and decomposability imply 
\begin{align}
\hat\sigma\lambda\calR(\bfDelta_{\bfB})
\le 2I+\frac{2H}{3} 
\le \frac{14}{3}\left(\nicefrac{F}{\sd_1}+\hat\sigma\sc_n\right)G.
\end{align}
From \eqref{proof:thm:improved:rate:q=1':eq0}, we get 
\begin{align}
2\hat\blacksquare_1 + x^2 + G^2
&\le D^2 + 2(\nicefrac{F}{\sd_1})G +
\frac{28}{3\hat\sigma}\left(\nicefrac{F}{\sd_1}+\hat\sigma\sc_n\right)G^2 + 2\triangle\\
&\le D^2 + 2\left(\nicefrac{2F}{\sd_1}+\hat\sigma\sc_n\right)G +
\frac{28}{3\hat\sigma}\left(\nicefrac{F}{\sd_1}+\hat\sigma\sc_n\right)G^2.
\end{align}
Using that $F\le\hat\sf_1$ and $56[(\nicefrac{\hat\sf_1}{\sd_1})+\hat\sigma\sc_n]\le3\hat\sigma$ --- as stated in (v) ---, we get
\begin{align}
2\hat\blacksquare_1 + x^2 + \frac{G^2}{2}
&\le D^2 + 2\left(2F+\hat\sigma\sc_n\right)G. 
\end{align}

From $2\left(2F+\hat\sigma\sc_n\right)G\le2\left(\nicefrac{2F}{\sd_1}+\hat\sigma\sc_n\right)^2 + \frac{G^2}{2}$, we get 
\begin{align}
2\hat\blacksquare_1 + x^2 &\le D^2 + 2\left(\nicefrac{2F}{\sd_1}+\hat\sigma\sc_n\right)^2.\label{proof:thm:improved:rate:q=1':eq4} 
\end{align}
If we use instead $2\left(\nicefrac{2F}{\sd_1}+\hat\sigma\sc_n\right)G\le4\left(\nicefrac{2F}{\sd_1}+\hat\sigma\sc_n\right)^2 + \frac{G^2}{4}$, we get 
\begin{align}
\frac{G^2}{4} &\le D^2 + 4\left(\nicefrac{2F}{\sd_1}+\hat\sigma\sc_n\right)^2.\label{proof:thm:improved:rate:q=1':eq5}
\end{align}

\item[Case 2:] $\left(\nicefrac{F}{\sd_1}+\hat\sigma\sc_n\right)G\le H$. From \eqref{proof:thm:improved:rate:q=1':eq0'}, $0\le H+\triangle=2H-I$ so that 
$\bfDelta_{\bfB}\in\calC_{\bfB}(6)$. By Lemma \ref{lemma:aug:rest:convexity}, 
\begin{align}
\triangle&\le(\nicefrac{3}{2})\hat r\Vert\bfDelta_{\bfB}\Vert_{\Pi}, \\
(\hat\sigma\lambda)\calR(\bfDelta_{\bfB})&\le 14\hat r \Vert\bfDelta_{\bfB}\Vert_\Pi.
\end{align}
The second inequality above and $\ATP$ imply that
\begin{align}
\sd_1\Vert\bfDelta_{\bfB}\Vert_{\Pi} &\le G
+(\nicefrac{\sd_2}{\lambda})\lambda\calR(\bfDelta_{\bfB})\label{proof:thm:improved:rate:q=1':eq6}\\
&\le G + (\nicefrac{\sd_2}{\lambda})
14(\nicefrac{r}{\hat\sigma})\Vert\bfDelta_{\bfB}\Vert_{\Pi}.
\end{align}
By condition \eqref{cond1:general:norm:q=1'}, $14[(\nicefrac{\sd_2}{\lambda})\vee(\nicefrac{\sd_4}{\tau})](\nicefrac{\hat r}{\hat\sigma})\le\sd_1/2$, implying $(\sd_1/2)\Vert\bfDelta_{\bfB}\Vert_{\Pi}\le G$. 
We conclude that 
\begin{align}
\triangle&\le(\nicefrac{3\hat r}{\sd_1})G,\\
\hat\sigma\lambda\calR(\bfDelta_{\bfB})&\le (\nicefrac{28\hat r}{\sd_1})G.
\end{align}

From \eqref{proof:thm:improved:rate:q=1':eq0}, 
\begin{align}
2\hat\blacksquare_1 + x^2 + G^2 &\le D^2 + 2(\nicefrac{F}{\sd_1})G + (\nicefrac{56\hat r}{\hat\sigma\sd_1})G^2
+ (\nicefrac{6\hat r}{\sd_1})G\\
&= D^2 + \left(\frac{2F + 6\hat r}{\sd_1}\right)G + (\nicefrac{56\hat r}{\hat\sigma\sd_1})G^2, 
\end{align}
which, together with $56\hat r\le\hat\sigma\sd_1/2$ --- as implied by condition \eqref{cond1:general:norm:q=1'} ---, entails
\begin{align}
2\hat\blacksquare_1 + x^2 + \frac{G^2}{2} &\le  D^2 + 2\left(\frac{F + 3\hat r}{\sd_1}\right)G. 
\end{align}

Proceeding similarly as before, we obtain from the displayed bound that 
\begin{align}
2\hat \blacksquare_1 + x^2 &\le D^2 + 2\left(\frac{F+3\hat r}{\sd_1}\right)^2,\label{proof:thm:improved:rate:q=1':eq8}\\
\frac{G^2}{4} &\le D^2 + 4\left(\frac{F+3\hat r}{\sd_1}\right)^2.\label{proof:thm:improved:rate:q=1':eq9}
\end{align}
\end{description}
The proof of \eqref{thm:improved:rate:q=1':eq1} follows by taking the largest of the bounds in  \eqref{proof:thm:improved:rate:q=1':eq4} and \eqref{proof:thm:improved:rate:q=1':eq8}. The proof of \eqref{thm:improved:rate:q=1':eq2} follows by taking the largest of the bounds in \eqref{proof:thm:improved:rate:q=1':eq5} and \eqref{proof:thm:improved:rate:q=1':eq9}. The proof of \eqref{thm:improved:rate:q=1':eq3} follows from \eqref{proof:thm:improved:rate:q=1':eq6} --- namely, $\ATP$ --- and \eqref{thm:improved:rate:q=1':eq1}-\eqref{thm:improved:rate:q=1':eq2}. 
\end{proof}

\section{Proof sketch of Theorem \ref{thm:response:sparse-low-rank:regression:q=1}, case (i')}\label{s:proof:thm:response:sparse-low-rank:regression:q=1:i'}
In the following $\calR:=\Vert\cdot\Vert_1$. Next, we assume that $n\ge C_0L^4(1+\log(1/\delta))$ and $n\ge C_0\sigma^2(1+\log(1/\delta))$ for an absolute constant to be determined next. We will also use that $L\ge1$. In the following, $C>0$ is the universal constant stated in Proposition  \ref{proposition:properties:subgaussian:designs}. No effort is made to optimize the numerical constants. 

Invoking Proposition \ref{proposition:properties:subgaussian:designs} and taking $C_0\gtrsim C^2$, we know that on the event 
$\calE_1\cap\calE_2\cap\calE_3$ --- stated in Section \ref{s:proof:thm:response:sparse-low-rank:regression:i} --- of probability 
$\ge1-3\delta$, the properties 
$\TP_{\Vert\cdot\Vert_1}(\sa_1,\sa_2)$, $\IP_{\Vert\cdot\Vert_1,0,\Vert\cdot\Vert_\sharp}(\sb_1(\delta),\sb_2,0,\sb_4)$, 
$\MP_{\Vert\cdot\Vert_1,0,\Vert\cdot\Vert_\sharp}(\sf_1(\delta),\sf_2,0,\sf_4)$ 
and $\ATP_{\Vert\cdot\Vert_1,0,\Vert\cdot\Vert_\sharp}(\sd_1,\sd_2,0,\sd_4)$ all hold --- see Section \ref{s:proof:thm:response:sparse-low-rank:regression:i} for the expression of the constants $\{\sa_i\}$, 
$\{\sb_i\}$, $\{\sd_i\}$ and $\{\sf_i\}$. Enlarging $C_0$ if necessary,  Bernstein's inequality implies that, on an event $\calE_4$ of probability at least $1-\delta$, $\sigma/2\le\hat\sigma\le3\sigma/2$. The proof will work on the event $\calE_1\cap\calE_2\cap\calE_3\cap\calE_4$ of probability 
$\ge1-4\delta$. 

Next, we will use the notation $\mu(\bb):=\mu\left(\calC_{\bb,\Vert\cdot\Vert_1}(6)\right)$ and 
\begin{align}
\calF&:=\calF(s,s\log p,\rho_1(\bfSigma),\sfC),\\
\calF_0&:=\calF(s,s\log p,\rho_1(\bfSigma),\sc_{0,n}^2,\sfC),
\end{align}
for $\sc_{0,n}$ and $\sfC$ to be determined. By definition
$\mu_*:=\sup_{\bb\in\calF}\mu(\bb)<\infty$. We will take 
\begin{align}
\sc_{0,n}\asymp r_{n,s\log p,\delta}(\rho_1(\bfSigma),\mu_*) = L\frac{1+\sqrt{\log(1/\delta)}}{\sqrt{n}} + L^2\rho_1(\bfSigma)\mu_*\sqrt{\frac{s\log p}{n}}.
\end{align}
Also, we take $\sc_n:=2\sc_{0,n}$ and $\sc_*\asymp\sc_n$. 

Next, we invoke Theorem \ref{thm:improved:rate:q=1'}. Conditions (i)-(iii) are met. Now, we note that, by definition of $\calF$ and assumptions of the theorem,   
\begin{align}
\calO(1)CL^2\rho_1(\bfSigma)
\mu_*\sqrt{\frac{s\log p}{n}}&<1,\label{thm:response:sparse-low-rank:reg:q=1':eq1}\\
CL\sqrt{\epsilon\log(e/\epsilon)}<c_1,\label{thm:response:sparse-low-rank:reg:q=1':eq2}
\end{align}
choosing appropriate absolute constants $c_1\in(0,1)$ and $\sfC\ge1$. In that case, by changing constants if necessary, we can assume $\sc_n,\sc_*\in(0,1)$ are small enough. Thus, if we choose 
$
\lambda \asymp C\sigma L^2\rho_1(\bfSigma)\sqrt{\nicefrac{\log p}{n}}
$
and 
$
\tau\asymp C\sigma L/\sqrt{n}, 
$
it is easy to check that conditions (iv)-(v) of Theorem \ref{thm:improved:rate:q=1'} are met.

In remains to verify conditions \eqref{cond4:general:norm:q=1'}-\eqref{cond9:general:norm:q=1'} for $\bb\in\calF_0$. Condition \eqref{cond4:general:norm:q=1'} is met: $D=\Vert\frX^{(n)}(\bb)-\boldf^{(n)}\Vert_2\le\sc_{0,n}\sigma\le2\sc_{0,n}\hat\sigma$. Let $r:=r_{\hat\sigma\lambda,\hat\sigma\tau\Omega,6}(\bfDelta_{\bb}|\bb)$. By the choice  of $(\lambda,\tau)$, one checks that \eqref{cond1:general:norm:q=1'} is satisfied if \eqref{thm:response:sparse-low-rank:reg:q=1':eq1}-\eqref{thm:response:sparse-low-rank:reg:q=1':eq2} hold.

Let $R:=(\hat\sigma\sc_n)\vee(3r)$ and $r_*:=\sup_{\bb\in\calF}r$. We have $R\lesssim r_*$. Next, we show that $D^2+\spadesuit_1(\sf_1,R)\le\sc_*^2\sigma^2$. For this to be true it is sufficient that $D\le\frac{\sc_*\sigma}{3}$ and 
$\spadesuit_1^{1/2}(\sf_1,R)\le\frac{3\sc_*\sigma}{4}$. By adjusting constants, we can have $D\le\hat\sigma\sc_n\le\frac{\sc_*\sigma}{3}$. Note that
$$
\spadesuit_1^{1/2}(\sf_1,R) \lesssim\sf_1+R
\lesssim C\sigma L\frac{1+\sqrt{\log(1/\delta)}}{\sqrt{n}}
+r_*,
$$
which, by definition of $\sc_*$ and \eqref{thm:response:sparse-low-rank:reg:q=1':eq1}-\eqref{thm:response:sparse-low-rank:reg:q=1':eq2}, can be shown to be  not greater than $\frac{3\sc_*\sigma}{4}$ --- adjusting the numerical constants if necessary. Hence, condition \eqref{cond8:general:norm:q=1'} holds. The verification of \eqref{cond9:general:norm:q=1'} is similar. 

We now verify the statement of Theorem \ref{thm:response:sparse-low-rank:regression}, case (i) --- using the bounds \eqref{thm:improved:rate:q=1':eq1}-\eqref{thm:improved:rate:q=1':eq2} of Theorem \ref{thm:improved:rate:q=1'}. Let $\hat r:=r_{\hat\sigma\lambda,0,6}(\bfDelta_{\bfB}|\bfB)$,  $\hat R:=(\hat\sigma\sc_n)\vee(3\hat r)$ and $\hat r_*:=\sup_{\bb\in\calF}\hat r$. We have $\hat R\lesssim \hat r_*$. We claim that, using $D\lesssim\sigma c_{0,n}$ and \eqref{thm:response:sparse-low-rank:reg:q=1':eq1}-\eqref{thm:response:sparse-low-rank:reg:q=1':eq2}, similar computations used in the proof of Theorem \ref{thm:response:sparse-low-rank:regression}(i) to bound the terms $D^2 + \spadesuit_1(F,\hat R)$ and $2D + \clubsuit_1(F,\hat R)$ entail the bounds in 
\eqref{thm:response:sparse-low-rank:regression:q=1':eq1}-\eqref{thm:response:sparse-low-rank:regression:q=1':eq2}. 

\section{Proof sketch of Theorem \ref{thm:response:sparse-low-rank:regression:q=1}, case (ii')}\label{s:proof:thm:response:sparse-low-rank:regression:q=1:ii'}

The proof of Theorem \ref{thm:response:sparse-low-rank:regression:q=1}(ii')
needs Definitions \ref{def:cones:slope:norm}-\ref{def:r:slope:norm} used in the proof of Theorem \ref{thm:response:sparse-low-rank:regression}, case (ii) --- see Section \ref{s:proof:thm:response:sparse-low-rank:regression:ii}. 

The proof with $\calR=\Vert\cdot\Vert_\sharp$, the Slope norm in $\re^p$, follows a similar path to the proof of Theorem \ref{thm:response:sparse-low-rank:regression:q=1}, case (i'). We claim that a very similar theorem to Theorem \ref{thm:improved:rate:q=1'} holds but with the minor changes:
\begin{quote}
We replace $r:=r_{\hat\sigma\lambda,\hat\sigma\tau\Omega,6}(\bfDelta_{\bfB}|\bfB)$ with $r:=r_{\hat\sigma\lambda,\hat\sigma\tau\Omega,6}(s)$ and $\hat r:=r_{\hat\sigma\lambda,0,6}(\bfDelta_{\bfB}|\bfB)$ with $\hat r:=r_{\hat\sigma\lambda,0,6}(s)$.
\end{quote}
Let us call it Theorem \ref{thm:improved:rate:q=1'}'. Using this theorem, setting $\mu(\bb):=\mu(\overline\calC_s(6))$ for given $\bb$ with 
$\Vert\bb\Vert_0\le s$, using the bound $\mathscr{G}(\bfSigma^{1/2}\mbB_\sharp)\lesssim\rho_1(\bfSigma)$ --- which follows from Proposition E.2 in  \cite{2018bellec:lecue:tsybakov} --- and the fact that 
$\bar\Omega_s\le2s\log(ep/s)$, the proof of Theorem \ref{thm:response:sparse-low-rank:regression:q=1}, case (ii') follows very similar arguments to the case (i'). 

Next, we highlight the minor changes in the proof of Theorem \ref{thm:improved:rate:q=1'}'. Lemmas \ref{lemma:recursion:1st:order:condition:q=1'}-\ref{lemma:MP+ATP:q=1'} are unchanged. Instead of Lemma \ref{lemma:aug:rest:convexity} we used Lemma \ref{lemma:aug:rest:convexity:slope} --- see Section \ref{s:proof:thm:response:sparse-low-rank:regression:ii}. Using these lemmas, we obtain a variation of Proposition \ref{prop:suboptimal:rate:q=1'} --- again with a similar proof, using 
$\triangle:=\triangle_{\hat\sigma\lambda,\hat\sigma\tau}(\bfDelta_{\bb},\bfDelta^{\hat\btheta})$ instead of $\triangle:=\triangle_{\hat\sigma\lambda,\hat\sigma\tau}(\bfDelta_{\bfB},\bfDelta^{\hat\btheta}|\bfB)$ and $r:=r_{\hat\sigma\lambda,\hat\sigma\tau\Omega,6}(s)$ instead of $r:=r_{\hat\sigma\lambda,\hat\sigma\tau\Omega,6}(\bfDelta_{\bfB}|\bfB)$.
Lemmas \ref{lemma:recursion:delta:bb:general:norm:q=1'}-\ref{lemma:MP+IP+ATP:q=1'} are unchanged. Using all these auxiliary results, we claim that the proof of Theorem \ref{thm:improved:rate:q=1'}' follows the same arguments in the proof of  Theorem \ref{thm:improved:rate:q=1'} --- using $\triangle:=\triangle_{\hat\sigma\lambda,0}(\bfDelta_{\bb},\bold{0})$ instead of 
$\triangle:=\triangle_{\hat\sigma\lambda,0}(\bfDelta_{\bfB},\bold{0}|\bfB)$, $r:=r_{\hat\sigma\lambda,0,6}(s)$ instead of $r:=r_{\hat\sigma\lambda,0,3}(\bfDelta_{\bfB}|\bfB)$ and $\overline\calC_s(6)$ instead of $\calC_{\bfB}(6)$.

\section{Proof sketch of Theorem \ref{thm:response:sparse-low-rank:regression:q=1}, case (iii')}\label{s:proof:thm:response:sparse-low-rank:regression:q=1:iii'}

The exact same comments in Section \ref{s:thm:response:sparse-low-rank:regression:iii} apply --- but invoking Theorem \ref{thm:improved:rate:q=1'} instead of Theorem \ref{thm:improved:rate}. 

\begin{appendix}

\section{Peeling lemmas}\label{Peeling lemmas}

This section presents several peeling lemmas. This is a well known technique in Empirical Process theory in order to lift confidence statements from a compact subset to the entire set. Throughout this section, $g,\bar g$  are  right-continuous, non-decreasing functions from $\mathbb{R}_+$ to $\mathbb{R}_+$, $V$ is an arbitrary set and $h,\bar h$ are functions from $V$ to $\mathbb{R}_+$. We let $g^{-1}$ be the generalized inverse of $g$ defined by $g^{-1}(x) = \inf\{a\in\mathbb{R}_+: g(a)\ge x\}$; we use the same notation for the generalized inverse of $\bar g$. $\mathbb{N}^*$ is the set of natural numbers excluding zero. In what follows, $C>0$ is an universal constant that may change within the text. 

\subsection{Peeling for $\PP$}

\begin{lemma}\label{lemma:peeling:product:process}
Let $b>0$ be a constant and $c\ge1$ be an universal constant. Assume that, for every $r,\bar r>0$ and every $\delta\in(0,1/c]$, the event $A(r,\bar r,\delta)$ defined by the inequality
\begin{align}
\inf_{\bv\in V: (h,\bar h)(\bv)\le (r,\bar r)}
M(\bv) &\ge -\frac{b}{n} g(r)\bar g(\bar r)
-\frac{b}{\sqrt{n}}(g(r) + \bar g(\bar r))
- b\left(\sqrt{\frac{\log(1/\delta)}{n}}
+\frac{\log(1/\delta)}{n}
\right),
\end{align}
has probability at least $1-c\delta$. 

Then, for every $\delta\in(0,1/c]$, with probability at least $1-c\delta$, it holds that, for all 
$\bv\in V$,
\begin{align}
M(\bv) &\ge -C\frac{b}{n}g\circ h(\bv)\cdot
\bar g\circ\bar h(\bv)
-Cb\left(\frac{1}{n}+\frac{1}{\sqrt{n}}\right)g\circ h(\bv)
-Cb\left(\frac{1}{n}+\frac{1}{\sqrt{n}}\right)\bar g\circ\bar h(\bv)\\
& -Cb\left(\frac{1}{\sqrt{n}}+\frac{1}{n}\right)
-Cb\left(\sqrt{\frac{\log(1/\delta)}{n}}
+\frac{\log(1/\delta)}{n}
\right).
\end{align}
\end{lemma}
\begin{proof}
Let $\mu>0$ and $\eta,\epsilon>1$ be parameters to be chosen later on.  We set $\mu_0=0$ and, for $k\ge1$, $\mu_k := \mu \eta^{k-1}$. We may partition the set $V$ with the sets 
$$
V_{k,\bar k} :=\{\bv\in V : \mu_k\le (g\circ h)(\bv)< \mu_{k+1},
\mu_{\bar k}\le (\bar g\circ \bar h)(\bv)< \mu_{\bar k+1}
\},
$$ 
defined for $k,\bar k\in\mathbb{N}$.
Given $m\in\mathbb{N}^*$, we set
$\nu_m  :=g^{-1}(\mu_m)$ and $\bar\nu_{m}:=\bar g^{-1}(\mu_{m})$. Clearly, $ g(\nu_m)=\mu_m$ and 
$\bar g(\bar\nu_{m})=\mu_{m}$.

Fix $\delta\in(0,1/c]$. An union bound and the fact that $\sum_{k\ge 1,\bar k\ge 1} (k\bar k)^{-1-\epsilon}\le
(1+\epsilon^{-1})^2$ imply that the event
$$
A := \bigcap_{k\ge1,\bar k\ge1}^\infty A\left(\nu_k, \bar\nu_{\bar k},\frac{\epsilon^2\delta}{(1+\epsilon)^2(k\bar k)^{1+\epsilon}}\right),	
$$
has a probability at least $1-c\delta$. 

To ease notation, define
$\triangle(\delta):=\log\{(1+\epsilon)^2/(\epsilon^2 \delta)\}$ and $\triangle_{k,\bar k}:=({1+\epsilon})\log (k\bar k)$. We assume in the sequel that the event $A$ is realized, that is, 
\begin{align}
		\forall k,\bar k\in\mathbb{N}^*
		\quad
		\begin{cases}
    \forall\bv\in V \text{ such that }(h,\bar h)(\bv)\le (\nu_{k},\bar \nu_{\bar k})\text{ we have } \\
		M(\bv) \ge 
		- \frac{b}{n}g(\nu_{k})\bar g(\bar\nu_{k})		
		- \frac{b}{\sqrt{n}}g(\nu_{k})
		- \frac{b}{\sqrt{n}}\bar g(\bar \nu_{\bar k})\\
		-(\nicefrac{b}{\sqrt{n}})\sqrt{\triangle(\delta)+\triangle_{k,\bar k}}
		-(\nicefrac{b}{n})[\triangle(\delta)+\triangle_{k,\bar k}].
		\end{cases}\label{lemma:peeling:PP:eq1}
\end{align}

For every $\bv\in V$, there are $\ell,\bar \ell\in\mathbb{N}$ such that
$\bv\in V_{\ell,\bar \ell}$. We now consider several cases. 

\begin{description}
\item[Case 1:] $\ell=\bar\ell=0$. Since $h(\bv)\le\nu_1$, $\bar h(\bv)\le\bar\nu_1$, \eqref{lemma:peeling:PP:eq1} with $k=\bar k=1$ leads to
\begin{align}
    M(\bv) \ge - \frac{b}{n}\mu^2
    - \frac{b}{\sqrt{n}}\mu
		- \frac{b}{\sqrt{n}}\mu
		-\frac{b}{\sqrt{n}}\sqrt{\triangle(\delta)}
		-\frac{b}{n}\triangle(\delta).
    \end{align}

\item[Case 2:] $\ell\ge1$ and $\bar \ell\ge 1$. Using that $h(\bv)\le \nu_{\ell+1}$ and $\bar h(\bv)\le \bar\nu_{\bar \ell+1}$, we will invoke \eqref{lemma:peeling:PP:eq1} for the indexes $(k,\bar k)=(\ell+1,\bar\ell+1)$. Additionally, we observe that, since 
$\bv\in V_{\ell,\bar\ell}$, and 
$\mu\eta^{\ell+1}=\eta^2\mu_{\ell}$,
\begin{align}
g(\nu_{\ell+1}) = \eta\mu^{\ell} =  \eta\mu^{\ell+1} + \eta\mu^{\ell} - \eta\mu^{\ell+1}\le \eta^2 g\circ h(\bv) + \mu\eta^\ell - \mu\eta^{\ell+1}.\label{lemma:peeling:PP:eq2}
\end{align}
Similarly, 
\begin{align}
\bar g(\nu_{\bar\ell+1}) &\le \eta^2 \bar g\circ \bar h(\bv) + \mu\eta^{\bar\ell} - \mu\eta^{\bar\ell+1},\label{lemma:peeling:PP:eq3}\\
g(\nu_{\ell+1})\cdot\bar g(\nu_{\bar\ell+1})&\le
\eta^4g\circ h(\bv)\cdot \bar g\circ \bar h(\bv) 
+ \mu^2\eta^\ell\eta^{\bar\ell} - \mu^2\eta^{\ell+1}\eta^{\bar\ell+1}. 
\label{lemma:peeling:PP:eq4}
\end{align}

Next, we will use the previous bounds in \eqref{lemma:peeling:PP:eq1}. To ease notation, we define the quantity
\begin{align}
\lozenge_{\ell,\bar \ell}&:=\frac{b}{n}\mu^2(\eta^\ell\eta^{\bar \ell}
- \eta^{\ell+1}\eta^{\bar \ell+1})
+\frac{b}{\sqrt{n}}\mu(\eta^\ell-\eta^{\ell+1})
+\frac{b}{\sqrt{n}}\mu(\eta^{\bar\ell}-\eta^{\bar\ell+1})\\
&+\frac{b}{\sqrt{n}}\sqrt{\triangle(\delta)+\triangle_{\ell+1,\bar \ell+1}} - \frac{b}{\sqrt{n}}\sqrt{\triangle(\delta)}\\
& +\frac{b}{n}[\triangle(\delta)+\triangle_{\ell+1,\bar \ell+1}]
-\frac{b}{n}\triangle(\delta).
\end{align}
In conclusion, from \eqref{lemma:peeling:PP:eq1} with 
$(k,\bar k)=(\ell+1,\bar\ell+1)$ and the bounds \eqref{lemma:peeling:PP:eq2}, \eqref{lemma:peeling:PP:eq3} and \eqref{lemma:peeling:PP:eq4}, we may write 
\begin{align}
    M(\bv) + \lozenge_{\ell,\bar \ell} &\ge -\frac{b}{n}\eta^4g\circ h(\bv)\cdot\bar g\circ\bar h(\bv)
    -\frac{b}{\sqrt{n}}\eta^2g\circ h(\bv)
    -\frac{b}{\sqrt{n}}\eta^2\bar g\circ\bar h(\bv)\\
    &-\frac{b}{\sqrt{n}}\sqrt{\triangle(\delta)}   
    -\frac{b}{n}\triangle(\delta).\label{lemma:peeling:PP:eq5}
\end{align}

To finish, we claim that, by appropriately fixing $\mu>0$ and 
$\epsilon,\eta>1$, we can show
$
\sup_{\ell,\bar \ell\ge1}\lozenge_{\ell,\bar \ell}\le 0.
$
We prove this claim next. 

First, 
\begin{align}
T_1(\ell,\bar\ell)&:=\frac{b}{n}\triangle_{\ell+1,\bar\ell+1} + \frac{b\mu^2}{n}(\eta^\ell\eta^{\bar \ell}
-\eta^{\ell+1}\eta^{\bar \ell+1}) \\
&=\frac{b}{n}\eta^\ell\eta^{\bar\ell}\left[
(1+\epsilon)\frac{\log(\ell+1)+\log(\bar\ell+1)}{\eta^\ell\eta^{\bar\ell}}
+\mu^2(1-\eta^2)
\right].
\end{align}
Since, for $\eta>1$,
$$
\sup_{\ell,\bar\ell\ge1}\frac{\log(\ell+1)+\log(\bar\ell+1)}{\eta^\ell\eta^{\bar\ell}}\le C, 
$$
we can fix $\epsilon>1$ and $\mu>0$ and take $\eta>1$ large enough such that
$
\sup_{\ell,\bar\ell\ge1}T_1(\ell,\bar\ell)\le0
$. 

Second, we note that 
\begin{align}
\frac{b}{\sqrt{n}}\sqrt{\triangle(\delta)+\triangle_{\ell+1,\bar \ell+1}} - \frac{b}{\sqrt{n}}\sqrt{\triangle(\delta)}
&\le \frac{b}{\sqrt{n}}\sqrt{\triangle_{\ell+1,\bar\ell+1}}. 
\end{align}
Hence, 
\begin{align}
T_2(\ell,\bar\ell)&:=\frac{b}{\sqrt{n}}\mu(\eta^\ell-\eta^{\ell+1})
+\frac{b}{\sqrt{n}}\mu(\eta^{\bar\ell}-\eta^{\bar\ell+1})\\
&+\frac{b}{\sqrt{n}}\sqrt{\triangle(\delta)+\triangle_{\ell+1,\bar \ell+1}} - \frac{b}{\sqrt{n}}\sqrt{\triangle(\delta)}\\
&\le \frac{b}{\sqrt{n}}\left[
\sqrt{\log(\ell+1)} + \mu(\eta^\ell-\eta^{\ell+1})
\right]\\
&+\frac{b}{\sqrt{n}}\left[
\sqrt{\log(\bar\ell+1)} + \mu(\eta^{\bar\ell}-\eta^{\bar\ell+1})
\right]\\
&= \frac{b}{\sqrt{n}}\eta^{\ell}\left[
\frac{\sqrt{\log(\ell+1)}}{\eta^\ell} + \mu(1-\eta)
\right]\\
&+\frac{b}{\sqrt{n}}\eta^{\ell}\left[
\frac{\sqrt{\log(\ell+1)}}{\eta^\ell} + \mu(1-\eta)
\right].\label{lemma:peeling:PP:eq4'}
\end{align}
Since, for $\eta>1$,
$
\sup_{\ell\ge1}\frac{\sqrt{\log(\ell+1)}}{\eta^\ell}\le C, 
$
again, we may fix $\epsilon>1$ and $\mu>0$ and take $\eta>1$ large enough such that
$
\sup_{\ell,\bar\ell\ge1}T_2(\ell,\bar\ell)\le0
$.

We thus conclude that 
$
\sup_{\ell,\bar \ell\ge1}\lozenge_{\ell,\bar \ell}\le 
\sup_{\ell,\bar \ell\ge1}[T_1(\ell,\bar \ell)+T_2(\ell,\bar \ell)]\le 0, 
$
as claimed. 

\item[Case 3:] $\ell\ge1$ and $\bar\ell=0$.
Since $h(\bv)\le \nu_{\ell+1}$ and $\bar h(\bv)\le\bar\nu_1$, we get from \eqref{lemma:peeling:PP:eq1} with $(k,\bar k)=(\ell+1,1)$ and \eqref{lemma:peeling:PP:eq2}, 
\begin{align}
    M(\bv) + \square_{\ell}&\ge -\frac{b}{n}\mu\eta^2g\circ h(\bv)
    -\frac{b}{\sqrt{n}}\eta^2g\circ h(\bv)
    -\frac{b}{\sqrt{n}}\mu
    -\frac{b}{\sqrt{n}}\sqrt{\triangle(\delta)}   
    -\frac{b}{n}\triangle(\delta), 
    \end{align}
where have defined
\begin{align}
\square_{\ell}&:=\frac{b}{n}\mu^2(\eta^\ell
- \eta^{\ell+1})
+\frac{b}{\sqrt{n}}\mu(\eta^\ell-\eta^{\ell+1})\\
&+\frac{b}{\sqrt{n}}\sqrt{\triangle(\delta)+\triangle_{\ell+1,1}} - \frac{b}{\sqrt{n}}\sqrt{\triangle(\delta)}\\
& +\frac{b}{n}[\triangle(\delta)+\triangle_{\ell+1,1}]
-\frac{b}{n}\triangle(\delta).
\end{align}

We claim that, by appropriately fixing $\mu>0$ and 
$\epsilon>1$ and taking $\eta>1$ large enough, we can show
$
\sup_{\ell\ge1}\square_{\ell}\le 0.
$
Indeed, the reasoning is analogous to the one in Case 2 so we omit it. 

\item[Case 4:] $\ell=0$ and $\bar\ell\ge1$. 
By analogy with Case 3, using that $h(\bv)\le \nu_1$ and $\bar h(\bv)\le \nu_{\bar \ell+1}$, we get from \eqref{lemma:peeling:PP:eq1} with $(k,\bar k)=(1,\bar\ell+1)$ and \eqref{lemma:peeling:PP:eq3},
\begin{align}
    M(\bv) + \bigcirc_{\bar\ell} &\ge -\frac{b}{n}\mu\eta^2\bar g\circ \bar h(\bv)
    -\frac{b}{\sqrt{n}}\mu
    -\frac{b}{\sqrt{n}}\eta^2\bar g\circ \bar h(\bv)
    -\frac{b}{\sqrt{n}}\sqrt{\triangle(\delta)}   
    -\frac{b}{n}\triangle(\delta),
    \end{align}
where the term $\bigcirc_{\bar\ell}$ can be shown to satisfy 
$
\sup_{\bar\ell\ge1}\bigcirc_{\bar \ell}\le0, 
$
for fixed $\mu>0$ and $\epsilon>1$ and large enough $\eta>1$. 
\end{description}

The proof of the statement of the lemma is finished once we join the lower bounds established in the four cases.
\end{proof}

\subsection{Peeling for $\IP$}

The following lemma is a restatement of Lemma 6 in \cite{2019dalalyan:thompson}. The proof follows by similar arguments used in the proof of Lemma \ref{lemma:peeling:product:process}. 
\begin{lemma}\label{lemma:peeling:chevet}
Let $b>0$ be a constant and $c\ge1$ be an universal constant. Assume that, for every $r,\bar r>0$ and every $\delta\in(0,1/c]$, the event $A(r,\bar r,\delta)$ defined by the inequality
\begin{align}
\inf_{\bv\in V: (h,\bar h)(\bv)\le (r,\bar r)}
M(\bv) &\ge -\frac{b}{\sqrt{n}}g(r)  -\frac{b}{\sqrt{n}}\bar g(\bar r)
- \frac{b}{\sqrt{n}}\sqrt{\log(2/\delta)},
\end{align}
has probability at least $1-c\delta$. 

Then, for every $\delta\in(0,1/c]$, with probability at least $1-c\delta$, it holds that, for all 
$\bv\in V$,
\begin{align}
M(\bv) &\ge -\frac{Cb}{\sqrt{n}}g\circ h(\bv)
-\frac{Cb}{\sqrt{n}}\bar g\circ\bar h(\bv)
-\frac{Cb}{\sqrt{n}}(1+\sqrt{\log(1/\delta)}).
\end{align}
\end{lemma}

\subsection{Peeling for $\MP$}

To prove Proposition \ref{prop:gen:MP}, we will use Lemma \ref{lemma:peeling:chevet} (with $\bar g=\bar h\equiv0$) and the following lemma. 
\begin{lemma}\label{lemma:peeling:multiplier:process}
Let $b>0$ be a constant and $c\ge1$ be universal constants. Assume that for every $r>0$ and every $\delta\in(0,1/c]$, the event $A(r,\delta)$ defined by the inequality
\begin{align}
\inf_{\bv\in V: h(\bv)\le r}
M(\bv) &\ge - \left[1+\frac{1}{\sqrt{n}}\sqrt{\log(1/\delta)}\right]b\frac{g(r)}{\sqrt{n}}\\
&-\frac{b}{\sqrt{n}}\sqrt{\log(1/\delta)}-\frac{b}{n}\log(1/\delta),
\end{align}
has probability at least $1-c\delta$. 

Then, with probability at least $1-c\delta$, we have that, for all $\bv\in V$, 
\begin{align}
M(\bv) &\ge - C\left\{[1+(\nicefrac{1}{\sqrt{n}})\sqrt{\triangle(\delta)}]\frac{b}{\sqrt{n}}
+\frac{b}{n}\right\}g\circ h(\bv)\\
&-C(\nicefrac{b}{\sqrt{n}})[1 + \sqrt{\triangle(\delta)}]
-C(\nicefrac{b}{n})[\triangle(\delta) + \sqrt{\triangle(\delta)}]. 
\end{align}
\end{lemma}
\begin{proof}
Let $\mu>0$ and $\eta,\epsilon>1$ be two parameters to be chosen later on.  We set $\mu_0=0$ and, for $k\ge1$, $\mu_k := \mu \eta^{k-1}$. Define, for $k\in\mathbb{N}$, the sets
$$
V_k :=\{\bv\in V : \mu_k\le (g\circ h)(\bv)< \mu_{k+1}\}.
$$ 
For $k\in\mathbb{N}^*$, we set $\nu_k  :=g^{-1}(\mu_k)$.

An union bound and the fact that $\sum_{k\ge 1} k^{-1-\epsilon}\le
1+\epsilon^{-1}$ imply that the event
$$
A := \bigcap_{k=1}^\infty A\left(\nu_k,\frac{\epsilon\delta}{(1+\epsilon)k^{1+\epsilon}}\right),	
$$
has a probability at least $1-c\delta$. For convenience, we define $\triangle(t):=\log\{(1+\epsilon)/(\epsilon t)\}$ and $\triangle_k:=({1+\epsilon})\log k$. Throughout the proof, assume that this event is realized:
\begin{align}
		\forall k\in\mathbb{N}^*
		\quad
		\begin{cases}
    \forall\bv\in V \text{ such that } h(\bv)\le \nu_k\text{ we have } \\
		M(\bv) \ge - [1+(\nicefrac{1}{\sqrt{n}})\sqrt{\triangle(\delta)+\triangle_k}]b\frac{g(\nu_k)}{\sqrt{n}}
		-(\nicefrac{b}{\sqrt{n}})\sqrt{\triangle(\delta)+\triangle_k}	
		-(\nicefrac{b}{n})[\triangle(\delta)+\triangle_k].
		\end{cases}\label{lemma:peeling:MP:eq1}
\end{align}

For every $\bv\in V$, there is 
$\ell\in\mathbb{N}$ such that
$\bv\in V_\ell$. 

\begin{description}
\item[Case 1:] $\ell=0$. In that case, \eqref{lemma:peeling:MP:eq1} with $k=1$ and using $g(\nu_1)=\mu$ lead to
\begin{align}
    M(\bv) &\ge - [1+(\nicefrac{1}{\sqrt{n}})\sqrt{\triangle(\delta)}](\nicefrac{b\mu}{\sqrt{n}})
    -(\nicefrac{b}{\sqrt{n}})\sqrt{\triangle(\delta)}
		-(\nicefrac{b}{n})\triangle(\delta).
    \end{align} 

\item[Case 2:] $\ell\ge1$. Since $v\in V_\ell$, $h(\bv)\le\nu_{\ell+1}$, so we can invoke \eqref{lemma:peeling:MP:eq1} for $k=\ell+1$. 

First, we make some observations. Recall that $g(\nu_{\ell+1})=\mu\eta^{\ell}$. Since $\eta>1$, there is universal constant $C\ge1$ such that
\begin{align}
\sqrt{\triangle_{\ell+1}}\mu\eta^{\ell}
\le C\sqrt{(1+\epsilon)}\mu\eta^{\ell+1}
\le C\sqrt{(1+\epsilon)}\eta^2g\circ h(\bv).  
\end{align}
Additionally, 
\begin{align}
\mu\eta^{\ell} = \mu\eta^{\ell+1} + \mu(\eta^{\ell} - \eta^{\ell+1})\le \eta^2 g\circ h(\bv) + \mu(\eta^{\ell} - \eta^{\ell+1}).
\end{align}
We thus conclude that
\begin{align}
[1+(\nicefrac{1}{\sqrt{n}})\sqrt{\triangle(\delta)+\triangle_{\ell+1}}]b\frac{g(\nu_{\ell+1})}{\sqrt{n}}
&\le [1+(\nicefrac{1}{\sqrt{n}})\sqrt{\triangle(\delta)}]\frac{b\mu\eta^{\ell}}{\sqrt{n}}
+ \sqrt{\triangle_{\ell+1}}\frac{b\mu\eta^{\ell}}{n}\\
&\le [1+(\nicefrac{1}{\sqrt{n}})\sqrt{\triangle(\delta)}]\frac{b\eta^2}{\sqrt{n}}g\circ h(\bv)\\
&+[1+(\nicefrac{1}{\sqrt{n}})\sqrt{\triangle(\delta)}]\frac{b\mu}{\sqrt{n}}(\eta^\ell-\eta^{\ell+1})\\
&+C\sqrt{(1+\epsilon)}\frac{\eta^2 b}{n}g\circ h(\bv).
\label{lemma:peeling:MP:eq2}
\end{align}

In order to use the previous bound in \eqref{lemma:peeling:MP:eq1}, it will be convenient to define the quantity:
\begin{align}
\lozenge_\ell&:= [1+(\nicefrac{1}{\sqrt{n}})\sqrt{\triangle(\delta)}]\frac{b\mu}{\sqrt{n}}(\eta^\ell-\eta^{\ell+1})\\
&+(\nicefrac{b}{\sqrt{n}})\sqrt{\triangle(\delta)+\triangle_{\ell+1}}-(\nicefrac{b}{\sqrt{n}})\sqrt{\triangle(\delta)}\\	
&+(\nicefrac{b}{n})[\triangle(\delta)+\triangle_{\ell+1}]
-(\nicefrac{b}{n})\triangle(\delta).
\end{align}

In conclusion, from \eqref{lemma:peeling:MP:eq1} (with $k=\ell+1$) and \eqref{lemma:peeling:MP:eq2}, we obtain
\begin{align}
M(\bv) + \lozenge_{\ell} &\ge - \left\{[1+(\nicefrac{1}{\sqrt{n}})\sqrt{\triangle(\delta)}]\frac{b\eta^2}{\sqrt{n}}
+C\sqrt{(1+\epsilon)}\frac{\eta^2 b}{n}
\right\}g\circ h(\bv)\\
		&-(\nicefrac{b}{\sqrt{n}})\sqrt{\triangle(\delta)}	
		-(\nicefrac{b}{n})\triangle(\delta).    
\end{align}

Next, we claim that, by appropriately fixing $\mu>0$ and 
$\epsilon,\eta>1$, we can show
\begin{align}
\sup_{\ell\ge1}\lozenge_{\ell}\le 0.\label{lemma:peeling:MP:eq3}
\end{align}
We present the proof of the claim next. 

First, since $\triangle(\delta)\ge1$, 
\begin{align}
T_1(\ell)&:= \sqrt{\triangle(\delta)}\frac{b\mu}{n}(\eta^\ell-\eta^{\ell+1}) + \frac{b}{n}\triangle_{\ell+1}\\
&\le \sqrt{\triangle(\delta)}\frac{b}{n}\left[
\mu(\eta^\ell-\eta^{\ell+1}) + \log(\ell+1)
\right]\\
&\le\sqrt{\triangle(\delta)}\frac{b}{n}\eta^{\ell} 
\left[
\mu(1-\eta) + \frac{\log(\ell+1)}{\eta^\ell}
\right].
\end{align}
Since, for $\eta>1$,
$
\sup_{\ell\ge1}\frac{\log(\ell+1)}{\eta^\ell}\le C, 
$
we may fix $\epsilon>1$ and $\mu>0$ and take $\eta>1$ large enough such that
$
\sup_{\ell\ge1}T_1(\ell)\le0
$.

Secondly, 
\begin{align}
T_2(\ell)&:=\frac{b}{\sqrt{n}}\mu(\eta^\ell-\eta^{\ell+1})
+\frac{b}{\sqrt{n}}\sqrt{\triangle(\delta) + \triangle_{\ell+1}} 
-\frac{b}{\sqrt{n}}\sqrt{\triangle(\delta)}\\
&\le \frac{b}{\sqrt{n}}\left[
\sqrt{\log(\ell+1)} + \mu(\eta^\ell-\eta^{\ell+1})
\right]\\
&\le \frac{b}{\sqrt{n}}\eta^\ell\left[
\frac{\sqrt{\log(\ell+1)}}{\eta^\ell} + \mu(1-\eta)
\right]. 
\end{align}
Again, since, for $\eta>1$,
$
\sup_{\ell\ge1}\frac{\sqrt{\log(\ell+1)}}{\eta^\ell}\le C, 
$
we may fix $\epsilon>1$ and $\mu>0$ and take $\eta>1$ large enough such that
$
\sup_{\ell\ge1}T_2(\ell)\le0
$.

In conclusion, since $\lozenge_\ell=T_1(\ell)+T_2(\ell)$, we have proved \eqref{lemma:peeling:MP:eq3}.
\end{description}

Joining the lower bounds from both cases, we prove the statement of the lemma.
\end{proof}

\section{The Multiplier Process}\label{s:multiplier:process}
In this section we prove Theorem \ref{thm:mult:process}. We refer to the definitions and notations in Section \ref{s:multiplier:process:main}. Given $f,g\in L_{\psi_2}$, we set
$
\langle f,g\rangle_n:=\hat\probn fg
$
and
$
\Vert f\Vert_n:=\sqrt{\langle f,f\rangle_n}.
$
We recall the H\"older-type inequality $\Vert fg\Vert_{\psi_1}\le\Vert f\Vert_{\psi_2}\Vert g\Vert_{\psi_2}$.

Our proof is inspired by Dirksen's method \cite{2015dirksen} which obtained concentration inequalities for the \emph{quadratic process}. One key observation used by Dirksen \cite{2015dirksen} and Bednorz \cite{2014bednorz} is that one must bound the chain differently for $k\le \lfloor\log_2 n\rfloor$, the so called ``subgaussian path'' and $k\ge \lfloor\log_2 n\rfloor$, the ``subexponential path''. In bounding the multiplier process, we additionally introduce a ``lazy walked'' chain,  a technique already present in Talagrand's original bound for the \emph{empirical process} \cite{2014talagrand}. 

We first present some preliminary lemmas. 
\begin{lemma}\label{lemma:increment:bounds:mult:process}
Let $f,f' \in L_{\psi_2}$. 

If for $k\in\mbN$, $2^{k/2}\le\sqrt{n}$, then for any $u\ge1$, with probability at least $1-2\exp(-(2^k+u))$,
\begin{align}
|M(f)-M(f')|\le \left[(1+\sqrt{2})\frac{2^{k/2}}{\sqrt{n}}+\sqrt{\frac{2u}{n}}+\frac{u}{n}\right]\Vert\xi\Vert_{\psi_2}\Vert f-f'\Vert_{\psi_2}.
\end{align}

If for $k\in\mbN$, $2^{k/2}\ge\sqrt{n}$, then for any $u\ge1$, with probability at least $1-2\exp(-(2^k+u))$,
\begin{align}
\Vert f-f'\Vert_{n}\le(\sqrt{u}+2^{k/2})\frac{[2(1+\sqrt{2})+1]^{1/2}}{\sqrt{n}}\dist(f,f').
\end{align}
\end{lemma}
\begin{proof}
Suppose $2^{\frac{k}{2}}\le \sqrt{n}$. We first note that, by the triangle and H\"older-type inequalities for the norm $\psi_1$, we have
$\xi f-\xi f'\in L_{\psi_1}$. The Bernstein-type inequality implies that, for all $v\ge0$, with probability at least $1-2e^{-v}$,
\begin{align}
|M(f)-M(f')|&=|\hat\probn[(\xi f-\xi f')-\probn(\xi f-\xi f')]|\\
&\le\Vert \xi f-\xi f'-\probn(\xi f-\xi f')\Vert_{\psi_1}
\left(\sqrt{\frac{2v}{n}}+\frac{v}{n}\right).
\end{align}
Taking $v:=2^{k}+u$ and using that $2^{\frac{k}{2}}\le \sqrt{n}$,  we get 
\begin{align}
\sqrt{\frac{2v}{n}}+\frac{v}{n}
\le (1+\sqrt{2})\frac{2^{k/2}}{\sqrt{n}}+\sqrt{\frac{2u}{n}}+\frac{u}{n},
\end{align}
establishing the the first inequality claimed in the lemma.

Suppose now $2^{k/2}\ge\sqrt{n}$. The second inequality is proved in \cite{2015dirksen}, Lemma 5.4. We include the proof for completeness with slightly better constants. We first claim that, for any $v\ge1$, with probability at least $1-2\exp(-nv)$,
\begin{align}
\Vert f-f'\Vert_n\le[2(1+\sqrt{2})+1]^{1/2}\Vert f-f'\Vert_{\psi_2}\sqrt{v}. 
\end{align}
Indeed, by Bernstein's inequality, for any $v\ge1$, with probability at least $1-2e^{-nv}$, 
\begin{align}
|\hat\probn[(f-f')^2-\probn(f-f')^2]|\le \Vert (f-f')^2-\probn(f-f')^2\Vert_{\psi_1}\left(\sqrt{2v}+v\right)
\le2(1+\sqrt{2})\Vert (f-f')^2\Vert_{\psi_1}v.
\end{align}
For $v\ge1$, we also have $|\probn (f-f')^2|\le\Vert (f-f')^2\Vert_{\psi_1}\le\Vert (f-f')^2\Vert_{\psi_1}v$. This, triangle inequality, $\Vert (f-f')^2\Vert_{\psi_1}=\Vert f-f'\Vert_{\psi_2}^2$ and the display imply the claim.

Now, since $2^{k/2}\ge\sqrt{n}$, setting $v:=2^{k}u\frac{1}{n}\ge1$ we get $\exp(-nv)=\exp(-2^ku)$, entailing the claim.
\end{proof}

The following result is a straightforward modification of Lemma A.4 by Dirksen \cite{2015dirksen}.
\begin{lemma}\label{lemma:union:bound:chaining:lazy}
Fix $1\le p<\infty$, $u\ge2$ and set $\ell:=\lfloor
\log_2p\rfloor$. For every $n>\ell$ let $(\Omega_i^{(k)})_{i\in I_k}$ be a collection of events satisfying 
\begin{align}
\prob(\Omega_i^{(k)})\le2\exp(-(2^k + u)), \quad\mbox{for all $i\in I_k$}
\end{align}
or
\begin{align}
\prob(\Omega_i^{(k)})\le2\exp(-2^ku), \quad\mbox{for all $i\in I_k$}. 
\end{align}

If $|I_k|\le 2^{2^k+1}$, then for an absolute constant $c>0$, 
\begin{align}
\prob\left(\cup_{k>\ell}\cup_{i\in I_k}\Omega^{(k)}_i\right)\le c\exp(-pu/4).
\end{align}	
\end{lemma}

\begin{proof}[Proof of Theorem \ref{thm:mult:process}]
Let $(F_k)$ be an optimal admissible sequences for 
$\gamma_{2}(F)$. Let $(\calF_k)$ be defined by $\calF_0:=F_0$ and $\calF_k:=\cup_{j\le k}F_j$ so that $|\calF_k|\le2|F_k|=2^{2^k+1}$. Set $k_0:=\min\{k\ge1:2^{k/2}>\sqrt{n}\}$ and let us define 
$
\calI:=\{k\in\mbN:\ell<k<k_0\}
$
and
$
\calJ:=\{k\in\mbN:k\ge k_0\}
$. 
Given $k\in\mbN$ and $f\in F$, let 
$
\Pi_k(f)\in\argmin_{f'\in\calF_k}\dist(f,f'),
$
Given $f\in F$, we take some $\Pi_0(f)\in F$ and for any $j\in\mathbb{N}$, we define the ``lazy walk'' chain selection by:  
\begin{align}
k_j(f):=\inf\left\{j\ge k_{j-1}(f):\dist(f,\Pi_{j}(f))\le\frac{1}{2}\dist(f,\Pi_{k_{j-1}(f)}(f))\right\}.
\end{align}
For simplicity of notation, we will set $\pi_j(f):=\Pi_{k_j(f)}(f)$. For $f\in F$, our proof will rely on the chain: 
\begin{align}
M(f)-M(\pi_0(f))
&=\sum_{j:k_j(f)\in\calJ}\left[M(\pi_{j+1}(f))-M(\pi_{j}(f))\right]
+\sum_{j:k_j(f)\in\calI}\left[M(\pi_{j}(f))-M(\pi_{j-1}(f))\right],\label{thm:mult:process:eq:chain}
\end{align}
where we have used that $\cup_{k\ge0}\calF_k$ is dense on $F$. 

Fix $u\ge1$. Given any $k\in\mathbb{N}$, define the event 
$\Omega_{k,\calI,u}$ for which, for all $f,f'\in \calF_k$, we have
\begin{align}
|M(f)-M(f')|\le \left[(1+\sqrt{2})\frac{2^{k/2}}{\sqrt{n}}+\sqrt{\frac{2u}{n}}+\frac{u}{n}\right]\Vert\xi\Vert_{\psi_2}\Vert f-f'\Vert_{\psi_2}.\label{thm:mult:process:eq1}
\end{align}
Define also the event $\Omega_{k,\calJ,u}$ for which, for all $f,f'\in \calF_k$, we have
\begin{align}
\Vert f-f'\Vert_n&\le(\sqrt{u}+2^{k/2})\frac{[2(1+\sqrt{2})+1]^{1/2}}{\sqrt{n}}\Vert f-f'\Vert_{\psi_2}.\label{thm:mult:process:eq2}
\end{align}
For simplicity, define the vector 
$\bxi:=(\xi_i)_{i\in[n]}$ and $\Vert\bxi\Vert_n:=\frac{1}{\sqrt{n}}\Vert\bxi\Vert_2$. Given $v\ge1$, we define the event 
$\Omega_{\xi,v}$, for which
\begin{align}
\Vert\bxi\Vert_n\le[2(1+\sqrt{2})+1]^{1/2}\Vert\xi\Vert_{\psi_2}\sqrt{v}.\label{thm:mult:process:eq2'} 
\end{align}

By an union bound over all possible pairs 
$(\pi_{k-1}(f),\pi_{k}(f))$ we have $|\Omega_{k,\calI,u}|\le|\calF_{k-1}||\calF_{k}|\le2^{2^{k+1}}$. If 
$\Omega_{\calI,u}:=\cap_{k\in\calI}\Omega_{k,\calI,u}$, the first bound on Lemma \ref{lemma:increment:bounds:mult:process} for $k\in\calI$ and Lemma \ref{lemma:union:bound:chaining:lazy} imply that there is universal constant $c>0$
$$
\prob(\Omega_{\calI,u}^c)\le ce^{-u/4}.
$$
Similarly, the second bound in Lemma \ref{lemma:increment:bounds:mult:process} for $k\in\calJ$ and Lemma \ref{lemma:union:bound:chaining:lazy} imply that for the event  
$\Omega_{\calJ,u}:=\cap_{k\in\calJ}\Omega_{k,\calJ,u}$, we have 
$$
\prob(\Omega_{\calJ,u}^c)\le ce^{-u/4}.
$$
Using Bernstein's inequality for $\{\xi_i\}_{i\in[n]}$ we get $\prob(\Omega_{\xi,v}^c)\le ce^{-vn}$. Hence, the event 
$\Omega_{u,v}:=\Omega_{\calI,u}\cap\Omega_{\calJ,u}\cap\Omega_{\xi,v}$ has 
$\prob(\Omega_{u,v}^c)\le ce^{-u/4}+ce^{-vn}$. 

We next fix $u\ge2$ and $v\ge1$ and assume that $\Omega_{u,v}$ always holds. We now bound the chain over $\calI$ and $\calJ$ separately.

\begin{description}

\item[\textbf{Part 1}: \emph{The subgaussian path $\calI$}.] Given $j$ such that $k_j(f)\in\calI$, since $\pi_j(f),\pi_{j-1}(f)\in\calF_{k_j(f)}$, we may apply \eqref{thm:mult:process:eq1} to $k:=k_j(f)$ so that  
\begin{align}
|M(\pi_{j}(f))-M(\pi_{j-1}(f))|\le \left[(1+\sqrt{2})\frac{2^{k_j(f)/2}}{\sqrt{n}}+\sqrt{\frac{2u}{n}}+\frac{u}{n}\right]\Vert\xi\Vert_{\psi_2}\Vert \pi_{j}(f)-\pi_{j-1}(f)\Vert_{\psi_2}.
\end{align}
We note that, by triangle inequality and minimality of $k_{j-1}(f)$,  
\begin{align}
\Vert \pi_{j}(f)-\pi_{j-1}(f)\Vert_{\psi_2}
\le\dist(f,\calF_{k_j(f)})+\dist(f,\calF_{k_{j-1}(f)})\le\dist(f,\calF_{k_j(f)})+2\dist(f,\calF_{k_{j}(f)-1}),
\end{align}
so that 
\begin{align}
\sum_{j:k_j(f)\in\calI}2^{k_j(f)/2}\Vert \pi_{j}(f)-\pi_{j-1}(f)\Vert_{\psi_2}
\le(1+2\sqrt{2})\gamma_2(F). \label{thm:mult:process:sum:gamma1}
\end{align}
Moreover, by the definition of the lazy walked chain and a geometric series bound,
\begin{align}
\sum_{j:k_j(f)\in\calI}\Vert \pi_{j}(f)-\pi_{j-1}(f)\Vert_{\psi_2}\le4\dist(f,\calF_0)\le4\bar\Delta(F). \label{thm:mult:process:sum:delta1}
\end{align}
We thus conclude that 
\begin{align}
\left|\sum_{j:k_j(f)\in\calI}[M(\pi_{j}(f))-M(\pi_{j-1}(f))]\right|&\le 
4\left(\sqrt{\frac{2u}{n}}+\frac{u}{n}\right)\Vert\xi\Vert_{\psi_2}\bar\Delta(F)
+(1+\sqrt{2})(1+2\sqrt{2})\Vert\xi\Vert_{\psi_2}\frac{\gamma_{2}(F)}{\sqrt{n}}.	\label{thm:mult:process:subgaussian}
\end{align}

\item[\textbf{Part 2}: \emph{The subexponential path $\calJ$}.] Let us denote by $\bfQ$ the joint distribution of $(\xi,X)$ and $\hat\bfQ$ the empirical distribution associated to $\{(\xi_i,X_i)\}_{i\in[n]}$. In particular, $M(f)=\hat\bfQ(\cdot)f-\bfQ\hat\bfQ(\cdot)f$. By Jensen's and triangle inequalities,
\begin{align}
\left|\sum_{j:k_j(f)\in\calJ}[M(\pi_{j+1}(f))-M(\pi_{j}(f))]\right|
&\le\sum_{j:k_j(f)\in\calJ}\hat\bfQ(\cdot)\left|\pi_{j+1}(f)-\pi_{j}(f)\right|\\
&+\sum_{j:k_j(f)\in\calJ}\bfQ\hat\bfQ(\cdot)\left|\pi_{j+1}(f)-\pi_{j}(f)\right|.\label{thm:mult:process:eq3}
\end{align}
For convenience, we set $\hat T_j:=\hat\bfQ(\cdot)\left|\pi_{j+1}(f)-\pi_{j}(f)\right|$.

Given $j$ such that $k_j(f)\in\calJ$, since 
$\pi_{j+1}(f),\pi_{j}(f)\in\calF_{k_{j+1}(f)}$, we may apply \eqref{thm:mult:process:eq2} to $k:=k_{j+1}(f)$. This fact, \eqref{thm:mult:process:eq2'} and Cauchy-Schwarz yield 
\begin{align}
\hat T_j&\le\Vert\bxi\Vert_n\Vert\pi_{j+1}(f)-\pi_{j}(f)\Vert_n
\le\sqrt{v}[2(1+\sqrt{2})+1]\Vert\xi\Vert_{\psi_2}(\sqrt{u}+2^{k_{j+1}(f)/2})\frac{1}{\sqrt{n}}\Vert \pi_{j+1}(f)-\pi_{j}(f)\Vert_{\psi_2}. 
\end{align}
In a similar fashion, we can also state that with probability at least $1-ce^{-u/4}$,  
\begin{align}
\frac{\hat T_j}{\Vert\bxi\Vert_n}\le(\sqrt{u}+2^{k_{j+1}(f)/2})\frac{[2(1+\sqrt{2})+1]^{1/2}}{\sqrt{n}}\Vert\pi_{j+1}(f)-\pi_{j}(f)\Vert_{\psi_2},
\end{align}
so integrating the tail leads to
\begin{align}
\left\{\esp\left(\frac{\hat T_j}{\Vert\bxi\Vert_n}\right)^2\right\}^{1/2}\le c2^{k_{j+1}(f)/2}\frac{[2(1+\sqrt{2})+1]^{1/2}}{\sqrt{n}}\Vert\pi_{j+1}(f)-\pi_{j}(f)\Vert_{\psi_2},
\end{align}
and by H\"older's inequality,
\begin{align}
\probn\hat T_j\le \left\{\esp\Vert\bxi\Vert_n^2\right\}^{1/2}\left\{\esp\left(\frac{\hat T_j}{\Vert\bxi\Vert_n}\right)^2\right\}^{1/2}
\le c[2(1+\sqrt{2})+1]^{1/2}\frac{\Vert\xi\Vert_{\psi_2}}{\sqrt{n}}\frac{2^{k_{j+1}(f)/2}\Vert\pi_{j+1}(f)-\pi_{j}(f)\Vert_{\psi_2}}{\sqrt{n}}.
\end{align}
We thus conclude from \eqref{thm:mult:process:eq3} and a analogous reasoning to \eqref{thm:mult:process:sum:gamma1} and \eqref{thm:mult:process:sum:delta1} that 
\begin{align}
\left|\sum_{j:k_j(f)\in\calJ}[M(\pi_{j+1}(f))-M(\pi_{j}(f))]\right|&\le 
c_1^2\sqrt{v}\Vert\xi\Vert_{\psi_2}\sqrt{u}\frac{4\bar\Delta(F)}{\sqrt{n}}\\
&+\left(c_1^2\sqrt{v}+c_1\frac{c}{\sqrt{n}}\right)\Vert\xi\Vert_{\psi_2}(1+2\sqrt{2})\frac{\gamma_2(F)}{\sqrt{n}},
\end{align}
with $c_1:=[2(1+\sqrt{2})+1]^{1/2}$.
\end{description}

From the above bound, \eqref{thm:mult:process:subgaussian} and \eqref{thm:mult:process:eq:chain} we conclude that, for any $u\ge2$ and $v\ge1$, on the event $\Omega_{u,v}$ of probability at least $1-ce^{-u/4}-ce^{-nv}$, we have the bound stated in the theorem.
\end{proof}

\section{The Product Process}\label{s:product:process}
In this section we prove Theorem \ref{thm:product:process} --- restate it here in a more general form. We refer to notation and definitions in Section \ref{s:multiplier:process:main}.
\begin{theorem}[Product process]\label{thm:product:process:2}
Let $F,G$ be subclasses of $L_{\psi_2}$. For any $1\le p<\infty$,
\begin{align}
\Lpnorm{\sup_{(f,g)\in F\times G}|A(f,g)|}\lesssim\frac{\gamma_{2,p}(F)\gamma_{2,p}(G)}{n}
+\bar\Delta(F)\frac{\gamma_{2,p}(G)}{\sqrt{n}}+\bar\Delta(G)\frac{\gamma_{2,p}(F)}{\sqrt{n}}+\bar\Delta(F)
\bar\Delta(G)\left(\sqrt{\frac{p}{n}}+\frac{p}{n}\right).
\end{align}
In particular, there exist universal constants $c,C>0$, such that for all $n\ge1$ and $u\ge 1$, with probability at least $1-e^{-u}$,
\begin{align}
\sup_{(f,g)\in F\times G}\left|A(f,g)\right|&\le C\left[\frac{\gamma_{2}(F)\gamma_{2}(G)}{n}
+\bar\Delta(F)\frac{\gamma_{2}(G)}{\sqrt{n}}+\bar\Delta(G)\frac{\gamma_{2}(F)}{\sqrt{n}}\right]\\
&+c\sup_{(f,g)\in F\times G}\Vert fg-\probn fg\Vert_{\psi_1}\left(\sqrt{\frac{u}{n}}+\frac{u}{n}\right).
\end{align}
\end{theorem}

Before proving the theorem we will need some auxiliary results. 
\begin{lemma}\label{lemma:increment:bounds}
Let $f,f'$ and $g,g'$ in $L_{\psi_2}$. 

If for $k\in\mbN$, $2^{k/2}\le\sqrt{n}$, then for any $u\ge1$, with probability at least $1-2\exp(-2^ku)$,
\begin{align}
|A(f,g)-A(f',g')|\le 2(1+\sqrt{2})\frac{u2^{k/2}}{\sqrt{n}}\Vert fg-f'g'\Vert_{\psi_1}.
\end{align}

If for $k\in\mbN$, $2^{k/2}\ge\sqrt{n}$, then for any $u\ge1$, with probability at least $1-2\exp(-2^ku)$,
\begin{align}
\Vert f-f'\Vert_{n}\le\sqrt{u}2^{k/2}\frac{[2(1+\sqrt{2})+1]^{1/2}}{\sqrt{n}}\dist(f,f').
\end{align}
\end{lemma}
\begin{proof}
The proof is as in the proof of Lemma \ref{lemma:increment:bounds:mult:process} using Bernstein's inequality. A minor difference is that it is used the larger scaling $v:=2^ku$ for all $u\ge1$ instead of $v:=2^k+u$ for all $u>0$. The fact that 
$\sqrt{2u}+u\le (\sqrt{2}+1)u$ for $u\ge1$ is used in the first inequality of the lemma. 
\end{proof}

As in the proof of Theorem \ref{thm:mult:process}, we combine Dirksen's method \cite{2015dirksen} with Talagrand's \cite{2014talagrand} ``lazy-walked'' chain. One difference now is that we will explicitly need Dirksen's bound for the quadratic process
$$
A(f):=\hat\probn(f^2-\probn f^2).
$$
The following proposition is a corollary of the proof of Theorem 5.5 in \cite{2015dirksen}.
\begin{proposition}[Dirksen \cite{2015dirksen}, Theorem 5.5]\label{prop:quadratic:process}
Let $F\subset L_{\psi_2}$. Given $1\le p<\infty$, set $\ell:=\lfloor\log_2p\rfloor$ and $k_0:=\min\{k>\ell:2^{k/2}>\sqrt{n}\}$. Let $(\calF_k)$ be an optimal admissible sequence for $\gamma_{2,p}(F,\dist)$ and, for any $f\in F$ and $k\in\mbN$, let $\pi_k(f)\in\argmin_{f'\in\calF_\ell}\dist(f,f')$.  Then there exists universal constant $c>0$ such that for all $n\in\mathbb{N}$ and $u\ge2$, with probability at least $1-ce^{-pu/4}$,
\begin{align}
\sup_{f\in F}\sup_{k\ge k_0}\left|A(\pi_{k}(f))\right|^{1/2}-\sup_{f\in F}\left|A(\pi_{k_0}(f))\right|^{1/2}&\le  
\sqrt{u}\left[25\frac{\gamma_{2,p}(F,\dist)}{\sqrt{n}}+\left(85	\frac{\bar\Delta(F)\gamma_{2,p}(F,\dist)}{\sqrt{n}}\right)^{1/2}\right].
\end{align}

Moreover, for all $n\in\mathbb{N}$ and $u\ge1$, with probability at least $1-ce^{-pu/4}$,
\begin{align}
\sup_{f\in F}\left|A(\pi_{k_0}(f))\right|^{1/2}\le\sqrt{u}[4(1+\sqrt{2})+2]^{1/2}\bar\Delta(F).
\end{align}
\end{proposition}
%\begin{proof}
%We only need to prove the second claim as the first is a corollary of the proof in Theorem 5.5 of \cite{2015dirksen}. Fix $u\ge1$ and define the event 
%$\Omega_{k_0,u}$ for which, for all $f\in F$,
%\begin{align}
%\Vert \pi_{k_0}(f)\Vert_{n}\le\sqrt{u}2^{k_{0}/2}\frac{[2(1+\sqrt{2})+1]^{1/2}}{\sqrt{n}}\Vert\pi_{k_0}(f)\Vert_{\psi_2}.
%\end{align}
%We have $2^{2^{k_0}}$ possibilities for 
%$\pi_{k_0}(f)$, so an union bound and the first bound in Lemma \ref{lemma:increment:bounds} (using that $2^{k_0}>\sqrt{n}$), give
%\begin{align}
%\prob(\Omega_{k_0,u}^c)\le2^{2^{k_0}}\cdot 2\exp(-2^{k_0}u)\le2^{2^{k_0}u}\cdot 2\exp(-2^{k_0}u)\le c\exp(-2^{k_0}u/4)
%\le c\exp(-pu/4),
%\end{align}
%where we used that $2^{k_0}>2^{\ell}\ge p$. To finish the proof just note that, by minimality of $k_0$, we have $2^{k_0/2}\le\sqrt{2n}$.
%\end{proof}

Finally, besides Lemma \ref{lemma:union:bound:chaining:lazy}, we need some additional lemmas from Dirksen \cite{2015dirksen}.

\begin{lemma}[Dirksen \cite{2015dirksen}, Lemma A.3]\label{lemma:dirksen:A3}
Fix $1\le p<\infty$, set $\ell:=\lfloor\log_2 p\rfloor$ and let $(X_t)_{t\in T}$ be a finite collection of real-valued random variables with $|T|\le2^{2^\ell}$. 

Then 
\begin{align}
\left(\esp\sup_{t\in T}|X_t|^p\right)^{1/p}\le
2\sup_{t\in T}(\esp|X_t|^p)^{1/p}.\label{lemma:dirksen:A3:moment}
\end{align}
\end{lemma}

\begin{lemma}[Dirksen \cite{2015dirksen}, Lemma A.5]\label{lemma:dirksen:A5}
Fix $1\le p<\infty$ and $0<\alpha<\infty$. Let $\gamma\ge0$ and suppose that $\xi$ is a positive random variable such that for some $c\ge1$ and $u_*>0$, for all $u\ge u_*$, 
\begin{align}
\prob(\xi>\gamma u)\le c\exp(-pu^\alpha/4).\label{lemma:dirksen:A5:tail}
\end{align}
Then, for a constant $\tilde c_\alpha>0$, depending only on $\alpha$, 
\begin{align}
(\esp\xi^p)^{1/p}\le\gamma(\tilde c_\alpha c+u_*).	
\end{align}
\end{lemma}

\begin{lemma}[Dirksen \cite{2015dirksen}, Lemma A.2]\label{lemma:moments:p} 
Let $0<\alpha<\infty$. If a random variable $X$ satisfies, for some $a_1,a_2>0$,
\begin{align}
\prob(|X|\ge a_1u+a_2\sqrt{u})\le \exp(-u),
\end{align}
for all $u\ge0$, then, for absolute constant $C>0$, for all $p\ge1$, 
\begin{align}
(\esp|X|^p)^{1/p}
\le C(a_1p+a_2\sqrt{p}).
\end{align}
\end{lemma}

\begin{proof}[Proof of Theorem \ref{thm:product:process:2}]
Let $(\calF_k)$ and $(\calG_k)$ be optimal admissible sequences for $\gamma_{2,p}(F)$  and 
$\gamma_{2,p}(G)$ respectively. Set $\ell:=\lfloor\log_2 p\rfloor$, $k_0:=\min\{k>\ell:2^{k/2}>\sqrt{n}\}$ and let us define 
$
\calI:=\{k\in\mbN:\ell<k<k_0\}
$
and
$
\calJ:=\{k\in\mbN:k\ge k_0\}
$. Given $(f,g)\in F\times G$, for any $k\in\mathbb{N}$, we take the usual selections
$
\pi_k(f)\in\argmin_{f'\in\calF_k}\dist(f,f'),
$
and 
$
\Pi_k(g)\in\argmin_{g'\in\calG_k}\dist(g,g').
$
For convenience, we also define 
$\calP_k(f,g):=A(\pi_{k}(f),\Pi_{k}(g))$ and 
$\mcP_k(f,g):=\pi_{k}(f)\Pi_{k}(g)$. Our proof will rely on the chain: 
\begin{align}
A(f,g)-A(\pi_\ell(f),\Pi_\ell(g))
&=\sum_{k\in\calJ}\left[\calP_{k+1}(f,g)-\calP_{k}(f,g)\right]
+\sum_{k\in\calI}\left[\calP_{k}(f,g)-\calP_{k-1}(f,g)\right],\label{thm:product:process:eq:chain}
\end{align}
where we have used that $\cup_{k\ge0}\calF_k\times\calG_k$ is dense on $F\times G$. 

Fix $u\ge2$. Given any $k\in\mathbb{N}$, define the event 
$\Omega_{k,\calI,u,p}$ for which, for all $f\in F$ and $g\in G$, we have
\begin{align}
|\calP_{k}(f,g)-\calP_{k-1}(f,g)|\le 2(1+\sqrt{2})u2^{k/2}\frac{\Vert\mcP_{k}(f,g)-\mcP_{k-1}(f,g)]\Vert_{\psi_1}}{\sqrt{n}}.\label{thm:product:process:eq1}
\end{align}
Define also the event $\Omega_{k,\calJ,u,p}$ for which, for all $f\in F$ and $g\in G$, we have both inequalities:
\begin{align}
\Vert \pi_{k+1}(f)-\pi_{k}(f)\Vert_{n}&\le\sqrt{u}2^{k/2}\frac{[2(1+\sqrt{2})+1]^{1/2}}{\sqrt{n}}\Vert\pi_{k+1}(f)-\pi_{k}(f)\Vert_{\psi_2},\\
\Vert \Pi_{k+1}(g)-\Pi_{k}(g)\Vert_{n}&\le\sqrt{u}2^{k/2}\frac{[2(1+\sqrt{2})+1]^{1/2}}{\sqrt{n}}\Vert\Pi_{k+1}(g)-\Pi_{k}(g)\Vert_{\psi_2}.\label{thm:product:process:eq2}
\end{align}
By an union bound over all possible 4-tuples 
$(\pi_{k-1}(f),\pi_{k}(f),\Pi_{k-1}(g),\Pi_{k}(g))$ we have that $|\Omega_{k,\calI,u,p}|\le|\calF_{k-1}||\calF_{k}||\calG_{k}||\calG_{k-1}|\le2^{2^{k+2}}$. If $\Omega_{\calI,u,p}:=\cap_{k\in\calI}\Omega_{k,\calI,u,p}$, the first bound on Lemma \ref{lemma:increment:bounds} for $k\in\calI$ and Lemma \ref{lemma:union:bound:chaining:lazy} (using that $k>\ell$ over $\calI$) imply that there is universal constant $c>0$
$$
\prob(\Omega_{\calI,u,p}^c)\le ce^{-pu/4}.
$$
Similarly, the second bound in Lemma \ref{lemma:increment:bounds} for $k\in\calJ$ and Lemma \ref{lemma:union:bound:chaining:lazy} (using that $k>\ell$ over $\calJ$) imply that for the event  
$\Omega_{\calJ,u,p}:=\cap_{k\in\calJ}\Omega_{k,\calJ,u,p}$, we have 
$$
\prob(\Omega_{\calJ,u,p}^c)\le ce^{-pu/4}.
$$
We now also define the event $\Omega_{u,p}$ as the intersection of $\Omega_{\calI,u,p}\cap\Omega_{\calJ,u,p}$ and the events for which both inequalities of Proposition \ref{prop:quadratic:process} hold for both classes $F$ and $G$. Clearly, by such proposition and the two previous displays we have 
$\prob(\Omega_{u,p}^c)\le ce^{-pu/4}$ from an union bound.

We next fix $u\ge2$ and assume that $\Omega_{u,p}$ always holds. We now bound the chain over $\calI$ and $\calJ$ separately. 

\begin{description}

\item[\textbf{Part 1}: \emph{The subgaussian path $\calI$}.] 
From \eqref{thm:product:process:eq1} and the inequality
\begin{align}
\Vert\mcP_{k}(f,g)-\mcP_{k-1}(f,g)\Vert_{\psi_1}
&\le\Vert\pi_{k}(f)-\pi_{k-1}(f)\Vert_{\psi_2}\Vert\Pi_{k}(g)\Vert_{\psi_2} 
+\Vert\pi_{k-1}(f)\Vert_{\psi_2} \Vert\Pi_{k}(g)-\Pi_{k-1}(g)\Vert_{\psi_2}\\ 
&\le \bar\Delta(G)[\dist(f,\pi_{k}(f))+\dist(f,\pi_{k-1}(f))]
+\bar\Delta(F)[\dist(g,\Pi_{k}(g))+\dist(g,\Pi_{k-1}(g))]
\end{align}
implying
\begin{align}
\left|\sum_{k\in\calI}[\calP_{k}(f,g)-\calP_{k-1}(f,g)]\right|&\le 
2(1+\sqrt{2})^2\frac{u}{\sqrt{n}}
\left[
\bar\Delta(F)\gamma_{2,p}(G)+\bar\Delta(G)\gamma_{2,p}(F)\right].	\label{thm:product:process:subgaussian}
\end{align}

\item[\textbf{Part 2}: \emph{The subexponential path $\calJ$}.]  
Note that $A(f,g)=\hat\probn fg-\probn(\hat\probn fg)$ and thus, by Jensen's and triangle inequalities,
\begin{align}
\left|\sum_{k\in\calJ}[\calP_{k+1}(f,g)-\calP_{k}(f,g)]\right|&\le
\left|\sum_{k\in\calJ}\hat\probn[\mcP_{k+1}(f,g)-\mcP_{k}(f,g)]\right|
+\left|\sum_{k\in\calJ}\probn\hat\probn[\mcP_{k+1}(f,g)-\mcP_{k}(f,g)]\right|\\
&\le\sum_{k\in\calJ}\hat\probn\left|\mcP_{k+1}(f,g)-\mcP_{k}(f,g)\right|
+\sum_{k\in\calJ}\probn\hat\probn\left|\mcP_{k+1}(f,g)-\mcP_{k}(f,g)\right|.\label{thm:product:process:eq3}
\end{align}
Let us denote $\hat T_k:=\hat\probn\left|\mcP_{k+1}(f,g)-\mcP_{k}(f,g)\right|$. We have the split
\begin{align}
\hat T_k&\le\left|\hat\probn\pi_{k}(f)[\Pi_{k+1}(g)-\Pi_{k}(g)]\right|
+\left|\hat\probn\Pi_{k+1}(g)[\pi_{k+1}(f)-\pi_{k}(f)]\right|.
\end{align}
By Cauchy-Schwarz,
\begin{align}
\left|\hat\probn\pi_{k}(f)[\Pi_{k+1}(g)-\Pi_{k}(g)]\right|
&\le\Vert\pi_{k}(f)\Vert_n\Vert\Pi_{k+1}(g)-\Pi_{k}(g)]\Vert_n\\
&\le\left[\{A(\pi_k(f))\}^{1/2}+\{\probn\pi_k^2(f)\}^{1/2}\right]\Vert\Pi_{k+1}(g)-\Pi_{k}(g)]\Vert_n,
\end{align}
which together with \eqref{thm:product:process:eq2}, bounds in Proposition \ref{prop:quadratic:process} and $\sqrt{u}\ge1$ give
\begin{align}
\left|\hat\probn\pi_{k}(f)[\Pi_{k+1}(g)-\Pi_{k}(g)]\right|&\le \frac{c_2u2^{k/2}}{\sqrt{n}}
\left[c_1\bar\Delta(F)+25\frac{\gamma_{2,p}(F)}{\sqrt{n}}
+\left\{\frac{85\bar\Delta(F)\gamma_{2,p}(F)}{\sqrt{n}}\right\}^{1/2}\right]\dist(\Pi_{k+1}(g),\Pi_{k}(g))\\
&\le c_2\frac{u2^{k/2}}{\sqrt{n}}\dist(\Pi_{k+1}(g),\Pi_{k}(g))
\left[c_3\bar\Delta(F)+c_4\frac{\gamma_{2,p}(F)}{\sqrt{n}}\right],
\end{align}
by Young's inequality and constants $c_1:=\{4(1+\sqrt{2}+2)\}^{1/2}+1$, $c_2:=\{2(1+\sqrt{2}+1)\}^{1/2}$, $c_3:=c_1+\frac{\sqrt{85}}{2}$ and $c_4:=25+\frac{\sqrt{85}}{2}$.
An identical bound gives 
\begin{align}
\left|\hat\probn\Pi_{k+1}(g)[\pi_{k+1}(f)-\pi_{k}(f)]\right|
&\le c_2\frac{u2^{k/2}}{\sqrt{n}}\dist(\pi_{k+1}(f),\pi_{k}(f))
\left[c_3\bar\Delta(G)+c_4\frac{\gamma_{2,p}(G)}{\sqrt{n}}\right].
\end{align}
We thus conclude that 
\begin{align}
\hat T_k&\le  c_2\frac{u2^{k/2}}{\sqrt{n}}\dist(\Pi_{k+1}(g),\Pi_{k}(g))
\left[c_3\bar\Delta(F)+c_4\frac{\gamma_{2,p}(F)}{\sqrt{n}}\right]\\
&+c_2\frac{u2^{k/2}}{\sqrt{n}}\dist(\pi_{k+1}(f),\pi_{k}(f))
\left[c_3\bar\Delta(G)+c_4\frac{\gamma_{2,p}(G)}{\sqrt{n}}\right].
\end{align}
Note that, in fact, we have proved that the above bound on $\hat T_k$ holds with probability at least $1-c\exp(-pu/4)$ for any $u\ge2$. Thus, from Lemma \ref{lemma:dirksen:A5} we have, for some universal constant $c_0>0$,  
\begin{align}
\probn\hat T_k&\le  c_0c_2\frac{u2^{k/2}}{\sqrt{n}}\dist(\Pi_{k+1}(g),\Pi_{k}(g))
\left[c_3\bar\Delta(F)+c_4\frac{\gamma_{2,p}(F)}{\sqrt{n}}\right]\\
&+c_0c_2\frac{u2^{k/2}}{\sqrt{n}}\dist(\pi_{k+1}(f),\pi_{k}(f))
\left[c_3\bar\Delta(G)+c_4\frac{\gamma_{2,p}(G)}{\sqrt{n}}\right].
\end{align}
Using the previous two bounds in \eqref{thm:product:process:eq3} gives, after using the triangle inequality for $\dist$, summing over $k\in\calJ$ and using the definition of $\gamma_{2,p}(F)$ and $\gamma_{2,p}(G)$ (recalling that $k>\ell$), 
\begin{align}
\left|\sum_{k\in\calJ}[\calP_{k+1}(f,g)-\calP_{k}(f,g)]\right|&\le (1+c_0)(1+2^{-1/2})c_2\frac{u}{\sqrt{n}}\gamma_{2,p}(G)
\left[c_3\bar\Delta(F)+c_4\frac{\gamma_{2,p}(F)}{\sqrt{n}}\right]\\
&+(1+c_0)(1+2^{-1/2})c_2\frac{u}{\sqrt{n}}\gamma_{2,p}(F)
\left[c_3\bar\Delta(G)+c_4\frac{\gamma_{2,p}(G)}{\sqrt{n}}\right].
\end{align}
\end{description}

From the above bound, \eqref{thm:product:process:subgaussian} and \eqref{thm:product:process:eq:chain} we conclude that, for any $u\ge2$, on the event $\Omega_{u,p}$ of probability at least $1-e^{-pu/4}$, we have 
\begin{align}
&\sup_{(f,g)\in F\times G}|A(f,g)|^{1/2}-\sup_{(f,g)\in F\times G}|A(\pi_\ell(f),\Pi_\ell(g))|^{1/2}\\
&\le \sqrt{u}\left[\frac{c_5}{\sqrt{n}}\left(\bar\Delta(F)\gamma_{2,p}(G)+\bar\Delta(G)\gamma_{2,p}(F)\right)+c_6\frac{\gamma_{2,p}(F)\gamma_{2,p}(G)}{n}\right]^{1/2},
\end{align}
with $c_5:=2(1+\sqrt{2}^2+(1+c_0)(1+2^{-1/2}))c_2c_3$ and
$c_6:=2(1+c_0)c_2c_4(1+2^{-1/2})$. This and Lemma \ref{lemma:dirksen:A5} (with $\alpha=2$) imply that 
\begin{align}
&\Lpnorm{\sup_{(f,g)\in F\times G}|A(f,g)|^{1/2}-\sup_{(f,g)\in F\times G}|A(\pi_\ell(f),\Pi_\ell(g))|^{1/2}}\\
&\le c\left[\frac{1}{\sqrt{n}}\left(\bar\Delta(F)\gamma_{2,p}(G)+\bar\Delta(G)\gamma_{2,p}(F)\right)+\frac{\gamma_{2,p}(F)\gamma_{2,p}(G)}{n}\right]^{1/2}.
\end{align}
We also have from Lemma \ref{lemma:dirksen:A3},
\begin{align}
\left(\esp\sup_{(f,g)\in F\times G}|A(\pi_\ell(f),\Pi_\ell(g))|^{p/2}\right)^{2/p}&\le4\sup_{(f,g)\in F\times G} \left(\esp|A(\pi_\ell(f),\Pi_\ell(g))|^{p/2}\right)^{2/p}\\
&\le c\sup_{(f,g)\in F\times G}\left[\Vert \mcP_\ell(f,g)-\probn\mcP_\ell(f,g)\Vert_{\psi_1}\left(\sqrt{\frac{p}{n}}+\frac{p}{n}\right)\right],%\\
%&\le2c\sup_{(f,g)\in F\times G}
%\Vert\pi_\ell(f)\Pi_\ell(g)\Vert_{\psi_1}%\left(\sqrt{\frac{p}{n}}+\frac{p}{n}
%\right),
\end{align}
where the second inequality follows from Bernstein's inequality for $A(\pi_\ell(f),\Pi_\ell(g))$ and Lemma \ref{lemma:moments:p}. The two previous displays finish the proof. 
\end{proof}

\end{appendix}

%----------------------------------
% BIBLIOGRAPHY
%----------------------------------
\bibliographystyle{plain}

\end{document}